%% file: mmp_root.tex
\documentclass[12pt]{amsart}
\usepackage[active]{srcltx}
\usepackage{calc,amssymb,amsthm,amsmath,amscd, eucal,ulem,stmaryrd}
\usepackage{alltt}
\usepackage[left=1.35in,top=1.25in,right=1.35in,bottom=1.25in]{geometry}
\usepackage{mathtools}
\usepackage{soul}
\usepackage{hyperref}

\makeatletter
\def\@tocline#1#2#3#4#5#6#7{\relax
  \ifnum #1>\c@tocdepth   \else
    \par \addpenalty\@secpenalty\addvspace{#2}    \begingroup \hyphenpenalty\@M
    \@ifempty{#4}{      \@tempdima\csname r@tocindent\number#1\endcsname\relax
    }{      \@tempdima#4\relax
    }    \parindent\z@ \leftskip#3\relax \advance\leftskip\@tempdima\relax
    \rightskip\@pnumwidth plus4em \parfillskip-\@pnumwidth
    #5\leavevmode\hskip-\@tempdima
      \ifcase #1
       \or\or \hskip 1em \or \hskip 2em \else \hskip 3em \fi      #6\nobreak\relax
    \hfill\hbox to\@pnumwidth{\@tocpagenum{#7}}\par    \nobreak
    \endgroup
  \fi}
\makeatother

\RequirePackage[dvipsnames,usenames]{xcolor}

\input{kmacros3.sty}
\input{xy}
\xyoption{all}
\usepackage{tikz}
\usepackage{tikz-cd}
\usetikzlibrary{decorations.markings}
\usetikzlibrary{cd}
\usepackage{mabliautoref}
\usepackage{bm}

            \DeclareRobustCommand{\bigplus}{\pmb{+}}

\usepackage{pifont}

\DeclareSymbolFont{rsfs}{OMS}{rsfs}{m}{n}
\DeclareSymbolFontAlphabet{\scr}{rsfs}

\DeclareMathOperator{\alt}{alt}

\usepackage{relsize}
\usepackage[bbgreekl]{mathbbol}
\usepackage{amsfonts}
\DeclareSymbolFontAlphabet{\mathbb}{AMSb} \DeclareSymbolFontAlphabet{\mathbbl}{bbold}
\newcommand{\Prism}{{\mathlarger{\mathbbl{\Delta}}}}

\renewcommand{\m}{\mathfrak{m}}

\renewcommand{\fram}{\mathfrak{m}}
\numberwithin{equation}{theorem}

\usepackage{fullpage}

\usepackage{setspace}

\usepackage{enumerate}
\usepackage{enumitem}
\usepackage{graphicx}
\usepackage[all,cmtip]{xy}

\usepackage{upgreek}
\newcommand{\mytau}{{\uptau}}
\usepackage{bbding}
\usepackage{verbatim}

\renewcommand{\O}{\mathcal O}

\begin{document}

\title{Globally $\bigplus$-regular varieties and the minimal model program for threefolds in mixed characteristic}
\author{Bhargav Bhatt, Linquan Ma, Zsolt Patakfalvi, Karl Schwede, Kevin Tucker, Joe Waldron, Jakub Witaszek}
\address{Department of Mathematics, University of Michigan, Ann Arbor, MI 48109, USA}
\email{bhattb@umich.edu}
\address{Department of Mathematics, Purdue University, West Lafayette, IN 47907, USA}
\email{ma326@purdue.edu}
\address{\'Ecole Polytechnique F\'ed\'erale de Lausanne (EPFL), MA C3 635, Station 8, 1015 Lausanne, Switzerland}
\email{zsolt.patakfalvi@epfl.ch}
\address{Department of Mathematics, University of Utah, Salt Lake City, UT 84112, USA}
\email{schwede@math.utah.edu}
\address{Department of Mathematics, University of Illinois at Chicago, Chicago, IL 60607, USA}
\email{kftucker@uic.edu}
\address{Department of Mathematics, Michigan State University, East Lansing, MI 48824, USA}
\email{waldro51@msu.edu}
\address{Department of Mathematics, Princeton University, Fine Hall, Washington Road, Princeton NJ 08544, USA
}
\email{jwitaszek@princeton.edu}

\maketitle

\begin{abstract}
We establish the Minimal Model Program for arithmetic threefolds whose residue characteristics are greater than five.  In doing this, we generalize the theory of global $F$-regularity to mixed characteristic and identify certain stable sections of adjoint line bundles.  Finally, by passing to graded rings, we generalize a special case of Fujita's conjecture to mixed characteristic.
\end{abstract}
\setcounter{tocdepth}{2}
\tableofcontents

\input{introduction.tex}

\input{notation.tex}

\input{vanishing.tex}

\input{B0Definitions}

\input{B_0_and_graded_rings.tex}

\input{GloballyRelativelyBCMRegular}

\input{lifting.tex}

\input{flips.tex}

\input{mmp.tex}

\input{applications.tex}

\bibliographystyle{skalpha}
\bibliography{MainBib}
\end{document}

%% file: introduction.tex
\section{Introduction}
The Kodaira and Kawamata-Viehweg vanishing theorems are among the most important tools used in algebraic geometry in characteristic zero and are a key component of the minimal model program \cite{BirkarCasciniHaconMcKernan}.  They are crucial to understanding linear systems as they allow the lifting of global sections of line bundles from lower dimensional subvarieties.  Unfortunately, these vanishing theorems are false in general when working over fields of positive characteristic (such as $\mathbb{F}_p$, \cite{raynaud_contre-exemple_1978}) or mixed characteristic rings (such as $\bZ$ or $\bZ_p$\footnote{Burt Totaro \cite{BurtTotaroPrivateCommunication} has pointed out to us that examples of failure of relative Kawamata-Viehweg vanishing in mixed characteristic can be obtained via methods similar to those in \cite{TotaroFailureOfKodairaVanishingForFanoVarieties}.}).

In characteristic $p > 0$, the Frobenius morphism and asymptotic Serre vanishing can be used as a replacement in some contexts.  An important class of such applications of Frobenius goes back to the development of tight closure theory and the notions of $F$-split and $F$-regular varieties \cite{HochsterHunekeTC1,MehtaRamanathanFrobeniusSplittingAndCohomologyVanishing,RamananRamanathanProjectiveNormality}. The discovery of connections between these notions and birational geometry led to a plethora of applications, for instance: 
 \cite{SmithFRatImpliesRat,MehtaSrinivasRatImpliesFRat,HaraRatImpliesFRat,HaraWatanabeFRegFPure,HaraYoshidaGeneralizationOfTightClosure, TakagiInterpretationOfMultiplierIdeals,SchwedeTuckerZhangAlterations, TakagiInversion,PatakfalviSemipositivity, ZhangYuchen-PluricanonicalMapsOfMaximalAlbenese, MustataSchwedeSeshadri, CasciniHaconMustataSchwedeNumDimPseudoEffective, HaconSingularitiesOfThetaDivisors,  BlickleSchwedeTuckerTestAlterations,das_different_different_different, gongyo_rational_2015,  CTW17,carvajal-rojas_fundamental_2016,  HaconPatakfalviGenericVanishingInCharacteristicPAndTheCharacterization, hacon_birational_2017,  aberbach_polstra, EjiriWhenIsTheAlbaneseMorphismAnAlgebraicFiberSpace,BernasconiBPFInLargeCharacteristic}.
In particular, building on \cite{KeelBasepointFreenessForNefAndBig} and \cite{SchwedeACanonicalLinearSystem}, Hacon and Xu proved the existence of minimal models for positive characteristic terminal threefolds over algebraically closed fields of characteristic $p>5$ \cite{HaconXuThreeDimensionalMinimalModel}; this was then extended in various directions  \cite{CTX15, Birkar16, Xu15Bpf, BW17, waldron2017lmmp, hashizume2019minimal, GNT06, DW19, XuLeiNonvanishing, HW19a, HaconWitaszekMMPp=5,HaconWitaszekMMP4fold}.

In the mixed characteristic setting, the theory of perfectoid algebras and spaces \cite{ScholzePerfectoidspaces} has led to spectacular advancements, including proofs of the direct summand conjecture and the existence of big Cohen-Macaulay algebras by Andr\'e  \cite{AndreDirectsummandconjecture}
 (see also \cite{BhattDirectsummandandDerivedvariant}).  Building on these techniques, the second and fourth authors developed a mixed characteristic analog of $F$-regularity called BCM-regularity in \cite{MaSchwedePerfectoidTestideal,MaSchwedeSingularitiesMixedCharBCM}, and, together with the fifth, sixth, and the seventh author, its adjoint {(plt)} variant (see \cite{MaSchwedeTuckerWaldronWitaszekAdjoint}).  In particular, it was shown that klt surface singularities of mixed characteristic $(0,p>5)$ are BCM-regular {and that inversion of adjunction holds for three-dimensional plt singularities}; the positive characteristic {analogs} of these results were key initial ingredients for the aforementioned work of Hacon and Xu.

What is missing is a mixed characteristic analog of the theory of global $F$-regularity
,
a strengthening of the log Fano condition
which was introduced in positive characteristic in \cite{SmithGloballyFRegular} (see also \cite{SchwedeSmithLogFanoVsGloballyFRegular}).  We establish such a theory, which we call \emph{globally $\bigplus$-regularity}, based upon the recent work of the first author, \cite{BhattAbsoluteIntegralClosure}, who showed that the absolute integral closure $R^+$ of an excellent domain $R$ is Cohen-Macaulay in mixed characteristic and deduced a variant of Kodaira vanishing up to finite covers.  Like in positive characteristic one may also define globally $\bigplus$-regularity by the study of BCM-regularity of section rings (normalizations of cones); in fact, this point of view will be important in proofs of some of our results.  Note that, while globally $\bigplus$-regular varieties (and pairs) could also reasonably be called \emph{global splinters}, our syntax  more closely matches existing terminology for global $F$-regularity.

As our main application we develop the mixed characteristic Minimal Model Program for threefolds when the residual characteristics are zero or  bigger than 5.
\begin{theoremA*} 
Let $R$ be a finite-dimensional excellent domain with a dualizing complex and containing $\bZ$ whose closed points have residual characteristics zero or greater than $5$.  Let $X$ be a klt integral scheme of dimension three which is projective  and surjective over $\Spec(R)$. Then we can run a Minimal Model Program on $X$ over $\Spec(R)$ which terminates with a minimal model or a Mori fiber space. 
\end{theoremA*}

\noindent In fact, our results are much stronger (see \autoref{ss:MMP-Intro} for more details).  They extend earlier results on the mixed characteristic case including {H.~Tanaka's work on the MMP for excellent surfaces (\cite{tanaka_mmp_excellent_surfaces}) and the work of} Y.~Kawamata on the MMP for mixed characteristic semistable threefolds \cite{KawamataMixedThreefolds}. Other related work appears in \cite[Theorem 4.1]{LipmanRationalSingularities}, \cite{LichtenbaumCurvesOverDVRS} and \cite{ShafarevichLecturesOnMinimalModels}.
We also point out that some variants of this theorem were obtained independently by Takamatsu and Yoshikawa in \cite{TakamatsuYoshikawaMMP} (see \autoref{remark:TakamatsuYoshikawa} for additional discussion). 

From now on, $(R,\fram)$ is a Noetherian complete local domain of mixed characteristic $(0,p>0)$ (although what follows also works when $R$ is of characteristic $p > 0$). For simplicity, in the introduction, we present our initial results in the non-boundary-case ($\Delta = 0$) and append references to full statements.

First, we discuss the analog of global $F$-regularity.  We say that a normal integral scheme $X$ proper over $R$ is \emph{globally $\bigplus$-regular} if $\sO_X \to f_* \sO_Y$ splits for every finite cover $f \colon Y \to X$, and observe the following as a straightforward consequence of generalizations and reformulations (see \autoref{sec.Vanishing}) of the vanishing theorems of \cite{BhattAbsoluteIntegralClosure}.
\begin{theoremB*}[\autoref{cor.RelativeKVVanishingForG+Regular}] \label{thmB} Suppose that $X$ is globally $\bigplus$-regular and proper over $\Spec(R)$.  If $\sL$ is a big and semiample line bundle on  $X$,  then Kawamata-Viehweg vanishing holds for $\sL$, that is $H^i(X, \omega_X \otimes \sL) = 0$ for $i>0$.
\end{theoremB*}
\noindent In positive characteristic, global $F$-regularity implies global $\bigplus$-regularity (\autoref{lemma:BregularINpositiveCharacteristic}), but the converse is an open problem even in the affine setting.

In fact, the previous result is an direct consequence of the following generalization of the vanishing theorem of Bhatt to more arbitrary excellent local bases \cite{BhattAbsoluteIntegralClosure}.  Indeed, this vanishing theorem will be used multiple times in key ways in this paper.

\begin{theoremC*}[\autoref{cor.VanishingWithoutRestrictingToPFiber}]
    Suppose that $(T,x)$ is an excellent local ring  of residue characteristic $p>0$. Let $\pi : X \to \Spec(T)$ be a proper map with $X$ integral. Suppose that $L \in \Pic(X)$ is a big and semiample line bundle. Then for all $b<0$ and all $i < \dim(X)$, we have that $H^{i}(\myR \Gamma_x(\myR\Gamma(X^+, L^b)))=0$.
\end{theoremC*}

Another key notion used in applications in positive characteristic birational geometry is that of Frobenius stable sections $S^0(X,\sM) \subseteq H^0(X,\sM)$, for a line bundle $\sM$, and its variant $T^0(X,\sM)$, introduced in \cite{SchwedeACanonicalLinearSystem} and \cite{BlickleSchwedeTuckerTestAlterations} respectively. These sections behave as if Kodaira vanishing was valid for them. In this article, we consider the following mixed characteristic analog thereof (see \autoref{def:B_0}):
\[
\myB^0(X, \sM):= \bigcap_{\substack{f \colon Y \to X\\ \textnormal{finite}}}\im \left( H^0(Y, \sO_Y( K_{Y/X} + {f^*M})) \to H^0(X, \sM) \right).
\]
We call these global sections \emph{$\bigplus$-stable}.
We also consider an adjoint (plt-like) version $\myB^0_S(X,S; \sM)$ for an irreducible divisor $S$ and a line bundle $\sM = \sO_X(M)$. 

\begin{theoremD*}The following holds for a normal integral scheme $X$ proper over $\Spec (R)$:
\begin{enumerate}
    \item Under appropriate assumptions,
		\[
		\myB^0_S(X,S; \sO_X(K_X+S+A)) \to \myB^0(S; \sO_S(K_S+A|_S))
		\]
		is surjective, where $A$ is an ample Cartier divisor,  see \autoref{thm:main-lifting}.
    \item $X$ is globally $\bigplus$-regular if and only if\ $\myB^0(X, \sO_X) = H^0(X, \sO_X)$, see \autoref{prop.GlobalBRegularSplits}.  
    \item If $X$ is $\bQ$-Gorenstein, then
		\[
		\myB^0(X, \sM):= \bigcap_{\substack{f \colon Y \to X\\ \textnormal{alteration}}}\im \left( H^0(Y, \sO_Y( K_{Y/X} + {f^*M})) \to H^0(X, \sM) \right),
		\]
		see 
        \autoref{cor.B0VsB0Alt}.  \label{itm.B0DescribedByAlterations}
	\item If $X = \Spec R$ is $\bQ$-Gorenstein and affine, then $\myB^0(X,\sO_X) = \mytau_{R^+}(R)$, where the latter term is the BCM-test ideal defined in \cite{MaSchwedeSingularitiesMixedCharBCM}, see
        \autoref{prop:B^0-agrees-with-test-ideal-for-affines}.	
	\item $\myB^0$ transforms as expected under finite maps and alterations, see \autoref{subsec.TransformationsOfB0UnderAlterations}.  	\item If $\sL$ is an ample line bundle on $X$, and $S = \bigoplus_{i \geq 0} H^0(X, \sL^i)$ is the section ring, then for $i > 0$ we have that $\myB^0(X, \omega_X \otimes \sL^i)$ is the $i$th graded piece of a test submodule $\mytau_{R^{+,\gr}}(\omega_S)$ on $S$, see \autoref{prop.B^0forgradedringsvsB^0}.  
	\item If $X$ is projective over $\Spec R$, is regular (or has sufficiently mild singularities) and $\sL$ is ample, then for $n \gg 0$
		\[
			\myB^0(X, \omega_X \otimes \sL^n) = H^0(X, \omega_X \otimes \sL^n),
		\]
		see \autoref{thm:B0-equals-H0-for-high-ample}.  \label{itm.MyB0Agrees}
\end{enumerate}
\end{theoremD*}
\noindent The proofs of the above results are based on \cite{BhattAbsoluteIntegralClosure,BhattLuriepadicRHmodp} as well as ideas developed in \cite{SchwedeACanonicalLinearSystem,BlickleSchwedeTuckerTestAlterations,MaSchwedeSingularitiesMixedCharBCM,MaSchwedeTuckerWaldronWitaszekAdjoint}. We should note that \autoref{itm.B0DescribedByAlterations} shows that $\myB^0$ agrees with the notion of $T^0$ introduced in \cite{BlickleSchwedeTuckerTestAlterations} for {$\bQ$-Gorenstein} varieties in characteristic $p > 0$ and defined and used in similar ways in mixed characteristic in the independent work \cite{TakamatsuYoshikawaMMP} mentioned above.
\begin{theoremE*}[\autoref{thm:Karen_Fujita_v1}]  Let $X$ be a $d$-dimensional scheme that is regular (or has sufficiently mild singularities) and which is flat and projective over $R$.  Set $t = \dim R$ and let $\sL$ be an ample globally generated line bundle on $X$.  Then $\omega_X \otimes \sL^{d - t + 1}$ is globally generated by $\myB^0(X, \omega_X \otimes \sL^{d-t+1})$.
\end{theoremE*}

We also note that we obtain related global generation results for adjoint line bundles $\omega_X \otimes \sL$ via Seshadri constants, see \autoref{theorem:seshadri}.

One should expect that this variant of the Fujita conjecture would hold for any $X$ admitting BCM-rational singularities (in the sense of \cite{MaSchwedeSingularitiesMixedCharBCM}), as in \cite{SmithFujitaFreenessForVeryAmple,KeelerFujita}.   Indeed, our argument would show this if we knew that the formation of our test ideals commuted with localization in a sufficiently strong sense.  Indeed, a limited localization result from \cite{MaSchwedeTuckerWaldronWitaszekAdjoint} was how we proved \autoref{itm.MyB0Agrees} above, which was used in our proof of this theorem.
The question of whether BCM-test ideals localize in general is one of the key remaining open problems about BCM-singularities.  In forthcoming work, we shall prove that localization holds in certain circumstances and derive geometric consequences.

We warn the reader however that, in general, the localization is false for $\myB^0(X,\sM)$ when $X$ is projective: 
\begin{theoremF*}[{\autoref{ex:elliptic_curve}}]
	Let $E$ be a smooth elliptic curve over $\bZ_p$. 
		 Then
	\begin{enumerate}
	\item $\myB^0(E, \sO_E) = 0$, but 
	\item $\myB^0(E_{\bQ_p}, \sO_{E_{\bQ_p}}) = \bQ_p$. 
	\end{enumerate}
\end{theoremF*}
\noindent This also shows that in contrast to positive characteristic, $\myB^0(X,\sM)$ cannot be calculated on a single finite cover (or an alteration).

Our definition of $\myB^0$ works most naturally when the base ring is complete. However, certain partial results on lifting sections can be obtained when the base is not complete, see \autoref{cor:lifting_from_BCM-regular}.
Since most geometric results can be deduced from the complete case, we shall always assume, when talking about $\myB^0$, that the base is complete. In particular, our setup allows for running the Minimal Model Program over {algebraic and analytic singularities}. Since many results of \cite{BhattAbsoluteIntegralClosure} assume that the base is finitely presented over a DVR, we provide generalizations thereof in \autoref{sec.Vanishing}.

\subsection{Minimal Model Program} \label{ss:MMP-Intro}
In this subsection, $R$ is an excellent domain of finite Krull dimension {admitting a dualizing complex. In most theorems, we will also} assume that the closed points of $R$ have residual characteristics {zero or} greater than $5$ (the cases $R=\mathbf{Z}[1/30]$ or $R=\mathbf{Z}_p$ for $p > 5$ are already interesting). Let $T$ be a quasi-projective scheme over $R$.

\begin{theoremG*}[MMP, \autoref{prop:psef_termination}, \autoref{thm:MFS}] Let $(X,B)$ be a three-dimensional $\bQ$-factorial dlt pair with $\mathbb{R}$-boundary, which is projective over $T$.  {Assume that {the image of $X$ in $T$ is of positive dimension and that} $T$ has no residue fields of characteristic $2,3$ or $5$}.  

	If $K_X+B$ is pseudo-effective, then {we can run a $(K_X+B)$-MMP and any sequence of steps of this MMP terminates with a log minimal model.} 

	If $K_X+B$ is not pseudo-effective, then we can run a $(K_X+B)$-MMP with scaling over $T$ which terminates with a Mori fiber space.
\end{theoremG*}

Note that the assumption on {the image of $X$ in $T$}  is needed {because we do not know that all flips terminate in purely positive characteristic. {In fact, even the MMP with scaling is not known to terminate when the base field is imperfect, however, log minimal models exist in this case by \cite{DW19}}.  Also, we  do not know the existence of Mori fibre spaces} when $T=\Spec(k)$ for an imperfect field $k$. Indeed,   we do not know the validity of the Borisov-Alexeev-Borisov conjecture in this setting, the version of which over an algebraically closed field was used in \cite{BW17}.

\begin{theoremH*}[{Base-point-free theorem, \autoref{thm:bpf}, \autoref{thm:MMP_bpf}}] Let $(X,B)$ be a three-dimensional $\bQ$-factorial klt pair with $\mathbb{R}$-boundary admitting a projective morphism $f \colon X \to T$.
 Let $L$ be an $f$-nef $\bQ$-Cartier divisor on $X$ such that $L-(K_X+B)$ is $f$-big and nef. Suppose that
\begin{enumerate}
    \item $L$ is $f$-big, or
    \item the image of $X$ in $T$ is positive dimensional and {$T$ has no residue fields of characteristic $2,3$ or $5$}.
\end{enumerate}
Then, $L$ is $f$-semiample.
\end{theoremH*}

Note that a similar result for $\mathbb{R}$-divisors is proved in \autoref{thm:bpf_for_R_boundary}.

\begin{theoremI*}[{Cone theorem, \autoref{thm:full-cone-theorem}}]
Let $(X,\Delta)$ be a three-dimensional $\mathbb{Q}$-factorial dlt pair with $\mathbb{R}$-boundary, projective over $T$ {having no residue fields of characteristic $2,3$ or $5$} {and such that the image of $X$ in $T$ is of positive dimension}.  Then there exists a countable collection of curves\footnote{curves in this article are assumed to be projective over the base, see the definition in \autoref{sec:preliminaries_LMMP}} $\{C_i\}$ over $T$ such that 
\begin{enumerate}
    \item 
\[
\overline{\mathrm{NE}}(X/T) = \overline{\mathrm{NE}}(X/T)_{K_X+\Delta\geq 0} + \sum_i \bR_{\geq 0}[C_i],
\]
\item The rays $[C_i]$ do not accumulate in the half space $(K_X+\Delta)_{<0}$, and

        \item For all $i$, there is a positive integer $d_{C_i}$ such that $$0<-(K_X+\Delta)\cdot_k C\leq 4 d_{C_i}$$ and if $L$ is any Cartier divisor on $X$, then $L\cdot_k C_i$ is divisible by $d_{C_i}$, where $k$ is the residue field of the closed point of $T$ lying under $C$.
\end{enumerate}
\end{theoremI*}
\noindent Note that we cannot expect the bounds on extremal rays to be as in characteristic zero, since the residue fields of $T$ might not be {algebraically closed (cf.\ \cite[Example 7.3]{tanaka_behaviour} and} \cite{tanaka_mmp_excellent_surfaces,DW19}).

Besides the above constructions and results on $\myB^0$,
the proofs of the above results are based on the recent generalization of Keel's theorem on the semiampleness of line bundles to mixed characteristic (see \cite{Witaszek2020KeelsTheorem}), the MMP for mixed characteristic surfaces (see \cite{tanaka_mmp_excellent_surfaces}), and all the previous work on the positive characteristic MMP (most notably: \cite{HaconXuThreeDimensionalMinimalModel} for the existence of pl-flips with standard coefficients, \cite{Birkar16} for the existence of pl-flips with arbitrary coefficients, \cite{BW17} for the termination of the MMP with scaling and the existence of Mori fiber spaces, and \cite{DW19} for the generalization of the cone and contraction theorems to non-perfect residue fields).

Our proof of the fact that pl-flips, with standard coefficients, exist follows the strategy of \cite{HaconXuThreeDimensionalMinimalModel}, see \autoref{Section:flips}. Although we employ all key ideas of their work, we are able to simplify each step. Further, we provide a new proof of the base point free theorem for nef and big line bundles; we infer it from the mixed characteristic Keel's theorem by employing the recent work of Koll\'ar, \cite{Kollar2020RelativeMMPWithoutQfactoriality}, on the  non-$\bQ$-factorial MMP, and the ideas of \cite{HW19a}.  In fact, our proof yields the validity of the base point free theorem for big and nef line bundles for threefolds in any positive characteristic $p>0$, a result which was not known before.

{The termination of all flips when the image of $X$ in $T$ has positive dimension and when $K_X+\Delta$ is pseudo-effective, is proven by the argument of Alexeev-Hacon-Kawamata, see \cite{AHK07}.  Our proof of the base point free theorem for non-big line bundles uses this together with abundance in lower dimensions to provide substantial simplifications over the argument from \cite{BW17}. }{Furthermore, our more general set-up also requires a different proof of the cone theorem.  These are used to deduce termination with scaling and the existence of Mori fiber spaces following \cite{BW17}.}

\begin{remark} \label{remark:TakamatsuYoshikawa}
    While finalizing our project, we were contacted by Teppei Takamatsu and Shou Yoshikawa, who informed us that they were working on related topics (see \cite{TakamatsuYoshikawaMMP}). In their article, among many other things, they show the validity of some special cases of the three-dimensional MMP in all (mixed) characteristics $p \geq 0$: for semistable threefolds (generalizing \cite{KawamataMixedThreefolds}) and for resolutions of singularities. Aside from \cite{BhattAbsoluteIntegralClosure} and   \cite{KawamataMixedThreefolds}, their work builds upon ideas from the proof of the existence of some flips  discovered recently in \cite{HaconWitaszekMMP4fold} and on the results of \cite{HW19a}.  They also define and study the notion of global $T$-regularity which is very closely related to our global $\bigplus$-regularity, and obtain results on lifting sections.
            \end{remark}

\subsection{Applications to moduli theory}

We have the following sample corollaries to the moduli theory of surfaces. We recall that stable surfaces are the two dimensional generalizations of stable curves. In particular, they are supposed to provide a good compactification of the moduli space of smooth canonically polarized surfaces. The present article concludes the last step needed to show that their moduli stack exists over $\bZ[1/30]$ (see \cite{PatakfalviProjectivityStableSrufaces} for a historical overview of the subject).

\begin{theoremJ*}
(Existence of $\overline{\sM}_{2,v}$ over $\bZ[1/30]$, \autoref{cor:moduli_exists})
\begin{enumerate}
    \item 
 The moduli stack $\overline{\sM}_{2,v}$ of stable surfaces of volume $v$ over $\bZ[1/30]$ exists as a separated Artin stack of finite type over $\bZ[1/30]$ with finite diagonal.
\item 
 The  coarse moduli space $\overline{\mathrm{M}}_{2,v}$ of stable surfaces of volume $v$  over $\bZ[1/30]$  exists as a separated algebraic space of finite type over $\bZ[1/30]$.
\end{enumerate}
\end{theoremJ*}

Unfortunately at this point it is not known whether $\overline{\sM}_{2,v}$ is proper, and $\overline{\mathrm{M}}_{2,v}$ is projective over $\bZ[1/30]$. The best we can say is the following. 

\begin{theoremK*}
(\autoref{thm:closure_moduli_proper})
{Fix an integer $v > 0$} and let 
\begin{equation*}
    d=\prod_{p \textrm{ prime, } p \leq f(v)} p,  
        \qquad \textrm{where } \qquad
        f(v)= \left\{
    \begin{array}{ll}
    373 & \textrm{if $v=1$} \\[10pt]
    213 v + 48 \qquad & \textrm{if $v \geq 2$.}
    \end{array}
    \right.
\end{equation*}
Then, the closure $\overline{\sM}_{2,v}^{\sm}$ of the locus of smooth surfaces  in $\overline{\sM}_{2,v}$ is proper over $\bZ[ 1 / d ]$. Additionally, it admits a projective coarse moduli space $\overline{\mathrm{M}}_{2,v}^{\sm}$ over $\bZ[ 1 / d ]$.
\end{theoremK*}

These results are shown in \autoref{sec:applications}.

\subsection{Applications to commutative algebra}

We highlight one more standard corollary of the minimal model program which we expect to be useful in commutative algebra.  It follows from the above results as in \cite[Exercises 108, 109]{KollarExercisesInBiratGeom}.
\begin{corollaryL*}
	Suppose $(X = \Spec R, \Delta)$ is a three-dimensional klt pair where $R$ is an excellent local domain of 
		residue characteristic $p$ for $p>5$.  Then for every Weil divisor $D$ on $X$ we have that the local section ring $\bigoplus_{i \geq 0} \sO_X(iD)$ is finitely generated.  In other words, if $I$ is an ideal of pure height one in $R$, then the symbolic power algebra
	\[
		R \oplus I \oplus I^{(2)} \oplus I^{(3)} \oplus \dots
	\]
	is finitely generated.
\end{corollaryL*}
    
This result in characteristic $p$ has applications to tight closure theory. In fact, combining the above Corollary with \cite[Theorem B]{aberbach_polstra} yields a generalization of \cite[Theorem A]{aberbach_polstra} from rings essentially of finite type over a field to the case of excellent local rings.

\begin{corollaryM*}
Let $(R,\m)$ be a four-dimensional excellent local ring  of equal characteristic $p>5$. If $R$ is $F$-regular then $R$ is strongly $F$-regular. 
\end{corollaryM*}

{\subsection{Applications to four-dimensional Minimal Model Program and liftability}
In \cite{HaconWitaszekMMP4fold}, it is shown that a variant of the four-dimensional semistable Minimal Model Program over curves and over singularities is valid in positive characteristic $p>5$ contingent upon the existence of resolutions of singularities. Using the techniques and results of our paper as well as the generalisation of the result of Cascini and Tanaka on relative semiampleness (now proven in \cite{Witaszek21}), this semistable MMP may be extended to mixed characteristic. In turn, this may be used to show that liftability of three-dimensional varieties of characteristic $p>5$ is stable under the Minimal Model Program. These results {are now} contained in an update to \cite{HaconWitaszekMMP4fold}. }   

\subsection{Technical notes} We summarize here the major technical points of the article.
\begin{enumerate}
\item Most of the theory developed in the article assumes we are working over a  complete local  base.  This lets us show, in \autoref{lem.B0AsInverseLimit}, that elements of $\myB^0(X, \cO_X(M))$ have \emph{compatible} systems of pre-images in $H^0(Y, \cO_Y(K_{Y/X} + f^* M))$.  
In fact, even in characteristic $p > 0$, \cite{DattaMurayamaTate} gives examples of excellent regular local (non-$F$-finite) rings that are not $F$-split.  It follows that there cannot be \emph{compatible systems} of pre-images for these examples for $X = \Spec R$.  Our proofs crucially use this compatibility (or Matlis dual versions).  In proofs, typically completeness comes as the necessary condition to apply Matlis-duality, e.g., \autoref{cor.B0VsB0Alt} and \autoref{thm:main-lifting}.

\item A priori plt pairs in the non-$\bQ$-factorial setting could have intersecting boundary components, cf., \autoref{lem:properties-of-plt}. 
\item We needed Bertini-type statements over a local ring of mixed characteristic, see \autoref{sec:Bertini}.
\item The known resolution theorems for Noetherian excellent schemes of dimension $3$ do not produce resolutions by sequences of blow-ups of non-singular centers. See \autoref{rem:dlt_anti_ample}.
\item When we pass to the localization or the completion of the base, then $\bQ$-factoriality or the Picard number being $1$ may be lost. In particular certain theorems and definitions had to be adapted, e.g.,  the paragraphs after \autoref{def:singularities_of_the_MMP} and \autoref{def:pl_flipping_contraction}, as well as the proof of  \autoref{cor:flips-exist2}. 
\item When working over arbitrary Noetherian excellent schemes, it can happen that the codimension and the dimension of a closed subscheme does not add up to the dimension of the ambient scheme, cf., \autoref{remark:divisors-of-unexpected-dimension}. 
\item For the technical advances related to the Minimal Model Program, see the beginnings of \autoref{Section:flips} and \autoref{section:MinimalModelProgram}.\end{enumerate}

\subsection*{Acknowledgements}
The authors thank 
Javier Caravajal-Rojas, Paolo Cascini, Christopher Hacon, Srikanth Iyengar, J\'anos Koll\'ar, Alicia Lamarche, Olivier Piltant, Hiromu Tanaka, Kazuma Shimomoto, Michael Temkin, and Burt Totaro for valuable conversations related to, and useful comments on, this paper.  We also thank the referees for numerous thoughtful and valuable comments on previous drafts of this article.
\begin{itemize}
	\item{} This material is partially based upon work supported by the National Science Foundation under Grant No. DMS-1440140 while the authors were in residence at the Mathematical Sciences Research Institute in Berkeley California during the Spring 2019 semester.  The authors also worked on this while attending an AIM SQUARE in June 2019.
	\item{} {Bhatt was supported by NSF Grant DMS \#1801689, NSF FRG Grant \#1952399, a Packard fellowship, and the Simons Foundation grant \#622511.}
	\item{} {Ma was supported by NSF Grant DMS \#190167, NSF FRG Grant \#1952366, and a fellowship from the Sloan Foundation.}
	\item{} {Patakalvi was partially supported by the following grants: grant \#200021/169639 from the Swiss National Science Foundation,  ERC Starting grant \#804334.}
	\item{} {Schwede was supported by NSF CAREER Grant DMS \#1252860/1501102, NSF Grants \#1801849 and \#2101800, NSF FRG Grant \#1952522 and a Fellowship from the Simons Foundation.}
	\item{} {Tucker was supported by NSF Grant DMS \#1602070, \#1707661 and \#2200716 and by a Fellowship from the Sloan Foundation.}
	\item{} Waldron was supported by a grant from the Simons Foundation \#850684.
	\item{} {Witaszek was supported by the National Science Foundation DMS \#2101897 and under Grant No.\ DMS-1638352 at the Institute for Advanced Study in Princeton.}
\end{itemize}

%% file: notation.tex
\section{Preliminaries}
\label{sec:notation}
\label{sec:preliminaries}

\emph{For much of the article we work over an excellent domain $R$ of finite Krull dimension with a dualizing complex.  Unless otherwise specified, we shall write $R^+$ to denote an absolute integral closure of $R$ in the sense of \cite{ArtinJoins} (i.e., the integral closure of $R$ in an algebraic closure of $\mathrm{Frac}(R)$); this object is unique up to isomorphism, and our constructions will be independent of the specific choice. {Moreover}, except for \autoref{sec:preliminaries}, \autoref{sec.Vanishing} and \autoref{section:MinimalModelProgram} or where otherwise noted, we will also assume that $(R, \fram)$ is a complete local Noetherian domain whose residual characteristic is $p > 0$ (in this case $R$ is excellent \cite[\href{https://stacks.math.columbia.edu/tag/07QW}{Tag 07QW}]{stacks-project}, it has finite Krull dimension \cite[\href{https://stacks.math.columbia.edu/tag/0323}{Tag 0323}]{stacks-project}, and it admits a dualizing complex as discussed in \autoref{sec:dualizing_complexes_local_duality}).  Most typically, we are interested in the case that $R$ is of mixed characteristic $(0,p>0)$.  Now suppose that a scheme $S$ is excellent with a dualizing complex (most typically $S = \Spec R$).  Observe that any scheme $X$ with a map $f : X \to S$ of finite type is also excellent \cite[\href{https://stacks.math.columbia.edu/tag/07QU}{Tag 07QU}]{stacks-project} and has a dualizing complex induced from $S$, see \cite[\href{https://stacks.math.columbia.edu/tag/0AUA}{Tag 0AUA}]{stacks-project}, which we take as $\omega_X^{\mydot} = f^! \omega_S^{\mydot}$ when $f$ is separated (our typical case).  
Furthermore, in \autoref{section:MinimalModelProgram} we will sometimes assume that our schemes $X$ have $X_{\mathbb{Q}}$ non-trivial.}

{In this article, we say that a scheme over $R$ is $n$-dimensional if its  \emph{absolute dimension} is equal to $n$ (as opposed to the relative dimension).} Furthermore, the underlying scheme of a pair is always assumed to be normal, excellent, Noetherian, integral and admitting a dualizing complex (see \autoref{def:log-pair} for the precise statement).

If $X$ is a normal integral scheme over $R$, then $X_{\fram}$ denotes the fiber over $\fram \in \Spec R$.

\begin{definition}
\label{def:alteration}
	Given an integral Noetherian scheme $X$, an \emph{alteration} $\pi:Y\to X$ is a surjective generically finite proper morphism with $Y$ integral. (We shall often be in the situation where $Y$ is normal.)
\end{definition}

{Note that constructibility of the level sets and the upper semi-continuity of the dimension of fibers function holds in the setting of \autoref{def:alteration}  \cite[\href{https://stacks.math.columbia.edu/tag/05F9}{Tag 05F9}]{stacks-project}, \cite[\href{https://stacks.math.columbia.edu/tag/0D4I}{Tag 0D4I}]{stacks-project}. Similarly, it holds that over the locus where the fibers are finite, $\pi$ is finite \cite[\href{https://stacks.math.columbia.edu/tag/02OG}{Tag 02OG}]{stacks-project}. In particular, if  $\pi$ is an alteration, then there exists a non-empty open set over which $\pi$ is finite. Additionally, the additivity of dimension also holds here \cite[\href{https://stacks.math.columbia.edu/tag/02JX}{Tag 02JX}]{stacks-project}, and so we have $\dim X = \dim Y$.
}

    Throughout this article, we will frequently use that local cohomology on the Noetherian ring $R$ commutes with direct limits (in other words, filtered colimits) just as sheaf cohomology does on Noetherian topological spaces, see \cite[Chapter III, Proposition 2.9]{Hartshorne} \cite[\href{https://stacks.math.columbia.edu/tag/01FF}{Tag 01FF}]{stacks-project}.  In particular, we have for a directed system of $R$-modules $M_\alpha$ that  
    \[
        \varinjlim_{M_\alpha} H^i_{\m}(M_\alpha) = H^i_{\m}(\varinjlim_{M_\alpha} M_\alpha),
    \]
    see \cite[Theorem 3.4.10]{BrodmannSharpLocalCohomology}.  More generally, if $X$ is a Noetherian scheme and $E \subseteq X$ is closed, by mimicking the argument of \cite[Chapter III, Proposition 2.9]{Hartshorne}, one immediately sees for a directed system of sheaves of $\sO_X$-modules $\sF_{\alpha}$ that 
    \begin{equation}
        \label{eq.CohomologyWithSupportsCommutesWithDirectLimits}
        \varinjlim_{\sF_{\alpha}} H^i_Z(X, \sF_{\alpha}) = H^i_Z(X, \varinjlim_{\sF_{\alpha}} \sF_{\alpha}).
    \end{equation}
    Recall also that tensor products commute with arbitrary colimits \cite[\href{https://stacks.math.columbia.edu/tag/00DD}{Tag 00DD}]{stacks-project}.

\subsection{Dualizing complexes and local duality}
\label{sec:dualizing_complexes_local_duality}

Recall that any complete Noetherian local ring $(R, \fram)$ has a dualizing complex $\omega_R^{\mydot}$, since such an $R$ is a quotient of a regular ring (\cite[\href{https://stacks.math.columbia.edu/tag/032A}{Tag 032A}]{stacks-project},
\cite[\href{https://stacks.math.columbia.edu/tag/0A7I}{Tag 0A7I}]{stacks-project},
\cite[\href{https://stacks.math.columbia.edu/tag/0A7J}{Tag 0A7J}]{stacks-project}). 
We always choose $\omega_R^{\mydot}$ to be normalized in the sense of \cite{HartshorneResidues}, that is $\myH^{-i} \omega_R^{\mydot} = 0$ for $i > \dim R$ and $\myH^{-\dim R} \omega_R^{\mydot} \neq 0$.  If then $\pi : X \to \Spec R$ is a proper morphism (or even separated morphism), we define the dualizing complex $\omega_X^\mydot$ of $X$ to be $\pi^! \omega_R^\mydot$  and the dualizing sheaf $\omega_X$ to be $\myH^{-\dim X}(\omega_X^\mydot)$.  We make these choices so that Grothendieck local duality can be applied as described below. Before doing that however, we observe that when $R$ is an excellent regular domain of finite Krull dimension, we can define $\omega_{X}^\mydot$ and $\omega_X$ similarly.  We shall work in this non-local generality in \autoref{section:MinimalModelProgram}.

Back in the complete local case, fix $E = E_R(R/\fram)$ to be an injective hull of the residue field.  This provides an exact Matlis duality functor $(-)^{\vee} := \Hom_R(-, E)$ which induces an anti-equivalence of categories of Noetherian $R$-modules with Artinian $R$-modules \cite[\href{https://stacks.math.columbia.edu/tag/08Z9}{Tag 08Z9}]{stacks-project}; by exactness, Matlis duality extends to the derived category as well, and we continue to denote it by $(-)^\vee$.  In particular, since $E$ is injective, we may harmlessly identify $\Hom_R(-, E)$ and $\myR \Hom_R(-, E)$. Note that here, and when working in the derived category in general, we shall also simplify notation by writing $E$ (rather than $E[0]$) for the relevant complex concentrated in degree zero.

There is also a Matlis duality when $(R, \m)$ is not complete (but still local and Noetherian).  In this, we still define $E = E_R(R/\fram)$ to be the injective hull of the residue field.  Then $(-)^{\vee} := \Hom_R(-, E)$ is an exact functor which takes Noetherian modules to Artinian modules (which are then canonically modules over $\widehat{R}$).  Note that for $M$ Noetherian, $(M^{\vee})^{\vee} \cong \widehat{M}$.  For additional discussion see \cite[10.2.18]{BrodmannSharpLocalCohomology}.

Since we work with normalized dualizing complexes, we have an isomorphism $\myR\Gamma_{\mathfrak{m}}(\omega^{\mydot}_R) \simeq E$  \cite[\href{https://stacks.math.columbia.edu/tag/0A81}{Tag 0A81}]{stacks-project}. Using this isomorphism and the complete-torsion equivalence (\cite[\href{https://stacks.math.columbia.edu/tag/0A6X}{Tag 0A6X}]{stacks-project}) shows the following compatibility of Grothendieck and Matlis duality:  for any $K \in D_{\mathrm{coh}}^b(R)$, the following natural maps give isomorphisms
\[ \myR\Hom_R(K,\omega_R^{\mydot}) \simeq \myR\Hom_R\big(\myR\Gamma_{\mathfrak{m}}(K), \myR\Gamma_{\mathfrak{m}}(\omega^{\mydot}_R)\big) \simeq \Hom_R\big(\myR\Gamma_{\mathfrak{m}}(K),E\big) = \myR\Gamma_{\mathfrak{m}}(K)^\vee. \]
As $R$ is complete and $\Hom_R(-,E)$ induces an anti-equivalence of Noetherian and Artinian $R$-modules, this yields
\[
    \big(\myR \Hom_R(K, \omega_R^{\mydot})\big)^{\vee} \simeq \myR \Gamma_{\fram}(K)
\]
for $K \in D^b_{\mathrm{coh}}(R)$. For more details see for instance \cite{HartshorneLocalCohomology, HartshorneResidues, BrunsHerzog} and \cite[\href{https://stacks.math.columbia.edu/tag/0A81}{Tag 0A81}]{stacks-project}.  

We will be particularly interested in applying the above considerations in the following situation. 
\begin{lemma}
    \label{lem:duality-general}
	Suppose that $(R, \fram)$ is a Noetherian complete local ring, $X$ is an integral scheme proper over $\Spec R$, and that $\sL$ is a line bundle on $X$.  Then
    \[        
        \myR \Gamma_{\fram} \myR \Gamma(X, \sL) \cong \Hom\big(\myR \Gamma(X, \sL^{-1} \otimes \omega_X^{\mydot}), E\big).
    \]    
    In the case that $X$ is Cohen-Macaulay, the right side becomes $\Hom\big(\myR \Gamma(X, \sL^{-1} \otimes \omega_X[\dim X]), E\big)$.  

    Furthermore, 
    \[
        \myR \Gamma_{\m} \myR \Gamma(X, \omega_X^{\mydot} \otimes \sL^{-1}) \cong \Hom\big(\myR \Gamma(X, \sL), E \big)
    \]
    and if $X$ is Cohen-Macaulay, the left side becomes $\myR \Gamma_{\m}(\myR \Gamma(X, \omega_X \otimes \sL^{-1}))[\dim X]$.
\end{lemma}
\begin{proof}
    Both statements follow by combining Grothendieck and local duality with the observations made above.  In the first case take $K = \myR\Gamma(X, \sL)$ and in the second take $K = \myR\Gamma(X, \omega_X^{\mydot} \otimes \sL^{-1})$.
\end{proof}

We will also use the following consequence of local duality frequently.

\begin{lemma}   
	\label{lem:duality}
	Suppose that $(R, \fram)$ is a Noetherian complete local ring, $X$ is an integral scheme proper over $\Spec R$, and that $\sF$ is a coherent sheaf on $X$. Then we have an isomorphism of $R$-modules
	\begin{equation*}
	    \big(\myH^d\myR \Gamma_{\fram} \myR \Gamma(X, \sF)\big)^\vee\cong \Hom_{\sO_X}(\sF, \omega_X),
	\end{equation*}
	where $d=\dim X$.\end{lemma}
\begin{proof}
    The fact that $X \to \Spec R$ is proper is essential in what follows.
    By local duality (\cite[Tag 0A84]{stacks-project}) and Grothendieck duality (cf.\ \cite[Tag 0AU3(4c)]{stacks-project}), we have
    \begin{multline}
    \label{eq:duality:dualities}
            \Big(\myH^d \myR\Gamma_\m\big(\myR\Gamma(X,\sF)\big)\Big)^\vee
        \cong \myH^{-d}\myR\Hom_R\big(\myR\Gamma(X,\sF), \omega_R^{\mydot}\big)
        \\ \cong \myH^{-d}\myR \Gamma \circ \myR\sHom_{\sO_X}(\sF, \omega_X^{\mydot})
        \cong \myH^{-d}\myR\Hom_{\sO_X}(\sF, \omega_X^{\mydot})
    .
    \end{multline}
    If $X$ was Cohen-Macaulay so that $\omega_X^{\mydot} = \omega_X[d]$, then we would be done.  However, we are taking the bottom cohomology, so the higher cohomologies  of the dualizing complex do not interfere, as we work out in detail now.
    Form a triangle $\omega_X[d]\to \omega_X^{\mydot} \to C \xrightarrow{+1}$. Applying $\myR\Hom_{\sO_X}(\sF, -)$ to this triangle we get:
    \[
        \myR\Hom_{\sO_X}(\sF, \omega_X[d])\to \myR\Hom_{\sO_X}(\sF, \omega_X^{\mydot})\to \myR\Hom_{\sO_X}(\sF, C)\xrightarrow{+1}.
    \]
    Note that $C$ and hence $\myR\Hom_{\sO_X}(\sF, C)$ only live in cohomological degree $\geq -d+1$, thus we have
    \begin{equation}
    \label{eq:duality:omegas}
    \myH^{-d} \myR\Hom_{\sO_X}(\sF, \omega_X^{\mydot})\cong \myH^{-d}\myR\Hom_{\sO_X}(\sF, \omega_X[d])  \cong  \Hom_{\sO_X}(\sF, \omega_X).
    \end{equation}
    Combining \autoref{eq:duality:dualities} and \autoref{eq:duality:omegas} yields exactly the statement of the claim. 
\end{proof}

\subsection{Big Cohen-Macaulay algebras} 
Let $(R,\m)$ be a Noetherian local ring of dimension $d$ and let $M$ be a (not necessarily finitely generated) $R$-module. A sequence of elements $x_1,\dots,x_n$ of $R$ is called a {\it regular sequence} on $M$ if $x_{i+1}$ is a nonzerodivisor on $M/(x_1,\dots,x_i)M$ for each $i$. We consider the following conditions on $M$ ({which are equivalent when $M$ is finitely generated}).
\begin{enumerate}
    \item \label{itm:BCM:local_coho_max_ideal} $H_\m^i(M)=0$ for all $i<d$.  
    \item  \label{itm:BCM:one_sop_is_reg_seq} There is a system of parameters $x_1,\dots, x_d$ of $R$ that is a regular sequence on $M$. 
    \item \label{itm:BCM:every_sop_is_reg_seq} Every system of parameters of $R$ is a regular sequence on $M$.
    \item \label{itm:BCM:local_coho_prime_ideal} $H_P^i(M_P)=0$ for all $P\in\Spec(R)$ and all $i<\dim(R_P)$.
\end{enumerate}

It is straightforward to see that  \autoref{itm:BCM:every_sop_is_reg_seq} $\Rightarrow$ \autoref{itm:BCM:one_sop_is_reg_seq} $\Rightarrow$ \autoref{itm:BCM:local_coho_max_ideal}. If $M$ satisfies condition \autoref{itm:BCM:local_coho_max_ideal} and $M/\m M\neq 0$, then the $\m$-adic completion $\widehat{M}$ satisfies condition \autoref{itm:BCM:every_sop_is_reg_seq} by \cite[Exercise 8.1.7, Theorem 8.5.1]{BrunsHerzog}. We will see below (\autoref{lem: coh CM vs balanced CM}) that, under mild assumptions on $R$, condition \autoref{itm:BCM:every_sop_is_reg_seq} and condition \autoref{itm:BCM:local_coho_prime_ideal} are equivalent. These implications are summarized in the diagram below.
\[
    \xymatrix{
        {\left({\begin{array}{cl} \text{ \autoref{itm:BCM:local_coho_max_ideal}} & H^i_{\m}(M) =0 \\ & \text{for $i < d$ } \end{array}}\right)} \ar@{=>}@/_12pc/[dd]_{\text{\parbox[c]{50pt}{if $M \neq \m M$ and $M := \widehat{M}$}}} & & & \\
        \ar@{=>}[u] {\left(\begin{array}{cl}\text{ \autoref{itm:BCM:one_sop_is_reg_seq}} & \text{$\exists x_1, \dots, x_d\;$ s.o.p. of $R$,}\\ & \text{which is a reg. seq. on $M$}\end{array}\right)} \\
        \ar@{=>}[u] {\left(\begin{array}{cl}\text{ \autoref{itm:BCM:every_sop_is_reg_seq}} & \text{$\forall x_1, \dots, x_d\;$ s.o.p. of $R$,}\\ & \text{is a reg. seq. on $M$}\end{array}\right)} \ar@{<=>}[dd]^{\text{\parbox[l]{80pt}{\raggedright if $R$ is \\ equidimensional\\ \& catenary}}}\\ \\
        {\left({\begin{array}{cl} \text{ \autoref{itm:BCM:local_coho_prime_ideal}} & H^i_{P}(M_P) =0 \\ & \text{for $i < \dim R_P$ } \end{array}}\right)}
    }
\]

The module $M$ is called:
\begin{itemize}{}
\item{} \emph{big Cohen-Macaulay} if $M$ satisfies condition \autoref{itm:BCM:one_sop_is_reg_seq} and $M/\m M\neq 0$, see \cite{HochsterTopicsInTheHomologicalTheory}, 
\item{} \emph{balanced big Cohen-Macaulay} if $M$ satisfies condition \autoref{itm:BCM:every_sop_is_reg_seq} and $M/\m M\neq 0$, see \cite[Chapter 8]{BrunsHerzog}.     
\item{} \emph{cohomologically Cohen-Macaulay} if $M$ satisfies condition \autoref{itm:BCM:local_coho_prime_ideal}, see \cite{BhattAbsoluteIntegralClosure}.
\end{itemize}
If $B$ is an $R$-algebra that is (big/balanced big/cohomologically) Cohen-Macaulay as an $R$-module, then it is called a \emph{(big/balanced big/cohomologically) Cohen-Macaulay algebra}.
Note that, in the definition of cohomologically Cohen-Macaulay, we do not require the non-triviality condition $M/\m M \neq 0$, so this definition passes to localization, which is convenient for some inductive arguments in \cite{BhattAbsoluteIntegralClosure}.

\begin{remark}
For our purpose, even the weakest notion \autoref{itm:BCM:local_coho_max_ideal} above suffices for most of our applications. In fact, we can usually replace $B$ by its $\m$-adic completion to obtain the strongest notion. Thus, for most practical purposes, the distinctions between these notions can be ignored. 
\end{remark}

Balanced big Cohen-Macaulay algebras always exist: in equal characteristic, this is a result of Hochster-Huneke \cite{HochsterHunekeInfiniteIntegralExtensionsAndBigCM,HochsterHunekeApplicationsofBigCM}, and in mixed characteristic, this is settled by Andr\'{e} \cite{AndreDirectsummandconjecture}. For our purposes, the following theorem (due to Hochster-Huneke in equal characteristic $p>0$, and the first author in mixed characteristic $(0,p>0)$) gives an explicit construction of balanced big Cohen-Macaulay algebras, and is the key behind our definitions and constructions.

\begin{theorem}
\label{RPlusCM}
Let $(R,\m)$ be an excellent local domain of residue characteristic $p>0$. Let $R^+$ be an absolute integral closure of $R$. Then $H^i_{\fram}(R^+) =0$ for $i < \dim R$ and the $p$-adic completion of $R^+$ is a balanced big Cohen-Macaulay algebra. \end{theorem}
\begin{proof}
In positive characteristic, the $p$-adic completion of $R^+$ is $R^+$ and this is \cite[Theorem 1.1]{HochsterHunekeInfiniteIntegralExtensionsAndBigCM}. In mixed characteristic, the statement about local cohomology is \cite[Theorem 5.1]{BhattAbsoluteIntegralClosure}.  The extension to the $p$-adic completion is explained below in \autoref{cor: p-adic completion of R^+}, also see \cite[Corollary 5.17]{BhattAbsoluteIntegralClosure}.
\end{proof}

We caution the reader that if $R$ has equal characteristic $0$ (i.e., contains $\mathbb{Q}$) with $\dim(R) \geq 3$, then $R^+$ is never big Cohen-Macaulay in any of the senses discussed above because of a simple trace obstruction. For example, one may first construct a finite normal domain extension $S$ of $R$ that is not Cohen-Macaulay and $H^i_\fram(S) \neq 0$ for some $i < \dim R$. Since the normalized (field) trace splits the inclusion $S \to S^+ = R^+$, $H^i_\fram(S)$ is a direct summand of $H^i_\fram(R^+)$ and thus $R^+$ fails to satisfy condition \autoref{itm:BCM:local_coho_max_ideal}. See also \cite[Proof of Proposition 2.1]{ShimomotoTavanfarLocalRingsNoSmallCMMixedChar} for a collection of explicit constructions.

We next want to explain how to drop the additional assumptions on the existence of Noether normalization in \cite[Corollary 5.17]{BhattAbsoluteIntegralClosure} in the local case.

\begin{lemma}
\label{lem: coh CM vs balanced CM}
Let $(R,\m)$ be a Noetherian, equidimensional, catenary local ring and let $M$ be an $R$-module. Then every system of parameters of $R$ is a regular sequence on $M$ if and only if $H_P^i(M_P)=0$ for all $P\in\Spec(R)$ and all $i<\dim(R_P)$. In particular, $M$ is balanced big Cohen-Macaulay if and only if $M$ is cohomologically Cohen-Macaulay and $M/\m M\neq 0$.
\end{lemma}
\begin{proof}
The if direction follows from \cite[Corollary 2.8]{BhattAbsoluteIntegralClosure}. For the only if direction, let $P$ be a prime ideal of height $h$. There exists $x_1,\dots,x_h$ part of a system of parameters such that $P$ is a minimal prime of $(x_1,\dots,x_h)$. Thus we know that $x_1,\dots,x_h$ is a regular sequence on $M$ and hence a regular sequence on $M_P$. But the image of $x_1,\dots,x_h$ is a system of parameters on $R_P$, and thus $H_P^i(M_P)=0$ for all $i<h$ as desired.
\end{proof}

\begin{lemma}
\label{lem: regular sequence}
Suppose $R$ is a commutative ring and $f,g \in R$ is a regular sequence on an $R$-module $N$. Then $g, f$ is a regular sequence on $\widehat{N}^f$, the $f$-adic completion of $N$.
\end{lemma}
\begin{proof}
First of all, $f$ is a nonzerodivisor on $N$ and hence a nonzerodivisor on $\widehat{N}^f$. Because $N/fN\cong \widehat{N}^f/f\widehat{N}^f$, $f, g$ is a regular sequence on $\widehat{N}^f$. This implies that $f$ is a nonzerodivisor on $\widehat{N}^f/g\widehat{N}^f$. 
It remains to prove that $g$ is a nonzerodivisor on $\widehat{N}^f$. So suppose $ga=0$ where $a=\sum_{i=0}^{\infty}a_if^i$ where $a_i\in N$. Then for each $k$,
$$g\cdot \sum_{i=0}^{k}a_if^i=-g\cdot\sum_{j=k+1}^{\infty}a_jf^j\in f^{k+1}\widehat{N}^f .$$
Thus we actually have $g\cdot \sum_{i=0}^{k}a_if^i\in f^{k+1}N$ and hence $\sum_{i=0}^{k}a_if^i\in f^{k+1}N$ for each $k$ since $f, g$ is a regular sequence on $N$. But then we have
$$a=\sum_{i=0}^{\infty}a_if^i= \sum_{i=0}^{k}a_if^i+  \sum_{j=k+1}^{\infty}a_jf^j\in f^{k+1}\widehat{N}^f$$
for all $k$, which implies $a=0$ since $\widehat{N}^f$ is $f$-adically separated.
\end{proof}

\begin{lemma}
\label{lem: completion}
Suppose $N$ is $f$-adically complete and $f$ is a nonzerodivisor on $N/gN$, then $N/gN$ is $f$-adically complete.
\end{lemma}
\begin{proof}
$N/gN$ is always derived $f$-adically complete. Since $f$ is a nonzerodivisor on $N/gN$, we know that the $f$-adic completion of $N/gN$ is the same as the derived $f$-adic completion of $N/gN$, which is $N/gN$. Hence $N/gN$ is $f$-adically complete.
\end{proof}

\begin{theorem}
\label{thm: balanced CM}
Let $(R,\m)$ be a Noetherian, equidimensional, catenary local ring and let $M$ be an $R$-module. Suppose $t\in R$ is a parameter such that
\begin{enumerate}[series=balancedcm]
  \item $t$ is a nonzerodivisor on $M$ \label{thm: balanced CM.cond1}
  \item $M/tM$ is balanced big Cohen-Macaulay over $R/tR$. \label{thm: balanced CM.cond2}
\end{enumerate}
Then $\widehat{M}^t$, the $t$-adic completion of $M$, is balanced big Cohen-Macaulay over $R$.
\end{theorem}
\begin{proof}
We prove by induction on $d=\dim(R)$. So we assume the conclusion of the theorem holds whenever the local ring has dimension $<d$.

We first prove that every system of parameters $x_1,\dots,x_d$ of $R$ such that $x_i=t$ for some $i$ is a regular sequence on $M$. This is clear if $i=1$ and so we assume $x_1\neq t$. We claim that
\begin{enumerate}[resume*=balancedcm]
  \item $t$ is a nonzerodivisor on $\widehat{M}^t/x_1\widehat{M^t}$. \label{thm: balanced CM.cond3}
  \item $\widehat{M}^t/(x_1, t)\widehat{M}^t$ is balanced big Cohen-Macaulay over $R/(x_1, t)$. \label{thm: balanced CM.cond4}
\end{enumerate}
Here \autoref{thm: balanced CM.cond4} is obvious since $\widehat{M}^t/(x_1, t)\widehat{M}^t=M/(t, x_1)M$, and \autoref{thm: balanced CM.cond3} follows from \autoref{lem: regular sequence} since $t, x_1$ is a regular sequence on $M$.

By induction, we know that the $t$-adic completion of $\widehat{M}^t/x_1\widehat{M}^t$ is balanced big Cohen-Macaulay over $R/x_1R$. Since $t$ is a nonzerodivisor on $\widehat{M}^t/x_1\widehat{M}^t$, by \autoref{lem: completion} $\widehat{M}^t/x_1\widehat{M}^t$ is $t$-adically complete. Therefore $\widehat{M}^t/x_1\widehat{M}^t$ is balanced big Cohen-Macaulay over $R/x_1R$. But since $x_1$ is a nonzerodivisor on $\widehat{M}^t$ by \autoref{lem: regular sequence}, $x_1,\dots, x_d$ is a regular sequence on $\widehat{M}^t$.

Now let $P$ be a prime ideal of height $h$. Suppose $t\in P$, then since $\widehat{M}^t/t\widehat{M}^t=M/tM$, we have
$H_P^i((\widehat{M}^t)_P/t(\widehat{M}^t)_P)=H_P^i((M/tM)_P)=0$ for all $i<h-1$, which by the long exact sequence of local cohomology implies that $H_P^i(\widehat{M}^t)=0$ for all $i<h$. Now suppose $t\notin P$, by prime avoidance, we can pick $x_1,\dots,x_h$ and $x_{h+2},\dots,x_d$ such that
\begin{enumerate}[resume*=balancedcm]
  \item $P$ is a minimal prime of $(x_1,\dots,x_h)$  \label{thm: balanced CM.cond5}
  \item $x_1,\dots,x_h, t, x_{h+2},\dots,x_d$ is a system of parameters of $R$.  \label{thm: balanced CM.cond6}
\end{enumerate}
By what we have already proved, $x_1,\dots,x_h, t, x_{h+2},\dots,x_d$ and hence $x_1,\dots,x_h$ is a regular sequence on $\widehat{M}^t$. Thus $x_1,\dots,x_h$ is a regular sequence on $(\widehat{M}^t)_P$ and so $H_P^i((\widehat{M}^t)_P)=0$ for all $i<h$. Therefore $\widehat{M}^t$ is cohomologically Cohen-Macaulay. Since $\widehat{M}^t/\m\widehat{M}^t=M/\m M\neq 0$ (by condition \autoref{thm: balanced CM.cond2}), $\widehat{M}^t$ is balanced big Cohen-Macaulay as desired.
\end{proof}

Now we can prove the promised extension of \cite[Corollary 5.17]{BhattAbsoluteIntegralClosure}.

\begin{corollary}
\label{cor: p-adic completion of R^+}
Let $(R,\m)$ be an excellent local domain of mixed characteristic $(0,p>0)$. Then $\widehat{R^+}^p$, the $p$-adic completion of $R^+$, is a balanced big Cohen-Macaulay.
\end{corollary}
\begin{proof}
This follows from \cite[Corollary 5.11]{BhattAbsoluteIntegralClosure} and \autoref{thm: balanced CM}.
\end{proof}

As we mentioned before, one advantage of the notion of cohomologically Cohen-Macaulay is that it behaves well under localization.  It is not clear that (balanced) big Cohen-Macaulay algebras behave well under localization, we record the following partial result for psychological comfort; it will not be used in this paper.

\begin{proposition}
Suppose $R$ is a complete Noetherian local domain and $B$ is a balanced big Cohen-Macaulay algebra, then $B_P$ is balanced big Cohen-Macaulay for $R_P$ for all $P\in\Spec(R)$.
\end{proposition}
\begin{proof}
Let $x_1,...,x_h$ be a system of parameters in $R_P$, by prime avoidance, we may assume that $x_1,...,x_h$ is also part of a system of parameters of $R$, thus it is a regular sequence on $B$ and hence a possibly improper regular sequence on $B_P$. But since $B$ is big Cohen-Macaulay and $R$ is a Noetherian complete local domain, $B$ is a solid $R$-algebra (see \cite[Corollary 2.4]{HochsterSolidClosure}) and thus $\Spec(B)\to \Spec(R)$ is surjective, so $B_P/PB_P\neq 0$ and hence $x_1,\dots,x_h$ is a regular sequence on $B_P$. 
\end{proof}

We conclude our discussion with a definition related to the discussion above.

\begin{definition}[Splinters]
    A Noetherian reduced ring $R$ is called a \emph{splinter} if for every finite extension of rings $R \subseteq S$ we have that $R \hookrightarrow S$ splits as a map of $R$-modules.  
\end{definition}

As mentioned above, in characteristic zero, every normal ring is a splinter (the trace can be used to split the inclusions).  However, in characteristic $p > 0$ or mixed characteristic $(0, p>0)$, if a local ring $(R, \m)$ is a splinter, then $H^i_{\m}(R) \to H^i_{\m}(R^+) = H^i_{\m}(\varinjlim_{S \subseteq R^+} S)$ is injective for every $i > 0$.  In particular, by \autoref{RPlusCM} we see that splinters are Cohen-Macaulay.

\subsection{Resolution of singularities}
\label{subsec.ResolutionOfSings}
In this section, we recall known results about resolutions of singularities for mixed characteristic three-dimensional schemes. {Note that resolutions of singularities exist for Noetherian excellent surfaces in full generality by \cite{LipmanDesingularizationOf2Dimensional}.}

\begin{theorem}[{\cite[Theorem 1.1]{CP19} and \cite[Corollary 1.5]{CossartJannsenSaito}}] \label{thm:proper-resolutions} Let $X$ be a reduced and separated Noetherian scheme which is quasi-excellent and of dimension at most three, and let $T$ be a subscheme of $X$. Then there exists a proper birational morphism $g \colon Y \to X$ from a regular scheme $Y$ such that both $g^{-1}(T)$ and $\mathrm{Ex}(g)$ are divisors and $\Supp(g^{-1}(T)\cup \mathrm{Ex}(g))$ is simple normal crossing.
\end{theorem}
\begin{proof}
    By \cite[Theorem 1.1]{CP19}, there is a projective morphism $f:Z\to X$ such that $X$ is regular.  Then applying \cite[Corollary 1.5]{CossartJannsenSaito} to $(Z, T')$ with $T'=f^{-1}(T)$ gives the required $g$. 
\end{proof}

\begin{proposition}\label{proj-resolutions}
Let $X$ be a reduced scheme of dimension $3$, quasi-projective over a Noetherian quasi-excellent affine scheme $\Spec(R)$.  Let $T$ be a subscheme of $X$.  Then there exists a projective birational morphism $g:Y\to X$ from a regular scheme $Y$ such that both $g^{-1}(T)$, and $\mathrm{Ex}(g)$ are divisors, $\Supp(g^{-1}(T)\cup\mathrm{Ex}(g))$ is simple normal crossing and $Y$ supports a $g$-ample $g$-exceptional divisor.
\end{proposition}
\begin{proof}
    {By taking normalization, we may assume that $X$ is normal and integral.} Let $g':Y'\to X$ be the proper birational morphism given by \autoref{thm:proper-resolutions}.  By Chow's lemma \cite[Theorem 5.6.1(a)]{EGA} applied to $g$, there exists a projective birational map $\tilde{g}:\tilde{Y}\to X$ which factors through $f:\tilde{Y}\to Y'$, and which is the blow up of some ideal sheaf $\sI$ by \cite[Theorem 1.24]{liu_algebraic_2002}. By the universal property of blow-ups \cite[Tag 0806]{stacks-project}, $\tilde{Y}$ is also the blow-up of $\sI'=\sI\sO_{Y'}$, which is the ideal sheaf of a subscheme $Z$.
        Now let $h:Y\to Y'$ be the projective  embedded resolution of $(Y', Z\cup (g')^{-1}(T)\cup\mathrm{Ex}(g'))$ given by \cite[Corollary 1.5]{CossartJannsenSaito}, which is projective since it is a composition of blowups. 
    Then $g:Y\to X$ factors through $\tilde{Y}$ by the universal property of blow-ups, and so $g$ is projective by \cite[0C4P]{stacks-project} since $Y$ is projective over $\tilde{Y}$ and $\tilde{Y}$ is projective over $X$.  Given this $Y$, we may replace it with a resolution supporting a $g$-ample $g$-exceptional divisor by \cite[Theorem 1]{kollar_witaszek}.
\end{proof}

\begin{remark}
    Note that the construction in \autoref{proj-resolutions} does not result in a morphism $g$ which is an isomorphism over the simple normal crossing locus of $(X,T)$.  Cossart and Piltant \cite{CP19} prove \autoref{thm:proper-resolutions} with this hypothesis, but they do not have the requirement that $g$ is projective or that $Y$ supports a $g$-ample $g$-exceptional divisor as in \autoref{proj-resolutions}.
        
{Furthermore, we do not know if \autoref{proj-resolutions} is valid over non-affine bases (due to the assumptions of \cite[Theorem 1.24]{liu_algebraic_2002}). For this reason, we assume in \autoref{section:MinimalModelProgram} that all the schemes are quasi-projective over an affine scheme.}
\end{remark}

We also need the following version of the negativity lemma from birational geometry \cite[Lem 3.39]{KollarMori}.

\begin{lemma}
\label{lem:negativity}
Let $f:Y\to X$ be a projective birational morphism of normal 
excellent integral schemes and $\Gamma$ is a $\bQ$-Cartier $\bQ$-divisor on $Y$ such that $f_* \Gamma$ is effective and $-\Gamma$ is $f$-nef. Then $\Gamma$ is effective.
\end{lemma}
\begin{proof}
Note first that $f$-nefness is preserved by localisation on $X$ \cite[Lem 2.6]{CasciniTanaka2020}, and so is the birationality of $f$. Additionally,   effectivity of divisors can be checked on all localizations of $X$. Hence, we may assume that $X= \Spec A$, where $(A, \fram)$ is local.
In particular then $Y$  has finite Krull dimension.
If $\dim Y \leq 2$, then we are done by 
\cite[Lem 2.11]{tanaka_mmp_excellent_surfaces}. Hence we may assume that $\dim Y >2$ and that the statement of the lemma is known for all dimensions smaller than $\dim Y$. 

Assume then that $\Gamma$ is not effective. Let $E$ be the prime divisor on $Y$ which has a negative coefficient in $\Gamma$. By localizing at the points of positive codimension, and using the induction hypothesis, we see that the components of $\Gamma$ that are mapping to the non-closed point of $X$ have non-negative coefficients. In particular, $E$ lies over the closed point of $X$. 
As $\dim Y > 2$ we can find  a hypersurface $H \subseteq Y$ such that  \begin{enumerate}
    \item \label{itm:negativity:intersects} $H \cap E \neq \emptyset$, and
    \item \label{itm:negativity:does_not_intersect}
 { no component of $H$ is  contained in any irreducible component of $\Exc(f)$.} \end{enumerate}
We introduce the following notation:
\begin{itemize}
    \item $Y'$ is the normalization of an irreducible component of $H$ that intersects $E$,
    \item $h : Y' \to Y$ is the induced morphism,
    \item  $X'$ is the normalization of $f(Y')$, where $f(Y')$ is also local as it is a closed subscheme of $X$, { and then $X'$ is semi-local},
    \item $f' : Y' \to X'$ is the induced morphism, which is birational due to assumption \autoref{itm:negativity:does_not_intersect} and the fact that $\codim_Y h(Y') =1$, we have $h(Y') \not\subseteq \Exc(f)$,
    \item $\Gamma':=h^* \Gamma$, for which we have that $f'_* \Gamma'$ is effective { as we know that the coefficients of $\Gamma'$ are already positive over the non-closed points of $X$}. \end{itemize}
 By the above observations we may apply the induction  hypothesis to $f' : Y' \to X'$ and to $\Gamma'$. By our choice of $Y'$, $\Gamma'$ has a negative coefficient, which is a contradiction. 
\end{proof}

\subsection{Bertini}
\label{sec:Bertini}

We will need certain Bertini theorems in mixed characteristic.

\begin{theorem}
    \label{thm.FinalBertini}
Let $R$ be a Noetherian local domain. Fix an integer $N \geq 1$. Let $X_1,...,X_n \subset \mathbf{P}^N_R$ be a finite collection of regular closed subschemes. Then there exist some $d \gg 0$ and $0 \neq h \in H^0(\mathbf{P}^N_R, \sO(d))$ such that $V(h) \cap X_i$ is regular for all $i$.
\end{theorem}
\begin{proof}
Let $k$ denote the residue field of $R$, and let $X_s = \cup_i X_{i,s} \subset \mathbf{P}^N_k$ be the subscheme of $\mathbf{P}^N_k$ obtained by taking the scheme-theoretic union of the special fibres $X_{i,s} \subset X_i$. Choose a stratification $\{U_j\}_{j \in J}$ of $X_s$ by locally closed subschemes such that each $U_j$ is connected, regular (and so $k$-smooth if $k$ {is perfect}, for instance if $k$ is finite), and such that each $X_{i,s} \subset X_s$ is (set-theoretically) a union of strata: this is clearly possible without assuming connectedness/regularity of the strata, and the connectedness/regularity can be ensured a posteriori by further refining the stratification. 

Next, we claim that there exists some $d \gg 0$ and some $0 \neq a \in H^0(\mathbf{P}^N_k,\sO(d))$ such that $V(a) \cap U_i$ is regular for all $i$. If $k$ is infinite, then this follows with $d=1$ from the classical Bertini theorem (see, e.g., 
{\cite[Corollary 3.4.14]{FlennerJoins}}): there is a Zariski dense open inside $\mathbf{V}(H^0(\mathbf{P}^N_k,\sO(1)))$ parametrizing the sections $a$ that solve the problem for each $U_i$ separately, and intersecting these opens gives a Zariski dense open inside $\mathbf{V}(H^0(\mathbf{P}^N_k,\sO(1)))$ parametrizing the sections $a$ solving the problem for all the $U_i$'s simultaneously; we then conclude by noting that any $k$-rational variety has a $k$-point as $k$ is infinite. When $k$ is finite, this follows with $d \gg 0$ from the variant of Poonen's Bertini theorem presented in \cite[Proposition 5.2]{GhoshKrishna} applied with $Z=Y=V_i=\emptyset$ and $T=\{0\}$, noting that $\zeta_{U_i}(s)$ does not have a zero or a pole at $s=\dim(U_i)+1$ (e.g., by the Weil conjectures).

Pick a section $0 \neq a \in H^0(\mathbf{P}^N_k,\sO(d))$ as constructed in the previous paragraph, and pick a lift $0 \neq h \in H^0(\mathbf{P}^N_R,\sO(d))$ of $a$. We shall show that $h$ solves our problem. First, by construction, for any closed point $u$ of any $U_j$, the image of $a$ in $\sO(d) \otimes_{\sO_{\mathbf{P}^N_k}} \sO_{U_j}/\mathfrak{m}_{U_j,u}^2$ is nonzero. Now each $X_{i,s}$ is a union of strata, so for each closed point $x \in X_{i,s}$, we can find some stratum $U_j \subset X_{i,s}$ containing $x$. As there is a natural restriction map $\sO_{X_{i,s}}/\mathfrak{m}_{X_{i,s},x}^2 \to \sO_{U_j}/\mathfrak{m}_{U_j,x}^2$, we conclude that the image of $a$ in $\sO(d) \otimes_{\sO_{\mathbf{P}^N_k}} \sO_{X_{i,s}}/\mathfrak{m}_{X_{i,s},x}^2$ is also nonzero for all closed points $x \in X_{i,s}$. But closed points of $X_i$ and $X_{i,s}$ are the same by properness of $\mathrm{Spec}(R)$. By the same reasoning used to pass from $U_j$ to $X_{i,s}$ and functoriality of restriction maps, we learn that for any index $i$ and any closed point $x \in X_i$, the image of $h$ in $\sO(d) \otimes_{\sO_{\mathbf{P}^N_R}} \sO_{X_i}/\mathfrak{m}_{X_i,x}^2$ is also nonzero. This means exactly that $V(h) \cap X_i$ is regular at all closed points of $X_i$ that it contains, i.e., $V(h) \cap X_i$ is regular at its closed points. As the regular locus is stable under generalization, we conclude that $V(h) \cap X_i$ is regular, as wanted.
\end{proof}

\begin{remark}
    \label{log_bertini}
    Now suppose that $X \to \Spec R$ is projective, $X$ is regular and $B$ is a snc divisor on $X$.  If we apply \autoref{thm.FinalBertini} to $X$ itself and the finitely many strata of $B$, then we obtain an $H = V(h)$ such that $(X, H+B)$ and $(H, B \cap H)$ are also snc pairs.
\end{remark}

\subsection{Log minimal model program}
\label{sec:preliminaries_LMMP}

We refer the reader to \cite{KollarMori} for the standard definitions and results in the Minimal Model Program. Here we briefly recall some basic notions, {in particular highlighting the adjustments required by our generality}.

\begin{definition}
\label{def:mobile_part}
Given a Cartier divisor $D$ on a Noetherian normal separated scheme $X$, we define $\Mob(D) = D - \Fix(D)$, where the divisor $\Fix(D)$ is defined by requiring that for each prime divisor $E$ on $X$
\begin{equation*}
    \coeff_E \Fix(D)= \min_{D' \in |D|} \coeff_E D' 
\end{equation*}
Note that as $D$ is Cartier the above coefficients are integers and hence the minimum exists. We also note that here, and in general in the article, the linear system $|D|$ simply means the \emph{set} of all effective divisors linearly equivalent to $D$. That is, we do not put any scheme structure on $|D|$.
\end{definition}
\begin{remark} \label{rem:mob}
In the situation of \autoref{def:mobile_part},
 there is a natural identification of $H^0(X,\sO_X(D))$ with $H^0(X, \sO_X(\Mob(D)))$. Note also that if $D' = D + F$ for a Cartier divisor $F \geq 0$, then $\Mob(D') \geq \Mob(D)$. {Further observe that when $D$ is effective, so is $\Mob(D)$.}
\end{remark}

A \emph{$\mathbb{Q}$-divisor} (resp.\ \emph{$\mathbb{R}$-divisor}) is a finite formal sum $\sum_{i=1}^n d_iD_i$ where $D_i$ is an integral codimension one subscheme of $X$, and $d_i\in\mathbb{Q}$ (resp.\ $d_i\in\mathbb{R}$).  Two divisors are $\bQ$-linearly (resp.\ $\bR$-linearly) equivalent if their difference is a $\bQ$-linear (resp.\ $\bR$-linear) combination of principal divisors.
A $\mathbb{Q}$-divisor (resp.\ $\mathbb{R}$-divisor) is \emph{$\bQ$-Cartier} (resp.\ \emph{$\mathbb{R}$-Cartier}) if some multiple of it is a Cartier divisor (resp.\ if it can be written as an $\mathbb{R}$-linear combination of Cartier divisors). Note that a $\bQ$-divisor which is $\bR$-Cartier is automatically $\bQ$-Cartier.

An $\mathbb{R}$-Cartier $\bR$-divisor $D$ is \emph{$\mathbb{R}$-ample} if it is
$\bR$-linearly equivalent to $\sum \alpha_i D_i$, where $\alpha_i \in \bR_{>0}$ and $D_i$ are ample Cartier divisors (not necessarily effective).  Note that if $D$ is \emph{$\mathbb{R}$-ample}, it is in fact equal to an $\bR$-linear combination $\sum \alpha_i D_i$ of ample $D_i$ with $\alpha_i \in \bR_{>0}$ (no $\bR$-combination of principal divisors is necessary as we may perturb them to ample divisors).  
Note that a $\bR$-ample $\bQ$-Cartier divisor is automatically ample. Henceforth, we will refer to $\bR$-ample $\bR$-Cartier divisors as ample $\bR$-Cartier divisors, as no confusion can arise.  
\begin{lemma}[{Nakai-Moishezon Criterion, cf.\ \cite[Remark 2.3]{tanaka_mmp_excellent_surfaces}}] \label{lem:Nakai-Moishezon}
Let $\pi \colon X \to Y$ be a proper morphism from an algebraic space $X$ to a Noetherian scheme $Y$.  Let $D$ be a $\bQ$-Cartier $\bQ$-divisor on $X$. Then $D$ is ample over $Y$ if and only if $D^{\dim V} \cdot V > 0$ for every $y \in Y$ and {every} positive dimensional closed integral subscheme $V$ of the fiber $X_y$ over $y$.

If $X$ is scheme, then the same condition characterizes ampleness of $\bR$-Cartier divisors $D$. 
\end{lemma}

\begin{proof}
It is enough to show that $D|_{X_{{y}}}$ is ample for every {$y \in Y$} (\cite[Tag 0D3A]{stacks-project}). By \cite[Tag 0D2P]{stacks-project}, we can assume that the residue field $k({y})$ is algebraically closed. Then, the statement follows from \cite[Theorem 3.11]{KollarProjectivityOfCompleteModuli}.

As for $\bR$-divisors on schemes, the statement over algebraically closed fields follows from \cite[Theorem 1.3]{FujinoNakaiMoishezonReal}; the reduction to that case can be done similarly to \cite[Lemma 6.2]{FujinoNakaiMoishezonReal}).
 \end{proof}

Given a projective morphism $f:X\to Z$, we define a \emph{curve over $Z$} to be a scheme $C$ of dimension $1$ such that $C$ is proper over some closed point $z\in Z$.  
Define $N_1(X)$ to be the vector space generated by integral curves over $Z$ modulo numerical equivalance: that is $\sum a_i C_i=0$ in $N_1(X)$ if and only if $(\sum a_i C_i)\cdot D=0$ for every Cartier divisor $D$ on $X$.  We say that a $\mathbb{R}$-Cartier divisor $D$ is $f$-nef if $D\cdot C\geq 0$ whenever $C$ is an integral curve over $Z$.

\begin{remark}[The relative Picard rank] \label{remark:relative-Picard-rank}
Let $f:X \to S$ be a proper morphism of Noetherian schemes. Write $\mathrm{Pic}^{\tau}(X/S) \subset \mathrm{Pic}(X)$ for the subgroup of line bundles $L$ on $X$ which are numerically trivial on all fibres $X_s$ of $f$, i.e., for every point $s \in S$ and every irreducible curve $C \subset X_s$, the restriction $L|_C$ has degree $0$ {(in fact, it is enough to verify this for closed points only)}.  Define $N^1(X/S) = \left(\mathrm{Pic}(X)/\mathrm{Pic}^{\tau}(X/S)\right) \otimes_{\mathbf{Z}} \mathbf{R}$. This $\mathbf{R}$-vector space is finite dimensional: the case of varieties over a field is explained in \cite[\S 4, Proposition 2]{KleimanNumerical}, and the same arguments go through in the general case (we learnt of the reference \cite{KleimanNumerical} from \cite{TakamatsuYoshikawaMMP}). The integer $\rho(X/S) := \dim_{\mathbf{R}} N^1(X/S)$ is called the {\em relative Picard rank} of $f$.
\end{remark}

\begin{remark} \label{remark:divisors-of-unexpected-dimension}
We warn the reader that in some situations we consider, a Cartier divisor may not have the expected dimension:  for example if $X=\Spec \bZ_p[t]$ and $Z = \Spec \bZ_p[t]/(pt-1) \simeq \Spec \bQ_p$, then $\dim X=2$, but $\dim Z=0$ despite $Z$ being a divisor.

{Furthermore, we make the following related observation. Although it is enough to check nefness of line bundles on proper curves only, it may still happen in mixed characteristic that some of these proper curves map to points of characteristic zero. For example, let $X = \Spec \bZ_p[x,y]$, let $\pi \colon Y \to X$ be the blow-up of $X$ along the subscheme $Z$ given by the ideal $(x,y)$, with the relatively ample line bundle $\cO_Y(1)$. Let $O$ be the point given by $(p,x,y)$, and let $\eta$ be the generic point of $Z$. Here $Z = \{O, \eta\}$. Let $X' = X \,\backslash\, \{O\}$ and $Y' = \pi^{-1}(X')$. In particular, $\eta$ is a characteristic zero closed point of $X'$. Then $\cO_{Y'}(-1)$ is non-negative (in fact, zero) on all positive characteristic proper curves, but it is not relatively nef. This may be checked on the proper characteristic zero curve $\pi^{-1}(\eta)$. Note that when $X = \bZ[x,y]$, the situation is different as there are many closed points of positive characteristic on $Z$.}

\end{remark}

\begin{definition}
{We say   that a proper map $f \colon X \to Z$ is \emph{small} if $\mathrm{Exc}(f)$ is of codimension at least two (all flips and flipping contractions are assumed to be \emph{small}) and that it is \emph{divisorial} if $\mathrm{Exc}(f)$ is of codimension one (but it could still happen that $\dim \mathrm{Exc}(f) \leq \dim X - 2$ as in \autoref{remark:divisors-of-unexpected-dimension}).} {Note that the codimension of a subscheme $Y$ in $X$ is {equal to} $\dim(\sO_{X,\xi})$, where $\xi$ is the generic point of $Y$ \cite[Tag 02IZ]{stacks-project}.}\end{definition}

\begin{remark} \label{remark:divisors-of-unexpected-dimension2}
The fact that curves on a three-dimensional scheme can be of codimension one may be a source of understandable confusion. However, when $T$ is a spectrum of an excellent local domain (denote the closed point of $T$ by $s$), it is always true that divisors on a proper integral scheme $X$  over $T$ are of dimension $\dim X -1$. 

{To see this, first  the following computation shows that  every closed point $x \in X_s$ has codimension $\dim X$:
\[
\dim \sO_{X,x}=\dim T+\text{trdeg}_{K(T)}K(X)-\text{trdeg}_{\kappa(s)}\kappa(x)=\dim T+\text{trdeg}_{K(T)}K(X) = \dim X,
\] 
where
\begin{itemize}
    \item in the first equality, we used  \cite[\href{https://stacks.math.columbia.edu/tag/02JT}{Tag 02JT}]{stacks-project}
    \item in the second equality, we used that $\text{trdeg}_{\kappa(s)}\kappa(x)=0$ since $X_s$ is a scheme of finite type over $\kappa(s)$ and $x$ is a closed point, and
    \item the last equality is given by \cite[\href{https://stacks.math.columbia.edu/tag/02JX}{Tag 02JX}]{stacks-project}.
\end{itemize}
Now, if $D$ is a divisor of $X$, then $f(D)$ is closed, where $f \colon X \to T$ is the structure morphism. Hence $f(D)$ contains $s \in T$, and so $D$ intersects $X_s$ in a non-empty closed subset of $X$. In particular, $X$ contains a closed point $x \in X_s$, { which must necessarily map to $s \in T$ since $f$ is proper: the argument gives this by construction}.  Then,
\[
\dim X > \dim D \geq \dim \sO_{D, x}=\dim \sO_{X,x}-1 =\dim X -1.
\]
where in the first equality we used that $X$ is catenary. We obtain that $\dim D = \dim X -1$.

}

Since the existence of contractions and flips in the Minimal Model Program can be checked after localisation at each point, in their proofs we may always assume that $T$ is a spectrum of a local domain. However, we cannot reduce to the local situation in the case of the cone theorem and termination of flips. 
\end{remark}

\begin{remark} \label{remark:divisors-of-unexpected-dimension3}
{Let $T$ be a quasi-projective scheme over a finite dimensional excellent ring.} The reader should be wary that $\dim T_{\mathbb{Q}}$ may be equal to $\dim T$ even when $T \neq T_{\mathbb Q}$. For example, take $T = (\mathrm{Spec}\,\mathbb{Z}_p{[[x,y]]}) \backslash (p,x,y)$ which is two-dimensional, as so is $T_{\mathbb Q} = {\mathrm{Spec}\, \mathbb{Z}_p[[x,y]] \otimes \mathbb{Q}_p}$.    In particular, it may happen that given a three-dimensional proper scheme $X$ over $T$, the localisation $X_{\mathbb Q}$ is still three-dimensional.

However, it is always true that $\dim X_{\mathbb{Q}} \geq \dim X - 1$ when all the generic points of $X$ have characteristic $0$. Pick a point $x\in X$ such that $d:=\dim X=\dim \sO_{X,x}=\dim (\sO_{X,x}/P)$ where $P$ is a minimal prime of $\sO_{X,x}$. Now if the residue field $\sO_{X,x}/\m_{x}$ has characteristic zero, then $\sO_{X,x}$ contains $\bQ$ and hence $\dim X_{\bQ}\geq \dim(\sO_{X,x}\otimes\bQ)=\dim\sO_{X,x}=\dim X$. If the residue field $\sO_{X,x}/\m_x$ has characteristic $p>0$, then by our assumption on generic points, we know that $p\notin P$ and thus we can complete $p$ to a system of parameters $(p, x_2,\dots, x_d)$ of the excellent local domain $\sO_{X,x}/P$ and we have $(\sO_{X,x}/P)\otimes \bQ\cong (\sO_{X,x}/P)[1/p]$. Since $p$ is not in any minimal prime $Q$ of $(x_2,\dots,x_d)$ and any such $Q$ has height $d-1$ in $\sO_{X,x}/P$, it follows that $\dim X_{\bQ}\geq \dim((\sO_{X,x}/P)\otimes \bQ) = \dim (\sO_{X,x}/P)[1/p] \geq \dim(\sO_{X,x}/P)_Q=\dim X-1.$
\end{remark}

Given a projective morphism  $f:X\to Z$, 
we say that a $\bQ$-Cartier divisor $D$ is $f$-big if $D|_{X_\eta}$ is big where $\eta$ is the generic point of $f(X)$.  Equivalently, $\rank f_*\sO_X(kD)>ck^{\dim X_\eta}$ for some constant $c$ for $k$ sufficiently large and divisible.    If $D$ is $f$-nef, then $D$ is $f$-big if and only if $D^{\dim(X_\eta)}|_{X_\eta}\neq 0$. We say that an $\mathbb{R}$-Cartier divisor is $f$-big if it can be written as $\sum \alpha_i D_i$, where $\alpha_i \in \bR_{>0}$ and $D_i$ are $f$-big Cartier divisors.

\begin{definition} \label{def:log-pair}
In this article, $(X,\Delta)$ is a (\emph{log}) \emph{pair} if $X$ is a normal Noetherian excellent integral $d$-dimensional scheme with a dualizing complex, $\Delta$ is an effective $\bR$-divisor.  Frequently, but not always, we also assume that $K_X+\Delta$ is $\bR$-Cartier. 

If $\Delta$ is a $\mathbb{Q}$-divisor (resp.\ $\mathbb{R}$-divisor), we call it a $\mathbb{Q}$-boundary (resp.\ $\mathbb{R}$-boundary). 
Outside of \autoref{section:MinimalModelProgram}, we will assume that our boundaries are $\mathbb{Q}$-boundaries unless otherwise stated. We say that $\Delta$ has \emph{standard coefficients} if they are contained in $\{ 1 - \frac{1}{m} \, \mid \, m \in \bZ_{>0}\} \cup \{1\}$. 
\end{definition}

{Before the next definition note that if $X$ is a Noetherian excellent integral scheme of dimension $d$ with a dualizing complex, then the canonical sheaf $\omega_X$ 
introduced in \autoref{sec:dualizing_complexes_local_duality} is $S_2$ by \cite[\href{https://stacks.math.columbia.edu/tag/0AWN}{Tag 0AWN}]{stacks-project}. Additionally $\omega_X^{\mydot}$
is compatible with localization 
\cite[\href{https://stacks.math.columbia.edu/tag/0A7G}{Tag 0A7G}]{stacks-project}. In particular, taking into account the normalization of dualizing complexes  (also explained in \autoref{sec:dualizing_complexes_local_duality}) we obtain that for the generic point $\eta \in X$ we have $\omega_{X,\eta}^{\mydot} \cong \omega_{\eta}^{\mydot}[-d] \cong \sO_{\eta}[d]$ and for any codimension $1$ point $x \in X$ we have $\omega_{X,x}^{\mydot} \cong \omega_{\Spec \sO_{X,x}}^{\mydot} [ -(d-1)]$. So,    if $X$ is normal, then  also at the points of the latter type we have    $\omega_{X,x}^{\mydot} \cong \sO_{\Spec \sO_{X,x}}[{d}]$, and hence $\omega_X$ is a rank $1$ divisorial sheaf   \cite{HartshorneGeneralizedDivisorsOnGorensteinSchemes}.  We denote the corresponding linear equivalence class of divisors by $K_X$. 

If $f : Y \to X$ is a proper birational morphism of Noetherian excellent integral schemes of finite Krull dimension with dualizing complexes, and $\Delta$ is an $\bR$-divisor on $X$ with $K_X + \Delta$ $\bR$-Cartier, then we can find an $\bR$-divisor $\Delta$ satisfying the equation
\begin{equation}
    \label{eq.DefinitionOfDeltaY}
    f^*(K_X+\Delta)=K_Y+\Delta_Y.
\end{equation}
Note that $\Delta_Y$ is uniquely defined if we add the assumption $f_* K_Y = K_X$, which we will always assume in such situations. 
}

\begin{definition} 
\label{def:singularities_of_the_MMP}
Consider a pair  $(X,\Delta)$ with $K_X + \Delta$ being $\bR$-Cartier such that every coefficient in $\Delta$ is at most $1$.
If for every birational morphism $f\colon Y\to X$ {from a normal scheme}, divisor $\Delta_Y$ as in \autoref{eq.DefinitionOfDeltaY} and for every prime divisor $E$ on $Y$ which is exceptional over $X$, we have
\begin{itemize}
    \item $\mathrm{mult}_E(\Delta_Y) <  0$, then $(X,\Delta)$ is \emph{terminal},
    \item $\mathrm{mult}_E(\Delta_Y) \leq 0$, then $(X,\Delta)$ is \emph{canonical},
    \item $\mathrm{mult}_E(\Delta_Y) <1$ and $\lfloor \Delta \rfloor =0$, then $(X,\Delta)$ is \emph{kawamata log terminal} (\emph{klt}),
    \item $\mathrm{mult}_E(\Delta_Y) <1$, then $(X,\Delta)$ is \emph{purely log terminal} (\emph{plt}),
    \item $\mathrm{mult}_E(\Delta_Y) <1$ unless the generic point of the image of $E$ on $X$ is contained in the simple normal crossing locus of $(X,\Delta)$,  then $(X,\Delta)$ is \emph{divisorially log terminal} (\emph{dlt}),
    \item $\mathrm{mult}_E(\Delta_Y) \leq 1$, then $(X,\Delta)$ is \emph{log canonical} (\emph{lc}).
\end{itemize}
\end{definition}
In the first definition, $\lfloor \Delta \rfloor = 0$ is automatic. Further, notice that $(X,\Delta)$ being plt does not imply $\lfloor \Delta \rfloor$ is irreducible for $(X,\Delta)$ plt. {This is not merely a technical subtlety, as otherwise plt would fail to be stable under certain base-changes.
On the other hand, the irreducibility of $\lfloor \Delta \rfloor$ is at times required in a number of standard arguments, which then we have to revise with extra care (c.f. \autoref{Section:flips}).}

We call the number $a(E,X,\Delta)=1-\mult_E(\Delta_Y)$ the \emph{log discrepancy} of $(X,\Delta)$ along $E$ (the number $-\mult_E(\Delta_Y)$ is called \emph{discrepancy}). If $(X,\Delta)$ admits a log resolution $f \colon Y \to X$, then it suffices to verify the above definitions (except the terminal and the plt case) for the divisors on $Y$ only  \cite[Section 2.10]{KollarKovacsSingularitiesBook}.

The base-change properties of the notions defined in \autoref{def:singularities_of_the_MMP} can be deduced from the following lemma.
\begin{lemma}
    \label{lem.BaseChangeOfResolution}
    Suppose $\pi : X \to \Spec R$ is a log resolution of some pair $(\Spec R, \Delta)$.  If $R \to R'$ is a flat map to an  excellent ring with geometrically regular fibers (for instance, an \'etale cover, the strict henselization at some point of $\Spec R$, or the completion thereof), then the base change 
    \[
       \pi' :  X' = X_{R'} \to \Spec R'
    \]
    is a log resolution of the base changed pair $(\Spec R', \Delta_{R'})$.  
\end{lemma}
\begin{proof}
    Since $X$ is regular and $X' \to X$ is flat with regular fibers, we see that $X'$ is regular (and in particular reduced).  But this also applies to all strata of the simple normal crossings divisor $\pi^{-1} \Delta$ and so its base change is also simple normal crossings.  This proves the lemma.
\end{proof}

\begin{remark} Let $(X,\Delta)$ be a three-dimensional klt pair and let $D$ be an effective divisor. Then $(X,\Delta + \varepsilon D)$ is klt for $0 < \varepsilon \ll 1$ as proper resolutions exist in this setting.
\end{remark}

\begin{definition}
We say that a projective birational morphism  $g \colon Y \to X$ is a \emph{terminalization} of a klt pair $(X,B)$ if when writing $K_Y+B_Y=f^*(K_X+B)$, the pair $(Y,B_Y)$ is terminal. 
\end{definition}

\begin{lemma}
\label{lem:plt_dlt_base_change}
Let $f : (X, \Delta) \to Z= \Spec R$ be a projective birational morphism from a three-dimensional plt (resp.\ klt, dlt)  pair to the spectrum of an excellent base ring $R$ with a dualizing complex, and let $h : R \to R'$ be a flat map between excellent local ring s with dualizing complexes and suppose that $h$ has geometrically regular fibers (for instance, an \'etale cover, or the strict henselization at a maximal ideal, or the completion thereof).   Then the base changed pair $\left(X_{R'}, \Delta_{R'} \right)$ is plt (resp.\ klt,\ dlt).
\end{lemma}
\begin{proof}
    This follows from \autoref{lem.BaseChangeOfResolution} since we can check these conditions on a single log resolution.
\end{proof}

\noindent Note that the above lemmas in the smooth case are discussed in \cite[2.14 and 2.15]{KollarKovacsSingularitiesBook}.

 We say that a scheme is \emph{normal up to a universal homeomorphism} if its normalization is a universal homeomorphism.
\begin{lemma} \label{lem:properties-of-plt}
Let $(X,\Delta)$ be a dlt pair such that all the irreducible components $S_1, \ldots, S_k$ of $\lfloor \Delta \rfloor$ are $\bQ$-Cartier. Then all $S_i$ are normal up to a universal homeomorphism (and normal in codimension one). Moreover, if $(X,\Delta)$ is plt, then $\lfloor \Delta \rfloor= S_1 \sqcup \ldots \sqcup S_k$.

The same holds for $(X', \Delta')$, where $\phi \colon X' \to X$ is a flat map with geometrically regular fibers (for example, a completion at a point $x \in X$) and $\Delta'= \phi^*(\Delta)$. 
\end{lemma}
\begin{proof}
The first part follows by exactly the same proof as \cite[Lemma 2.1]{HaconWitaszekMMP4fold} (we learnt this result from J\'anos Koll\'ar). Suppose that $(X,\Delta)$ is plt and $S_i \cap S_j \neq \emptyset$ for some $i\neq j$. Since both $S_i$ and $S_j$ are $\bQ$-Cartier, then $S_i\cap S_j$ contains a codimension two point $\eta$. By localizing at $\eta$, we may assume that $X$ is two-dimensional, and so the result follows from the classification of plt surfaces (cf.\ \cite[Theorem 2.31]{KollarKovacsSingularitiesBook}). By the same argument $S_i$ are normal in codimension one.

To prove the last statement, we may assume that $x \in S_i$. Since normalizations are stable under flat maps with geometrically regular fibers (cf.\ $S_2$ is preserved under flat maps by \cite[15.1, 23.3]{MatsumuraCommutativeRingTheory},  $R_1$ is preserved by the argument of \autoref{lem.BaseChangeOfResolution}),
we get that $S'_i = \phi^*(S_i)$ is normal up to a universal homeomorphism. In particular, $S'_i$ is a disjoint union of its irreducible components. 
\end{proof}

\begin{lemma}[{{\cite[Lemma 9.2]{Birkar16}}}]\label{lemma:add_ample}
        
    {Let $g: (X,B) \to \Spec R$ be a projective morphism from a klt (resp. plt, dlt) pair with a $\bQ$-boundary over a Noetherian local domain.}
    Suppose that there exists $g : W \to X$, a log resolution of $(X, B)$ and of $X_{\fram}$ such that there exists a $g$-exceptional divisor $E \geq 0$ on $W$ such that $-E$ is ample.  In the case that $(X, B)$ is dlt, additionally assume that this resolution has no exceptional divisors with discrepancy $-1$ (this condition is automatic for the other cases).  Finally suppose that $A$ is an ample divisor on $X$.  Then there exists a divisor $0 \leq A' \sim_{\bQ} A$ such that $(X, B+A')$ is klt (respectively plt, dlt)
\end{lemma}

\begin{proof}
    The proof follows \cite[Lemma 9.2]{Birkar16} (mimicking his argument from the dlt case) {with the following adjustments:}  we set $E' := \frac{E}{m}$ for some $m \gg 0$, and we use our Bertini theorems \autoref{thm.FinalBertini} (in particular \autoref{log_bertini}) where the ``general'' $A_W'$ is chosen.
\end{proof}

\begin{remark}
\label{rem:dlt_anti_ample}
    If $X$ is a $\bQ$-factorial threefold, then any projective resolution $\pi : Y \to X$ in the klt/plt case will satisfy the hypotheses of \autoref{lemma:add_ample}.  Indeed, if $H$ on $Y$ is relatively ample, then $H - \pi^* \pi_* H$ will be relatively ample and $\pi$-exceptional. {When $X$ is not necessarily $\bQ$-factorial, the existence of such a resolution locally is guaranteed for Noetherian quasi-excellent three-dimensional reduced schemes by Proposition \ref{proj-resolutions}. Unfortunately, in contrast to positive or zero characteristics, we do not know if the resolution as in the dlt case above exists in dimension three.     } 
            \end{remark}

Given a proper birational map $f \colon Y \to X$ between normal integral schemes over $\Spec R$, a Cartier divisor $D$ on $X$, and an exceptional effective divisor $E$ on $Y$, we have that $f_*\sO_Y(f^*D+E) = \sO_X(D)$. The following result, used extensively throughout this paper, is an easy generalisation of the above fact to $\bQ$-Cartier divisors.
\begin{lemma} \label{lemma:pushforward} Let $f \colon Y \to X$ be a proper birational morphism between {normal Noetherian schemes}.
Let $D_Y$ and $D_X$ be $\bQ$-Cartier Weil divisors on $Y$ and $X$, respectively, such that $f_*D_Y = D_X$ and $D_Y \geq \lfloor f^*D_X \rfloor$ (equivalently, $\lceil D_Y - f^*D_X \rceil \geq 0$). Then $f_*\sO_Y(D_Y) = \sO_X(D_X)$.
\end{lemma}

The aim of the log minimal model program is to take a projective scheme with mild singularities and perform certain birational operations on it, to arrive at a projective scheme of the one of the following two special kinds. Here, a morphism $f \colon X \to Z$ is called \emph{a contraction} if it is projective and satisfies $f_* \sO_X = \sO_Z$.
\begin{definition}\label{def:lmm}
Let $(X,\Delta)$ be a pair and $f:X\to Z$ a projective contraction.  We say that $(Y,\Delta_Y)$ with projective contraction $g:Y\to Z$ is a \emph{log birational model} of $(X,\Delta)$ over $Z$ if $X$ is birational to $Y$ and $\Delta_Y$ is the sum of the birational transform of $\Delta$ and the reduced exceptional divisor of $Y\dashrightarrow X$.

We say that a log birational model $(Y,\Delta_Y)$ is a \emph{log minimal model} of $(X,\Delta)$ over $Z$ if 
\begin{enumerate}
\item $(Y,\Delta_Y)$ is $\mathbb{Q}$-factorial dlt,
\item $K_Y+\Delta_Y$ is nef over $Z$,
\item for any divisor $E$ on $X$ which is exceptional over $Y$, $a(E,X,\Delta)<a(E,Y,\Delta_Y)$, and 
\item\label{itm:lmm_contract_divisors} the induced map $Y\dashrightarrow X$ does not contract any divisors.
\end{enumerate}

We say that a log birational model $(Y,\Delta_Y)$ is a \emph{Mori fiber space} for $(X,\Delta)$ over $Z$ if 
\begin{enumerate}
\item $(Y,\Delta_Y)$ is $\mathbb{Q}$-factorial dlt,
\item there is a projective contraction $\phi \colon Y\to V$ over $Z$ such that
\begin{itemize}
\item the contraction $\phi$ is $(K_Y+\Delta_Y)$-negative,
    \item $\dim(V)<\dim(Y)$,
\item $\rho(Y/V)=1$,
\end{itemize}
\item for any divisor $E$ on $X$ which is exceptional over $Y$, $a(E,X,\Delta)<a(E,Y,\Delta_Y)$, and
\item the induced map $Y\dashrightarrow X$ does not contract any divisors.
\end{enumerate}

\end{definition}
\noindent If $(X,\Delta)$ is klt, then so is $(Y,\Delta_Y)$. We say that a log minimal model $(Y,\Delta_Y)$ of $(X,\Delta)$ is \emph{good} if $K_Y+\Delta_Y$ is semiample.

\begin{remark}
Note that for some authors e.g. \cite{Birkar16}, \autoref{def:lmm}\autoref{itm:lmm_contract_divisors} is not assumed in these definitions.  We include this assumption since the log minimal models and Mori fiber spaces we construct will satisfy this.\end{remark}

\begin{definition}
A \emph{flipping contraction} $f \colon X \to Z$ of a pair $(X, \Delta)$ is a small projective birational morphism such that $-(K_X+\Delta)$ is $f$-ample. \end{definition}

Note that it is usually assumed, and is the case when running the usual LMMP, that $\rho(X/Z)=1$.  However, we will need to make use of the above more general notion. 

\begin{definition}
Given a flipping contraction $f\colon X\to Z$ of a pair $(X,\Delta)$, the \emph{flip} of $f$ (if it exists) is a small projective birational morphism $f^+\colon X^+\to Z$ such that $K_{X^+}+\Delta_{X^+}$ is $f^+$-ample.\footnote{Notice that this $X^+$ is not the one corresponding to the absolute integral closure of $\sO_{X}$.}
\end{definition}

\subsection{Minimal Model Program for Noetherian excellent surfaces}\label{sec:surface_mmp}
We review the Minimal Model Program for Noetherian excellent surfaces following \cite{tanaka_mmp_excellent_surfaces}. Throughout this subsection {the base ring} $R$ is assumed to be a finite dimensional, excellent ring {admitting a dualizing complex}, and $T$ to be a quasi-projective scheme over $R$. In particular, this covers the key cases from the viewpoint of applications such as when
\begin{itemize}
    \item $T$ is a quasi-projective scheme over a field or a Dedekind domain, or
    \item $T= \Spec A$ for any complete Noetherian  local domain $A$ (see \cite[Tag 032D]{stacks-project}).
\end{itemize}

\begin{remark}
{Note that the assumption in \cite{tanaka_mmp_excellent_surfaces} is that the base ring $R$ is regular.  However all the arguments go through with the weaker assumption that $R$ admits a dualizing complex \cite{TanakaPrivateCommunication}.} 
\end{remark}

\begin{theorem}[{MMP, \cite[Theorem 1.1]{tanaka_mmp_excellent_surfaces}}]
\label{thm:surface-excellent-mmp}
Let $(X,\Delta)$ be a log canonical pair over $R$ of dimension two with $\bR$-boundary and admitting a projective morphism $f \colon X \to T$. Then we can run a $(K_X+\Delta)$-MMP over $T$ which terminates with a minimal model or a Mori fibre space.
\end{theorem}

\begin{theorem}[{$\bQ$-factoriality of dlt singularities, \cite[Corollary 4.11]{tanaka_mmp_excellent_surfaces}, cf.\  \cite{LipmanRationalSingularities}}] \label{thm:surface-Qfactoriality-of-dlt}
Let $(X,\Delta)$ be a two-dimensional dlt pair with $\bR$-boundary. Then $X$ is $\bQ$-factorial.
\end{theorem}

\begin{theorem}[{Base point free theorem, \cite[Theorem 4.2]{tanaka_mmp_excellent_surfaces}}] \label{thm:surface-bpf-theorem} Let $(X,B)$ be a klt pair of dimension two with $\bR$-boundary and admitting a projective morphism $f \colon X \to T$  over $R$. Let $L$ be an $f$-nef $\bQ$-Cartier divisor such that $L-(K_X+B)$ is $f$-nef and $f$-big. Then $L$ is $f$-semiample.
\end{theorem}
\begin{proof}
When $X$ is projective over a field, this follows from abundance (\cite[Theorem 1.1]{TanakaAbundanceImperfectFields}). Specifically, let $E$ be an effective divisor such that $A_\varepsilon=L-(K_X+B)-\varepsilon E$ is ample for all $\varepsilon$ sufficiently small.  Fix $\varepsilon$ such that $(X,B+\varepsilon E)$ is klt and by \autoref{lemma:add_ample} choose  $0\leq A'\sim_{\mathbb{Q}} A_{\varepsilon}$ such that $(X,B+\varepsilon E+A')$ is klt.  Then we can conclude by \cite[Theorem 1.1]{TanakaAbundanceImperfectFields} using the fact that $L\sim_{\mathbb{Q}} K_X+B+\varepsilon E+A'$.
If $X$ is not projective over a field the result is implied by \cite[Theorem 4.2]{tanaka_mmp_excellent_surfaces}. 
\end{proof}
\noindent Note that when $X$ is not defined over a field we even know that $nL$ is base point free for all $n\gg 0$ and not just divisible enough. Unfortunately, this does not hold in general, specifically when the numerical
dimension of $L$ is equal to one and the base field has characteristic two and three (see \cite[Theorem 1.2]{tanaka_pathologies}).

The following theorem is well-known in characteristic zero, and has been recently established for varieties which are projective over a field of positive characteristic \cite{TanakaAbundanceImperfectFields}. We prove the general case later on.
\begin{theorem}[\autoref{abundance}]
Let $(X,\Delta)$ be a log canonical pair of dimension $2$, projective over $T$ with $\mathbb{Q}$-boundary. Assume in addition that $T$ is the spectrum of a local ring with positive residue characteristic.  If $K_X+\Delta$ is nef over $T$, then it is semiample over $T$.
\end{theorem}

We present a strengthening of \cite[Theorem 2.14]{tanaka_mmp_excellent_surfaces} following \cite[Theorem 4.3]{DW19}.
As our residue fields are not necessarily algebraically closed, the bound on the length of extremal rays involves a term $d_C$ introduced in [op.\ cit.].

\begin{theorem}[Cone theorem] \label{thm:surface_cone}
Let $\pi:X\to T$ be a projective morphism with $X$ integral, normal, and of dimension at most  two. Let 
$\Delta\geq 0$ be such that $K_X+\Delta$ is $\mathbb{R}$-Cartier.  Then there exist countably many curves $\{C_i\}_{i\in I}$ on $X$ such that 
\begin{enumerate}
    \item $\pi(C_i)$ is a closed point\footnote{this is automatic by definition of a curve over $T$}.
    \item $$\overline{\mathrm{NE}}(X/T) = \overline{\mathrm{NE}}(X/T)_{K_X+\Delta\geq 0} + \sum_{i\in\mathbb{N}} \bR_{\geq 0}[C_i].$$
    \item\label{itm:surface_cone_ample} For any ample $\mathbb{R}$-divisor $A$, there is a finite $n$ such that 
    $$\overline{\mathrm{NE}}(X/T) = \overline{\mathrm{NE}}(X/T)_{K_X+\Delta+A\geq 0} + \sum_{i\leq n} \bR_{\geq 0}[C_i].$$
    \item\label{item:bound} For each $C_i$, either 
    \begin{enumerate}
    \item\label{item:case_non_lc} $C_i$ is contained in the non-lc locus of $(X,\Delta)$.
        \item\label{case:bound} $0<-(K_X+\Delta)\cdot_k C_i\leq 4 d_{C_i}$ where $d_{C_i}$ is as in \autoref{lem:d_C}.
    \end{enumerate}
\end{enumerate}
\end{theorem}

\begin{proof}
If $\dim(X)=1$, then the result is obvious. So we assume $\dim(X)=2$.  Furthermore if $\dim(\pi(X))=0$, the result is proved in \cite[Theorem 4.3]{DW19}.  Note that this did not assume that the field had positive characteristic, and while our phrasing of \autoref{item:case_non_lc} is slightly stronger than that of \cite{DW19}, it is actually what is given by the proof there. 
   
So we may assume that $\dim(\pi(X))\geq 1$.  The first three parts are implied by the stronger \cite[Lemma 2.13]{tanaka_mmp_excellent_surfaces}, so it remains to prove \autoref{item:bound}.
    For this we must show that each $(K_X+\Delta)$-negative extremal ray $\Sigma$  contains a curve satisfying the bound or contained in $\Supp(\Delta)$
    Using the argument of \cite[Proposition 4.5, Step 1]{DW19} we may assume that $X$ is regular and $(X,\Delta_{\leq 1})$ is dlt. 
    
    The extremal ray $\Sigma$ contains some curve $C$ by \cite[Lemma 2.13]{tanaka_mmp_excellent_surfaces}, and as $X$ is regular we claim that $C^2\leq 0$.  If $\pi$ is birational this follows from \autoref{lem:negativity}, while if $\pi$ has image of dimension $1$, it follows because $C\cdot F=0$ for $F$ a fiber of $\pi$.  Let $D$ be the normalization of an irreducible component of $(C\otimes_k\overline{k})_{\mathrm{red}}$.  Then $$(K_X+\Delta) \cdot_k C\geq (K_X+\Delta+aC) \cdot_k C=\deg_k(K_C+\Delta_C)\geq d_C\deg_{\overline{k}}(K_D+\Delta_D)\geq -2d_C$$ where $a$ is chosen such that $C$ has coefficient one in $\Delta+aC$, and $\Delta_C$ and $\Delta_D$ are effective divisors on $C$ and $D$ respectively.
\end{proof}

We used the following lemma in the proof of the above theorem.
\begin{lemma}\label{lem:d_C}
Let $X$ be a scheme over a Dedekind domain $V$ containing a proper curve $C$ over a point $v\in \Spec(V)$ with residue field $k$.  Let $\phi: X_v\otimes_k\overline{k}\to X_v$ be the natural projection.
Then there is a positive integer $d_C$ such that
for any $\mathbb{R}$-Cartier divisor $D$, if $C^{\overline{k}}$ is any integral curve on $X_v\otimes_k\overline{k}$ whose image on $X_v$ is $C$ we have $$D\cdot_k C=d_C(\phi^*D\cdot_{\overline{k}} C^{\overline{k}})$$
In particular if $L$ is any Cartier divisor on $X$, then $L\cdot_k C$ is divisible by $d_C$.
\end{lemma}
\begin{proof}
    This is \cite[Lemma 4.1]{DW19} applied to $C\subset X_v$.  Note that the statement of \cite[Subsection 3]{DW19} required that $X_v$ be proper, however the proofs only require that the curve $C$ be proper.
\end{proof}

\subsection{Characteristic zero base point free theorem}

We note that the base point free theorem for Noetherian excellent schemes of characteristic zero follows from the vanishing theorems in \cite{takumi}.

\begin{proposition}\label{prop:char_zero_bpf}
Suppose that $T$ is a scheme which is quasi-projective over a finite dimensional excellent ring $R$ {admitting a dualizing complex and containing} $\mathbb{Q}$. 

Let $(X,\Delta)$ be a {$\mathbb{Q}$-factorial} klt pair with $\bR$-boundary. Let $f:X\to T$ be a  projective morphism, and let $L$ be an $f$-nef $\bQ$-Cartier $\bQ$-divisor on $X$ such that $L-K_X-\Delta$ is $f$-nef and $f$-big.  Then $L$ is $f$-semiample. 
\end{proposition}
\begin{proof}
{By a perturbation we may assume that $\Delta$ is a $\mathbb{Q}$-divisor and $L-K_X-\Delta$ is ample.}
    We may assume that $T$ is integral, and then use the argument of \cite[Theorem 3-1-1]{KMM}.  This has three main imputs: relative Kawamata-Viehweg vanishing \cite{takumi}, the existence of a projective resolution with ample exceptional divisor (\cite{TemkinAbsoluteChar0}), and the non-vanishing theorem on the generic fiber $X_\eta$ of $X\to T$ (that is, $H^0(X_{\eta}, \mathcal{O}_{X_{\eta}}(mL)) \neq 0$ for some $m \geq 1$).  As this generic fiber is a variety over a field of characteristic zero, the non-vanishing theorem \cite[Theorem 2-1-1]{KMM} applies directly via the base change of its Stein factorization to the algebraic closure of $K(T)$.
\end{proof}

\subsection{Mixed characteristic Keel's theorem} \label{ss:preliminaries-mixed-Keel-theorem}

{In what follows, we say that a nef Cartier divisor $L$ on a scheme $X$ proper over a Noetherian excellent base scheme $T$ is \emph{EWM} over $T$ if there exists a proper morphism $f \colon X \to Y$ to a proper (over $T$) algebraic space $Y$ such that a closed integral subscheme $V \subseteq X$ is contracted (that is, $\dim f(V) < \dim V$) if and only if $L|_V$ is not big.}
{\begin{remark}
The original definition of EWM in \cite{Witaszek2020KeelsTheorem} differed from the one above (which is the same as in \cite{CasciniTanaka2020,KeelBasepointFreenessForNefAndBig}). It was weaker, as it only required $f$ to contract proper curves $C$ such that $L \cdot C = 0$. This was corrected in an update to \cite{Witaszek2020KeelsTheorem}. 
\end{remark}}
We start by recalling the main results of \cite{Witaszek2020KeelsTheorem}.
\begin{theorem}[{\cite[Theorem 6.1]{Witaszek2020KeelsTheorem}}]
Let $L$ be a nef Cartier divisor on a scheme $X$ projective over a Noetherian excellent base scheme $T$. Then $L$ is semiample (EWM.\ resp.) over $T$ if and only if  $L|_{\mathbb{E}(L)}$ and $L|_{X_{\bQ}}$ are semiample (EWM.\ resp.) over $T$. \label{thm:mixed-characteristic-Keel} 
\end{theorem}
\noindent Here, $X_{\bQ}$ denotes the characteristic zero fiber of $X \to \Spec \bZ$ and  $\mathbb{E}(L)$ denotes the union of closed integral subschemes $V \subseteq X$ such that $L|_V$ is not relatively big over $T$. 
\begin{proof}
This is \cite[Theorem 6.1]{Witaszek2020KeelsTheorem}. Note that the EWM case of this theorem assumed that the base scheme $T$ is of finite type over a mixed characteristic Dedekind domain. This assumption was needed to invoke \cite[Theorem 3.1 and Theorem 6.2]{ArtinGluing}, but the only reason Artin stated it in his article was because the Popescu approximation theorem was not known at that time (\cite[Tag 07GC]{stacks-project}). {This assumption was retained in \cite{Witaszek2020KeelsTheorem} out of abundance of caution.}
\end{proof}

\begin{proposition} \label{prop:bpf_plt} Let $T$ be a quasi-projective scheme over a finite dimensional excellent ring $R$ admitting a dualizing complex. Let $(X,S+B)$ be a three-dimensional dlt pair which is projective over $T$, where $S$ is a prime divisor and $B$ is an effective $\bQ$-divisor. Suppose that each irreducible component of $\lfloor S+B \rfloor$ is $\bQ$-Cartier. Let $L$ be a nef Cartier divisor on $X$ such that $L-(K_X+S+B)$ is ample and $\mathbb{E}(L) \subseteq S$. Then $L$ is semiample. 

Moreover, if $\phi \colon X \to Z$ is the associated semiample fibration, then every relatively numerically trivial $\bQ$-Cartier $\bQ$-divisor $D$ on $X$ descends to $Z$.
\end{proposition}
\begin{proof} By means of perturbation,  we can assume that $(X,S+B)$ is plt  and $S=\lfloor S + B \rfloor$. By \autoref{lem:properties-of-plt}, we also know that $S$ is normal up to a universal homeomorphism. Since $L|_{X_{\bQ}}$ is semiample by \autoref{prop:char_zero_bpf}, it is enough to show that $L|_{\mathbb{E}(L)}$ is semiample by \autoref{thm:mixed-characteristic-Keel}, and so that $L|_S$ is semiample. First, note that $L|_{\tilde S}$ is semiample, where $\tilde S$ is the normalization of $S$. 
Indeed, write $K_{\tilde S} + B_{\tilde S} = (K_X+S+B)|_{\tilde S}$. Since $(\tilde S, B_{\tilde S})$ is klt { and $\dim \tilde S \leq 2$, we have that} $L|_{\tilde S}$ is semiample by \autoref{thm:surface-bpf-theorem}. {Then $L|_S$ is semiample in view of $\tilde S \to S$ being a universal homeomorphism by \cite[Theorem 2.22]{Witaszek21}.
}

The second part follows by applying the first part to $L+D$ over $Z$.
\end{proof}

\begin{proposition} \label{prop:Das-Waldron-adjunction}
Let $(X,S+B)$ be a pair with $K_X + S + B$ $\bR$-Cartier, and with $X$ projective over a Noetherian  excellent scheme $T$ admitting a dualizing complex such that $S$ is a Weil divisor not contained in $\Supp(B)$ whose image in $T$ is a closed point with residue field $k$.  Let $Z$ be the normalization of $S_{\overline{k}}$.  
Then there are effective divisors $C$, $M$ and $F$, and a $\mathbb{R}$-divisor $B_Z$ on $Z$ such that

$$(K_X+S+B)|_{Z}\sim_{\mathbb{R}}K_{T}+C+M+F+B_Z$$
where
\begin{itemize}
    \item $\Supp(C)$ is the {pullback to $Z$ of the} locus on which the normalization $S^\nu\to S$ fails to be an isomorphism.
    \item $\Supp(F)$ is the locus on which $Z\to ((S^\nu)_{\overline{k}})_{\mathrm{red}}$ fails to be an isomorphism.
    \item $\Supp(M)=0$ if and only if $S_{\overline{k}}$ is reduced.
\end{itemize}
\end{proposition}
\begin{proof}
    First, by adjunction, $(K_X+S+B)|_{S^\nu}=K_{S^\nu}+C_{S^\nu}+B_S$ where $C_{S^\nu}\geq 0$ is the conductor {of the normalization $S^\nu\to S$} and $B_S\geq 0$.
        Then we have $K_{S^\nu}|_Z=K_Z+M+F$ where $M$ and $F$ are elements of the linear systems $(p-1)\mathfrak{F}$ and $(p-1)\mathfrak{M}$ from \cite{ji_waldron}.  Note that \cite{ji_waldron} assumes that the ground field is a function field, but our situation can be reduced to this as explained in \cite[Subsection 2.1]{ji_waldron} and \cite[Theorem 4.12, Step 1, (1)]{DW19}.
\end{proof}

\begin{corollary} \label{cor:EWM-bpf-theorem}
\label{thm:keel_EWM_imperfect}
Let $(X,B)$ be a klt pair of dimension three admitting a projective morphism $f \colon X \to T$ to a {finite dimensional} Noetherian excellent scheme $T$ . Let $L$ be an $f$-nef and $f$-big  Cartier divisor such that $L - (K_X+B)$ is $f$-nef and $f$-big as well. Then $L$ is EWM over $T$.
\end{corollary}
\begin{proof}
This is proven in \cite[Corollary 6.7]{Witaszek2020KeelsTheorem} under the assumption that $T$ is a spectrum of a mixed characteristic Dedekind domain with perfect residue fields. 

The fact that the base is a Dedekind domain was used three times in the proof: to employ the mixed characteristic Keel theorem, to invoke \cite[Proposition 6.6]{Witaszek2020KeelsTheorem}, and to deduce that $L|_{X_{\bQ}}$ is semiample. These results hold in our more general setting by \autoref{thm:mixed-characteristic-Keel}, \autoref{lem:auxiliary-for-EWM-bpf-theorem}, {and by \autoref{prop:char_zero_bpf}}, respectively. Note that $L$ is semiample over every non-closed point of $T$ by \autoref{thm:surface-bpf-theorem}.

The assumption on the residue fields was  used  to deduce that $L$ restricted to  an appropriately chosen surface $D_i \subseteq X$, {which is projective over a field} of positive characteristic, is EWM. 
This can be resolved by arguing as in Case 1 of \cite[Theorem 4.12]{DW19}. {Indeed, the semiampleness of $L|_{(D_i)_{\overline k}}$ follows by the same argument as that of $L|_{D_i}$ in
\cite[Corollary 6.7]{Witaszek2020KeelsTheorem} thanks to \autoref{prop:Das-Waldron-adjunction}. Here $\overline{k}$ is the algebraic closure of the base field $k$. Then $L|_{D_i}$ is semiample  by \cite[Lemma 2.2]{KeelBasepointFreenessForNefAndBig}.}
\end{proof}

\begin{lemma}[{\cite[Proposition 6.6]{Witaszek2020KeelsTheorem}}]
Let $X$ be a two-dimensional normal integral scheme projective and surjective over a Noetherian excellent scheme $T$ such that $\dim T \geq 1$. Let $L$ be a line bundle on $X$ which is nef over $T$ and suppose that $L|_{X_{\eta}}$ {(and $L|_{X_{\bQ}}$ if $X_{\bQ} \neq \emptyset$) are} semiample for the fiber $X_{\eta}$ over the generic point $\eta \in T$. Then $L$ is EWM { over $T$}.     \label{lem:auxiliary-for-EWM-bpf-theorem}
\end{lemma}
\begin{proof}
{Replacing $T$ by the Stein factorization of $f:X\to T$,  we may assume that $T$ is normal and $f_*\sO_X=\sO_T$.} If  $L|_{X_{\eta}}$ is big or $\dim T = 2$, then $\dim \mathbb{E}(L) = 1$. Thus $L|_{\mathbb{E}(L)}$ is EWM, and so $L$ is EWM by \autoref{thm:mixed-characteristic-Keel}. Otherwise, $\dim T = 1$, $\dim X_{\eta} = 1$, and $L|_{X_{\eta}} \sim_{\bQ} 0$. {In this case, the normality of $T$ ensures that $T$ is regular and so} we can apply \cite[Lemma 2.17]{CasciniTanaka2020} to deduce that $L$ is relatively torsion. 
\end{proof}

\subsection{Seshadri constants}
\label{ss:seshadri-constants}
Recall that for a projective scheme $X$ over a Noetherian excellent {base scheme $T$}, a nef and big $\bQ$-Cartier $\bQ$-divisor $A$, and a closed point $x \in X$, we define the \emph{Seshadri constant} 
\[
  \epsilon(A;x) = \sup\left\{t \in \bQ \mid \pi^*A - tE \text{ is nef}\,\right\},
\]
where $\pi \colon X' \to X$ is the blow-up of $x$ and $\sO_X(-E) = \fram_x \cdot \sO_{X'}$ is the  exceptional divisor. When $A$ is in addition semiample, then, with notation as above, we also define the \emph{semiample Seshadri constant} 
\[
  \epsilon_{\mathrm{sa}}(A;x) = \sup\left\{t \in \bQ \mid \pi^*A - tE \text{ is semiample}\,\right\}.
\]

In particular, the Seshadri and the semiample Seshadri constants are non-negative, and positive if $A$ is ample. Further, note that $\epsilon(A+B;x) \geq \epsilon(A;x) + \epsilon(B;x)$ (resp.\ $\epsilon_{\mathrm{sa}}(A+B;x) \geq \epsilon_{\mathrm{sa}}(A;x) + \epsilon_{\mathrm{sa}}(B;x)$), where $A$ and $B$ are  nef and big (resp.\ semiample and big) $\bQ$-Cartier $\bQ$-divisors on $X$.
 
For the proof of the existence of flips, we will need the following results.
\begin{lemma}
\label{lem:birational-surfaces-with-rational-exceptional}
Let $f \colon Y \to X$ be projective birational morphism, where $Y$ is a two-dimensional regular integral scheme, and $X$ is affine and klt. Assume that the reduced exceptional divisor $F$ is of positive characteristic. 

Then every nef Cartier divisor $L$ on $Y$ is relatively semiample over $X$. In particular, if $A$ is a semiample $\bQ$-Cartier $\bQ$-divisor on $Y$, then $\epsilon_{\mathrm{sa}}(A;x) = \epsilon(A;x)$ for every closed point $x \in F$.\end{lemma}
\noindent Since $f$ is birational, every $\bQ$-Cartier $\bQ$-divisor is automatically big over $X$.
\begin{proof}
Since semiampleness is stable under strict henselization, we can assume that $X$ is strictly henselian. Note that  $F$ is simple normal crossing and is a tree of regular conics, because the morphism $f$ may be constructed from the minimal resolution of $X$ by successively blowing up closed points, and the claim holds for the minimal resolution of $X$ by  \cite[Section 3]{KollarKovacsSingularitiesBook}. 
With notation as in \autoref{ss:preliminaries-mixed-Keel-theorem}, we have that $\mathbb{E}(L) \subseteq F$.
Hence, by \autoref{thm:mixed-characteristic-Keel}, it is enough to show that $L|_F$ is semiample.  To do this we may assume that $F$ is contracted to a single point $x$ with separably closed residue field $k$. By \cite[Lemma 4.4]{DW19}, $L$ is semiample on every irreducible component of $F$, and so $L|_F$ is semiample by \cite[Corollary 2.9]{KeelBasepointFreenessForNefAndBig} as $F$ is a tree of regular conics over a separably closed field, and so the intersection points are geometrically connected.
\end{proof}

\begin{lemma}
\label{lem:flips_seshadri_lower_bound}
Let $f \colon Y \to X$ be a projective birational morphism, where $Y$ is a two-dimensional regular integral scheme, and $X$ is affine and klt. Assume that the reduced exceptional divisor $F$ is of positive characteristic. 

Let $M$ be an effective semiample Cartier divisor on $Y$ with no exceptional curve of $Y \to X$ in its support, and let $x \in M \cap F$ be of multiplicity $k \in \bZ_{>0}$ in $M$. Then 
\[
\epsilon_{\mathrm{sa}}(M;x) = \epsilon(M;x) \geq k.
\]

More generally, let $D$ be a fixed divisor and let $A$ be a semiample $\bQ$-Cartier $\bQ$-divisor such that $A \sim_{\bQ} M + \Lambda$, where $M$ is an effective Cartier divisor with no exceptional curve of $Y \to X$ in its support, and $ -\delta D\leq \Lambda \leq \delta D$ for $\delta >0$. Take $x \in F \cap M$ of multiplicity $k \in \bZ_{>0}$ in $M$. Then $\epsilon_{\mathrm{sa}}(A;x)$
converges to $k$ when $\delta \to 0$.
\end{lemma}
\begin{proof}
We show the second statement. Then the first one follows by the same argument. Suppose that $0 < \delta \ll \gamma \ll 1$ and let $\pi \colon W \to Y$ be the blow-up at $x$. Since $x \in M$ is of multiplicity $k$, we have that $\pi^*M = M_W + kE$, where $M_W$ is the strict transform of $M$ and $E$ is the exceptional divisor of the blow-up $\pi$.

By \autoref{lem:birational-surfaces-with-rational-exceptional}, it is enough to verify that $\epsilon(A;x) \geq k - \gamma$, that is
\[
\pi^*A - (k-\gamma) E
\]
is nef. Let $C$ be an exceptional irreducible curve on $W$ over $X$. We need to check that $(\pi^*A - (k-\gamma) E) \cdot C \geq 0$. We consider the following cases:
\begin{itemize}
    \item $C=E$, then
      \[
        (\pi^*A - (k-\gamma) E) \cdot C = -(k-\gamma)E^2 >0,
      \]
      \item $C \neq E$ and $C \cap E \neq \emptyset$, then
      {\begin{align*}
        (\pi^*A - (k-\gamma) E) \cdot C &= (\pi^*(M+\Lambda) - (k-\gamma) E) \cdot C\\
        &= (M_W + \gamma E + \pi^*\Lambda )\cdot C \\
        &\geq (\gamma E + \pi^*\Lambda )\cdot C \\
        &= \gamma E \cdot C + \Lambda \cdot \pi_*C\\ 
        &\geq\gamma E \cdot C - \delta \vert D \cdot \pi_*C \vert\\
        &\geq 0,          
      \end{align*}}
      for $0 < \delta \ll \gamma \ll 1$ where the last inequality follows as $E \cdot C \geq 1$, $D$ is fixed, and there are only finitely many possible curves $C$. The first inequality follows as $M_W$ contains no curves in its support which are exceptional over $X$, and so $C \not \subseteq \Supp M_W$.
      \item$C \neq E$ and $C \cap E = \emptyset$, then
      \[
(\pi^*A - (k-\gamma) E) \cdot C = A \cdot \pi_*C \geq 0
\]
as $A$ is nef.
\end{itemize}
\end{proof}

%% file: vanishing.tex
\section{Vanishing in mixed characteristic}
\label{sec.Vanishing}

The goal of this section is to extend the first author's vanishing theorem \cite[Theorem 6.28(b)]{BhattAbsoluteIntegralClosure} from the case of essentially finitely presented algebras over excellent henselian DVRs in mixed characteristic\footnote{In fact, any DVR of mixed characteristic $(0, p>0)$ is excellent, see \cite[Tag 07QW]{stacks-project}.} to the case of arbitrary excellent local domains of mixed characteristic. As in the corresponding local story in \cite[\S 5]{BhattAbsoluteIntegralClosure}, our main tools are Popescu's approximation theorem \cite[Tag 07BW]{stacks-project} together with limit arguments \cite[Tag 01YT]{stacks-project}. We follow the notation from \cite{BhattAbsoluteIntegralClosure} in this section; in particular, we write $X_{p=0} := X \times_{\mathrm{Spec}(\mathbf{Z})} \mathrm{Spec}(\mathbf{Z}/p)$ for any scheme $X$.

\begin{proposition}
    \label{prop.BhattVanishing}
    Suppose that $(T,x)$ is an excellent local domain of mixed characteristic $(0,p>0)$ that admits a dualizing complex. Let $\pi : X \to \Spec(T)$ be a proper surjective map with $X$ reduced, equidimensional and $p$-torsion free. Suppose that $L \in \Pic(X)$ is a semiample line bundle. 
    \begin{enumerate}
        \item There exists a finite surjective map $Y\to X$ such that 
        \[ \tau^{>0}\myR\Gamma(X_{p=0}, L^a) \to \tau^{>0}\myR\Gamma(Y_{p=0}, L^a)\] 
        is $0$ for all $a\geq 0$. In particular,  
        \[\myH^j(\myR\Gamma(X^+_{p=0}, L^a))=0\] 
        for all $j>0$ and all $a\geq0$.
        
        \item If $L$ is also big, then for all $b<0$ there exists a finite surjective map $Y\to X$ such that
   \[
        \myR \Gamma_x(\myR \Gamma(X_{p = 0}, L^b)) \to \myR \Gamma_x(\myR \Gamma(Y_{p = 0}, L^b))
    \]
    is the zero map on $H^j$ for $j < \dim(X_{p=0})$.  In particular,  $$H^{j}(\myR \Gamma_x(\myR\Gamma(X^+_{p=0}, L^b)))=0$$ for all $j < \dim(X_{p=0})$ and all $b<0$.
    \end{enumerate}
\end{proposition}

In what follows, we will only explain part (b) carefully. Part (a) follows from a similar and slightly easier argument so we omit it. We begin by proving a variant of \cite[Theorem 6.28(b)]{BhattAbsoluteIntegralClosure} where we allow non-closed points and do not require that the base DVR is henselian.

\begin{proposition}
    \label{prop.BhattVanishingAtNonClosed}
    Let $V$ be an excellent DVR of mixed characteristic $(0,p>0)$ and let $\pi : X \to \Spec(T)$ be a proper surjective map of integral flat finitely presented $V$-schemes.  Fix a (not necessarily closed) point $x \in \Spec(T)_{p = 0}$ and a big and semiample line bundle $L \in \Pic(X)$.  Then for all $b<0$ there exists a finite surjective map $Y \to X$ such that
    \[
        \myR \Gamma_x(\myR \Gamma(X_{p = 0}, L^b) \otimes_T T_x) \to \myR \Gamma_x(\myR \Gamma(Y_{p = 0}, L^b)\otimes_T T_x)
    \]
    is the zero map on $H^i$ for $i < \dim((X \times_T T_x)_{p=0})$.  Here $T_x$ is the localization of $T$ at the prime ideal $x$.
\end{proposition}
\begin{proof}
Without loss of generality, we can assume $X$ is normal. We first assume $x$ is a closed point. Let $V^h$ be the henselization of $V$. So $V^h=\colim V_j$ where each $V_j$ is a pointed \'{e}tale extension of $V$. We have a commutative diagram
\[\xymatrix{
X \ar[r]  & \Spec(T) \ar[r]  & \Spec(V) \\
X_j \ar[r]  \ar[u]& \Spec(T_j) \ar[r]  \ar[u]& \Spec(V_j)\ar[u] \\
X' \ar[r]  \ar[u]& \Spec(T') \ar[r] \ar[u]& \Spec(V^h)\ar[u] 
}
\]
such that each square is Cartesian. By \cite[Theorem 6.28(b)]{BhattAbsoluteIntegralClosure} applied to the bottom row of the above diagram,\footnote{Since $X$ is normal, each connected component of $X'$ is integral so technically we are applying \cite[Theorem 6.28(b)]{BhattAbsoluteIntegralClosure} to each connected component of $X'$.} there exists a finite surjective map $Y'\to X'$ such that the map $\myR \Gamma_x(\myR \Gamma(X_{p = 0}, L^b)) \to \myR \Gamma_x(\myR \Gamma(Y'_{p = 0}, L^b))$ is zero on $H^i$ for $i<\dim X'_{p=0}$ (here we abuse notation and use $L$ to denote the corresponding line bundle on $Y'_{p= 0}$). Moreover, we may assume $Y'=Y_j\times_{X_j}X'$ is the base change of a finite surjective map $Y_j\to X_j$ for some index $j$. Since $V/p=V_j/p=V^h/p$, we have $X_{p=0}\cong X_{j,p=0} \cong X'_{p=0}$ and $Y_{j,p=0} \cong Y'_{p=0}$. Thus the map $\myR \Gamma_x(\myR \Gamma(X_{j,p=0}, L^b)) \to \myR \Gamma_x(\myR \Gamma(Y_{{j},p=0}, L^b))$ is zero on $H^i$ for $i<\dim X_{p=0}$. Next we note that by \cite[Lemma 4.4]{BhattAbsoluteIntegralClosure}\footnote{Here we are using the scheme version of \cite[Lemma 4.4]{BhattAbsoluteIntegralClosure}, the proof is the same.}, there exists a finite cover $Y\to X$ such that the base change $Y\times_{X}X_j\to X_j$ factors through $Y_j$. Therefore the map $\myR \Gamma_x(\myR \Gamma({X}_{p = 0}, L^b)) \to \myR \Gamma_x(\myR \Gamma({Y}_{p = 0}, L^b))$ is zero on $H^j$ for $i<\dim X_{p=0}$ as it factors through $\myR \Gamma_x(\myR \Gamma(Y_{j,p=0}, L^b))$.

We next handle the case that $x$ is not necessarily a closed point. By \cite[Lemma 4.8]{BhattAbsoluteIntegralClosure}, there exists an extension of DVRs $V\to W$ that is essentially of finite type and a (flat) finite type $W$-algebra $S$ such that $T_x\cong S_y$ where $y \in \Spec(S)_{p=0}$ is a closed point. Choose $\Tilde{X}$ an integral finitely presented scheme over $S$ (and flat over $W$) such that $\Tilde{X}\times_{\Spec(S)}\Spec(S_y)\cong X\times_{\Spec(T)}\Spec(T_x)$, which is possible as the latter is finitely presented over $S_y$ which is a localization of $S$. Consider the diagram
\[\xymatrix{
X\times_{\Spec(T)}\Spec(T_x) \ar[r]  & \Spec(T_x) \ar[r]  & \Spec(V) \\
\Tilde{X}\times_{\Spec(S)}\Spec(S_y) \ar[r] \ar[u]^\cong \ar[d] & \Spec(S_y) \ar[r] \ar[u]^\cong \ar[d] & \Spec(W) \ar[u] \ar[d]^=\\
\Tilde{X} \ar[r] & \Spec(S) \ar[r] & \Spec{W}
}
\]
By applying the first part above to $\Tilde{X}\to \Spec(S)\to \Spec(W)$ and the closed point $y\in \Spec(S)_{p=0}$, we learn that there exists a finite surjective map $\Tilde{Y}\to \Tilde{X}$ such that the map
\[
    \myR \Gamma_y(\myR \Gamma(\Tilde{X}_{p = 0}, L^b)) \to \myR \Gamma_y(\myR \Gamma(\Tilde{Y}_{p = 0}, L^b))
\] 
is zero on $H^i$ for $i<\dim (\Tilde{X}\times_SS_y)_{p=0}$. Finally, by taking suitable integral closures, we can choose a finite surjective map $Y\to X$ such that $Y\times_{\Spec(T)}\Spec(T_x)$ factors through (in fact, equals) $\Tilde{Y}\times_{\Spec(S)}\Spec(S_y)$, so that the map $\myR \Gamma_x(\myR \Gamma(X_{p = 0}, L^b) \otimes_T T_x) \to \myR \Gamma_x(\myR \Gamma(Y_{p = 0}, L^b)\otimes_T T_x)$ is zero on $H^i$ for $i < \dim((X \times_T T_x)_{p=0})$.
\end{proof}

This directly leads to the following statement.

\begin{corollary}
\label{cor.BhattVanishingAtPluslevel}
    Let $V$ be an excellent DVR of mixed characteristic $(0,p>0)$ and let $\pi : X \to \Spec(T)$ be a proper surjective map of integral flat finitely presented $V$-schemes.  Fix a big and semiample line bundle $L \in \Pic(X)$.  Then for all $b<0$, and all $x\in \Spec(T)_{p=0}$, $H^{i}(\myR \Gamma_x(\myR\Gamma(X^+_{p=0}, L^b) \otimes_TT_x))=0$ for all $i<\dim (X\times_TT_x)_{p=0}$.  
\end{corollary}
\begin{proof}
    Simply notice that 
    \[        
        \myR \Gamma_x(\myR\Gamma(X^+_{p=0}, L^b) \otimes_TT_x) = {\displaystyle \colim_{Y \shortrightarrow X}} \myR \Gamma_x( (\myR \Gamma(Y, L^b))_{p = 0} \otimes_TT_x)      
    \]
where the colimit is over all finite surjective maps $Y \to X$. Now the statement follows from \autoref{prop.BhattVanishingAtNonClosed}.
\end{proof}

\begin{remark}
    In the case that $\dim X = \dim T$, we may interpret \autoref{cor.BhattVanishingAtPluslevel} as saying that $\myR\Gamma(X^+_{p=0}, L^b)$ is a Cohen-Macaulay complex over $T/p$ in the sense of \cite[Definition 2.1]{BhattAbsoluteIntegralClosure}.
\end{remark}

We now extend our results to Noetherian complete local bases. 

\begin{proposition}
\label{prop.BhattVanishingCompleteCase}
    Let $T$ be a complete Noetherian local domain of mixed characteristic $(0,p>0)$.  Let $\pi : X \to \Spec(T)$ be a proper surjective map such that $X$ reduced, equidimensional, and $p$-torsion free.  Fix a big and semiample line bundle $L \in \Pic(X)$.  Then for all $b<0$ and all $j<\dim X_{p=0}$, $H^{j}(\myR \Gamma_x(\myR\Gamma(X^+_{p=0}, L^b)))=0$  where $x\in \Spec(T)$ is the closed point. 
\end{proposition}
\begin{proof}
    The strategy is similar to that of \cite[Theorem 5.1]{BhattAbsoluteIntegralClosure}. By Cohen's structure theorem, we may assume that $T$ is finite over a power series ring $V\llbracket x_2, \dots, x_n \rrbracket$ where $V$ is a coefficient ring of $T$ (hence a complete DVR). Thus without loss of generality, we may assume $T=V\llbracket x_2, \dots, x_n \rrbracket$.      Moreover, we may replace $X$ by its normalization and work with each connected component to assume $X$ is normal and integral.

 By Popescu's theorem \cite[Tag 07GC]{stacks-project}, we can write $T=\colim Q_i$ such that $Q_0=V[x_2,\dots,x_n]$ and each $Q_i$ is smooth over $Q_0$. Since $X\to \Spec(T)$ is proper and surjective, we may assume that there exists a proper surjective map $X_i\to\Spec(Q_i)$ such that $X\cong X_i\times_{\Spec(Q_i)}\Spec(T)$ and the line bundle $L$ is pulled back from some big and semiample line bundle $L_i$ on $X_i$, {see for instance \cite[Lemma 4.1]{takumi}}. Now by \autoref{cor.BhattVanishingAtPluslevel} applied to $X_i\to \Spec(Q_i) \to \Spec(V)$, we know that for all $b<0$ and all $y\in \Spec(Q_i)_{p=0}$, we have $H^{j}(\myR \Gamma_y(\myR\Gamma(X^+_{i,p=0}, L_i^b) \otimes_{Q_i}Q_{i,y}))=0$ for all $j<\dim (X_i\times_{Q_i} Q_{i,y})_{p=0}$. In particular, for any $y\in \Spec(Q_i)$ that contains $(p, x_2,\dots,x_n)$, the $H^j$ vanish for all 
 $$j<\dim X_{i,p=0}-\dim (Q_i/(p,x_2,\dots,x_n)Q_i)=n+\dim X_i-\dim Q_i-1.$$ 
Note that for $i\gg0$, we have $\dim X_i-\dim Q_i=\dim X-\dim T$ {by \cite[Tag 0EY2]{stacks-project}.} Thus for $i\gg0$, the $H^j$ vanish for all $j<\dim X-1=\dim X_{p=0}$. At this point, we apply \cite[Proposition 2.10]{FoxbyIyengarDepth} or \cite[Section 3]{Gabber.tStruc} 
to the $Q_i$-complex $\myR\Gamma(X^+_{i,p=0}, L_i^b)$ and the ideal $I=(p, x_2,\dots,x_n)\subseteq Q_i$, we see that $H^{j}(\myR \Gamma_{(p,x_2,\dots,x_n)}(\myR\Gamma(X^+_{i,p=0}, L_i^b)))=0$ for all $j<\dim X_{p=0}$. 
 
 Finally, fix $j<\dim X_{p=0}$, for each $\eta\in H^{j}(\myR \Gamma_{x}(\myR\Gamma(X_{p=0}, L^b)))$, it is the image of some  $\eta'\in H^{j}(\myR \Gamma_{(p,x_2,\dots,x_n)}(\myR\Gamma(X_{i,p=0}, L_i^b)))$ for some index $i$. The previous paragraph shows that there is a finite cover $Y_i\to X_i$ such that $\eta'$ maps to zero in $H^{j}(\myR \Gamma_{(p,x_2,\dots,x_n)}(\myR\Gamma(Y_{i,p=0}, L_i^b)))$. Base change along $\Spec(T)\to \Spec(Q_i)$, we see that there exists a finite cover $Y\to X$ such that the image of $\eta$ is zero in $H^{j}(\myR \Gamma_{x}(\myR\Gamma({Y}_{p=0}, L^b)))$. Therefore $H^{j}(\myR \Gamma_x(\myR\Gamma(X^+_{p=0}, L^b)))=0$ for all $j<\dim X_{p=0}$. 
\end{proof}

Now we can prove the case of an excellent local base. This is precisely part (b) of \autoref{prop.BhattVanishing}.

\begin{proposition}
    \label{prop.BhattVanishing(b)}
    Suppose that $(T,x)$ is an excellent local domain of mixed characteristic $(0,p>0)$. Let $\pi : X \to \Spec(T)$ be a proper surjective map with $X$ reduced, equidimensional and $p$-torsion free. Suppose that $L \in \Pic(X)$ is a big and semiample line bundle. Then for all $b<0$, $H^{j}(\myR \Gamma_x(\myR\Gamma(X^+_{p=0}, L^b)))=0$ for all $j < \dim(X_{p=0})$.
    If, in addition, $(T,x)$ admits a dualizing complex $\omega_T^{\mydot}$, then there exists a finite cover $Y\to X$ such that
   \[
        \myR \Gamma_x(\myR \Gamma(X_{p = 0}, L^b)) \to \myR \Gamma_x(\myR \Gamma(Y_{p = 0}, L^b))
    \]
    is the zero map on $H^j$ for $j < \dim(X_{p=0})$.  
\end{proposition}
\begin{proof}
    We first assume $(T,x)$ is normal and henselian. By Popescu's theorem again, we can write $\widehat{T}=\colim T_i$ where each $T_i$ is smooth over $T$, and $\widehat{T}$ is a Noetherian complete local domain. Let $\widehat{X}$ and $X_i$ be the base change of $X$ along $T\to \widehat{T}$ and $T\to T_i$ respectively (note that $\widehat{X}$ and $X_i$ are still reduced, equidimensional and $p$-torsion free). Given a class $\eta\in H^j(\myR \Gamma_x(\myR \Gamma(X_{p = 0}, L^b)))$, by \autoref{prop.BhattVanishingCompleteCase}, there exists a finite cover $\widehat{Y}$ of $\widehat{X}$ such that the image of $\eta$ is $0$ in $H^j(\myR \Gamma_x(\myR \Gamma(\widehat{Y}_{p = 0}, L^b)))$. We can descend $\widehat{Y}$ to a finite cover $Y_i$ over $X_i$ for $i\gg0$, and enlarging $i$ if necessary, we know the image of $\eta$ is $0$ in $H^j(\myR \Gamma_x(\myR \Gamma(Y_{i,p=0}, L^b)))$. Now $(T,x)$ is henselian and the map $T\to T_i$ is smooth with a specified section over the residue field (via the map to the completion); thus, the map $T\to T_i$ admits a section $T_i \to T$ by \cite[Tags 07M7, 04GG]{stacks-project}. Base change $Y_i\to \Spec(T_i)$ along this section yield a finite cover $Y$ of $X$ such that the image of $\eta$ is $0$ in $H^j(\myR \Gamma_x(\myR \Gamma(Y_{p = 0}, L^b)))$. Running this argument for all finite covers $X'$ of $X$ and taking a direct limit, we find that $H^{j}(\myR \Gamma_x(\myR\Gamma(X^+_{p=0}, L^b)))=0$ for all $j < \dim(X_{p=0})$.

Next we assume $T$ is an excellent normal local domain. We may assume $X$ is normal. Let $T\to T^h$ be the henselization of $T$. Then $X\times_TT^h$ is also normal, by working with each connected component, we simply assume that $X\times_TT^h$ is normal and integral. Consider $X^+\times_TT^h$, this is a cofiltered limit of \'{e}tale $X^+$-schemes (in particular it is normal). Since $X^+$ is absolute integrally closed, each connected component of $X^+\times_TT^h$ is absolute integrally closed. But each connected component is also integral over $X\times_TT^h$, thus can be identified with $(X\times_TT^h)^+$. By the henselian case we already proved, we have that $H^{j}(\myR \Gamma_x(\myR\Gamma((X\times_TT^h)^+_{p=0}, L^b)))=0$ for all $j < \dim(X_{p=0})$. This implies $H^{j}(\myR \Gamma_x(\myR\Gamma((X^+\times_TT^h)_{p=0}, L^b)))=0$ by \cite[Lemma 5.9]{BhattAbsoluteIntegralClosure}. Since $T^h$ is faithfully flat over $T$, this implies $H^{j}(\myR \Gamma_x(\myR\Gamma(X^+_{p=0}, L^b)))=0$.

Finally, if $(T,x)$ is an excellent local domain, then the normalization $T'$ of $T$ is a excellent semi-local domain finite over $T$. Moreover, for any $y \in \mathrm{Spec}(T)$, we have an isomorphism
\[\myR\Gamma_{y}( (-)_y )=\oplus_{y'}\myR\Gamma_{y'}( (-)_{y'})\]
of functors on $T'$-complexes, where $y' \in \mathrm{Spec}(T')$ runs over the finitely many preimages of $y$ in $\Spec(T')$. Applying the above isomorphism when $y=x$, we can obtain the first part of the proposition from the excellent normal case (applied to localizations of $T'$ at preimages of $x$). Applying the above isomorphism for all $y \in \Spec(T/p)$, the last conclusion follows from \cite[Lemma 2.17 and Lemma 2.18]{BhattAbsoluteIntegralClosure} applied to the ind-object $\{\myR \Gamma(Y_{p = 0}, L^b)\}_Y$ where $Y$ runs over all finite covers of $X$ in $X^+$.
\end{proof}

Finally, we reformulate the above result in a form that does not require passing to the $p=0$ fibre; this will be convenient for us and also allows us to give a uniform statement that includes the equal characteristic $p>0$ case. 

\begin{corollary}
    \label{cor.VanishingWithoutRestrictingToPFiber}
    Suppose that $(T,x)$ is an excellent local ring  of residue characteristic $p>0$. Let $\pi : X \to \Spec(T)$ be a proper map with $X$ integral. Suppose that $L \in \Pic(X)$ is a big and semiample line bundle. Then for all $b<0$ and all $i < \dim(X)$, we have that $H^{i}(\myR \Gamma_x(\myR\Gamma(X^+, L^b)))=0$.
\end{corollary}
\begin{proof}
    Since $X\to \Spec(T)$ is proper and $X$ is integral, we can replace $T$ by $\pi_*\sO_X$ to assume $X\to\Spec(T)$ is proper and surjective and that $T$ is a domain. 
    If $(T,x)$ has mixed characteristic, then we consider the exact sequence 
    \[
        0 = H^{i-1}(\myR \Gamma_x(\myR\Gamma(X^+_{p = 0}, L^b))) \to H^{i}(\myR \Gamma_x(\myR\Gamma(X^+, L^b))) \xrightarrow{p} H^{i}(\myR \Gamma_x(\myR\Gamma(X^+, L^b))).
    \]
    This implies that the multiplication-by-$p$ map on $H^{i}(\myR \Gamma_x(\myR\Gamma(X^+, L^b)))$ is injective, which is impossible unless $H^{i}(\myR \Gamma_x(\myR\Gamma(X^+, L^b)))=0$ since any element of the module is $x^n$-torsion and so $p^n$-torsion for $n \gg 0$.
        
    Now suppose $(T,x)$ has equal characteristic $p>0$. By the same argument as in \autoref{prop.BhattVanishing(b)}, we may assume $(T,x)$ is a Noetherian complete local domain, and then by the same reduction as in \autoref{prop.BhattVanishingCompleteCase} and \autoref{prop.BhattVanishingAtNonClosed} (the steps are easier as we are working over a field and not a mixed characteristic DVR), we may assume $(T,x)$ is essentially finite type over a field $k$. We can write $k$ as a filtered colimt of finite type fields $k_j$ and thus $T$ is a filtered colimit of $T_j$ essentially finite type over $k_j$. Note that $X$ descends to $X_j$ over $T_j$ for large $j$ (and similarly for the big and semiample line bundle $L$, for instance see \cite[Lemma 4.1]{takumi}), and the dimension is preserved. If we can find a finite cover $Y_j\to X_j$ such that $H^{i}(\myR \Gamma_x(\myR\Gamma(X_j, L^b))) \to H^{i}(\myR \Gamma_x(\myR\Gamma(Y_j, L^b)))$ is zero, then after base change to $X$ we see that the image of $H^{i}(\myR \Gamma_x(\myR\Gamma(X_j, L^b)))$ is zero in $H^{i}(\myR \Gamma_x(\myR\Gamma(X^+, L^b)))$ and we will be done. Therefore, replacing $T$ by $T_j$ and $X$ by $X_j$, we may assume that $T$ is essentially finite type over an $F$-finite field $k$. In particular, $X$ and $T$ are $F$-finite.
    
    The rest argument essentially follows from the proof of \cite[Theorem 6.28]{BhattAbsoluteIntegralClosure}, replacing the mixed characteristic results there by their equal characteristic counterparts in \cite{BhattDerivedDirectSummand}. We can replace $X$ by a finite cover to assume $L=f^*N$ where $f: X\to Z$ is a proper surjective map (of proper integral schemes over $\Spec(T)$) and $N$ is ample on $Z$. Now by \cite[Theorem 1.5]{BhattDerivedDirectSummand}, there is a finite cover $Y\to X$ such that the map $\myR f_*\sO_X\to \myR g_*\sO_Y$ factors through $g_*\sO_Y$, where $g$ is the composition map $Y\to Z$. Set $Z'=\Spec_Z(g_*\sO_Y)$, we find that $\myR\Gamma_x\myR\Gamma(X, L^b)\to \myR\Gamma_x\myR\Gamma(Y, L^b)$ factors through $\myR\Gamma_x\myR\Gamma(Z', N^b)$. Since $L$ is big, $\dim X=\dim Z$ and hence by the above discussion, to show there is a finite cover of $X$ such that $H^{i}(\myR \Gamma_x(\myR\Gamma(X, L^b)))$ maps to zero for $i<\dim(X)$, it is enough to show there is a finite cover of $Z$ such that $H^{i}(\myR \Gamma_x(\myR\Gamma(Z, N^b)))$ maps to zero for $i<\dim(Z)$. Thus replacing $X$ by $Z$ and $L$ by $N$, we may assume $L$ is ample.
    
   By local duality, for any finite cover $Y$ of $X$, $\myH^i\myR \Gamma_x(\myR\Gamma(Y, L^b))$ is the Matlis dual of
\[
       \myH^{-i}\myR\Hom_T(\myR\Gamma(Y,L^b),\omega_T^\mydot)\cong\myH^{-i}\myR\Gamma(Y, \omega_Y^\mydot\otimes L^{-b})
\]
Applying \cite[Proposition 6.2]{BhattDerivedDirectSummand}, there exists a further finite cover $Y'$ of $Y$ such that the map 
    $$\myH^{-i}\myR\Gamma(Y', \omega_{Y'}^\mydot\otimes L^{-b})\to \myH^{-i}\myR\Gamma(Y, \omega_Y^\mydot\otimes L^{-b})$$
factors through $\myH^{-i}\myR\Gamma(Y, \omega_Y[\dim(X)]\otimes L^{-b})$. Since $X$ is $F$-finite, we can take $Y$ to be the $e$-th Frobenius of $X$ so $\myH^{-i}\myR\Gamma(Y, \omega_Y[\dim(X)]\otimes L^{-b})=0$ for $e\gg0$ and $i<\dim(X)$ by Serre vanishing (note that $L$ is ample and $b<0$). Therefore the composition map 
$$\myH^{-i}\myR\Gamma(Y', \omega_{Y'}^\mydot\otimes L^{-b})\to \myH^{-i}\myR\Gamma(X, \omega_X^\mydot\otimes L^{-b})$$
is the zero map. Thus its Matlis dual $\myH^i\myR \Gamma_x(\myR\Gamma(X, L^b))\to \myH^i\myR \Gamma_x(\myR\Gamma(Y', L^b))$ is also the zero map. Running this argument for all finite covers of $X$ and taking a colimit, we find that $H^{i}(\myR \Gamma_x(\myR\Gamma(X^+, L^b)))=0$ as desired.
\end{proof}

\begin{remark}
In the context Corollary~\ref{cor.VanishingWithoutRestrictingToPFiber}, if $\myH^i(\myR\Gamma_x(\myR\Gamma(X, L^{-1})))$ is bounded $p$-power-torsion, then it follows that there exists a finite cover that $Y \to X$ that annihilates that cohomology group. Dual versions can then be phrased in terms of canonical modules and dualizing complexes; see Remark~\ref{KVcharpRelative} for the characteristic $p$ analog.  This approach is explored in \cite{TakamatsuYoshikawaMMP}.  
\end{remark}

\begin{remark}[Kodaira vanishing up to finite covers in positive characteristic]
\label{KVcharpRelative}
Continue in the setup and notation of Corollary~\ref{cor.VanishingWithoutRestrictingToPFiber} and assume that $T$ has characteristic $p$. The proof given above then shows the following finer statement: there exists a finite surjective map $Y \to X$ such that the induced trace map 
$$\myH^{-i}\myR\Gamma(Y, \omega_{Y}^\mydot\otimes L^{-b})\to \myH^{-i}\myR\Gamma(X, \omega_X^\mydot\otimes L^{-b})$$
is the $0$ map for $i < \dim(X)$.  
\end{remark}

%% file: B0Definitions.tex
\section{The subset of \texorpdfstring{$\bigplus$-stable}{+-stable}  sections (\texorpdfstring{$\myB^0$}{B-zero})}

Let $X$, $\Delta$ and $M$ be as in \autoref{def:B_0} below.
In this section, we define special global sections inside $H^0(X, \sO_X(M))$, which will be important especially when $M - K_X - \Delta$ is ample (or big and semiample).  Like $S^0$ and $T^0$ in characteristic $p > 0$ from \cite{BlickleSchwedeTuckerTestAlterations,SchwedeACanonicalLinearSystem}, these special linear systems behave as though Kawamata-Viehweg vanishing is true.  We will use this extensively later in the paper.

\begin{convention}
    \label{conv.CategoryOfFiniteMaps}
    In the remainder of the paper, we will often work with intersections, limits or colimits over the category of all finite covers of an integral scheme $X$. In this situation, we always mean the following: fix an algebraic closure $\overline{K(X)}$ of the function field of $X$, and consider the category of all finite integral covers $f:Y \to X$ equipped with an embedding $K(Y) \subset \overline{K(X)}$ over $X$ (in particular, the morphisms must respect this embedding). Thus, our intersections, limits or colimits take place over a poset. Note that a cofinal collection in this category is given by the finite covers with $Y$ normal when $X$ is excellent. Moreover 
    \[
        X^+ = \varprojlim_{\substack{f \colon Y \to X\\ \textnormal{finite}}} Y,
    \]
    see \cite[\href{https://stacks.math.columbia.edu/tag/01YV}{Tag 01YV}]{stacks-project}.
            A similar convention applies to categories of alterations. 
\end{convention}

\begin{definition}[$\bigplus$-stable sections]
	\label{setup}
\label{def:B_0}

Consider the following situation:
\begin{itemize}
\item $X$ is a normal, integral scheme proper over a complete Noetherian local domain $(R, \fram)$ with characteristic $p > 0$ residue field,
\item $\Delta \geq 0$ is a $\bQ$-divisor on $X$, and
\item  $M$ is a $\bZ$-divisor and $\sM = \sO_X(M)$.  In fact, the following definition only depends on the linear equivalence class of $M$.
\end{itemize}
Then, define 
\begin{equation*}
\myB^0(X, \Delta; \sM):= \bigcap_{\substack{f \colon Y \to X\\ \textnormal{finite}}}\im \left( H^0(Y, \sO_Y( K_Y + \lceil{f^* (M - K_X - \Delta)}\rceil)) \to H^0(X, \sM) \right) 
\end{equation*}
where the intersection is taken as $R$-submodules of $H^0(X, \sM)$, and runs over all $f:Y \to X$ as in \autoref{conv.CategoryOfFiniteMaps} where $Y$ is normal.  One sees by Galois conjugation that the above module is independent of the choice of geometric generic point of $X$.

We call the global sections $\myB^0(X, \Delta; \sM)$ the \emph{$\bigplus$-stable sections of $H^0(X, \sM)$ (with respect to $(X, \Delta)$)}.

Additionally, assuming that $M-(K_X + \Delta)$ is $\bQ$-Cartier,   define also
\begin{equation*}
\myB^0_{\alt}(X, \Delta; \sM):= \bigcap_{\substack{f \colon Y \to X\\\textnormal{alteration}}}\im \left( H^0(Y, \sO_Y( K_Y + \lceil f^* (M - K_X - \Delta)\rceil)) \to H^0(X, \sM) \right)
\end{equation*}
where the intersection runs over all alterations $f:Y \to X$ from a normal integral schemes as in \autoref{conv.CategoryOfFiniteMaps}.

If $\Delta=0$, then we use the simplified notation: $\myB^0(X; \sM):=\myB^0(X, \Delta; \sM)$ and $\myB^0_{\alt}(X; \sM):=\myB^0_{\alt}(X, \Delta;\sM)$.
\end{definition}

\begin{remark}[$\myB^0$ for non-integral $X$]
    \label{rem.NonintegralB^0}
    If $X$ is not integral but still normal where each component has the same dimension $d$, we define $\myB^0(X, \Delta; \sM)$ as the direct sum of $\myB^0$ for each connected (hence irreducible) component of $X$.
\end{remark}

\begin{remark}
	\label{def.B^0ViaCoversWhereDivisorHasBecomeCartier}
	Alternately, we may assume that $Y \to X$ factors through some finite $h : W \to X$ such that $h^*(M - K_X - \Delta)$ is integral.  In that case, the roundings are also not needed.  If $M - K_X - \Delta$ is $\bQ$-Cartier, we may also assume that $h^*(M - K_X - \Delta)$ is Cartier (see, for example, \cite[Section 2.4]{KollarMori} or \cite{TomariWatanabeFilteredRings}; \textit{cf.} \cite[Lemma 4.5]{BlickleSchwedeTuckerTestAlterations}).  In this latter case, we do not even need to restrict to normal $Y$.
\end{remark}

\begin{remark}
	\label{rem.SimpleAdjointFormulationNoPair}
Frequently, one applies \autoref{def:B_0} to  $\sM = \omega_X \otimes \sL$ and $\Delta = 0$ with $\sL$ a line bundle, in which case the first notion of \autoref{def:B_0} simplifies to
\begin{equation*}
\myB^0(X; \omega_X \otimes \sL):= \bigcap_{\substack{f \colon Y \to X\\ \textnormal{finite}}}\im \left(  H^0(X, \sL \otimes f_* \omega_Y ) \to H^0(X,\sL \otimes \omega_X ) \right).
\end{equation*}
\end{remark}

\begin{remark}[Non-complete $R$]
	\label{rem.B0ForNonCompleteR}
	If $(R, \fram)$ is an excellent non-complete local ring, with completion $\widehat{R}$, we may base change by the completion $\widehat{R}$ of $R$ to obtain $X_{\widehat{R}}$,  and define $\myB^0$ and $\myB^0_{\alt}$ as above but restrict to finite covers (respectively, alterations) that arise as the base change of a finite cover of $X$.  In this case, the resulting intersection, which we denote by $\widehat{\myB}^0$, is a subset of $H^0(X_{\widehat{R}}, \sL \otimes_R \widehat{R})$ and so is a finitely generated $\widehat{R}$-module since $X_{\widehat{R}} \to \Spec \widehat{R}$ is proper and $\widehat{R}$ is Noetherian.  However, this intersection need not be finitely generated as an $R$-module as $\widehat{R}$ is not.  See also \autoref{subsec.B0andcompletion} where we show that $\widehat{\myB}^0$ is equal to $\myB^0(X_{\widehat{R}}, \Delta|_{X_{\widehat{R}}}, \sL \otimes_R \widehat{R})$
\end{remark}

\noindent The following basic property of $\myB^0$ is immediate from the definition.
\begin{lemma} \label{lem.comparison_betwen_B0_for_different_delta} With notation as in \autoref{setup}, we have that
\[
\myB^0(X, \Delta; \sM) \subseteq \myB^0(X, \Delta'; \sM)
\]
for every effective $\bQ$-divisor $\Delta' \leq \Delta$.
\end{lemma}

Our next goal is to identify $\myB^0$ with a Matlis dual of a direct limit, and also with a certain inverse limit.  These alternate descriptions of $\myB^0$ will be both convenient and crucial in what follows.  Before doing that, we make the following observations related to passing direct limits through cohomology in our setting:
    \[
        \varinjlim_{\substack{f \colon Y \to X\\ \textnormal{finite}}} \myH^d \myR \Gamma_{\fram} \myR \Gamma(Y,  \sO_Y(\lfloor f^*(K_X+\Delta-M) \rfloor)) = \varinjlim_{\substack{f \colon Y \to X\\ \textnormal{finite}}} H^d_{g^{-1} \m}(X, f_* \sO_Y(\lfloor f^*(K_X+\Delta-M) \rfloor))
    \]
    which, in view of \autoref{eq.CohomologyWithSupportsCommutesWithDirectLimits} may be identified with 
    \[
        H^d_{g^{-1} \m}(X, \pi_* \sO_{X^+}(\pi^*(K_X+\Delta-M))) = \myH^d \myR \Gamma_{\fram} \myR \Gamma(X^+, \sO_{X^+}(\pi^*(K_X+\Delta-M)))
    \]
    where $g : X \to \Spec R$ is the given map and $\pi : X^+ \to X$ is the induced map.
    In other words, 
    \begin{equation}
        \label{eq.directLimitPassesThroughLocalCohomAndMore}
        \begin{array}{rcl}
        & {\displaystyle \varinjlim_{\substack{f \colon Y \to X\\ \textnormal{finite}}} \myH^d \myR \Gamma_{\fram} \myR \Gamma(Y,  \sO_Y(\lfloor f^*(K_X+\Delta-M) \rfloor))} \\ \\
        = & \myH^d \myR \Gamma_{\fram} \myR \Gamma(X^+, \sO_{X^+}(\pi^*(K_X+\Delta-M))).
        \end{array}
    \end{equation}
    Of course, this identification can be obtained in other ways as well.

{\begin{lemma}[Alternate descriptions of $\myB^0$]
	\label{lem.B0AsInverseLimit}
 	Work in the situation of \autoref{def:B_0} and suppose $d = \dim X$.
 	\begin{enumerate}[leftmargin=20pt]
    \item We then have that $\myB^0(X, \Delta; \sM)$ is the $R$-Matlis dual of \label{eq.lem.B0AsInverseLimit.DualImageForFinite}
	\begin{equation*}		
        		\qquad \im\Big( \myH^d \myR \Gamma_{\fram} \myR \Gamma(X, \sO_X(K_X - M)) \to 
		\varinjlim_{\substack{f \colon Y \to X\\ \textnormal{finite}}} \myH^d \myR \Gamma_{\fram} \myR \Gamma(Y,  \sO_Y(\lfloor f^*(K_X+\Delta-M) \rfloor)) \Big)     \end{equation*}
    or equivalently, by \autoref{eq.directLimitPassesThroughLocalCohomAndMore}, the $R$-Matlis dual of:
    \[ 
        \im\Big( \myH^d \myR \Gamma_{\fram} \myR \Gamma(X, \sO_X(K_X - M))  \to 
		\myH^d \myR \Gamma_{\fram} \myR \Gamma(X^+, \sO_{X^+}(\pi^*(K_X+\Delta-M))) \Big).		
	\]
    \item Dually, we have that:
    \label{lem.B0AsInverseLimit.DualImageForFinite.InverseLimitProjection}
 	\begin{equation*}
		\qquad \myB^0(X, \Delta; \sM)= \im \left( \Bigg( \varprojlim_{\substack{f \colon Y \to X\\ \textnormal{finite}}} H^0(Y, \sO_Y( K_Y + \lceil f^* (M - K_X - \Delta)\rceil))\Bigg) \to H^0(X, \sM) \right).
	\end{equation*}

		\item
	Similarly, when $M - (K_X + \Delta)$ is $\bQ$-Cartier, $\myB^0_{\alt}$ is the $R$-Matlis dual of
    \[		
		\qquad \im\Big( \myH^d \myR \Gamma_{\fram} \myR \Gamma(X, \sO_X(K_X - M)) \to  
		\varinjlim_{\substack{f \colon Y \to X\\ \textnormal{alteration}}} \myH^d \myR \Gamma_{\fram} \myR \Gamma(Y,  \sO_Y(\lfloor f^*(K_X+\Delta-M) \rfloor)) \Big) 	\]
    where $Y$ runs over alterations.  
    \item Dually, we have that \label{lem.B0AsInverseLimit.DualImageForAlteration.InverseLimitProjection}
	\begin{equation*}
		\qquad \myB^0_{\alt}(X, \Delta; \sM)= \im \left( \Bigg( \varprojlim_{\substack{f \colon Y \to X\\ \textnormal{alteration}}} H^0(Y, \sO_Y( K_Y + \lceil f^* (M - K_X - \Delta)\rceil))\Bigg) \to H^0(X, \sM) \right).
	\end{equation*}	
	\end{enumerate}
\end{lemma}
\vskip 12pt
	An alternate description of \autoref{lem.B0AsInverseLimit.DualImageForFinite.InverseLimitProjection}, is that for every $s \in \myB^0(X, \Delta; \sM)$ there exists a compatible system as follows such that $s_X =s$:
	\begin{equation*}
	    \left(\parbox{240pt}{\begin{center}$s_Y \in H^0\Big(Y, \sO_Y\big( K_Y + \lceil f^* (M - K_X - \Delta)\rceil\big)\Big)$ \\   $\qquad \forall f \colon Y \to X $ finite \end{center}} \Bigg| \ \parbox{200pt}{ such that $s_Z \mapsto s_Y$ for any factorization of finite maps $Z \to Y \to X$} \  \right).
	\end{equation*}
	Similarly for \autoref{lem.B0AsInverseLimit.DualImageForAlteration.InverseLimitProjection}.

\begin{proof}
			For each finite map $f \colon Y\to X$ with $Y$ normal, we have a natural map $\sO_X(K_X-M)\to f_* \sO_Y(\lfloor f^*(K_X+\Delta-M) \rfloor)$ (for alterations, where the argument will be the same, we also require that $K_X+\Delta-M$ is $\bQ$-Cartier).
		Thus we have 
	\begin{equation}
	\label{eq:B0AsInverseLimit:image}
	\myH^d \myR \Gamma_{\fram} \myR \Gamma(X,\sO_X(K_X-M)) \twoheadrightarrow \im_Y \hookrightarrow \myH^d \myR \Gamma_{\fram} \myR \Gamma(Y, \sO_Y(\lfloor f^*(K_X+\Delta-M) \rfloor)).
	\end{equation}
	Taking filtered colimit for all $Y$, we have
	\begin{equation}
		\label{eq:CohomologyMappingToDirectLimitViaImage}
		\myH^d \myR \Gamma_{\fram} \myR \Gamma(X, \sO_X(K_X-M)) \twoheadrightarrow \varinjlim_Y\im_Y \hookrightarrow \varinjlim_Y \myH^d \myR \Gamma_{\fram} \myR \Gamma(Y, \sO_Y(\lfloor f^*(K_X+\Delta-M) \rfloor)).
	\end{equation}
    Notice also that $\varinjlim_Y \im_Y$ is the image of the map \autoref{eq.lem.B0AsInverseLimit.DualImageForFinite}.  
	We shall show that the Matlis dual of the limit of the images satisfies the following:
	\begin{equation}
		\label{eq:B0AsMatlisDualToImages}
		(\varinjlim_Y\im_Y)^{\vee} = \varprojlim_Y \im_Y^\vee = \myB^0(X, \Delta; \sM).
	\end{equation}
    To see this, first observe that the Matlis dual of $\myH^d \myR \Gamma_{\fram} \myR \Gamma(X, \sO_X(K_X-M))$ is $H^0(X, \sM)$ by \autoref{lem:duality}.  Similarly, and using the fact that $\sHom_{\sO_Y}(\sO_Y(\lfloor f^*(K_X+\Delta-M) \rfloor), \omega_Y)\cong \sO_Y( K_Y + \lceil f^* (M - K_X - \Delta)\rceil)$ since $Y$ is normal, the Matlis dual of $\myH^d \myR \Gamma_{\fram} \myR \Gamma(Y, \sO_Y(\lfloor f^*(K_X+\Delta-M) \rfloor))$ is $H^0\big(Y, \sO_Y( K_Y + \lceil f^* (M - K_X - \Delta)\rceil)\big)$ by \autoref{lem:duality}.  Hence, applying Matlis duality to \autoref{eq:CohomologyMappingToDirectLimitViaImage}, and noticing that Matlis duality turns colimits into limits, we obtain
\begin{equation}	
\label{eq:keeptrackofsurjection}
 H^0(X, \sM)=H^0(X, \sO_X(M)) \hookleftarrow \varprojlim_Y \im_Y^\vee \twoheadleftarrow \varprojlim_Y H^0\big(Y, \sO_Y( K_Y + \lceil f^* (M - K_X - \Delta)\rceil)\big).
 \end{equation}
	 It follows that 
	\begin{equation}
	\label{eq:B0InverseLimit}
	\varprojlim_Y \im_Y^\vee = \im \left( \Bigg( \varprojlim_{\substack{f \colon Y \to X\\ \textnormal{finite}}} H^0(Y, \sO_Y( K_Y + \lceil f^* (M - K_X - \Delta)\rceil))\Bigg) \to H^0(X, \sM) \right).
	\end{equation}
	But since applying Matlis-duality to \autoref{eq:B0AsInverseLimit:image} yields
	\begin{equation*}
	\im_Y^\vee = \im \left(H^0(Y, \sO_Y( K_Y + \lceil f^* (M - K_X - \Delta)\rceil)) \to H^0(X, \sM) \right),
	\end{equation*}
	we know that 
	\[
		\varprojlim_Y \im_Y^\vee = \bigcap_{\substack{f \colon Y \to X\\\textnormal{finite}}}\im \left( H^0(Y, \sO_Y( K_Y + \lceil f^* (M - K_X - \Delta)\rceil)) \to H^0(X, \sM) \right)=\myB^0(X, \Delta; \sM)
	\]
	Now $(a)$ follows from \autoref{eq:CohomologyMappingToDirectLimitViaImage} and \autoref{eq:B0AsMatlisDualToImages}, and $(b)$ follows from \autoref{eq:B0AsMatlisDualToImages} and \autoref{eq:B0InverseLimit}. The argument for $\myB^0_{\alt}(X, \Delta; \sM)$ is the same.     
\end{proof}

\begin{remark}
The proof above uses that $(R,\m)$ is complete in an essential way since the Matlis dual of local cohomology modules supported at the maximal ideal $\m$ are finitely generated $\widehat{R}$-modules (and not necessarily finitely generated over $R$). 

Furthermore, without the complete hypothesis \autoref{lem.B0AsInverseLimit} \autoref{lem.B0AsInverseLimit.DualImageForFinite.InverseLimitProjection} is false.  Even in equal characteristic $p > 0$, suppose $(R, \m)$ is as in \cite[Corollary C]{DattaMurayamaTate} an excellent regular local ring, $X = \Spec R$, $M = 0$, and $\Delta = 0$.  Then we have
\[ 
\varprojlim_{\substack{R \subseteq S}} \omega_{S/R} = \varprojlim_{\substack{R \subseteq S}} \Hom_R(S, R) = \Hom_R(\varinjlim_{\substack{R \subseteq S}} S, R) = \Hom_{R}(R^+, R) = 0
\]
where $R \subseteq S$ runs over finite extensions of $R$ in $R^+$.  Hence the image in \autoref{lem.B0AsInverseLimit} \autoref{lem.B0AsInverseLimit.DualImageForFinite.InverseLimitProjection} is zero.  On the other hand, each map $\omega_{S/R} = \Hom_R(S, R) \to R$ is surjective for any finite extension $R \subseteq S$ by the direct summand theorem in characteristic $p > 0$ \cite{HochsterContractedIdealsFromIntegralExtensions}.

We will also need completeness in the proof of our section-lifting result \autoref{thm:main-lifting} (which uses the vanishing of \autoref{sec.Vanishing}).  For our geometric applications, this will not be a substantial restriction as we can reduce to this case.  Also see \autoref{rem.B0ForNonCompleteR}.
\end{remark}

\autoref{lem.B0AsInverseLimit} essentially asserts that the formation of images and inverse limits commutes in certain situations. Such an assertion would be automatic if the relevant inverse limits were exact functors. This is in fact true more generally, and we extrapolate the following observation from the method of proof\footnote{Specifically, the observation comes from extracting what is needed to ensure the surjective arrow in \eqref{eq:keeptrackofsurjection}.} of \autoref{lem.B0AsInverseLimit} above:

\begin{lemma}   
\label{RlimExact}
Let $R$ be a complete Noetherian local ring. Let $\{K_i\}_{i \in I}$ be a projective system of finitely generated $R$-modules with cofiltered indexing category $I$. Then $\myR\varprojlim_i K_i$ is concentrated in degree $0$. Consequently, the functor $\{M_i\} \mapsto \varprojlim_i M_i$ is exact on $I$-indexed diagrams of finitely generated $R$-modules.
\end{lemma}

\begin{proof}
Let $E$ be the injective hull of the residue field of $R$. Write $(-)^\vee := \myR\Hom_R(-,E)$ for the Matlis duality functor regarded as a functor on the derived category $D(R)$, so $(-)^\vee:D(R) \to D(R)$ is $t$-exact (because $E$ is an injective $R$-module), and we have a natural isomorphism $K \simeq (K^\vee)^\vee$ for $K \in D^b_{coh}(R)$. Now take $\{K_i\}$ as in the lemma. We then have
\[ 
    \myR\varprojlim_i K_i = \myR\varprojlim_i ((K_i^\vee)^\vee) = \myR\varprojlim_i \myR\Hom_R( (K_i^\vee), E) = \myR\Hom_R(\colim_i (K_i^\vee), E).
\]
As $(-)^\vee$ is $t$-exact, each $(K_i)^\vee$ lies in degree $0$. But filtered colimits are exact, so $\colim_i (K_i^\vee)$ also lies in degree $0$. Finally, $E$ is injective, so the last term above also lies in degree {$0$}, whence $\myR\varprojlim_i K_i$ lies in degree $0$, as wanted in the first part.

The second part is formal given the first part. For instance, say $\{f_i\}:\{M_i\} \to \{N_i\}$ is an $I$-indexed diagram of surjections of finitely generated $R$-modules. To show $\varprojlim_i M_i \to \varprojlim_i N_i$ is surjective, we simply use that $\myR\varprojlim_i \ker(f_i)$ is concentrated in degree $0$ by the first part, and that $\myR\varprojlim_i$ takes any short exact sequence of $I$-indexed diagrams of $R$-modules to an exact triangle in $D(R)$.
\end{proof}

Applying this lemma to $I$ being the category of all finite covers (resp. alterations) 
and the map of projective systems $$
    \left\{H^0(Y, \sO_Y(D_Y))\right\}_{Y\to X}\twoheadrightarrow \left\{\im\left(H^0(Y,\sO_Y(D_Y))\to H^0(X,\sO_X(\sM))\right)\right\}_{Y\to X}
$$ 
with $D_Y=K_Y+\lceil f^*(M-K_X-\Delta)\rceil$, 
appearing in \autoref{lem.B0AsInverseLimit} then gives an alternative proof of the lemma.

\begin{remark}
    \label{rem.InfinityCategoriesAndBeyond}
The proofs of \autoref{RlimExact} and \autoref{cor.B0VsB0Alt} below feature filtered colimits in the derived category. Literally interpreted in the triangulated category setting, this does not give a sensible object. For example, the formation of filtered colimits in the derived category $D(R)$ of a commutative ring $R$ (when they exist) does not commute with taking cohomology groups (even when everything is in a degree $0$), making the former a rather obscure notion\footnote{At the request of the referee, we give an example where colimits in the triangulated category $D(R)$ work poorly. Given a countable diagram $M_0 \to M_1 \to M_2 \to ...$ in $D(R)$, if the colimit $M := \colim_i M_i$ in $D(R)$ exists, then the map $\oplus_i M_i \to M$ must be a categorical epimorphism as $\mathrm{Hom}_{D(R)}(M,-) \to \prod_i \mathrm{Hom}_{D(R)}(M_i,-)$ is injective by the defining property of a colimit. But any epimorphism $f:x \to y$ in a triangulated category splits: the canonical map $g:y \to \mathrm{cone}(f)$ is $0$ as $g \circ f$ is $0$. So we learn that $\oplus_i M_i \to M$ admits a right inverse. This is clearly not the case for colimits of interest, e.g., if we take $R=\mathbf{Z}$ and $M_i = \mathbf{Z}/p^i$ with maps $M_i \to M_{i+1}$ determined by $1 \mapsto p$, then the ``correct'' colimit is $\mathbf{Q}_p/\mathbf{Z}_p$, but the map $\oplus_i \mathbf{Z}/p^i \to \mathbf{Q}_p/\mathbf{Z}_p$ does not have a right inverse: the right side admits a nonzero map from $\mathbf{Q}_p$ while the left side does not.}. Instead, to obtain the notion of filtered colimits for which passing to cohomology is exact, one can work with $\infty$-categories. Alternate approaches include dg-categories, or a $1$-categorical substitute such as the notion of homotopy colimits over suitable diagram categories, e.g., see \cite[Tag 0A5K]{stacks-project} for colimits over the poset $\mathbb{N}$). We will elide this issue in the sequel.  
\end{remark}

\begin{remark}
We explain why the completeness of $R$ is essential to \autoref{RlimExact}; we shall use the theory of derived completions, see \cite[Tags 091N,0BKF,0BKH]{stacks-project}. Suppose $(R,\mathfrak{m})$ is a Noetherian local ring. Then $R$ is $\m$-adically complete exactly when it is derived $\m$-complete (since $R$ is Noetherian and $\m$ is finitely generated), and the latter happens exactly when $R$ is derived $f$-complete for every $f\in \m$, i.e., $\myR^1\varprojlim(\cdots R \xrightarrow{f} R \xrightarrow{f} R) = 0$ (noting that $\myR^0 \varprojlim$ always vanishes in this case by Krull's intersection theorem). Therefore if $(R,\m)$ is not $\m$-adically complete, then there exists $f\in \m$ such that $\myR^1\varprojlim(\cdots R \xrightarrow{f} R \xrightarrow{f} R) \neq 0$, i.e., $\myR\varprojlim(\cdots R \xrightarrow{f} R \xrightarrow{f} R)$ is not concentrated in degree $0$ so \autoref{RlimExact} is false.
\end{remark}

The next result relies on deep results on $p$-adic Riemann-Hilbert correspondence \cite{BhattLuriepadicRHmodp} in the form of \cite[Theorem 3.12]{BhattAbsoluteIntegralClosure}.

\begin{corollary}[Alterations vs finite covers]
    \label{cor.B0VsB0Alt}
    With notation as above, and assuming that $M - K_X - \Delta$ is $\bQ$-Cartier, we have that 
    \[
        \myB^0(X, \Delta; \sM)=\myB^0_{\alt}(X, \Delta; \sM).
    \]
\end{corollary}
\begin{proof}
	We follow the notation of the statement and proof of \autoref{lem.B0AsInverseLimit}, keeping in mind \autoref{rem.InfinityCategoriesAndBeyond}.   It suffices to demonstrate that 
	\[
		\varinjlim_{\substack{f \colon Y \to X\\\textnormal{finite}}} \myH^d\myR \Gamma_{\fram} \myR \Gamma(Y, \sO_Y(\lfloor f^*(K_X+\Delta-M) \rfloor))=\varinjlim_{\substack{f \colon Y \to X\\\textnormal{alteration}}} \myH^d\myR \Gamma_{\fram} \myR \Gamma(Y, \sO_Y(\lfloor f^*(K_X+\Delta-M) \rfloor)).
	\]
	Since $p \in \mathfrak{m}$, we have $\myR\Gamma_\m(\widehat{M})=\myR\Gamma_{\m}(M)$
	for all $R$-complexes $M$, where $\widehat{M}$ denotes the derived $p$-completion
	(see \cite[Tag 091N]{stacks-project} for definitions and details about derived completion). 
	Since filtered colimits are exact (cf.\ \cite[Tag 00DB]{stacks-project}), it is thus enough to show that the natural map identifies
	\small
	\[
		\myH^d\myR\Gamma_\m \left(\widehat{\varinjlim_{\substack{f \colon Y \to X\\\textnormal{finite}}} \myR\Gamma(Y, \sO_Y(\lfloor f^*(K_X+\Delta-M) \rfloor))}\right)
		=\myH^d\myR\Gamma_\m \left(\widehat{\varinjlim_{\substack{f \colon Y \to X\\\textnormal{alteration}}} \myR\Gamma(Y, \sO_Y(\lfloor f^*(K_X+\Delta-M) \rfloor))}\right).
	\]
	\normalsize
	At this point, we recall that derived $p$-complete complexes obey a derived Nakayama lemma, i.e., in order to show a given map $M \to N$ of derived $p$-complete objects in $D(\mathrm{Ab})$ is an isomorphism, it is enough to show that  $M\otimes^L\Z/p \to N\otimes^L\Z/p$ is an isomorphism (cf.\ \cite[Tag 0G1U]{stacks-project}). Therefore, it is enough to show that
	\[
		\varinjlim_{\substack{f \colon Y \to X\\\textnormal{finite}}} \myR\Gamma(Y, \sO_Y(\lfloor f^*(K_X+\Delta-M) \rfloor)) \otimes^L \bZ/p
		=\varinjlim_{\substack{f \colon Y \to X\\\textnormal{alteration}}} \myR\Gamma(Y,\sO_Y(\lfloor f^*(K_X+\Delta-M) \rfloor))\otimes^L \bZ/p
	\]
	via the natural map. As a corollary of the $p$-adic Riemann-Hilbert functor from \cite{BhattLuriepadicRHmodp} (see \cite[Theorem 3.12]{BhattAbsoluteIntegralClosure}), we know that 
	\[
		\varinjlim_{\substack{f \colon Y \to X\\\textnormal{finite}}} \myR f_*\sO_Y\otimes^L \bZ/p =\varinjlim_{\substack{f \colon Y \to X\\\textnormal{alteration}}}\myR f_*\sO_Y\otimes^L \bZ/p
	\]
    via the natural map. Because twisting by a divisor and applying $\myR\Gamma(X, -)$ commutes with filtered colimits, we are done.\end{proof}
}

In characteristic $p > 0$, the analogs of $\myB^0$ typically stabilize, in other words we might expect that there exists a finite cover or alteration such that the image of 
\[
	H^0(Y, \sO_Y( K_Y + \lceil{f^* (M - K_X - \Delta)}\rceil)) \to H^0(X, \sM)
\] 
is exactly equal to $\myB^0$.  In characteristic $p > 0$, when one restricts the finite covers to iterates of Frobenius, this is essentially Hartshorne-Speiser-Gabber-Lyubeznik stabilization \cite{HartshorneSpeiserLocalCohomologyInCharacteristicP, Gabber.tStruc, LyubeznikVanishingoflocalcohomologycharP}, see for instance \cite[Section 2.4]{HaconXuThreeDimensionalMinimalModel} for a version of this in the relative setting.  If one instead considers arbitrary finite covers in characteristic $p > 0$, certain stabilization results in the case where $X \to \Spec R$ is an alteration can be found in \cite{BlickleSchwedeTuckerTestAlterations,SchwedeTuckerTestIdealsOfNonprincipalIdeals,ChiecchioEnescuMillerSchwede}, these are then all consequences of the equational lemma  \cite{HochsterHunekeInfiniteIntegralExtensionsAndBigCM,HunekeLyubeznikAbsoluteIntegralClosure,BhattDerivedDirectSummand}.

However, in mixed characteristic such stabilization is not possible.

\begin{example}
	\label{ex:elliptic_curve}
	Let $E \to \mathrm{Spec}(\bZ_p)$ be an elliptic curve\footnote{That is, $E \to \mathrm{Spec}(\bZ_p)$ is a proper smooth morphism whose geometric fibers are connected curves of genus one together with a prescribed section.}, so $\omega_{E/\bZ_p}\cong \sO_E$ and $H^0(E,\omega_{E/\bZ_p}) \simeq \mathbb{Z}_p$.
	We claim that
	\begin{enumerate}
	\item \label{itm:elliptic_curve:zero} $\myB^0\left(E, \omega_{E/\bZ_p} \right) = 0$, but 
	\item \label{itm:elliptic_curve:not_zero} $\im (\Tr_f:H^0(Y , \omega_{Y/\bZ_p}) \to H^0(E, \omega_{E/\bZ_p})) \neq 0$ for every alteration $f \colon Y \to E$.
	\end{enumerate}	 
	
	To prove \autoref{itm:elliptic_curve:zero}, fix an integer $n \geq 1$ and consider the $p^n$-power map $[p^n]:E \to E$. We claim that the corresponding trace map $\Tr_{[p^n]}:H^0(E,\omega_{E/\bZ_p}) \to H^0(E,\omega_{E/\bZ_p})$ is multiplication by $p^n$; this will imply that $$\myB^0(E,\omega_{E/\mathbb{Z}_p}) \subset \bigcap_n p^{n} H^0(E,\omega_{E/\bZ_p}) = \bigcap_n p^n \mathbb{Z}_p = 0,$$ as wanted. By duality, the claim for $\Tr_{[p^n]}$ is equivalent to showing that the pullback map $[p^n]^*:H^1(E,\sO_E) \to H^1(E,\sO_E)$ is given by $p^n$. But this is a general and standard fact about multiplication maps on abelian schemes, as we briefly recall. The map $[p^n]:E \to E$ factors as $E \xrightarrow{\Delta} E^{\times p^n} \xrightarrow{\mu} E$, where $\mu$ denotes the addition map and $\Delta$ is the diagonal, so we have $[p^n]^* = \Delta^* \circ \mu^*$. Now the K\"{u}nneth formula gives $H^1(E^{\times p^n},\sO_{E^{\times p^n}}) \simeq H^1(E,\sO_E)^{\oplus p^n}$, with projection to the $i$-th summand (resp. inclusion of the $i$-th summand) on the right given by the inclusion $E \to E^{\times p^n}$ in the $i$-th factor (resp.\ the projection $E^{\times p^n} \to E$ to the $i$-th factor). It is then immediate that $\mu^*:H^1(E,\sO_E) \to H^1(E^{\times p^n},\sO_{E^{\times p^n}}) \simeq H^1(E,\sO_E)^{\oplus p^n}$ is the diagonal map, so postcomposing with $\Delta^*$ gives $p^n$, as asserted.

	To prove \autoref{itm:elliptic_curve:not_zero}, it suffices to show that for every integral alteration $f \colon Y \to E$, the map $\Tr_f:H^0(Y,\omega_{Y/\bZ_p}) \to H^0(E,\omega_{E/\bZ_p})$ is surjective after inverting $p$. Let $\eta \in \Spec(\mathbb{Z}_p)$ be the generic point. As $f_\eta:Y_\eta \to E_\eta$ is an alteration of integral curves over $\mathbb{Q}_p$, it is in fact a finite map. The claim now follows as $E_\eta$ is a global splinter; explicitly, the map $\Tr_{f_\eta} = (\Tr_f)[1/p]$ is dual to the pullback map $f_\eta^*:H^1(E_\eta,\sO_{E_\eta}) \to H^1(Y_\eta,\sO_{Y_\eta})$, and the latter is injective since the map on sheaves $\sO_{E_\eta} \to f_{\eta,*} \sO_{Y_\eta}$ is split injective, with splitting coming from the normalized trace map on functions.
	\end{example}
	
	\begin{remark}
	The phenomenon in Example~\ref{ex:elliptic_curve} is not specific to elliptic curves and in fact generalizes significantly. Indeed, for any mixed characteristic DVR $V$ and a normal integral proper flat $V$-scheme $X$ of relative dimension $d \geq 1$ {such that $H^0(X,\omega_{X/V})\neq 0$}, we have the following:
	\begin{enumerate}
	    \item $\myB^0(X, \omega_{X/V}) = 0$.
	    \item $\im (\Tr_f:H^0(Y , \omega_{Y/V}) \to H^0(X, \omega_{X/V})) \neq 0$ for every finite cover $f \colon Y \to X$. (More generally, the same holds true for every alteration if we additionally assume that $X_\eta$ has rational singularities.)
	\end{enumerate}
	The proof of (b) is identical to that of \autoref{ex:elliptic_curve} (b). For (a), observe that the duality $\myR \Hom_V(\myR\Gamma(X,\sO_X),V) \simeq \myR\Gamma(X,\omega^\mydot_{X/V})$ and the fact that $\myR\Gamma(X,\sO_X) \in D^{\leq d}$ imply that $H^0(X,\omega_{X/V}) \simeq H^{-d}(X,\omega^\mydot_{X/V})$ is naturally identified with $\mathrm{Hom}_V(H^d(X,\sO_X),V)$, and similarly for all finite normal covers of $X$. Following the argument in the proof of \autoref{ex:elliptic_curve} (a), it is enough to show that for each $n \geq 1$, there exists a finite normal cover $f:Y\to X$ such that the pullback map $f^*:H^d(X,\sO_X) \to H^d(Y,\sO_Y)$ is divisible by $p^n$ as a map. This follows from \cite[Theorem 3.12]{BhattAbsoluteIntegralClosure}.
	\end{remark}

Working in equicharacteristic $p > 0$, we may form an analog of \autoref{ex:elliptic_curve} by considering a family of elliptic curves over $k\llbracket t \rrbracket$.  However, such an example does not satisfy \autoref{itm:elliptic_curve:not_zero}.
Indeed, the generic fiber of $E$ over $k((t))$ is not globally $F$-regular, and so there exists an alteration which is zero on global sections.
\subsection{$\myB^0$ in the affine case}

In the case where $X=\Spec(R)$, our definition produces a test ideal which we denote by $\tau_{+}(R,\Delta):=\myB^0(\Spec(R),\Delta;\sO_X)\subset \sO_X$.  We prove here that this agrees with a special case of the big Cohen-Macaulay test ideal defined in \cite{MaSchwedeSingularitiesMixedCharBCM}, which we first recall.

\begin{definition}
	 Suppose $\Gamma\geq 0$ is a $\mathbb{Q}$-Cartier divisor on $X=\Spec(R)$, where $(R,\m)$ is a Noetherian complete local {normal} domain, such that $\Div(f)=n\Gamma$ for some $f \in R$.
	We also fix a canonical divisor $K_X \geq 0$ and a big Cohen-Macaulay $R^+$-algebra $B$. Then define 
	\[
		0^{B,\Gamma}_{H^d_{\fram}(R)}:=\ker(H^d_{\fram}(R)\xrightarrow{\cdot f^{1/n}} H^d_{\fram}(B))
	\]
	and the BCM-test submodule of $(\omega_R,\Gamma)$ with respect to $B$: 
	\[
		\mytau_B(\omega_R,\Gamma):= \Ann_{\omega_R}0^{B,\Gamma}_{H^d_{\fram}(R)}.
	\]
	Equivalently, $\mytau_B(\omega_R,\Gamma)$ is the Matlis dual of the image of $H^d_{\fram}(R)\xrightarrow{\cdot f^{1/n}} H^d_{\fram}(B)$.

	Now given $\Delta\geq0$ such that $K_X+\Delta$ is $\mathbb{Q}$-Cartier we define the BCM-test ideal with respect to $B$ to be $\mytau_B(R,\Delta):=\mytau_B(\omega_R, K_R+\Delta)$.  Via our embedding $\sO_X \subseteq \sO_X(K_R)$, $\mytau_B(R,\Delta)$ is contained in $R$.  Note this definition requires that $K_X + \Delta$ is $\bQ$-Cartier.
\end{definition}

In this article we are interested in the particular big Cohen-Macaulay algebra $B=\widehat{R^+}$, the $p$-adic completion of the absolute integral closure of $R$, see \autoref{cor: p-adic completion of R^+}. Since $H^d_{\fram}(R^+) = H^d_{\fram}(\widehat{R^+})$, we can ignore the $p$-adic completion for the purposes of defining $\mytau_B(X, \Delta)$ and thus in what follows we will write $\mytau_{R^+}(R,\Delta)$ for $\mytau_{\widehat{R^+}}(R,\Delta)$.

\begin{proposition} \label{prop:B^0-agrees-with-test-ideal-for-affines}
	$\mytau_{R^+}(R,\Delta)=\tau_+(R,\Delta) := \myB^0(\Spec(R),\Delta;\sO_X)$ if $K_R+\Delta$ is $\mathbb{Q}$-Cartier.
\end{proposition}
\begin{proof}
	Set $X = \Spec R$ and assume that $K_X \geq 0$.  Define $\Gamma = K_X + \Delta$ and write $\Div_X(f) = n\Gamma = n(K_X + \Delta)$.  By  \autoref{lem.B0AsInverseLimit} \autoref{eq.lem.B0AsInverseLimit.DualImageForFinite}, we see that $\myB^0(X,\Delta;\sO_X)$ is the Matlis dual of the image, where $d = \dim R$, of
	\[
		H^d_\m(\sO_X(K_X)) \to \varinjlim_Y H^d_\m\big( \sO_Y(\lfloor f^*(K_X+\Delta) \rfloor)\big) = H^d_\m\big(\varinjlim_Y\sO_Y(\lfloor f^*(K_X+\Delta) \rfloor)\big).
	\]
		where $Y = \Spec S \xrightarrow{f} \Spec R = X$ is finite, in other words $R \subseteq S \subseteq R^+$ is a finite extension.  Because $R \to \omega_R$ has cokernel $\omega_R/R$ of dimension $< d$, we see that $H^d_{\m}(\omega/R) = 0$ by \cite[\href{https://stacks.math.columbia.edu/tag/0DXC}{Tag 0DXC}]{stacks-project} which implies that $H^d_\m(R) \twoheadrightarrow H^d_\m(\omega_R)$ surjects.  Hence their images in $H^d_\m\big(\varinjlim_Y\sO_Y(\lfloor f^*(K_X+\Delta) \rfloor)\big)$ are the same.  By restricting to those $S$ which are large enough to contain $f^{1/n}$, we see that $\sO_Y(\lfloor f^*(K_X+\Delta) \rfloor) = {1 \over f^{1/n}} \cdot \sO_Y$.  Finally, putting this all together, 
	$R^{+} = \varinjlim S$
	we see that $\myB^0(X,\Delta;\sO_X)$ is Matlis dual to the image of
	\[
		H^d_\m(R) \xrightarrow{\cdot f^{1/n}} H^d_\m(R^+).
	\]
	But this image is Matlis dual to $\mytau_{R^+}(R, \Delta)$.
			\end{proof}

\subsection{Transformation of $\myB^0$ under alterations}
\label{subsec.TransformationsOfB0UnderAlterations}

In this section we record for later use a number of transformation rules for $\myB^0$ as we pass from an alteration to the base $X$.

The first transformation rule allows us to do away with the divisor $\Delta$ by absorbing it into $\sM$, at least on some cover.  
\begin{lemma}
	\label{lem.FiniteCoverToRemoveDelta}
	With notation as in \autoref{def:B_0}. Suppose that $\pi : Y \to X$, where $Y$ is normal, is either:
	\begin{itemize}
		\item[(a)] a finite surjective map, or
		\item[(b)] $M - K_X - \Delta$ is $\bQ$-Cartier and $\pi$ is an alteration.
	\end{itemize}
	In either case, assume that $\pi^*(M - K_X - \Delta)$ has integer coefficients
	and consider the map
	\[
		\Tr : H^0(Y, \sO_Y(K_Y + \pi^*(M - K_X - \Delta))) \to H^0(X, \sM).	
	\]
	Then we have that
	\[
		\Tr\big(\myB^0(Y, \sO_Y(K_Y + \pi^*(M - K_X - \Delta)))\big) = \myB^0(X, \Delta, \sM).
	\]
\end{lemma}
\begin{proof}
	This is an immediate consequence of \autoref{lem.B0AsInverseLimit}.
\end{proof}

We now record a transformation for a birational $\pi : W \to X$.

\begin{lemma} \label{lemma-B0-under-pullbacks}  
	Let $X$ be a normal integral scheme proper over $\Spec(R)$ as in \autoref{def:B_0} and $B \geq 0$ a $\bQ$-divisor on $X$ such that $K_X + B$ is $\bQ$-Cartier.  Let $\pi \colon W \to X$ be a proper birational morphism from a normal {integral} scheme $W$ and write $K_W + B_W = \pi^*(K_X+B)$. Let $B' \geq 0$ be an effective $\Q$-divisor such that $B' \geq B_W$. Then for every Cartier divisor $L$ on $X$, we have
	\[
		\myB^0(X,B;\sO_X(L)) \supseteq \myB^0(W,B'; \sO_W(\pi^*L)). 
	\]
	Furthermore, if $B' = B_W$ (in particular, this assumes that $B_W$ is effective), then this containment is an equality.
\end{lemma}
\begin{proof}
	For every alteration $f \colon Y \to W$ we have the following diagram
	\begin{center}
		\begin{tikzcd}
			H^0\big(Y,\sO_Y(K_Y + \lceil f^*(\pi^*L - (K_W+B')) \rceil)\big) \arrow{r} \arrow{d}{\subseteq} & H^0(W, \pi^*L) \arrow{d}{=} \\
			H^0\big(Y,\sO_Y(K_Y + \lceil (\pi \circ f)^*(L - (K_X+B)) \rceil)\big) \arrow{r} & H^0(X, L).
		\end{tikzcd}
	\end{center}
	Note that in the case that $B' = B_W$, the left vertical containment is an equality.
	An application of \autoref{cor.B0VsB0Alt} completes the proof.
\end{proof}

In the proof of the existence of flips, we will need a technical variant of \autoref{lemma-B0-under-pullbacks}.  We record it here.

\begin{lemma} \label{lemma-B0-under-pullbacks-fancy} Let $X$ be a normal integral scheme proper over $\Spec (R)$ as in \autoref{def:B_0} and $B \geq 0$ a $\bQ$-divisor on $X$ such that $K_X + B$ is $\bQ$-Cartier. 
	Let $\pi \colon Y \to X$ be a proper birational morphism from a normal {integral} scheme $Y$ and write $K_Y + B_Y = \pi^*(K_X+B)$. Let $L$ be a $\Q$-Cartier $\Q$-divisor on $X$ such that $(X,B+\{-L\})$ is klt. Then {$H^0(X,\sO_X(\lceil L \rceil)) = H^0(Y,\sO_Y(\lceil \pi^*L + A_Y\rceil))$} and
\[
\myB^0(X,B+\{-L\};\sO_X(\lceil L \rceil)) = \myB^0(Y,\{B_Y-\pi^*L\}; \sO_Y(\lceil \pi^*L + A_Y\rceil)), 
\]
where $A_Y := -B_Y = K_Y - \pi^*(K_X + B)$.  Here $\{ \Delta \} = \Delta - \lfloor \Delta \rfloor$ is the fractional part of $\Delta$.
\end{lemma}
\begin{proof}
	First, since $(X, B +\{-L\})$ is klt, implicitly $\{-L\}$ is also $\bQ$-Cartier.  Thus so is $\lceil L \rceil = L + \{-L\}$.
	Notice that 
	\[
		\begin{array}{rl}
			& \lceil \pi^* L  + A_Y \rceil - (K_Y + \{ B_Y - \pi^* L\}) \\
			= & \lceil \pi^* L  + A_Y\rceil - K_Y - B_Y + \pi^* L {+} \lfloor B_Y - \pi^* L \rfloor\\
			= & \lceil \pi^* L  + A_Y\rceil - K_Y - B_Y + \pi^* L - \lceil \pi^* L + A_Y \rceil\\
			= & \pi^*(L - K_X - B)\\
			= & \pi^*(\lceil L \rceil - (K_X + B + \{ -L \})).  		\end{array} 
	\]
	Therefore, for every sufficiently large alteration $f \colon W \to Y$ we have the following diagram
	\begin{center}
		\begin{tikzcd}			
			H^0\big(W,K_W +  f^*(\lceil \pi^*L+A_Y \rceil - (K_Y+\{B_Y-\pi^*L\}))\big) \arrow{r} \arrow{d}{=} & H^0\big(Y, \sO_Y(\lceil \pi^*L+A_Y \rceil)\big) \arrow{d}{\kappa, \; =} \\
			H^0\big(W,K_W +  (\pi \circ f)^*(\lceil L\rceil - (K_X+B+\{-L\}))\big) \arrow{r} & H^0\big(X, \sO_X(\lceil L \rceil)\big).
		\end{tikzcd}
	\end{center}
	The equality of the left vertical arrow follows from our chain of equalities above.  However, we need to justify the equality, and in fact existence, of the right vertical arrow labeled $\kappa$ (this is where we use that $(X, B+\{-L\})$ is klt).  
	
	Now, since $(X, B + \{-L\})$ is klt, the components of $-B - \{-L\} = -B + L + \lfloor -L \rfloor = L - B - \lceil L \rceil$ have coefficients $\leq 0$ and $>- 1$.  Thus $\lceil L - B \rceil = \lceil L \rceil$ and so since $\pi_*$ of a divisor simply removes exceptional components, we have that:
	\[
		\pi_* \lceil \pi^* L + A_Y \rceil = \pi_* \lceil \pi^*L - B_Y\rceil = \lceil L - B \rceil = \lceil L \rceil.
	\]
	This at least implies that the map $\kappa$ exists.

	Next, again because $(X,B+\{-L\})$ is klt, $\lceil A_Y - \pi^*\{-L\} \rceil = \lceil K_Y - \pi^*(K_X + B + \{-L\}) \rceil$
	is effective and exceptional over $X$.  Therefore:
	\begin{align*}
				&\lceil \lceil \pi^*L + A_Y \rceil -\pi^*\lceil L\rceil \rceil \geq \lceil \pi^*L + A_Y  -\pi^*\lceil L\rceil \rceil = \lceil A_Y - \pi^*\{-L\} \rceil \geq 0.
	\end{align*}
	Hence the map $\kappa$ is an isomorphism (\autoref{lemma:pushforward}) and the diagram exists as claimed.
		Once we have the diagram in place, the result follows immediately by \autoref{cor.B0VsB0Alt}.
										\end{proof}

\subsection{Adjoint analogs}
\label{subsec.AdjointAnalogsOfS^0}

The subspace $\myB^0$ of $H^0$ provides a global analog of the test ideal in positive characteristic and the multiplier ideal in mixed characteristic. In fact, we will see it frequently as a graded piece of the $R^+$-test ideal for a cone. Therefore, the subspace $\myB^0$, in contrast to $S^0$ of \cite{SchwedeACanonicalLinearSystem} (a global analog of a non-$F$-pure ideal / lc ideal), cannot satisfy the sharpest possible adjunction to a divisor.  To address this problem we will create an adjoint-ideal version of $\myB^0$, to which we can lift sections.
	With notation as in \autoref{setup} assume that $\Delta = S + B$ where $S$ is a reduced divisor whose components do not appear in $B$.  	
	
	For each irreducible component $S_i$ of $S$ ($i = 1, \dots, t$), choose an integral subscheme $S_i^+$ of $X^+$ which lies over $S$.  Notice that this $S_i^+$ is indeed an absolute integral closure of $S$ so this is not an abuse of notation.  Equivalently this means that for every {normal} finite cover $Y \to X$ we pick a compatible choice of prime divisor $S_{i,Y}$ lying above $S_i$.  In that case, we set $S_Y$ to be the sum of the $S_{i,Y}$.  We define:
	\[
	    S^+ := \coprod_{i = 1}^t S_i^{+}.
	\]
	There is an affine map $f : S^+ \to X^+$ but it is not in general a closed immersion unless $t = 1$.  Indeed, we notice that when $S$ has multiple irreducible components, the map $\sO_{X^+} \to f_* \sO_{S^+} \cong \oplus_{i = 1}^t \sO_{S_i^+}$ is not surjective (the isomorphism follows since $S^+$ is a disjoint union).  From here on out, we abuse notation slightly and omit the $f_*$ on $\sO_{S^+}$.   Notice that $\sO_{X^+} \to \sO_{S_i^+}$ is surjective for each $i$.

		We will define the adjoint-ideal version of $\myB^0$ as the $R$-Matlis dual of the image of
	\[
	\myH^d \myR\Gamma_{\fram}\myR\Gamma(X, \sO_X(K_X - M)) \to \myH^d \myR \Gamma_\m\myR\Gamma(X^+, \bigoplus_{i = 1}^t \sO_{X^+}(-S_i^+ + \pi^*(K_X + S + B - M))).
	\]
	The origin of this map is carefully described below.  This dual is also identified with the intersection 
	\[
	    \myB^0_{S}(X,S + B; \sM) = \bigcap_Y \Image \Bigg( H^0\Big(Y, \bigoplus_{i = 1}^t \sO_Y(K_Y + S_{i,Y} + f^*  (M-K_X - S - B))  \Big) \to H^0(X,  \sM) \Bigg)
    \]
	see \autoref{lem.ComplexKVersionEqualsOrdinaryVersion}.

	Consider the short exact sequence (a direct sum of short exact sequences):
	\[
	    \xymatrix@R=12pt{
		    0 \ar[r] & \bigoplus_{i = 1}^t \sO_{X^+}(-S_i^+) \ar[r] & \bigoplus_{i = 1}^t \sO_{X^+} \ar[r] & \bigoplus_{i = 1}^t \sO_{S_i^+} \ar@{=}[d] \ar[r] & 0\\
		             &                                              &                                      & \sO_{S^+}
		}
	\]
	where $\sO_{X^+}(-S_i^+)$ is the colimit of the $\sO_Y(-S_{Y_i})$.  We notice that there is a map of short exact sequences where the bottom vertical maps correspond to the diagonals:
	\[
		\xymatrix{
			0 \ar[r] & \sO_X(-S) \ar[d] \ar[r] & \sO_X \ar[r] \ar[d] & \sO_S \ar[r] \ar[d] & 0\\
			0 \ar[r] & \sO_Y(-S_Y) \ar[d] \ar[r] & \sO_Y \ar[r] \ar[d] & \sO_{S_Y} \ar[r] \ar[d] & 0\\
			0 \ar[r] & \bigoplus_{i = 1}^t \sO_{X^+}(-S_i^+) \ar[r] & \bigoplus_{i = 1}^t \sO_{X^+} \ar[r] & \sO_{S^+} = \bigoplus_{i = 1}^t \sO_{S_i^+}\ar[r] & 0.
		}
	\]
	Assume that $f' : Y' \to X$ is such that $f'^*(K_X + S + B)$ is integral.  
	Twisting the top row by $K_X + S - M$ and the second and third by $f'^*(K_X + S + B - M)$ (and using that $B$ is effective for the second map), we obtain a factorization
\begin{equation}
	\label{eq.FactorizationAdjointMap}
	\begin{array}{rl}
		\sO_X(K_X - M) \to & \sO_{Y'}(-S_{Y'} + f'^*(K_X + S- M)) \\
		\to & \sO_{Y'}(-S_{Y'} + f'^*(K_X + S + B - M)) \\
		\to & \bigoplus_{i = 1}^t \sO_{Y'}(-S_{i,Y'} + f'^*(K_X + S + B - M))\\
		\to & \varinjlim_{Y} \bigoplus_{i = 1}^t \sO_{Y}(-S_{i,Y} + f^*(K_X + S + B - M))\\
		= & \bigoplus_{i = 1}^t \sO_{X^+}(-S_i^+ + \pi^*(K_X + S + B - M)).
	\end{array}
\end{equation}

\begin{definition} \label{setup_adjoint}
With notation as above, and in particular fixing $S^+ = \coprod_{i = 1}^t S_i^+ \to X^+$, define $\myB^0_{S}(X,S + B; \sM)$ to be the $R$-Matlis dual of {the image of}
	{\[
	\myH^d \myR\Gamma_{\fram}\myR\Gamma(X, \sO_X(K_X - M)) \to \underbrace{\varinjlim_Y  \myH^d \myR \Gamma_\m\myR\Gamma(Y, \bigoplus_{i = 1}^t \sO_{Y}(-S_{i,Y} + f^*(K_X + S + B - M)))}_{\myH^d \myR \Gamma_\m\myR\Gamma(X^+, \bigoplus_{i = 1}^t \sO_{X^+}(-S_i^+ + \pi^*(K_X + S + B - M)))} 
	\]}where $d = \dim X$ and $Y$ runs over finite maps with $Y$ normal and $f^*(K_X + S + B)$ has integer coefficients.  Notice that $\myB^0_{S}(X,S + B; \sM) \subseteq H^0(X, \sM)$ since its Matlis dual is a quotient of $\myH^d \myR\Gamma_{\fram}\myR\Gamma(X, \sO_X(K_X - M))$, see \autoref{lem:duality}.  
	
Similarly, we define $\myB^0_{S, \alt}(X,S+B; \sM)$ to be the $R$-Matlis dual of the image of 
{\[
\myH^d \myR\Gamma_{\fram}\myR\Gamma(X, \sO_X(K_X - M))  \to \varinjlim_Y \myH^d \myR \Gamma_\m\myR\Gamma(Y, \bigoplus_{i = 1}^t \sO_Y(-S_{Y,i} + \lfloor f^*(K_X+S+B-M)\rfloor))		
\]}where $Y$ runs over all normal alterations and we define $S_{Y,i}$ to be the strict transform of the corresponding divisors on the Stein factorization.  We may restrict to those where $f^*(K_X + S + B)$ is Cartier if desired. \end{definition}

A priori, these definitions \emph{depend} on the choice of $S^+ = \coprod_{i=1}^t S_i^+ \to X^+$.  Thus, our first order of business is to show that this choice does not matter.  We begin with the case that $S$ is integral.

\begin{lemma}
    \label{lem.B0_SIndependentOfChoiceForIntegral}
	Suppose $S$ is integral.  The objects $\myB^0_{S}(X,\Delta;\sM)$ and $\myB^0_{S,alt}(X,\Delta;\sM)$ are independent of the choice of $S^+ \subseteq  X^+$.
\end{lemma}
\begin{proof}
		We prove only the case of $\myB^0_{S}(X,\Delta;\sM)$ as the alteration case is very similar.
	For any two choices $S^+$ and $S'^+$ mapping to $X^+$, 
	pick an element $\sigma$ of $\mathrm{Gal}(X^+/X)$ which sends $S^+$ to $S'^+$.  Then one obtains the trace maps in the tower computing $\myB^0_{S'}$ by precomposing those computing $\myB^0_{S}$ by the isomorphism $\sigma$.  Therefore the images are pairwise equal and the intersections are the same.  
\end{proof}

The following lemma allows us to assume that $S$ is integral in some cases, and finishes the proof that $\myB^0_{S}$ is independent of $S^+ \to X^+$.

\begin{lemma}
	\label{lem.B0_SIsSumOfB0_SOfComponents}
	With notation as above, 
	\[
		\myB^0_{S}(X,S + B; \sM) = \sum_{i = 1}^t \myB^0_{S_i}(X,S + B; \sM).
	\]
	Likewise with $\myB^0_{S,\alt}(X, S+B; \sM)$ when $K_X + S + B$ is $\bQ$-Cartier.  As a consequence, $\myB^0_{S}(X,S + B; \sM)$ and $\myB^0_{S,\alt}(X, S+B; \sM)$ are independent of the choice $S^+ = \coprod_{i = 1}^t S_i^+ \to X^+$.
\end{lemma}
\begin{proof}
	The first statement is a direct application of Matlis duality.  Indeed suppose that $A \twoheadrightarrow B \hookrightarrow \bigoplus_{i = 1}^t C_i$ is a surjective map followed by a injective map of $R$-modules.  The Matlis dual $B^{\vee}$ is then the sum of the images of of the $C_i^{\vee} \to A^{\vee}$.  The alteration statement is proven in the same way.  The statement about independence of choice now follows from \autoref{lem.B0_SIndependentOfChoiceForIntegral} as each $S_i$ is integral.
\end{proof}

Our next goal is to study several alternate characterizations of $\myB^0_{S}$.  

\begin{lemma}
	\label{lem.ComplexKVersionEqualsOrdinaryVersion}
	With notation as above, then 
	\[
		\myB^0_{S}(X,S + B; \sM) = \bigcap_Y \Image \Bigg( H^0\Big(Y, \bigoplus_{i = 1}^t \sO_Y(K_Y + S_{i,Y} + f^*  (M-K_X - S - B))  \Big) \to H^0(X,  \sM) \Bigg)
	\]	
	where $d = \dim X$ and $Y$ runs over finite maps where $f^*(K_X + S +B)$ is a Weil divisor.  
	Likewise with $\myB^0_{S, \alt}$ (with alterations instead of finite maps).  	Furthermore, the elements in those sets correspond to compatible systems of elements 
	\[
		s_Y \in H^0\Big(Y,  \bigoplus_{i = 1}^t \sO_Y(K_Y + S_{i,Y} + f^*  (M-K_X - S - B))  \Big)
	\]
	as in \autoref{lem.B0AsInverseLimit}.  
\end{lemma}
\begin{proof}
	The statement about compatible systems and Matlis duality follows exactly as in \autoref{lem.B0AsInverseLimit}.  
\end{proof}

\begin{lemma}
	\label{lem.B0AlongDAsInverseLimit}
	With notation as above, and assuming that $K_X + S + B$ is $\bQ$-Cartier, then 
	\[
		\myB^0_{S}(X,S + B; \sM) = \myB^0_{S, \alt}(X,S+B; \sM).
	\]
\end{lemma}
\begin{proof}
	By \autoref{lem.B0_SIsSumOfB0_SOfComponents}, we may assume that $S$ is integral.  
	For each alteration $f : Y\to X$ we have an exact sequence
	\[
		0\to \sO_Y(-S_Y)\to \sO_Y\to \sO_{S_Y}\to 0.
	\]
	Notice that $S_{Y} \to S$ is an alteration as well.

	For the equality of $\myB^0_{S}$ with $\myB^0_{S, \alt}$, by the same argument as in \autoref{cor.B0VsB0Alt}, it is enough to show the following:	
	\[
		\sO_{X^+}(-S^+) \otimes^L \bZ/p \; \cong \varinjlim_{{\substack{f \colon Y \to X \\ \mathrm{alteration}}}} \myR f_* \sO_Y(-S_Y) \otimes^L \bZ/p.
	\]	
	
	Now, we have an exact triangle
	\[
		\myR f_*\sO_Y(-S_Y)\otimes^L \bZ/p\to \myR f_*\sO_Y\otimes^L \bZ/p\to \myR f_*\sO_{S_Y}\otimes^L \bZ/p \xrightarrow{+1}.
	\]
	By taking filtered colimits and applying the isomorphism:
	\[
		\varinjlim_{{\substack{f \colon W \to Z \\ \mathrm{alteration}}}} \myR f_* \sO_W/p = \sO_{Z^+}/p,
	\]
	implied by \cite[Theorem 3.12]{BhattAbsoluteIntegralClosure} (which in turn relies on \cite{BhattLuriepadicRHmodp}) as in \autoref{cor.B0VsB0Alt}, to both $Z = X$ and $Z= S$, gives an exact triangle 
	\[
		\varinjlim_{{\substack{f \colon Y \to X \\ \mathrm{alteration}}}} \myR f_* \sO_Y(-S_Y) \otimes^L \bZ/p \to \sO_{X^+}\otimes^L \bZ/p\to \sO_{S^+}\otimes^L \bZ/p \xrightarrow{+1}, 
	\]
	so the desired quasi-isomorphism follows.	
\end{proof}

We now compare $\myB^0_{S}$ with $\myB^0$.

\begin{lemma} \label{lem.comparison_betwen_B0_and_B0D} 
	With notation as in \autoref{setup_adjoint}, we have that
	\[
		\myB^0_{S}(X, S + B; \sM) \subseteq \myB^0(X, aS + B; \sM)
	\]
	for every $0 \leq a < 1$.
\end{lemma}
\begin{proof}
	By \autoref{lem.B0_SIsSumOfB0_SOfComponents}, we may assume that $S$ is integral.  
	Fix such an $0 \leq a < 1$.  For sufficiently large finite covers $f: Y \to X$ with $f^*(K_X + S + B)$ Cartier {and $f^*(aS + B)$ having integer coefficients}, observe that 
	\[
		f^*(aS + B) \leq f^*(S + B) - S_Y.
	\]
	Therefore, by \autoref{eq.FactorizationAdjointMap} the map
	\[
		\myH^d \myR\Gamma_{\fram}\myR\Gamma(X, \sO_X(K_X - M)) \to 
		\myH^d \myR \Gamma_\m\myR\Gamma(Y, \sO_Y(-S_Y + f^*(K_X+S+B-M)))
	\]
	factors through $\myH^d \myR \Gamma_\m\myR\Gamma(Y, \sO_Y(f^*(K_X+aS + B-M)))$.  The result follows by Matlis duality.
\end{proof}

Next we point out that $\myB^0_S$ behaves well with respect to birational maps, in analogy with \autoref{lemma-B0-under-pullbacks}. 

\begin{lemma} \label{lemma-B0S-under-pullbacks}  
	Let $X$ be a normal integral scheme, proper over $\Spec(R)$ as in \autoref{def:B_0}, $S$ a reduced divisor and $B \geq 0$ a $\bQ$-divisor on $X$ with no common components with $S$, such that $K_X + S + B$ is $\bQ$-Cartier.  Let $\pi \colon W \to X$ be a proper birational morphism from a normal {integral} scheme $W$ and write $K_W + S_W + B_W = \pi^*(K_X+S+B)$ where $S_W$ is the strict transform of $S$. Let $B' \geq 0$ be an effective $\Q$-divisor such that $B' \geq B_W$. Then for every Cartier divisor $L$ on $X$, we have
	\[
		\myB^0_S(X,S+B;\sO_X(L)) \supseteq \myB^0_{S_W}(W,S+B'; \sO_W(\pi^*L)). 
	\]
	Furthermore, if $B' = B_W$ (in particular, this assumes that $B_W$ is effective), then this containment is an equality.
\end{lemma}
\begin{proof}
	The proof is analogous to that of \autoref{lemma-B0-under-pullbacks}.  
	For every alteration $f \colon Y \to W$ with $S_{i,Y}$ as above, we have the following diagram
	\begin{center}
		\begin{tikzcd}
			H^0(Y,\bigoplus_{i = 1}^t\sO_Y(K_Y + S_{i,Y} + \lceil f^*(\pi^*L - (K_W + S_W +B'))) \rceil) \arrow{r} \arrow{d}{\subseteq} & H^0(W, \pi^*L) \arrow{d}{=} \\
			H^0(Y,\bigoplus_{i = 1}^t\sO_Y(K_Y + S_{i,Y} + \lceil (\pi \circ f)^*(L - (K_X+S+B)) \rceil)) \arrow{r} & H^0(X, L).
		\end{tikzcd}
	\end{center}
	Note that in the case that $B' = B_W$, the left vertical containment is an equality.
	An application of \autoref{lem.ComplexKVersionEqualsOrdinaryVersion} and \autoref{lem.B0AlongDAsInverseLimit} completes the proof.
\end{proof}

\subsubsection{Comparison with alternate versions}
\label{subsubsec.ComparisionWithB0SVersion1FiniteCase}

	In the first arXiv version of this article, we did not take a direct sum of $\sO_X^+(-S_i^+)$.  Instead, we primarily worked by forming an exact triangle:
	\[
		\sD_Y^{\mydot} \to \sO_{Y} \to \bigoplus_{i = 1}^t \sO_{S_{i,Y}} \xrightarrow{+1}
	\]
	for each finite cover $Y \to X$.  We then used $\varinjlim \sD_Y^{\mydot}$ instead of $\bigoplus_{i = 1}^t \sO_{X^+}(-S_{i}^+)$.  Of course, when $S$ has only one component, these two definitions agree.
	
	In general case, this had several disadvantages compared to our current approach.  First, it was not clear whether $\myB^0_{S}$ was independent of the choice of $S^+$ when $S$ was not integral.   Furthermore, we ended up working with a complex instead of a sheaf in all essential proofs.  In particular, the lemma that said we could work with a sheaf (Lemma 4.25 of that first arXiv version) was incorrect, although it was not used in a crucial way.  We notice the object $\myB^0_{S}$ defined in this paper is always at least contained in the one from the first arXiv version, essentially since the map to $\bigoplus_{i = 1}^t \sO_{X^+}$ factors through the diagonal map $\sO_{X^+} \to \bigoplus_{i = 1}^t \sO_{X^+}$ (see \autoref{lem.OldVersionOkIsh}).

	Back to the first arXiv version of this article, when working with alterations to define $\myB^0_{S, \alt}$ we could restrict ourselves to alterations $f : Y \to X$ that separated the individual components of $S$.  This also yields a satisfactory theory although it is still not clear whether it depends on the choice of $S^+$.  
	
	However, when $M - (K_X + S + B)$ is big and semiample, it turns out that the two approaches coincide (a fact we will not use).

	\begin{lemma}
		\label{lem.OldVersionOkIsh}
		With notation as above, assume additionally that $M - (K_X + S +B)$ is big and semiample.  Then 
		\[
			\myB^0_S(X, S+B, \sM) = \bigcap_{Y}\Image\Big(H^0\big(Y, \sO_Y(K_Y + S_Y + f^*(M - K_X - S -B))\big) \to H^0(X, \sM) \Big)
		\]
		where $f : Y \to X$ runs over alterations such that $f^*(K_X + S + B)$ is a Cartier divisor.
	\end{lemma}
	\begin{proof}
		The containment $\subseteq$ follows from the dual of the diagonal maps $\sO_Y(-S_Y) \to \bigoplus_{i = 1}^t \sO_Y(-S_{i,Y})$ so we prove the reverse. 

		Fix $\pi : W \to X$ a birational map that separates the components of $S$.  We have the commutative diagram where the vertical maps are induced by the diagonal:
		\[
			\xymatrix{
				0 \ar[r] &  \sO_{W^+}(-S_{W^+}) \ar[d] \ar[r] &  \sO_{W^+} \ar[r] \ar[d]  &  \sO_{S_{W^+}} \ar[r] \ar[d]^{\sim}  & 0 \\
				0 \ar[r] &  \bigoplus_{i = 1}^t \sO_{W^+}(-S_{i, W^+}) \ar[r] & \bigoplus_{i = 1}^t \sO_{W^+} \ar[r] & \bigoplus_{i = 1}^t \sO_{S_{i,W^+}} \ar[r] & 0
			}
		\]
		Note that the right vertical map is an isomorphism and the middle vertical map is split injective (simply project onto one of the coordinates).  We cannot say something similar about the left vertical arrow however.  
		Twisting by the pullback $\sL^+$ to $W^+$ of the line bundle $\sO_Y(f^* (M - (K_X + S + B)))$ (for some finite cover $f : Y \to W$), and taking local cohomology, we obtain:
		\[
			{\scriptsize
				\xymatrix{
					0 \ar[r] & \myH^{d-1} \myR\Gamma_{\fram}\myR\Gamma(\sO_{S_{W^+}} \otimes \sL^+) \ar[r] \ar[d]^{\sim} & \myH^d \myR\Gamma_{\fram}\myR\Gamma(\sO_{W^+}(-S_{W^+}) \otimes \sL^+) \ar[r] \ar[d] & \myH^d \myR\Gamma_{\fram}\myR\Gamma(\sL^+) \ar@{^{(}->}[d] \\
					0 \ar[r] & \myH^{d-1} \myR\Gamma_{\fram}\myR\Gamma(\sO_{S_{W^+}} \otimes \sL^+) \ar[r] & \myH^d \myR\Gamma_{\fram}\myR\Gamma(\bigoplus_{i=1}^t\sO_{W^+}(-S_{i,W^+}) \otimes \sL^+) \ar[r] & \myH^d \myR\Gamma_{\fram}\myR\Gamma(\bigoplus_{i=1}^t \sL^+) \\
				}
			}
		\]
		The left zeros are due to \autoref{cor.VanishingWithoutRestrictingToPFiber} and the fact that $\sL^+$ is the pullback of a big and semiample line bundle.  The five lemma then shows that the middle arrow is injective.  Dualizing and applying \autoref{lem.B0AlongDAsInverseLimit} implies the containment $\supseteq$ as desired.
	\end{proof}
	
\subsection{$\bigplus$-stable sections and completion}
\label{subsec.B0andcompletion}

The importance of working over a complete base has been highlighted in the presentation above. The goal of this subsection is to show that, when working over a non-complete excellent local base, the base change to the completion can still reasonably be used to define $\myB^0$ and $\myB^0_{\alt}$.  However, if one wishes to work without base changing to the completion, there are a number of (potentially non-equivalent) analogs of $\myB^0$ and $\myB^0_{\alt}$ that one might consider; see also \cite{DattaTuckerOpenness}.

\begin{proposition}
	\label{prop.B0completion}
	Consider
	\begin{itemize}
		\item $(R,\fram, k)$ is a normal local excellent domain with a dualizing complex and with characteristic $p > 0$ residue field,
        \item $X$ is a normal, integral scheme proper over $R$ with $H^0(X, \sO) = R$,
        \item $\Delta \geq 0$ is a $\bQ$-divisor on $X$, and
        \item  $M$ is a $\bZ$-divisor and $\sM = \sO_X(M)$.  
    \end{itemize}
    and for any flat $R$-algebra $S$ denote by $(\blank)_S$ the corresponding base change to $S$. We have that
 	\begin{equation*}
 	    \begin{array}{l}
            \myB^0\big(X_{\widehat{R}}, \Delta_{\widehat{R}}; \sM_{\widehat{R}}\big) = \medskip \\ \qquad \displaystyle     \bigcap_{\substack{f \colon Y \to X\\ \textnormal{finite}}}\im \left( H^0(Y, \sO_Y( K_Y + \lceil{f^* (M - K_X - \Delta)}\rceil)) \otimes_R \widehat{R} \to H^0(X, \sM) \otimes_R \widehat{R} \right) 
        \end{array}
    \end{equation*}
    where the intersection is taken as $\widehat{R}$-submodules of $H^0(X, \sM) \otimes_R \widehat{R}$, and runs over finite covers $f:Y \to X$ as in \autoref{conv.CategoryOfFiniteMaps}, and where $Y$ is normal.  Equivalently, this intersection is the Matlis dual of 
    \[
        \qquad \im\bigg( \myH^d \myR \Gamma_{\fram} \myR \Gamma(X, \sO_X(K_X - M)) \to 
		\varinjlim_{\substack{f \colon Y \to X\\ \textnormal{finite}}} \myH^d \myR \Gamma_{\fram} \myR \Gamma(Y,  \sO_Y(\lfloor f^*(K_X+\Delta-M) \rfloor)) \bigg).     \]
    If additionally $M - K_X - \Delta$ is $\bQ$-Cartier, we also have
    \begin{equation*}
        \begin{array}{l}
            \myB^0\big(X_{\widehat{R}}, \Delta_{\widehat{R}}; \sM_{\widehat{R}}\big) = \medskip \\ \qquad \displaystyle \bigcap_{\substack{f \colon Y \to X\\\textnormal{alteration}}}\im \left( H^0(Y, \sO_Y( K_Y + \lceil f^* (M - K_X - \Delta)\rceil)) \otimes_R \widehat{R} \to H^0(X, \sM) \otimes_R \widehat{R} \right)
        \end{array}
    \end{equation*}
    where the intersection is taken as $\widehat{R}$-submodules of $H^0(X, \sM) \otimes_R \widehat{R}$, and runs over all alterations $f:Y \to X$ as in \autoref{conv.CategoryOfFiniteMaps} and where $Y$ is normal.
    In other words, when computing $\myB^0(X_{\widehat{R}}, \Delta_{\widehat{R}}; \sM_{\widehat{R}})$, it suffices to consider only the completions of the finite covers (respectively, alterations) of $X$.
\end{proposition}

\begin{proof}
	We prove only the statement for finite covers, as the alteration version follows in a similar fashion.
	For any coherent sheaf $\sF$ on $X$, applying \autoref{lem:duality} to $\sF\otimes_R \widehat{R}$ gives  that 
    \begin{equation*}
        \big(\myH^d\myR \Gamma_{\fram} \myR \Gamma(X, \sF))^\vee\cong  \Hom_{\sO_X}(\sF \otimes_R \widehat{R}, \omega_{X_{\widehat{R}}}) 
    \end{equation*}
	where $d=\dim X$ and $(-)^\vee$ denotes Matlis duality $\Hom_R(-, E_R(k))$. 
		Arguing as in the proof of \autoref{lem.B0AsInverseLimit}, we see that
	\begin{equation*}
		\bigcap_{\substack{f \colon Y \to X\\ \textnormal{finite}}}\im \left( H^0(Y, \sO_Y( K_Y + \lceil{f^* (M - K_X - \Delta)}\rceil)) \otimes_R \widehat{R} \to H^0(X, \sM) \otimes_R \widehat{R} \right)
	\end{equation*}
	is Matlis dual to the image of 
    \begin{equation*}
        \xymatrix{
        \myH^d \myR \Gamma_{\fram} \myR \Gamma(X, \sO_X(K_X-M)) \ar[r]^-{\alpha} &
			\displaystyle \varinjlim_{{\substack{f \colon Y \to X \\ \mathrm{finite}}}} 
			\myH^d \myR \Gamma_{\fram} \myR \Gamma(Y, \sO_Y(\lfloor f^*(K_X+\Delta-M) \rfloor)  }.
    \end{equation*}
On the other hand, we have that $\myB^0(X_{\widehat{R}}, \Delta_{\widehat{R}}; \sM_{\widehat{R}})$ is Matlis dual to the image of
\begin{equation*}
	\xymatrix{
	\myH^d \myR \Gamma_{\fram} \myR \Gamma(X, \sO_X(K_X-M)) \ar[r]^-{\beta}
		 & \displaystyle \varinjlim_{{\substack{g \colon Z \to X_{\widehat{R}} \\ 
		 \mathrm{finite}}}} \myH^d \myR \Gamma_{\fram} \myR \Gamma(Z, \sO_Z(\lfloor g^*(K_{X_{\widehat{R}}}+\Delta_{\widehat{R}}-M_{\widehat{R}}) \rfloor).
	}
\end{equation*}
To show the desired equality, it suffices to verify that the kernels of $\alpha$ and $\beta$ coincide. Since a finite cover of $X$ completes to one for $X_{\widehat{R}}$, we need only check that the kernel of $\beta$ is contained in the kernel of $\alpha$.

We shall do this in three steps. An element $\eta$ of the kernel of $\beta$ is necessarily in the kernel of
    \begin{equation*}
        \myH^d \myR \Gamma_{\fram} \myR \Gamma(X, \sO_X(K_X-M)) \to \myH^d \myR \Gamma_{\fram} \myR \Gamma(Z, \sO_Z(\lfloor g^*(K_{X_{\widehat{R}}}+\Delta_{\widehat{R}}-M_{\widehat{R}}) \rfloor)
    \end{equation*}
	for some finite $g \colon Z \to X_{\widehat{R}}$. We first pass from the completion $\widehat{R}$ down to the henselization $R^h$, showing that there is some finite $f' \colon Y' \to X_{R^h}$ with $\eta$ in the kernel of 
   \begin{equation*}
        \myH^d \myR \Gamma_{\fram} \myR \Gamma(X, \sO_X(K_X-M)) \to \myH^d \myR \Gamma_{\fram} \myR \Gamma(Y', \sO_{Y'}(\lfloor f'^*(K_X+\Delta-M) \rfloor).
    \end{equation*}
    Second, we pass from $R^h$ down to a certain pointed \`etale extension $S_i$ of $R$, showing that there is a finite $f_i' \colon Y_i' \to X_{S_i}$ so that $\eta$ is in the kernel of 
       \begin{equation*}
	\myH^d \myR \Gamma_{\fram} \myR \Gamma(X, \sO_X(K_X-M)) \to  \myH^d \myR \Gamma_{\fram} \myR \Gamma(Y'_i, \sO_{Y'_i}(\lfloor f'^*_i(K_{X_{S_i}} +\Delta_{S_i}-M_{S_i}) \rfloor).
\end{equation*}
Finally, in the third and last step, we find a normal and finite $f \colon Y \to X$ so that 
\begin{equation*}
	\myH^d \myR \Gamma_{\fram} \myR \Gamma(X, \sO_X(K_X-M)) \to  \myH^d \myR \Gamma_{\fram} \myR \Gamma(Y, \sO_{Y}(\lfloor f^*(K_X +\Delta-M) \rfloor)
\end{equation*}
verifying that $\eta$ is in the kernel of $\alpha$.

\medskip
\noindent
\textit{Step 1: Passing from $\widehat{R}$ down to $R^h$.}
An element $\eta$ of the kernel of $\beta$ is necessarily in the kernel of
    \begin{equation*}
        \myH^d \myR \Gamma_{\fram} \myR \Gamma(X, \sO_X(K_X-M)) \to \myH^d \myR \Gamma_{\fram} \myR \Gamma(Z, \sO_Z(\lfloor g^*(K_{X_{\widehat{R}}}+\Delta_{\widehat{R}}-M_{\widehat{R}}) \rfloor)
    \end{equation*}
	for some finite $g \colon Z \to X_{\widehat{R}}$. Consider first the henselization $R^h$ of $R$. By Popescu's Theorem \cite[Tag 07GC]{stacks-project} applied to the regular morphism $R^h \to \widehat{R}$, we have that $\widehat{R} = \varinjlim R_i$ is the filtered colimit of smooth $R^h$-algebras $R_i$. We can descend $Z$ to a finite level, so say without loss of generality that there is a finite cover $g_0 \colon Z_0 \to X_{R_0}$ that completes to $g \colon Z \to X_{\widehat{R}}$. Base change to $R_i$ for all $i \geq 0$ gives a finite cover $g_i \colon Z_i \to X_{R_i}$ so that $Z = \varprojlim Z_i$. As
	\begin{equation*}
		\begin{array}{l}
		 \myH^d \myR \Gamma_{\fram} \myR \Gamma(Z, \sO_Z(\lfloor g^*(K_{X_{\widehat{R}}}+\Delta_{\widehat{R}}-M_{\widehat{R}}) \rfloor) \\ = \varinjlim_i  \myH^d \myR \Gamma_{\fram} \myR \Gamma(Z_i, \sO_{Z_i}(\lfloor g_i^*(K_{X_{R_i}} +\Delta_{R_i}-M_{R_i}) \rfloor)
		\end{array}
    \end{equation*}
	we must have that $\eta$ is in fact in the kernel of 
    \begin{equation*}
         \myH^d \myR \Gamma_{\fram} \myR \Gamma(X, \sO_X(K_X-M)) \to  \myH^d \myR \Gamma_{\fram} \myR \Gamma(Z_i, \sO_{Z_i}(\lfloor f_i^*(K_{X_{R_i}} +\Delta_{R_i}-M_{R_i}) \rfloor)
    \end{equation*}
    for some $i$. 
        Now, $R^h \to R_i$ is a smooth map, and using $R_i \to \widehat{R} \to k $ we have a surjection $R_i \to k = R^h / \fram R^h$.
    By \cite[Tag 07M7]{stacks-project}, there is an \'etale $R^h$-algebra $\overline{R_i}$ and $R^h$-algebra homomorphism $R_i \to \overline{R_i}$ so that the surjection $R_i \to k = R^h / \fram R^h$ factors as $R_i \to \overline{R_i} \to k = R^h / \fram R^h$. In particular, $\overline{R_i} \to k = R^h / \fram R^h$ is surjective, so there is a prime $\mathfrak{q}$ of $\overline{R_i}$ lying over $\fram R^h$ with residue field $k$.
    By \cite[Tag 04GG]{stacks-project} as $R^h$ is henselian,  $R^h \to \overline{R_i}$ has a section, and so also (pre-composing that section with $R_i \to \overline{R_i}$) $R^h \to R_i$ must have a section $R_i \to R^h$.
    Base change along this section yields a finite cover $f' \colon Y' = Z_i \otimes_{R_i} R^h \to X_{R^h}$ so that $\eta$ is in the kernel of 
    \begin{equation*}
        \myH^d \myR \Gamma_{\fram} \myR \Gamma(X, \sO_X(K_X-M)) \to \myH^d \myR \Gamma_{\fram} \myR \Gamma(Y', \sO_{Y'}(\lfloor f'^*(K_X+\Delta-M) \rfloor).
    \end{equation*}

 \medskip
\noindent
\textit{Step 2: Passing from  $R^h$ down to a pointed \`etale extension.}
Now, $R^h = \varinjlim S_i$ is the directed colimit of pointed \'etale extensions $S_i$ of $R$, which in turn are localizations of finite extensions of $R$ with $R \subseteq S_i \subseteq R^h \subseteq \widehat{R}$. Once again, the finite cover $f' \colon Y' \to X_{R^h}$ must descend to a finite level, so say without loss of generality that there is a finite cover $f'_0 \colon Y'_0 \to X_{S_0}$ that henselizes to $f'$. Base change to $S_i$ for all $i \geq 0$ gives a finite cover $f'_i \colon Y'_i \to X_{S_i}$ for all $i \geq 0$ so that $Y' = \varprojlim Y'_i$. As 
	\begin{equation*}
		\begin{array}{l}
		\myH^d \myR \Gamma_{\fram} \myR \Gamma(Y', \sO_{Y'}(\lfloor f'^*(K_{X_{R^h}}+\Delta_{R^h}-M_{R^h}) \rfloor) \\ =   \varinjlim_i \myH^d \myR \Gamma_{\fram} \myR \Gamma(Y'_i, \sO_{Y'_i}(\lfloor f'^*_i(K_{X_{S_i}} +\Delta_{S_i}-M_{S_i}) \rfloor)
		\end{array}
   \end{equation*}
   we must have that $\eta$ is in fact in the kernel of
   \begin{equation}
	\label{eq.yiprime}
	\myH^d \myR \Gamma_{\fram} \myR \Gamma(X, \sO_X(K_X-M)) \to  \myH^d \myR \Gamma_{\fram} \myR \Gamma(Y'_i, \sO_{Y'_i}(\lfloor f'^*_i(K_{X_{S_i}} +\Delta_{S_i}-M_{S_i}) \rfloor)
\end{equation}
for some $i$. 

 \medskip
\noindent
\textit{Step 3: Passing from  the pointed \`etale extension down to $R$.}
Let $S$ be the integral closure of $R$ in the fraction field of $S_i$, so that $S$ is a finite extension of $R$ and $S_i$ is the localization of $S$ at one of the (finitely many) maximal ideals $\fram_i$ lying over $\fram$ in $R$. Take $L$ to be normal closure of the function field of $Y_i'$ inside the fixed geometric generic point of $X$, with $G$ the corresponding group of automorphisms of $L$ over the function field of $X$. 
The fixed field $L^G$ is then such that $L^G \subseteq L$ is a Galois extension with Galois group $G$, and $L^G$ is a purely inseparable extension of the function field of $X$.
Set $f \colon Y \to X$ to be the normalization of $X$ inside of $L$. We have that $T = \myH^0(Y,\O_Y)$ is a finite normal extension of $S$, and hence also of $R$. The group $G$ acts on $Y$ and hence also on $T$, and the invariant ring $T^G \subseteq L^G$ is a finite and purely inseparable extension of $R$. Letting $\mathfrak{n}_0, \ldots, \mathfrak{n}_\ell$ denote the (finitely many) maximal ideals of $T$ lying over $\fram$, we have that $G$ acts transitively on the $\mathfrak{n}_j$'s \cite[\href{https://stacks.math.columbia.edu/tag/0BRK}{Tag 0BRK}]{stacks-project}. Without loss of generality, we may assume that $\mathfrak{n}_0 \cap S = \fram_i$.

We have that $Y$ is normal and finite over $X$, and we will argue that $\eta$ is in the kernel of 
\begin{equation}
	\label{eq.y}
	\myH^d \myR \Gamma_{\fram} \myR \Gamma(X, \sO_X(K_X-M)) \to  \myH^d \myR \Gamma_{\fram} \myR \Gamma(Y, \sO_{Y}(\lfloor f^*(K_X +\Delta-M) \rfloor).
\end{equation}
To do so, it suffices to show that $\eta$ is in the kernel of
\begin{equation}
	\label{eq.yn}
	\myH^d \myR \Gamma_{\fram} \myR \Gamma(X, \sO_X(K_X-M)) \to  \myH^d \myR \Gamma_{\fram} \myR \Gamma(Y_{\mathfrak{n}_j}, \sO_{Y_{\mathfrak{n}_j}}(\lfloor f^*(K_X +\Delta-M) \rfloor).
\end{equation}
for $j = 0, \ldots, \ell$, where $Y_{\mathfrak{n}_j} = Y \otimes_T T_{\mathfrak{n}_j}$. Moreover, using the transitive action of $G$ on the set of the $\mathfrak{n}_j$'s, it suffices to show that $\eta$ is in the kernel of \eqref{eq.yn} for $j = 0$. By construction, we have a factorization
\begin{equation*}
	f_{\mathfrak{n}_0} \colon Y_{\mathfrak{n}_0} \to Y_i' \xrightarrow{f_i'} X_{S_i} \to X
\end{equation*}
so that \eqref{eq.yn} factors through \eqref{eq.yiprime}. Thus, we conclude $\eta$ is in the kernel \eqref{eq.y} and hence too of $\alpha$ as desired, completing the proof.
\end{proof}

\begin{remark}
    \label{rem.B0AdjointVersionUpToCompletion}
    With $X$ as in \autoref{prop.B0completion}, suppose we can write $\Delta = S + B$ where $S = \sum_{i = 1}^t S_i$ is reduced and $B$ has no common components with $S$.  Fixing $S_i^+$ in $X^+$ as in \autoref{subsec.AdjointAnalogsOfS^0}, it would be natural to hope that 
    \begin{equation}
        \label{eq.B0AdjointVersionUpToCompletion}
        						\bigcap_{\substack{f \colon Y \to X\\ \textnormal{finite}}}\im \left( \bigoplus_{i = 1}^t H^0(Y, \sO_Y( K_Y + S_{i,Y} + \lceil{f^* (M - K_X - \Delta)}\rceil)) \otimes_R \widehat{R} \to H^0(X, \sM) \otimes_R \widehat{R} \right)         \\
                    	\end{equation}
	or equivalently the Matlis dual of
	\begin{equation*}
		\mathrm{im}\left( \myH^d \myR \Gamma_\m \myR \Gamma(X, \sO_X(K_X - M)) \to \myH^d \myR \Gamma_\m\myR\Gamma(X^+, \bigoplus_{i = 1}^t \sO_{X^+}(-S_i^+ + \pi^*(K_X + \Delta - M))) \right)
	\end{equation*}
    agrees with $\myB^0_{S}(X_{\widehat{R}},S_{\widehat{R}} + B_{\widehat{R}}; \sM_{\widehat{R}})$.  However, we do not see how to prove that -- even when $S$ is irreducible (which may not be preserved under completion).  The problem is we do not seem to have fine enough control over the Galois actions to mimic the end of the proof of \autoref{prop.B0completion} (the reduction to the Henselian case) since we have to simultaneously control $S_{i,Y}$ and maximal ideals lying over $\fram \subseteq R$. In other words, and in the notation used at the end of the proof of \autoref{prop.B0completion}, one must be able to use the Galois action to permute the ideals $\mathfrak{n}_j$ independently of the $S_{i,Y}$. 
    Regardless however, we do define $\hat \myB^0_S(X, S+B; \sM)$ to be the $R$-Matlis dual of the displayed image above.  
\end{remark}

We finally explain what happens when $H^0(X, \sO_X)$ is only semi-local.

\begin{remark}
    \label{rem.B0completionH0NotEqualToR}
    With notation as in \autoref{prop.B0completion}, instead assume that $H^0(X, \sO_X) =: T$ is semi-local with a finite map $R \to T$.  For each maximal ideal $\frn_i$ of $T$ let $T_i = T_{\frn_i}$ denote the localization and set $X_i = X_{T_i} = X \times_{\Spec T} {\Spec T_i}$.  Then $H^0(X_{i}, \sO_{X_i}) = T_i$ and since $T \otimes_R \widehat{R} = \bigoplus_i \widehat{T_i}$, we obtain that $X_{\widehat{R}} = \coprod X_{\widehat{T_i}}$.
    Now by \autoref{prop.B0completion}
    \[
        \begin{array}{l}
            \myB^0(X_{\widehat{T_i}}, \Delta_{\widehat{T_i}}; \sM_{\widehat{T_i}}) = \medskip \\ \qquad \displaystyle     \bigcap_{\substack{f \colon Y_i \to X_i\\ \textnormal{finite}}}\im \left( H^0(Y_i, \sO_{Y_i}( K_{Y_i} + \lceil{f^* (M - K_{X_i} - \Delta)}\rceil)) \otimes_{T_i} \widehat{T_i} \to H^0(X_i, \sM_{T_i}) \otimes_{T_i} \widehat{T_i} \right) 
        \end{array}
    \]
    where the finite covers $f : Y_i \to X_i$ are as in \autoref{conv.CategoryOfFiniteMaps} and each $Y_i$ is normal.
    Each finite cover $Y_i \to X_i$ is the localization of a finite cover $Y \to X$.  Therefore, we have that 
    \[
        \begin{array}{l}
            \myB^0(X_{\widehat{R}}, \Delta_{\widehat{R}}; \sM_{\widehat{R}}) = \bigoplus_i \myB^0(X_{\widehat{T_i}}, \Delta_{\widehat{T_i}}; \sM_{\widehat{T_i}}) = 
            \medskip \\ \qquad \displaystyle    \bigcap_{\substack{f \colon Y \to X\\ \textnormal{finite}}}\im \left( H^0(Y, \sO_{Y}( K_{Y} + \lceil{f^* (M - K_{X} - \Delta)}\rceil)) \otimes_R {\widehat{R}} \to H^0(X, \sM) \otimes_R {\widehat{R}} \right) 
        \end{array}
    \]
    where the intersection  runs over finite covers $f:Y \to X$ as in \autoref{conv.CategoryOfFiniteMaps}, and where $Y$ is normal.
\end{remark}

%% file: B_0_and_graded_rings.tex
\section{Section rings and \texorpdfstring{$\bigplus$-stable}{+-stable} sections} 
\label{sec.B0AndGradedRings}

The goal of this section to relate $\myB^0$ with the test ideal of the section ring $S$ (the affine cone).  As a consequence, we will deduce that $H^0 = \myB^0$ at least when working with sufficiently ample divisors on non-singular schemes since we know that the test ideal agrees with $S$ on the nonsingular locus by \cite{MaSchwedeTuckerWaldronWitaszekAdjoint}.

To avoid dealing with technical issues, we make some simplifying assumptions.
In particular, we assume that $\Delta = 0$ and we work with $\sM = \omega_X \otimes \sL$ as in \autoref{rem.SimpleAdjointFormulationNoPair}.  By \autoref{lem.FiniteCoverToRemoveDelta}, one may frequently reduce to this case.  

\begin{setting}
	\label{set.SettingForGraded}
	With notation as in \autoref{sec:notation}, suppose that $\pi : X \to \Spec (R)$ is a projective morphism where $X$ is a normal integral $d$-dimensional scheme and $R$ is a complete Noetherian local domain of mixed characteristic $(0,p)$.  Choose $\sL$ an ample line bundle.  Write 
	\[
		S = \mathcal{R}(X, \sL) := \bigoplus_{i \geq 0} H^0(X, \sL^i).
	\]
	It is important to note that $S$ is normal, see \cite[Chapter III, Exercise 5.14]{Hartshorne}.  We notice that $R':=H^0(X, \sO_X)$ is a finite $R$-algebra which is integral and normal (as $X$ is so), so $R'$ is itself a complete Noetherian local domain.

	By \cite{BhattAbsoluteIntegralClosure}, once we fix an absolute integral closure $X^+ \to X$, we have graded algebras $S^{+,\gr} \subseteq S^{+, \GR}$ defined as follows.	First, set
	\[
		S^{+,\gr} := \underset{{f \colon Y \to X}}{\colim} \mathcal{R}(Y, f^* \sL) = \bigoplus_{i \in \bZ_{\geq 0}} H^0(X^+, \sL^i)
	\]
	where the colimit runs over all finite normal covers of $X$ dominated by $X^+$.  Likewise after fixing a compatible system of roots $\{\sL^{1/n}\}_{n \geq 1}$ of $\sL$ pulled back to $X^+$ (such systems exist and are unique up to isomorphism, see \cite[Lemma 6.6]{BhattAbsoluteIntegralClosure}), we can define
	\[
		S^{+,\GR} := 	\bigoplus_{i \in \bQ_{\geq 0}} H^0(X^+, \sL^i),
	\]
	 Notice that $S^{+,\gr}$ is a $S^{+,\gr}$-module direct summand of $S^{+,\GR}$. In \cite[Section 6]{BhattAbsoluteIntegralClosure}, it is proved that $S^{+,\gr}/p$ and  $S^{+,\GR}/p$ are big Cohen-Macaulay over $S/p$ under the set up that $X$ is projective over $R$ which is finite type and flat over a henselian DVR. Here we need a version when $R$ is a Noetherian complete local domain and we deduce it from \cite{BhattAbsoluteIntegralClosure}.
	
\begin{theorem}
\label{prop.S+grIsBigCM}
With notation as in \autoref{set.SettingForGraded}, we have $H_{\m+S_{>0}}^j(S^{+,\gr})=0$ for all $j<d+1$. Therefore, $\widehat{S^{+,\gr}}$ is a balanced big Cohen-Macaulay algebra over $\widehat{S}$, where the completion is at the ideal $\m+S_{>0}$. Here $S_{>0}$ denotes the irrelevant ideal, i.e., the ideal generated by all homogeneous elements in $S$ of degree $>0$. 
\end{theorem}
\begin{proof}
We have an exact triangle $\myR\Gamma_{S_{>0}}(S^{+,\gr})\to S^{+,\gr} \to \oplus_{i\in\mathbb{Z}} \myR\Gamma(X^+, \sL^i)$ coming from \cite[\href{https://stacks.math.columbia.edu/tag/0G71}{Tag 0G71}]{stacks-project} and using the fact that $\oplus_{i\in\mathbb{Z}} \myR\Gamma(X^+, \sL^i) \cong \myR \Gamma(\Spec S \setminus V(S_{>0}), \widetilde{S^{+, \gr}})$; which can be seen from a computation of \Cech cohomology (\cf \cite[Theorem A.4.1]{EisenbudCommutativeAlgebraWithAView}). After derived tensoring with $\mathbb{Z}/p$ we have
\[
	\myR\Gamma_{S_{>0}}(S^{+,\gr}/p)\to S^{+,\gr}/p \to \oplus_{i\in\mathbb{Z}} \myR\Gamma(X^+_{p=0}, \sL^i).
\]
\begin{claim}
\label{clm:BhattLemma6.12}
$S^{+,\gr}/p\cong \oplus_{i\in\mathbb{Z}_{\geq0}} \myR\Gamma(X^+_{p=0}, \sL^i)$.
\end{claim}
\begin{proof}
This is essentially \cite[Proposition 6.12]{BhattAbsoluteIntegralClosure}. We briefly recall the argument. Using our chosen compatible system $\{\sL^{1/n}\}$ of roots of $\sL$ over $X^+$, for each $n$ we have a proper birational map $T_n:=\Spec_{X^+}(\oplus_{i\in \mathbb{Z}_{\geq0}} \sL^{\frac{i}{n}}) \to \Spec(\oplus_{i\in \mathbb{Z}_{\geq0}}H^0(X^+, \sL^{\frac{i}{n}}))$, where the latter is considered as an affine scheme over $R$, see \cite[Notation 6.7]{BhattAbsoluteIntegralClosure}. By compatibility we have a system of maps indexed by divisible $n$ with affine transition maps thus we can take limit: $f$: $T_\infty\to \Spec(S^{+,\GR})$, which is pro-proper. Note that $f$ is an isomorphism outside $\Spec(R^+)\subset \Spec(S^{+,\GR})$, and when pulled back along $\Spec(R^+)\subset \Spec(S^{+,\GR})$, it gives $g$: $X^+\to \Spec(R^+)$. Since $X^+$ and $R^+$ are absolute integrally closed, $\myR g_*\mathbf{F}_{p, X^+}\cong \mathbf{F}_{p,\hspace{0.1em} \Spec(R^+)}$ by \cite[Proposition 3.10]{BhattAbsoluteIntegralClosure} and so $\myR f_*\mathbf{F}_{p, T_\infty}\cong \mathbf{F}_{p, \hspace{0.1em} \Spec(S^{+,\GR})}$. Now the $p$-adic completion of $T_\infty$ and $S^{+, \GR}$ are perfectoid by \cite[Lemma 6.10 and 6.11]{BhattAbsoluteIntegralClosure} (these results do not require we are working over an absolute integrally closed DVR). Therefore we have 
\begin{align*}
S^{+, \GR}/p & = \text{RH}_{\overline{\Prism}}(\mathbf{F}_{p, \hspace{0.1em}  \Spec(S^{+,\GR})}) \cong  \text{RH}_{\overline{\Prism}}(\myR f_*\mathbf{F}_{p, T_\infty)}) \cong \myR f_*\text{RH}_{\overline{\Prism}}(\mathbf{F}_{p, T_\infty})  \\
& = \myR f_* \sO_{T_\infty} / p \cong \myR\Gamma(X^+, \sO_{T_\infty}) /p \cong \oplus_{i\in\mathbb{Q}_{\geq0}} \myR\Gamma(X^+_{p=0}, \sL^i).
\end{align*}
Here $\text{RH}_{\overline{\Prism}}$ denotes the $p$-adic Riemann-Hilbert functor of Bhatt-Lurie \cite{BhattLuriepadicRHmodp} (see \cite[Section 3]{BhattAbsoluteIntegralClosure}), the two equalities above follow from \cite[Theorem 3.4 (1)]{BhattAbsoluteIntegralClosure} as the $p$-adic completion of $T_\infty$ and $S^{+, \GR}$ are perfectoid, and the last isomorphism on the first line follows from \cite[Theorem 3.4 (2)]{BhattAbsoluteIntegralClosure} and taking colimit (each $T_n\to \Spec(\oplus_{i\in \mathbb{Z}_{\geq0}}H^0(X^+, \sL^{\frac{i}{n}}))$ is proper). Now passing to the summand, we get $S^{+,\gr}/p\cong \oplus_{i\in\mathbb{Z}_{\geq0}} \myR\Gamma(X^+_{p=0}, \sL^i)$ as desired.
\end{proof}
By Claim \ref{clm:BhattLemma6.12} we have 
$$\myR\Gamma_{S_{>0}}(S^{+,\gr}/p) \cong \oplus_{i<0}\myR\Gamma(X^+_{p=0}, \sL^i)[-1].$$
Applying $\myR\Gamma_\m(-)$ and taking cohomology, we thus have 
$$\myH^j\myR\Gamma_\m\myR\Gamma_{S_{>0}}(S^{+,\gr}/p) \cong \oplus_{i<0}\myH^{j-1}\myR\Gamma_\m\myR\Gamma(X^+_{p=0}, \sL^i).$$
Since $\sL$ is ample, by \autoref{prop.BhattVanishing}, $\myH^{j-1}\myR\Gamma_\m\myR\Gamma(X^+_{p=0}, \sL^i)=0$ for all $j<d$. Thus $\myH^j\myR\Gamma_\m\myR\Gamma_{S_{>0}}(S^{+,\gr}/p)=H_{\m+S_{>0}}^j(S^{+,\gr}/p)=0$ for all $j<d$. But note that we have 
$$\cdots \to H_{\m+S_{>0}}^{j-1}(S^{+,\gr}/p) \to H_{\m+S_{>0}}^j(S^{+,\gr}) \xrightarrow{\cdot p} H_{\m+S_{>0}}^j(S^{+,\gr}) \to H_{\m+S_{>0}}^j(S^{+,\gr}/p) \to \cdots. $$
Since $H_{\m+S_{>0}}^j(S^{+,\gr})$ is $p^\infty$-torsion, multiplication by $p$ is not injective on $H_{\m+S_{>0}}^j(S^{+,\gr})$ unless it vanishes. Thus it follows from the long exact sequence above that $H_{\m+S_{>0}}^j(S^{+,\gr})=0$ for all $j<d+1$. 
\end{proof}
	
	We recall, as explained in \cite[2.6.2]{HyrySmithOnANonVanishingConjecture}, that the graded canonical module $\omega_S$ is the graded dual of $\myH^{d+1} \myR \Gamma_{\fram} \myR \Gamma_{S_{>0}} S$ and that in degree $i > 0$, $[\omega_S]_i = H^0(X, \omega_X \otimes \sL^i)$.  Other potential definitions of the graded canonical have a different shift but we use this choice.
	
	As in \cite{MaSchwedeSingularitiesMixedCharBCM}, we define $\mytau_{S^{+,\gr}}(\omega_S) \subseteq \omega_S$ to be the graded Matlis dual of 
	\[
	    \Image(\myH^{d+1} \myR \Gamma_{\fram} \myR \Gamma_{S_{>0}} S \to \myH^{d+1} \myR \Gamma_{\fram} \myR \Gamma_{S_{>0}} S^{+,\gr}).
	\]
	Note that $S^{+,\gr}$ and $S^{+,\GR}$ are not complete (or perfectoid) but since we are taking local cohomology we can ignore this detail.
	Notice that we can also define $\mytau_{S^{+,\GR}}(\omega_S)$ analogously, but since $S^{+,\gr} \to S^{+,\GR}$ splits, this provides no new information.
\end{setting}

\begin{definition}
	With notation as in \autoref{set.SettingForGraded},  we define for $i > 0$
    \[
	    \myB^0_{\gr}(X,\omega_X \otimes \sL^i) :=  [\mytau_{S^{+,\gr}}(\omega_S)]_i \subseteq [\omega_S]_i = H^0(X,\omega_X \otimes \sL^i).
	\]
\end{definition}

\begin{proposition}
\label{prop.B^0forgradedringsvsB^0}
	In the above situation, $\myB^0_{\gr}(X,\omega_X \otimes \sL^i)=\myB^0(X,\omega_X \otimes \sL^i)$ for all $i > 0$.
	\end{proposition}
\begin{proof}
By graded local duality, we have 
$$\myB^0_{\gr}(X,\omega_X \otimes \sL^i)=[\im(\myH^{d+1}\myR\Gamma_{\m}\myR\Gamma_{S_{>0}}(S)\to \myH^{d+1} \myR\Gamma_{\m}\myR\Gamma_{S_{>0}}(S^{+,\gr}))]_{-i}^\vee$$
where $(-)^\vee$ is Matlis duality over $R$.

Note that we have a commutative diagram of exact triangles: 
\[\xymatrix{
\myR\Gamma_{S_{>0}}(S)\ar[r] \ar[d] & S \ar[r] \ar[d] & \oplus_{i\in\mathbb{Z}} \myR\Gamma(X, \sL^i) \ar[r]^-{+1} \ar[d] & {} \\
\myR\Gamma_{S_{>0}}(S^{+,\gr}) \ar[r] &  S^{+,\gr} \ar[r] &  \oplus_{i\in\mathbb{Z}} \myR\Gamma(X^+, \sL^i)  \ar[r]^-{+1} & {} 
}.
\]
Applying $\myR\Gamma_\m$ and taking cohomology, we have 
\[\xymatrix{
[\myH^{d}\myR\Gamma_\m(S)]_{-i} \ar[r] \ar[d] &  \myH^{d}\myR\Gamma_\m\myR\Gamma(X, \sL^{-i}) \ar[r] \ar[d] & [\myH^{d+1}\myR\Gamma_\m\myR\Gamma_{S_{>0}}(S)]_{-i} \ar[d]\ar[r] & 0\\
[\myH^{d}\myR\Gamma_\m(S^{+,\gr})]_{-i} \ar[r]  &  \myH^{d}\myR\Gamma_\m\myR\Gamma(X^+, \sL^{-i}) \ar[r] & [\myH^{d+1}\myR\Gamma_\m\myR\Gamma_{S_{>0}}(S^{+,\gr})]_{-i} \ar[r] & 0
}
\]
Note that, $[\myH^{d}\myR\Gamma_\m(S)]_{-i}=[\myH^{d}\myR\Gamma_\m(S^{+,\gr})]_{-i}=0$ when $i>0$: this is because $\m\subseteq R$ lives in degree $0$ so $[\myH^{d}\myR\Gamma_\m(S)]_{-i}=\myH^{d}\myR\Gamma_\m([S]_{-i})=0$ and similarly for $S^{+,\gr}$. Therefore the diagram shows that 
\begin{align*}
&[\im(\myH^{d+1}\myR\Gamma_{\m}\myR\Gamma_{S_{>0}}(S)\to \myH^{d+1} \myR\Gamma_{\m}\myR\Gamma_{S_{>0}}(S^{+,\gr}))]_{-i} \\
=& \im(\myH^{d}\myR\Gamma_\m\myR\Gamma(X, \sL^{-i})\to \myH^{d}\myR\Gamma_\m\myR\Gamma(X^+, \sL^{-i}))
\end{align*}
Taking Matlis dual over $R$ and using \autoref{eq.lem.B0AsInverseLimit.DualImageForFinite} in \autoref{lem.B0AsInverseLimit}, we see that $\myB^0_{\gr}(X,\omega_X \otimes \sL^i)=\myB^0(X,\omega_X \otimes \sL^i)$ for all $i > 0$ as desired.
\end{proof}

In what follows, we will be studying $H^0(X, \omega_X \otimes \sL^N)$ for $N$ sufficiently large when $X$ has sufficiently mild singularities.  For our purposes, sufficiently mild means the following.

\begin{definition}
	We say that a Noetherian ring $R$ has \emph{finite summand singularities} if there exists a finite extension $R \subseteq S$ such that $S$ is regular and the map splits as a map of $R$-modules.
\end{definition}

We note that by \cite{CRMPSTCoversofRDP}, 2-dimensional klt singularities of residual characteristic $p > 5$ are finite summand singularities. For an excellent ring, the locus of finite summand singularities is readily verified to be open.  We also note that if a Noetherian ring $R$ has finite summand singularities, then any finite extension $R \hookrightarrow S$ splits as a map of $R$-modules as a consequence of the direct summand theorem \cite{AndreDirectsummandconjecture}.  In particular, using the notation from the next section \autoref{def:globally_B_regular}, we see that $\Spec R$ is globally $\bigplus$-regular (that is $R \subseteq S$ splits for every finite extension, in other words $R$ is a splinter).  Note in equal characteristic $p > 0$, being globally $\bigplus$-regular is quite closely related to $F$-regularity (and they are conjectured to be equivalent), an analog of klt singularities.  Not all rings $R$ that are globally $\bigplus$-regular have finite summand singularities however, even in equal characteristic $p > 0$. 

In \cite[Theorem 4.1]{MaSchwedeTuckerWaldronWitaszekAdjoint} it was shown that if $(R, \Delta)$ has simple normal crossings at $Q$ with $\lfloor \Delta_Q \rfloor = 0$, then $\mytau_{B}(R, \Delta)_Q = R_Q$.  We will use this below, which will later help us study $H^0(X, \omega_X \otimes \sL^N)$.  

In the next theorem we assume that $X$ has finite summand singularities, which implies that $S$ has finite summand singularities (and so globally $\bigplus$-regular singularities) away from the irrelevant ideal $S_{>0}$.  It is natural to try to compute the some (local) $\bigplus$-test ideal on $S$ to measure this.  However, we don't know that such ideals commute with localization.  On the other hand, the ideal $\im({}^*\Hom_S(S^{+,\gr}, S)\to S)$, which can be viewed as a sort of test ideal, can be thought of as a measure of the obstruction to the global $\bigplus$-regularity of $S$ (again, its formation does not obviously commute with localization since $S^{+,\gr}$ is note finitely presented over $S$).  Regardless of these difficulties, we are able to that that image ideal contains $S_{>m}$ for some $m \gg 0$.

In what follows, we use graded $\Hom$ and graded injective hulls, denoted ${}^{*}\Hom$ and ${}^* E$ respectively, see \cite[Chapter 3, Section 6]{BrunsHerzog}.  

\begin{theorem}
	\label{thm.GradedTestIdealAnnihilated}
	Suppose that $X$, $\sL$, $R$ and $S$ are as in \autoref{set.SettingForGraded}.  Let $\fram_S = \fram \cdot S+ S_{>0}$ denote the homogeneous maximal ideal of $S$. Suppose $X$ has finite summand singularities.  Then for $m \gg 0$, $S_{>m}$ annihilates the kernel of 
	\[
		{}^*E_S \to {}^*E_S \otimes S^{+, \gr}
	\] 
	where ${}^*E_S = \myH^{d+1}\myR\Gamma_{\fram_S}(\omega_S)$ is the graded injective hull of the residue field of $S$.  Dually, 
	$$S_{>m}\subseteq \im({}^*\Hom_S(S^{+,\gr}, S)\to S).$$
\end{theorem}

\begin{proof}
	Begin by choosing a finite affine cover $\{U_i\}$ of $X$, such that for each such $U_i$ there exists a finite surjective map $f_i : V_i \to U_i$ where $V_i$ is regular and such that $\sO_{U_i} \to (f_i)_* \sO_{V_i}$ splits.  Without loss of generality, we may assume that $\sL|_{U_i} \cong \sO_{U_i}$, $\omega_{V_i} \cong \sO_{V_i}$, and $U_i$ is the complement of some $V(t_i)$ with $t_i \in H^0(X, \sL^n)$ for some $n$ (which we may pick independently of $i$).   For each $i$, let $X_i$ denote the normalization of $X$ in $K(V_i)$ and fix $\pi_i : X_i \to X$ to be the induced map.   Let $S_i$ denote the graded section ring of $X_i$ with respect to $\pi_i^* \sL$, let $\fram_i$ denote the homogeneous maximal ideal, and note that  $S \subseteq S_i$ is finite. Set $\widehat{S_i}$ to be the $\m_i$-adic completion of $S_i$. Notice we also abuse notation to view $t_i = \pi_i^* t_i$ as an element of $S_i$ and also as an element of $\widehat{S_i}$.  
	Forgetting the grading for now, embed $\omega_{\widehat{S_i}}^{(-1)} \subseteq \widehat{S_i}$ such that $\omega_{\widehat{S_i}}^{(-1)}[t_i^{-1}] = \widehat{S_i}[t_i^{-1}]$.  By Flenner's local Bertini theorem (see \cite[Satz 2.1]{FlennerLocalBertini}, \cite[Theorem 1]{TrivediLocalBertini} and \cite{TrivediLocalBertiniErratum}), there exists $f \in \omega_{\widehat{S_i}}^{(-2)}$, such that $f$ is not contained in $Q^{(2)}$ for all $Q\in\Spec(\widehat{S_i})$ not containing $\omega_{\widehat{S_i}}^{(-2)}$. In particular $f$ is not contained in $Q^{(2)}$ for all $Q\in\Spec(\widehat{S_i}[t_i^{-1}])$, it follows that $\widehat{S_i}[t_i^{-1}]/(f)$ is regular.  Set $D_i$ to be the effective divisor corresponding to $f \in \omega_{\widehat{S_i}}^{(-2)}$ and let $\Delta_i = \frac{1}{2}D_i$.  By construction, $(\widehat{S_i}, \Delta_i)$ is simple normal crossing at all $Q\in\Spec(\widehat{S_i}[t_i^{-1}])$ and $K_{\widehat{S_i}} + \Delta_i=\frac{1}{2}\Div(f)$ is $\bQ$-Cartier. Applying \cite[Theorem 4.1]{MaSchwedeTuckerWaldronWitaszekAdjoint} with the perfectoid big Cohen-Macaulay $S_i$-algebra $\widehat{S_i^+} = \widehat{S^+}$, we have that $\mytau_{\widehat{S^+}}(\widehat{S_i}, \Delta_i)_Q = \widehat{S_i}_Q$ for all $Q\in\Spec(\widehat{S_i}[t_i^{-1}])$. In particular, there exists $a$ such that $t_i^a\in \mytau_{\widehat{S^+}}(\widehat{S_i}, \Delta_i)$.
	
	Now since $\widehat{S_i}\xrightarrow{\cdot f^{1/2}}\widehat{S^+}$ factors through $\widehat{S_i}(K_{\widehat{S_i}})\cong\omega_{\widehat{S_i}}$ by construction, we have induced maps
	$$\cdot f^{1/2}: H_{\m_S}^{d+1}(\omega_{\widehat{S_i}})\to H_{\m_S}^{d+1}(\omega_{\widehat{S_i}}\otimes \widehat{S^+})\to H_{\m_S}^{d+1}(\widehat{S^+}).$$
	Applying Matlis duality, we have 
	\[\xymatrix{
	H_{\m_S}^{d+1}(\omega_{\widehat{S_i}})^\vee \ar[d]^\cong  & H_{\m_S}^{d+1}(\omega_{\widehat{S_i}}\otimes \widehat{S^+})^\vee \ar[d]^\cong \ar[l] & H_{\m_S}^{d+1}(\widehat{S^+})^\vee \ar[l] \ar[d]^\cong \\
	\widehat{S_i} & \Hom_{\widehat{S_i}}(\widehat{S^+}, \widehat{S_i}) \ar[l] & \Hom_{\widehat{S_i}}(\widehat{S^+}, \omega_{\widehat{S_i}}) \ar[l]
	}
	\]
Since the image of the composition map is equal to $\mytau_{\widehat{S^+}}(\widehat{S_i}, \Delta_i)$ by \cite[Proof of Theorem 6.12]{MaSchwedeSingularitiesMixedCharBCM}, we see that $\im(\Hom_{\widehat{S_i}}(\widehat{S^+}, \widehat{S_i})\to \widehat{S_i})$ contains $\mytau_{\widehat{S^+}}(\widehat{S_i}, \Delta_i)$, so it contains $t_i^a$, i.e., there exists a map $\psi_i : \widehat{S^+} \to \widehat{S_i}$ such that $t_i^a$ is in the image.

	Now if we view $S$ as a subring of $S_i$, then by hypothesis, $t_i^b$ is in the image of some $\rho_i : S_i \to S$.  Completing, we see that $t_i^{a+b}$ is in the image of  $\widehat{S^+} \xrightarrow{\psi_i} \widehat{S_i} \xrightarrow{\rho_i} \widehat{S}$.  Since $S \to \widehat{S^+}$ factors through $S^{+, \gr}$, we see that $t_i^{a+b}$ annihilates the kernel of 
	\[
		{}^*E_S \to {}^*E_S \otimes S^{+, \gr}.
	\]
	
	Finally, since the $U_i = D(t_i)$ cover $X$, we see that the $t_i^{a+b}$ generate the prime ideal $S_{>0}$ up to radical.  Thus $S_{>0}^{m_1} \subseteq \langle t_1^{a+b}, \dots, t_n^{a+b} \rangle$ for some $m_1$.  But since $S$ is Noetherian, a sufficiently high veronese subalgebra $S^{(e)} \subseteq S$ is generated in degree 1, \cite[Chapter III, Proposition 3]{Bourbaki1998}.  Thus by \cite[Chapter III, Proposition 2, Lemma 2]{Bourbaki1998}, for all $l \gg 0$ and $k \geq 0$ we have that $S_{ke} \cdot S_{l} = S_{ke + l}$.  It follows that $S_{>m} \subseteq S_{>0}^{m_1}$ for some sufficiently large $m > 0$.  This completes the proof of the first statement.
	
	For the final statement, by graded Matlis duality, we know that the cokernel of
	${}^*\Hom_S(S^{+,\gr}, S)\to S$ is annihilated by $S_{>m}$, i.e., 
	$S_{>m}\subseteq \im({}^*\Hom_S(S^{+,\gr}, S)\to S)$ as desired.
\end{proof}

\begin{theorem} \label{thm:B0-equals-H0-for-high-ample}
	Suppose that $X$, $\sL$, $R$ and $S$ are as in \autoref{set.SettingForGraded}.  Let $\fram_S = \fram \cdot S+ S_{>0}$ denote the homogeneous maximal ideal of $S$. Suppose $X$ has finite summand singularities.  Further suppose that $\sL$ is ample on $X$.  Then there exists $m>0$ such that $S_{>m}\cdot \omega_S\subseteq \mytau_{S^{+,\gr}}(\omega_S)$. As a consequence, for $n \gg 0$, we have that 
	\[
		\myB^0(X, \omega_X \otimes \sL^n) = H^0(X, \omega_X \otimes \sL^n).
	\]
\end{theorem}

\begin{proof}

	By \autoref{thm.GradedTestIdealAnnihilated}, we know $S_{>m}\subseteq \im({}^*\Hom_S(S^{+,\gr}, S)\to S).$
	This means for all (homogeneous) $x\in S_{>m}$, there is a (homogeneous) map $\phi\in {}^*\Hom_S(S^{+,\gr}, S)$ such that $\phi(1)=x$. 
	Therefore the composition map: 
	$$H_{\m_S}^{d+1}(S)\to H_{\m_S}^{d+1}(S^{+,\gr}) \xrightarrow{H_{\m_S}^{d+1}(\phi)} H_{\m_S}^{d+1}(S)$$
	is multiplication by $x$ on $H_{\m_S}^{d+1}(S)$. Thus we find that $S_{>m}$ annihilates the kernel of $H_{\m_S}^{d+1}(S)\to H_{\m_S}^{d+1}(S^{+,\gr})$. 
 By the definition of $\mytau_{S^{+,\gr}}(\omega_S)$ and using graded local duality, it follows $S_{>m}\cdot \omega_S\subseteq \mytau_{S^{+,\gr}}(\omega_S)$.
 	
	Finally, since $\omega_S$ is finitely generated, for all $n\gg0$, $[\omega_S]_n\subseteq S_{>m}\cdot \omega_S \subseteq \mytau_{S^{+,\gr}}(\omega_S)$. Therefore $[\omega_S]_n=[\mytau_{S^{+,\gr}}(\omega_S)]_n$. Hence by \autoref{prop.B^0forgradedringsvsB^0}, we have 
	$$\myB^0(X, \omega_X  \otimes \sL^n) = H^0(X, \omega_X \otimes \sL^n)$$
	for all $n\gg0$ as desired.
\end{proof}

\subsection{An application to Fujita's conjecture in mixed characteristic}
We conclude with a mixed characteristic version of a special case of Fujita's conjecture, analogous to the main result of \cite{SmithFujitaFreenessForVeryAmple}.  Indeed, our proof very closely follows the strategy of K.~Smith.

\begin{theorem} \label{thm:Karen_Fujita_v1} Let $X$ be a $d$-dimensional regular scheme (or a scheme with finite summand singularities) which is flat and projective over $R$.  Set $t = \dim R$ and let $\sL$ be an ample globally generated line bundle on $X$.  Then $\omega_X \otimes \sL^{d - t + 1}$ is globally generated by $\myB^0(X, \omega_X \otimes \sL^{d-t+1})$.
\end{theorem}

 We first prove the following result, whose proof is nearly the same as, and heavily inspired by, \cite[Proposition 3.3]{SmithFujitaFreenessForVeryAmple}.

\begin{proposition} \label{prop:Karen_Fujita_local_coh_v1} With notation as in \autoref{thm:Karen_Fujita_v1}, let $(S,\fram_S)$ be the section ring of $X$ with respect to $\sL$ as above. Further suppose that $y_0, \dots, y_{t-1}$ are a system of parameters for $R$ and $x_t, \dots, x_d \in S_1$ are such that $y_0, \dots, x_d$ are a system of parameters for $S$.  

Then there exists $N_0 \in \bN$ such that every homogeneous $0 \neq \eta \in H^{d+1}_{\fram_S}(S)$ of  degree less than $-N_0$ ($\deg \eta < -N_0$) admits a non-zero multiple $\eta'$ of degree $-d -1 + t = -\dim X - 1+ \dim R = -\dim S + \dim R$.
Furthermore, any such $\eta'$ has non-zero image in $H^{d+1}_{\fram_S}(S^{+,\gr})$.  
\end{proposition}
\begin{proof}
We begin with a claim.
\begin{claim}
	\label{clm.KernelOfLocalCohomMapBoundedDegree}
	There exists $N_0 \in \bN$ such that the kernel $K$ of $H^{d+1}_{\fram_S}(S) \to H^{d+1}_{\fram_S}(S^{+,\gr})$ is zero in degrees $< -N_0$.
\end{claim}
\begin{proof}[Proof of claim] 
	Let $K$ be the kernel  of  $H^{d+1}_{\fram_S}(S) \to H^{d+1}_{\fram_S}(S^{+,\gr})$.  The graded Matlis dual $K^{\vee}$ fits into an exact sequence $0 \to \mytau_{S^{+,\gr}}(\omega_S) \to \omega_S \to K^{\vee} \to 0$.  Now, \autoref{thm:B0-equals-H0-for-high-ample} implies that $[K^{\vee}]_{i} = 0$ for $i \gg 0$.  Thus $[K]_{n} = 0$ for $n \ll 0$, which proves \autoref{clm.KernelOfLocalCohomMapBoundedDegree}.
\end{proof}

We now come to our main computation.

\begin{claim}
	\label{clm.ImageOfEtaInLocalCohomSPlusIs0}
	Suppose $\eta \in H^{d+1}_{\fram_S}(S)$ is a homogeneous element of degree $-N < -d + t -1$ such that every $S$-multiple of degree $-d + t-1$ has zero image in $H^{d+1}_{\fram_S}(S^{+,\gr})$ (that is $\Image(S_{N-d+t -1}\cdot \eta) = 0 \in H^{d+1}_{\fram_S}(S^{+,\gr})$). Then the image of $\eta$ in $H^{d+1}_{\fram_S}(S^{+,\gr})$ is zero.
\end{claim}
\begin{proof}[Proof of claim]
	Write $\eta = [\frac{z}{\overline{y}^v\overline{x}^v}]$ where $\overline{x}= x_t\cdots x_d$ and $\overline{y}=y_0 \cdots y_{t-1}$ and $z$ is homogeneous of degree $(d-t+1)v-N$. 	Because $S_{N-d+t -1} \cdot \eta$ has zero image in $H^{d+1}_{\fram_S}(S^{+,\gr})$, there exists some $s \geq 0$ so that
	\begin{equation}\label{eq:clm2.fujita_step_2}
		(x_t, \ldots, x_d)^{N-d + t - 1} \cdot (\overline{y}^s \overline{x}^s) \cdot z \subseteq (y_0^{v+s}, \ldots, x_d^{v+s}) \widehat{S^{+,\gr}}.
	\end{equation}		
	Thus, since $\widehat{S^{+,\gr}}$ is Cohen-Macaulay and $y_0, \dots, x_d$ is a regular sequence on it, we have that 
	\[
		 z \in (y_0^{v}, \ldots, x_d^{v})\widehat{S^{+,\gr}} : (x_t, \ldots, x_d)^{N-d + t - 1}.
	\]	
	Now working modulo $y_0^v,\dots, y_{t-1}^v$, we see that 
\begin{align*}
    \overline{z} & \in (x_t^v,\dots,x_d^v)(\widehat{S^{+,\gr}}/(y_0^v,\dots,y_{t-1}^v)): (x_t, \ldots, x_d)^{N-d + t - 1} \\
    &=\Big((x_t^v, \ldots, x_d^v) + (x_t, \ldots, x_d)^{(d-t+1)v - N + 1}\Big)(\widehat{S^{+,\gr}}/(y_0^v,\dots,y_{t-1}^v))
\end{align*}
where the equality follows because $x_t,\dots,x_d$ is a regular sequence on $(\widehat{S^{+,\gr}}/(y_0^v,\dots,y_{t-1}^v))$ so the computation of colon ideal is the same as if the $x_i$'s are indeterminates in a polynomial ring (see \cite[(3.3.3)]{SmithFujitaFreenessForVeryAmple}). It follows that 
	\[
		z \in \Big((y_0^v, \ldots, x_d^v) + (x_t, \ldots, x_d)^{(d-t+1)v - N + 1}\Big)\widehat{S^{+,\gr}}.
	\]
	However, since $z$ has degree $(d - t+1)v - N$, we see that $z \in (y_0^v, \ldots, x_d^v)\widehat{S^{+,\gr}}$.  Thus the image of $\eta$ in $H^{d+1}_{\fram_S}(S^{+,\gr})$ is zero, proving \autoref{clm.ImageOfEtaInLocalCohomSPlusIs0}.
\end{proof}
To finish the proposition, choose $N_0$ as in \autoref{clm.KernelOfLocalCohomMapBoundedDegree} and a nonzero $\eta \in H^{d+1}_{\fram_S}(S)$ of degree $< -N_0$.  Hence $\eta \notin K = \ker\big(H^{d+1}_{\fram_S}(S) \to H^{d+1}_{\fram_S}(S^{+,\gr})\big)$ by \autoref{clm.KernelOfLocalCohomMapBoundedDegree}.  But now by the contrapositive of \autoref{clm.ImageOfEtaInLocalCohomSPlusIs0}, we see that $\eta$ has a nonzero $S$-multiple $\eta'$ of degree $-d + t -1 = -\dim S + \dim R$ whose image in $H^{d+1}_{\fram_S}(S^{+,\gr})$ is also nonzero.  This completes the proof.
\end{proof}

\begin{proof}[Proof of Theorem \ref{thm:Karen_Fujita_v1}] 
We first show that there exists a finite \'{e}tale extension $R'$ of $R$ such that the section ring $S'$ of $X':= X\times_RR'$ with respect to $\sL|_{X'}$ admits a homogeneous system of parameters $y_0,\dots,y_{t-1},x_t,\dots, x_d$ as in the statement of \autoref{prop:Karen_Fujita_local_coh_v1}. Let $R^{sh}$ be the strict hensalization of $R$ (so $R^{sh}$ has an infinite residue field). Then $X^{sh}:=X\times_RR^{sh}$ is flat and projective over $R^{sh}$ of relative dimension $d-t$ and so $X^{sh}_0:= X^{sh}\times_{R^{sh}}(R^{sh}/\m R^{sh})$ is projective over an infinite field of dimension $d-t$. Since $\sL$ is globally generated on $X$, the image of the linear system $|\sL|$ in $H^0(X^{sh}_0, \sL|_{X^{sh}_0})$ is base point free.  As $X^{sh}_0$ is projective over an infinite field, we can pick general linear combinations of sections in the image of $|\sL|$, call them $\overline{x_t},\dots,\overline{x_d}$, such that they form a homogeneous system of parameters in $\mathcal{R}(X^{sh}_0, \sL|_{X^{sh}_0})$. Since $R^{sh}$ is a colimit of finite \'{e}tale extensions of $R$, there exists a finite \'{e}tale complete domain extension $R'$ of $R$ such that $\overline{x_i}$ is the image of $x_i\in H^0(X', \sL|_{X'})$. Now it is straightforward to check that $y_0,\dots,y_{t-1},x_t,\dots, x_d$ form a system of parameters in $S'=\mathcal{R}(X', \sL|_{X'})$ for every system of parameters $y_0,\dots,y_{t-1}$ of $R$: modulo $\m$ (the radical of $(y_0,\dots,y_{t-1})$), $S'/\m S'$ is a homogeneous coordinate ring of $X'_0$ and so $\mathcal{R}(X^{sh}_0, \sL|_{X^{sh}_0})$ is integral over $S'/\m S'$ of the same dimension, thus by our choice, $x_t,\dots,x_d$ form a homogeneous system of parameters in $S'/\m S'$ (as they are so in $\mathcal{R}(X^{sh}_0, \sL|_{X^{sh}_0})$).

Next we claim that in order to show $\omega_X\otimes\sL^{d-t+1}$ is globally generated by $\myB^0(X,\omega_X\otimes\sL^{d-t+1})$, it is enough to prove this when we base change $X$ to $X'$.  Indeed, we have a surjective map of sheaves $T : \omega_{X'} \otimes \sL^{d-t+1} \to \omega_{X} \otimes \sL^{d-t+1}$.  Furthermore, if $\myB^0(X', \omega_{X'} \otimes \sL^{d-t+1})$ (globally) generates the left side its image via $T$ generates the right sheaf.  But $\myB^0(X', \omega_{X'} \otimes \sL^{d-t+1}) \twoheadrightarrow {\myB^0(X, \omega_{X} \otimes \sL^{d-t+1})}$ surjects by \autoref{lem.FiniteCoverToRemoveDelta}.  Therefore, without loss of generality, we now replace $R$ and $X$ by $R'$ and $X'$ to assume that $S=\mathcal{R}(X, \sL)$ admits a homogenous system of parameters $y_0,\dots,y_{t-1}, x_t,\dots,x_d$ as in \autoref{prop:Karen_Fujita_local_coh_v1}. Note that $X'$ is still regular (or has finite summands singularities) since it is finite \'{e}tale over $X$.

By the discussion above, it is enough to show that the multiplication map (which is well defined since $\mytau_{S^{+,\gr}}(\omega_S)$ is an $S$-module)
\[
\underbrace{H^0(X, \sL^{N-d+t-1})}_{S_{N-d+t-1}} \otimes_R \underbrace{\myB^0(X, \omega_X \otimes \sL^{d-t+1})}_{[\mytau_{S^{+,\gr}}(\omega_S)]_{d-t+1}} \to \underbrace{\myB^0(X, \omega_X \otimes \sL^N)}_{[\mytau_{S^{+,\gr}}(\omega_S)]_{N}} = \underbrace{H^0(X, \omega_X \otimes \sL^N)}_{[\omega_S]_{N}} 
\]
is surjective for $N \gg 0$. By graded local duality on $S$, this is equivalent to the injectivity of the map
\[
	\begin{array}{rl}
	& [H^{d+1}_{\fram_S}(S)]_{-N}\\
	\cong & \big[\Image\big(H^{d+1}_{\fram_S}(S) \to H^{d+1}_{\fram_S}(S^{+,\gr})\big)\big]_{-N} \\
	\to & \Hom_R\big(S_{N-d+t-1} \otimes_R [\mytau_{S^{+,\gr}}(\omega_S)]_{d-t+1}, E\big) \\
	\cong & \Hom_R\Big(S_{N-d+t-1}, \Image\big([H^{d+1}_{\fram_S}(S)]_{-d+t-1} \to{}  
	[H^{d+1}_{\fram_S}(S^{+,\gr})]_{-d+t-1}\big)\Big)	\end{array}
\]
where $E$ is the injective hull of the residue field of the complete local ring $(R, \fram)$ and the final isomorphism is $\Hom$-tensor adjointness and duality.  Just as in \cite[Lemma 1.3]{SmithFujitaFreenessForVeryAmple}, this map sends $\eta \in [H^{d+1}_{\fram_S}(S)]_{-N}$ to the map which is multiplication by $\eta$.  Hence this map is injective by \autoref{prop:Karen_Fujita_local_coh_v1} and our proof is complete.
\end{proof}

\begin{remark}
	It would be natural to try to obtain the following stronger result.  Suppose that $X$ has the property that for each closed point $x \in X$, we have that $H^d_{x}(\sO_X) \to H^d_{x}(\sO_{X^+})$ injects (in other words, $\sO_{X,x}$ is $\sO_{X,x}^+$-rational in the sense of \cite{MaSchwedeSingularitiesMixedCharBCM}, but without the Cohen-Macaulay hypothesis).  We expect that if $\sL$ is a globally generated ample line bundle on $X$, then
	\[
		\omega_X \otimes \sL^{d - t + 1}
	\]
	is globally generated by $\myB^0(X, \omega_X \otimes \sL^{d - t + 1})$.  The missing piece is a proof that $\mytau_{S^{+,\gr}}(\omega_S)$ agrees with $\omega_S$ except at the irrelevant ideal (a generalization of \autoref{thm:B0-equals-H0-for-high-ample}).  
\end{remark}

%% file: GloballyRelativelyBCMRegular.tex
\section{Globally \texorpdfstring{$\bigplus$}{+}-regular pairs}
\label{sec:globally-+-regular-pairs}
In this section we define and discuss various properties of globally $\bigplus$-regular pairs; analogous to globally $F$-regular pairs in positive characteristic. The reader interested in the results on globally $F$-regular pairs is  referred to \cite{SchwedeSmithLogFanoVsGloballyFRegular}.  
The reader unfamiliar with this story is invited to imagine that this means the section ring / cone has singularities which are a mixed characteristic analog of klt singularities.  Throughout this section, we work under the following assumptions unless otherwise stated:

\begin{enumerate}
    \item $X$ is a normal, integral, $d$-dimensional, excellent scheme with a dualizing complex where every closed point has residue field of positive characteristic.  \label{General+RegularHypothesis.1}
    \item $\Delta \geq 0$ is a $\bQ$-divisor on $X$. \label{General+RegularHypothesis.2}
        \end{enumerate}
Whenever there is a base scheme $\Spec R$, we also assume that $R$ is excellent with a dualizing complex and that every closed point of $\Spec R$ has positive characteristic residue field. 

Frequently, we also assume that $R$ is complete and $X$ is proper over $\Spec R$. However, the above setting also applies when the base is a positive or mixed characteristic Dedekind domain.

\begin{definition}
    \label{def:globally_B_regular}
    We say that $(X, \Delta)$ is \emph{globally $\bigplus$-regular} if for every finite dominant map $f : Y \to X$ with $Y$ normal, the map $\sO_X \to f_* \sO_{Y}(\lfloor f^*\Delta \rfloor)$ splits as a map of $\sO_X$-modules.  

    If we have $X \to \Spec R$ proper, then we say that $(X, \Delta)$ is \emph{completely globally $\bigplus$-regular over $R$} if for every closed point $z$ of $\Spec R$, the base change $(X_{\widehat{R_z}}, \Delta_{\widehat{R_z}})$ is globally $\bigplus$-regular.  If $R$ is clear from the context, we will omit the ``over $R$''.
\end{definition}

Notice that globally $\bigplus$-regular is an absolute notion but \emph{completely} globally $\bigplus$-regular requires a base.    

\begin{remark}
    In the above definition, we may restrict ourselves to $f : Y \to X$ such that $f^* \Delta$ has integer coefficients, since any $f' : Y' \to X$ is dominated by such a $Y$.
\end{remark}

\begin{remark}[Characteristic zero]
    \label{rem.CharacteristicZeroSplinterNotAsGood}
    If we did not require that our closed points have residual characteristic $p > 0$, then our definition would not always yield what the reader might expect. For instance, when $X$ is purely of characteristic zero, our condition defining global $\bigplus$-regularity simply means that $X$ is normal and that the coefficients of $\Delta$ are $< 1$.  If one additionally assumes that $K_X + \Delta$ is $\bQ$-Cartier, then one could alternately require that for every alteration $\pi : Y \to X$ {the map} $\sO_X \to \myR \pi_* \sO_Y(\lfloor \pi^* \Delta \rfloor)$ splits.  In characteristic zero, this {again does not provide any global information and} only means that $(X, \Delta)$ has rational singularities in the sense of \cite{SchwedeTakagiRationalPairs}. Lastly, one could require the trace map 
    \begin{equation}
        \label{eq.TRegularitySurjection}
        H^0(Y, \sO_Y( K_Y - \lfloor{\pi^* (K_X + \Delta)}\rfloor)) \to H^0(X, \sO_X)
    \end{equation}
    to be surjective for every alteration (as discussed in \cite{TakamatsuYoshikawaMMP}, where they called it global $T$-regularity). This in characteristic zero is equivalent to $(X,\Delta)$ being klt. 
    When $X \to \Spec R$ is proper and $R$ only admits positive characteristic closed points (the latter is always assumed throughout this section), we see that global $\bigplus$-regularity is equivalent to global $T$-regularity (the surjection of \autoref{eq.TRegularitySurjection}).   This follows from  \autoref{prop.GlobalBRegularSplits} in view of \autoref{prop.B0completion}.
\end{remark}

\begin{remark}[Non-integral $X$]
    \label{rem.Globall+RegularForNonIntegralX}
    If $X$ is not integral, but still normal with all connected components $d$-dimensional, we define $(X, \Delta)$ to be \emph{globally $\bigplus$-regular} if all its connected components are.  This coincides with the variant of $\myB^0$ in this setting as explained in \autoref{rem.NonintegralB^0}.  The results of this section go through since they may all be checked working one component at a time.
\end{remark}

\begin{remark}
    \label{rem.WhatToDoIfH0IsNotR}
    If $(R, \fram)$ is complete local, $X \to \Spec R$ is proper, and $X$ is integral, then $R \to H^0(X, \sO_X)$ is a finite map of rings. Since $X$ is integral, we see that $H^0(X, \sO_X)$ is an integral domain. But since $R$ is complete and in particular henselian, we also know that $H^0(X, \sO_X)$ is a product of local rings \cite[Tag 04GG (9)]{stacks-project}. Such a product cannot be an integral domain unless it only has one factor, thus we know that $H^0(X, \sO_X)$ is a local ring. 
    
    On the other hand if $(R, \fram)$ is not complete but only a Noetherian local ring, then $H^0(X, \sO_X)$ is only semi-local (it has finitely many maximal ideals).  In many cases though, we localize $T = H^0(X, \sO_X)$ at a maximal ideal to obtain a local ring $T'$ (and perhaps even take completion of that if desired) and consider the base change $X_{T'} = X \times_T {T'}$.  Replacing $R$ by $T'$ and $X$ by $X_{T'}$ we have that $H^0(X, \sO_X) = R$.
\end{remark}

\begin{lemma}
    \label{lem.GloballyPlusRegularIsLocalOnTheBase}
    Suppose we are given $X\to \Spec(R)$, the following are equivalent:     
    \begin{enumerate}
        \item $(X, \Delta)$ is globally $\bigplus$-regular.
        \item for each closed point $z \in \mSpec R$ we have that the base change to the localization $(X_{R_z}, \Delta_{R_z})$ is globally $\bigplus$-regular.        
    \end{enumerate}    
\end{lemma}
\begin{proof}
    The pair is globally $\bigplus$-regular if and only if the evaluation-at-1 map
    \begin{equation}
        \label{lem.GloballyPlusRegularIsLocalOnTheBase.eq.EvalMap}
        \Hom_{\cO_X}(f_* \sO_Y(\lfloor f^* \Delta \rfloor), \sO_X) \to H^0(X, \sO_X)
    \end{equation}
    surjects for each finite dominant $f : Y \to X$ with $Y$ normal.  Indeed, that map is surjective if and only if there exists $\phi \in \Hom(f_* \sO_Y(\lfloor f^* \Delta \rfloor), \sO_X)$ sending $1$ to $1$.      
   
    Now, we observe that since $\Hom(f_* \sO_Y(\lfloor f^* \Delta \rfloor), \sO_X) = H^0(X, \sHom_{\sO_X}(f_* \sO_Y(\lfloor f^* \Delta \rfloor))$ and $f_* \sO_Y$ is a coherent $\sO_X$-module, the formation of this $\Hom$-set commutes with localization on the base (a flat base change).  In other words:
    \[
        \begin{array}{rl}
            & \Hom(f_* \sO_Y(\lfloor f^* \Delta \rfloor), \sO_X) \otimes_R R_z\\
            \cong & \Gamma(X, \sHom_{\cO_X}(f_* \sO_Y(\lfloor f^* \Delta \rfloor), \sO_X) \otimes_R R_z) \\
            \cong & \Gamma(X_{R_z}, \sHom_{\cO_X}(f_* \sO_Y(\lfloor f^* \Delta \rfloor) \otimes_R R_z, \sO_{X_{R_z}})) \\
            \cong & \Gamma(X_{R_z}, \sHom_{\cO_{X_{R_z}}}(f_* \sO_{Y_{R_z}}(\lfloor f^* \Delta|_{X_{R_z}} \rfloor), \sO_{X_{R_z}})).
        \end{array}
    \]
    Note the evaluation-at-1 map also base changes to the evaluation-at-1 map of the localization.  Hence, since a map of modules is surjective if and only if it is surjective after localization at all maximal ideals, for each $Y \to X$ finite surjective, we see that \autoref{lem.GloballyPlusRegularIsLocalOnTheBase.eq.EvalMap} surjects if and only if 
    \[
        \Hom_{\cO_{X_{R_z}}}(f_* \sO_{Y_{R_z}}(\lfloor f^* \Delta|_{X_{R_z}} \rfloor), \sO_{Y_{R_z}}) \to H^0(X_{R_z}, \sO_{X_{R_z}})
    \]
    surjects for each $z \in \mSpec R$.

    Finally, notice that a finite surjective $h : Y' \to X_{R_z}$ with $Y'$ integral produces a finite surjective $Y \to X$ that localizes to $h$ (simply take the normalization of $\sO_X$ in the fraction field $Y'$) and we are indexing over the same set of finite surjective maps (which we can take with a fixed geometric generic point).
\end{proof}

\begin{lemma}\label{lem.GloballyBregularBox}
If $(X, \Delta)$ is globally $\bigplus$-regular and $0 \leq \Delta' \leq \Delta$, then $(X, \Delta')$ is globally $\bigplus$-regular as well.
\end{lemma}
\begin{proof}
This follows from the definition.\end{proof}

\begin{proposition}
    \label{prop.GlobalBRegularSplits}    
    Suppose that $X \to \Spec R$ is proper and $(R, \fram)$ is local. Then 
    $(X, \Delta)$ is globally $\bigplus$-regular if and only if $\myB^0(X_{\widehat R}, \Delta_{\widehat R}; {\sO_{X_{\widehat R}}} )= H^0(X_{\widehat R}, \sO_{X_{\widehat R}})$.  In the case that $K_X + \Delta$ is $\bQ$-Cartier, this is also equivalent to $\myB^0_{\alt}(X_{\widehat R}, \Delta_{\widehat R}; {\sO_{X_{\widehat R}}} )= H^0(X_{\widehat R}, \sO_{X_{\widehat R}})$.
\end{proposition}
\begin{proof}
                Notice that the map $R \to H^0(X, \sO_X) =: T$ is finite (although it will not be injective if $X \to \Spec R$ is not dominant) and so its base change $T \otimes_R \widehat{R}$ to the completion of $R$ may break up into a product of normal domains $\prod T_i$.  In particular, the normal scheme $X_{\widehat{R}}$ may have several connected components.  In such a case, working one component at a time, we may replace $R$ by $T_i$, a localization of $T$ at a maximal ideal, and $X$ by the base change $X \otimes_{T} T_i$ and so assume that $H^0(X, \sO_X) = R$, also see \autoref{rem.B0completionH0NotEqualToR}.

    By \autoref{prop.B0completion} we have that $\myB^0(X_{\widehat R}, \Delta_{\widehat R}; {\sO_{X_{\widehat R}}})$ equals
    \[
        \bigcap_{\substack{f \colon Y \to X\\ \textnormal{finite}}}\im \left( H^0(Y, {\sO_Y( K_Y - \lfloor{f^* (K_X + \Delta)}\rfloor)}) \otimes_R \widehat{R} \to H^0(X, {\sO_X}) \otimes_R \widehat{R} \right). 
    \]
    Suppose that $\myB^0(X_{\widehat R}, \Delta_{\widehat R}; {\sO_{X_{\widehat R}}} )= H^0(X_{\widehat R}, \sO_{X_{\widehat R}})$ and let $f \colon Y \to X$ be a normal finite cover. Then, $\Tr \colon f_* \sO_Y(K_Y - \lfloor f^*(K_X + \Delta)\rfloor) \to \sO_X$ is surjective on global sections (after completion, hence also before it), and so there exists a map $\phi$ such that
    \[
        \sO_X \xrightarrow{\phi} f_* \sO_Y(K_Y - \lfloor f^*(K_X + \Delta)\rfloor) \xrightarrow{\Tr} \sO_X
    \]
    is the identity (to define $\phi$, send $1 \in \Gamma(X, \sO_X)$ to a global section which $\Tr$ sends to $1$).  Hence, $\Tr$ is split surjective.  But now apply $\sHom(-, \sO_X)$ to $\Tr$ and observe that the obtained map $\sO_X \to f_* \sO_{Y}(\lfloor f^*\Delta \rfloor)$ also splits, as desired. 

    For the converse, note that when the map $\sO_X \to f_* \sO_{Y}(\lfloor f^*\Delta \rfloor)$ splits, then the dual map\footnote{obtained by applying $\sHom(-, \sO_X)$} $\Tr : f_* \sO_Y(K_Y - \lfloor f^*(K_X + \Delta)\rfloor) \to \sO_X$ is split surjective, and so surjective on global sections.  Hence, $\myB^0(X_{\widehat R}, \Delta_{\widehat R}; {\sO_{X_{\widehat R}}} )= H^0(X_{\widehat R}, \sO_{X_{\widehat R}})$ by \autoref{prop.B0completion}.     
    
    The final assertion follows from \autoref{cor.B0VsB0Alt}.
\end{proof}

\begin{corollary}
    \label{cor.GlobalBRegularEqualsCompletely}   
    Suppose that $X \to \Spec R$ is proper.  Then $(X, \Delta)$ is globally $\bigplus$-regular if and only if it is completely globally $\bigplus$-regular over $R$.   
\end{corollary}
\begin{proof}
    By \autoref{lem.GloballyPlusRegularIsLocalOnTheBase} we may assume that $R$ is local.  
            We then see that $(X, \Delta)$ is globally $\bigplus$-regular if and only if $\myB^0(X_{\widehat R}, \Delta_{\widehat R}; {\sO_{X_{\widehat R}}} )= H^0(X_{\widehat R}, \sO_{X_{\widehat R}})$ by \autoref{prop.GlobalBRegularSplits}.  But this latter statement is also equivalent to requiring that $(X_{\widehat{R}}, \Delta_{\widehat{R}})$ is globally $\bigplus$-regular.
\end{proof}

We now show that globally $\bigplus$-regular pairs have controlled singularities.

\begin{proposition}[Global to local]
\label{prop.GlobalBRegularRationalklt}
    Suppose $X$ is globally $\bigplus$-regular. Then $X$ is pseudo-rational and in particular Cohen-Macaulay (and so has rational singularities in the sense of \cite{KovacsRationalSingularities}).
    Further, suppose that $(X, \Delta)$ is globally $\bigplus$-regular and $K_X + \Delta$ is $\bQ$-Cartier.  Then $(X, \Delta)$ is klt (and Cohen-Macaulay).
\end{proposition}
\begin{proof}
    Since the question is local, we may localize $X$ at a closed point $x \in X$ and take $R=\sO_{X,x}$ so that $X = \Spec R$.  Furthermore, we may assume that $R$ is complete by \autoref{cor.GlobalBRegularEqualsCompletely}.
    
    First suppose that $X$ is globally $\bigplus$-regular and $\Delta = 0$.  By definition, $R$ is a splinter, hence it is Cohen-Macaulay and pseudo-rational (\cite[Corollary 5.10 and Remark 5.14(1)]{BhattAbsoluteIntegralClosure}).  This proves the first statement.

    Now suppose that $(X = \Spec R, \Delta)$ is globally $\bigplus$-regular and $K_X + \Delta$ is $\bQ$-Cartier. By \autoref{prop.GlobalBRegularSplits}
    the trace map
    \[
        H^0(Y, \sO_Y( K_Y - \lfloor f^* (K_X+ \Delta)\rfloor)) \to H^0(X, \sO_X) = R
    \]
    is surjective for every projective birational morphism $f \colon Y \to X$ from a normal integral scheme $Y$. This is the case exactly when $\lceil K_Y - f^* (K_X+ \Delta) \rceil$ is effective and exceptional over $X$, which, in turn, is equivalent to $\lfloor \Delta \rfloor = 0$ and all the exceptional divisors on $Y$ having log discrepancy greater than $0$. As this is true for every projective birational morphism, $(X,\Delta)$ is klt.  Further, $X$ is Cohen-Macaulay {by our work above} since $X$ is globally $\bigplus$-regular by \autoref{lem.GloballyBregularBox}.
    \end{proof}

\begin{lemma} \label{lem.B0EqualsH0ForGlobally+Regular}
    Suppose that $(X, \Delta)$ is globally $\bigplus$-regular, $X \to \Spec R$ is proper and $R$ is local.  Then for any line bundle $\sL = \sO_X(L)$ we have $\myB^0(X_{\widehat{R}}, \Delta_{\widehat{R}}; \sL_{\widehat{R}}) = H^0(X_{\widehat{R}}, \sL_{\widehat{R}})$.  In particular, if $R$ is complete then $\myB^0(X_, \Delta; \sL) = H^0(X, \sL)$.
\end{lemma}
    
\begin{proof}
    Without loss of generality, using \autoref{prop.B0completion} and \autoref{cor.GlobalBRegularEqualsCompletely}, we may assume that $R$ is complete.  Since $\sO_X \to f_* \sO_{Y}(\lfloor f^*\Delta \rfloor)$ splits for every normal finite cover $f : Y \to X$, then so does $\sO_X(K_X-L) \to f_* \sO_Y(\lfloor f^*(K_X + \Delta - f^* L) \rfloor)$.  
    Hence 
    \[
        \myH^{d} \myR \Gamma_{\fram}\myR\Gamma(X, \sO_X(K_X - L)) \to \myH^d \myR \Gamma_{\fram} \myR\Gamma(Y, \sO_Y(\lfloor f^*(K_X + \Delta - f^* L) \rfloor))
    \]
    is injective and so by \autoref{lem.B0AsInverseLimit}, we see that $\myB^0(X, \Delta; \sL) = H^0(X, \sL)$ as desired.
\end{proof}

\begin{corollary}[Relative Kawamata-Viehweg vanishing for globally $\bigplus$-regular varieties]
    \label{cor.RelativeKVVanishingForG+Regular}
    Suppose that $X \to \Spec R$ is proper, $(X, \Delta)$ is globally $\bigplus$-regular and $L$ is a Cartier divisor such that $L - (K_X + \Delta)$ is $\bQ$-Cartier, big, and semiample.  Then $H^i(X, \sO_X(L)) = 0$ for all $i > 0$.
\end{corollary}
\begin{proof}
    Via \autoref{cor.GlobalBRegularEqualsCompletely} and flat base change for cohomology, we may assume that $R$ is complete and local.  
    By \autoref{lem:duality-general}, to show that $H^i(X, \sO_X(L)) = 0$ for $i > 0$, it suffices to show that $\myH^{-i}\myR\Gamma_\m\myR\Gamma(X, \sO_X(-L)\otimes\omega_X^{\mydot})=0$. Since $X$ is globally $\bigplus$-regular, it is Cohen-Macaulay by \autoref{prop.GlobalBRegularRationalklt}.  Hence, we must show that $\myH^{d-i}\myR\Gamma_\m\myR\Gamma(X, \sO_X(K_X -L))=0$ for all $i>0$ where $d = \dim X$.  Consider the map 
    \[
        \sO_X(K_X -L) \to \underset{{f \colon Y \to X}}{\colim} f_* \sO_Y(f^*(K_X + \Delta - L)) = {\pi_*}\sO_{X^+}(\pi^*(K_X + \Delta - L))
    \]
    where $\pi: X^+\to X$ and we restrict ourselves to finite $f \colon Y \to X$ such that $f^* \Delta$ has integer coefficients.  Note that while $\pi$ is not finite, it is affine so its higher direct images vanish for quasi-coherent sheaves by \cite[\href{https://stacks.math.columbia.edu/tag/01XC}{Tag 01XC}]{stacks-project}. This is a colimit of split maps since $X$ is globally $\bigplus$-regular and hence 
    \[
        \myH^{d-i}\myR\Gamma_\m\myR\Gamma(X, \sO_X(K_X -L)) \to \myH^{d-i}\myR\Gamma_\m\myR\Gamma(X, \pi_* \sO_{X^+}(\pi^*(K_X + \Delta - L)))
    \]
    injects.  But the right side is zero for all $i>0$ by \autoref{cor.VanishingWithoutRestrictingToPFiber}, completing the proof.
    
\end{proof}

We also obtain vanishing of the structure sheaf.

\begin{proposition}
\label{prop:bigplusOvanish}
    Suppose that $X$ is globally $\bigplus$-regular and $X \to \Spec R$ is proper.  Then $H^i(X, \sO_X) = 0$ for all $i > 0$.
\end{proposition}
\begin{proof}
    We may assume $R$ is local with residue field $R/\m$ of characteristic $p > 0$.
    By \autoref{prop.BhattVanishing}(a), taking $L = \sO_X$, we can find a finite cover $\pi : Y \to X$ where the map $H^i(X_{p=0}, \sO_X) \to H^i(Y_{p=0}, \sO_Y)$ is zero, but it is also split injective, therefore $H^i(X_{p=0}, \sO_X)=0$. Since we have an exact sequence $H^i(X, \sO_X)\xrightarrow{\cdot p} H^i(X, \sO_X)\to H^i(X_{p=0}, \sO_X)=0$, it follows that $H^i(X, \sO_X)$ is $p$-divisible, but as $H^i(X, \sO_X)$ is a finitely generated $R$-module and $p\in\m$, thus $H^i(X, \sO_X)=0$ by Nakayama's lemma. 
\end{proof}

\begin{lemma} \label{lemma:BregularINpositiveCharacteristic}
    Suppose that $X$ is an $F$-finite normal integral scheme of characteristic $p > 0$.      If $(X,\Delta)$ is globally $F$-regular, then it is globally $\bigplus$-regular.  \end{lemma}
\begin{proof}
    Suppose that $(X, \Delta)$ is globally $F$-regular and that $\pi : Y \to X$ is a finite dominant map with $Y$ normal and integral.  By replacing $Y$ by a higher cover if necessary, we may assume that $\pi^* \Delta$ has integer coefficients.
    \begin{claim}
        There exists a divisor $D \geq 0$ on $X$ and a map
    \[
        \phi \in \Hom_{\sO_X}(\pi_* \sO_Y(\pi^* \Delta), \sO_X(D))
    \]
    which sends $1 \mapsto 1$.  
    \end{claim}
    \begin{proof}[Proof of claim]
        We begin by explaining what the claim is asserting.  Let $K(X)$ and $K(Y)$ denote the constant sheaves associated to the fraction fields of $X$ and $Y$ respectively.  Notice that $\pi_* \sO_Y(\pi^* \Delta)$ is a subsheaf of $\pi_* K(Y)$.  Since $\Delta$ is effective, $1 = 1_Y$ is a global section of $\pi_* \sO_Y(\pi^* \Delta)$.  Likewise since $D$ is effective, $1 = 1_X$ is a global section of $\sO_X(D)$.  The claim asserts that we can find a $\phi$ that sends $1$ to $1$.
        
        Now, working on an affine chart {$j \colon U \hookrightarrow X$} whose complement is a divisor $D' \geq 0$, set $V = \pi^{-1}(U)$. It follows from \cite[Proposition 4.2]{BlickleSchwedeTuckerTestAlterations}\footnote{In that paper, it was assumed that $K_X + \Delta$ is $\bQ$-Cartier, but that hypothesis is not needed when $\pi$ is finite.} that there exists $\phi_U \in \Hom(\pi_* \sO_V(\pi^* \Delta), \sO_U) \cong \pi_* \O_V(K_Y  -\pi^*(K_X + \Delta) )$ sending $1 \mapsto 1$.  By working from the stalk of the generic point, we see that $\phi_U$ induces a map $\phi_K : \pi_* K(Y) \to K(X)$.  Note, restricting $\phi_K$ to $U$ and restricting the domain to $\pi_* \sO_V(\pi^* \Delta)\subseteq (\pi_* K(Y))|_U$ we recover $\phi_U$ since both $X$ and $Y$ are integral schemes.  Next, restrict the source of $\phi_K$ to $\pi_* \sO_Y(\pi^* \Delta) \subseteq \pi_* K(Y)$ to obtain 
        \[ 
            \phi' : \pi_* \sO_Y(\pi^* \Delta) \to j_* \sO_U = \bigcup_{n \geq 0} \sO_X(nD').
        \]  
        But since the source of $\phi'$ is coherent, the image of $\phi'$ is contained in $\sO_X(nD')$ for some sufficiently large $n > 0$.  Set $D = nD'$.  This proves the claim.
    \end{proof}

    Now, since $(X, \Delta)$ is globally $F$-regular, there exists $e > 0$ and 
    \[
        \psi \in \Hom(F^e_* \sO_X(\lceil (p^e - 1) \Delta \rceil + D), \sO_X)
    \]
    which sends $F^e_*1$ to $1$.  Twisting the $\phi$ from the claim by $\lceil (p^e-1) \Delta \rceil$, and pushing forward by Frobenius, we obtain a map 
    \[
        \phi' \colon F^e_* \pi_*\sO_Y((F^e)^* \pi^* \Delta) \subseteq F^e_* \pi_*\sO_Y(\pi^* \Delta + \pi^* \lceil (p^e - 1) \Delta \rceil) \to F^e_* \sO_X(\lceil (p^e-1) \Delta \rceil+D)
    \]
    sending $F^e_* 1 \mapsto F^e_* 1$.  Composing with $\psi$ we see that the composed map $\psi \circ \phi'$ sends $F^e_*1$ to $1$. Thus $\sO_X\to F^e_* \pi_*\sO_Y((F^e)^* \pi^* \Delta)$ splits, and since this map factors through $\pi_*\sO_Y(\pi^*\Delta)$, we have $\sO_X\to \pi_*\sO_Y(\pi^*\Delta)$ splits. This proves that $(X, \Delta)$ is globally $\bigplus$-regular.
\end{proof}

\begin{remark}
    It is reasonable to expect that there is a converse to \autoref{lemma:BregularINpositiveCharacteristic}.  Even in the local case where $X = \Spec R$ and $\Delta = 0$ (but $R$ is not $\bQ$-Gorenstein) this is an open question.  It specializes in that setting to the conjecture that splinters are strongly $F$-regular, see for instance \cite{SinghQGorensteinSplinters,ChiecchioEnescuMillerSchwede}.  Note that we do not even know that splinters are klt for some appropriate boundary if $R$ is not $\bQ$-Gorenstein.  In the non-local case, we expect that $(X,\Delta)$ is of log Fano type but we do not know how to show that.
\end{remark}

We also state a related open question, analogous to the main result of \cite{SchwedeSmithLogFanoVsGloballyFRegular}.

\begin{conjecture}
    With notation as in the start of the section, suppose that $X$ is globally $\bigplus$-regular and that $X \to \Spec R$ is projective.  Then there exists an effective $\bQ$-divisor $\Delta$ on $X$ such that $(X, \Delta)$ is globally $\bigplus$-regular and $-K_X - \Delta$ is ample.  
\end{conjecture}

This conjecture is open in characteristic $p > 0$, even in the local case when $X = \Spec R$.  In mixed characteristic, even if $X$ is nonsingular, we do not even know how to construct a boundary $\Delta$ where $(X, \Delta)$ is lc and $K_X + \Delta \sim_{\bQ} 0$.

\begin{corollary}
\label{cor:BregularINpositiveCharacteristicWithoutF-finite}
    Suppose that $(X,\Delta)$ is proper over a complete Noetherian local ring $(R,\m, k)$ of characteristic $p>0$. Let $R':=R\widehat{\otimes}_kk^{1/p^\infty}$ be the complete tensor product (so $R'$ is an $F$-finite complete local ring). If $(X_{{R'}},\Delta_{R'})$ is globally $F$-regular, then $(X, \Delta)$ is globally $\bigplus$-regular.   
\end{corollary}
\begin{proof}
The natural maps $X_{R'}^+\to X_{R'}\to X$ induce:
    \[
        \begin{array}{rl}
            & \myH^d\myR\Gamma_\m\myR\Gamma(X, \sO_X(K_X))\\
            \to & \myH^d\myR\Gamma_\m\myR\Gamma(X_{R'}, \sO_{X_{R'}}(K_{X_{R'}}))\\
            \to & \myH^d\myR\Gamma_\m\myR\Gamma(X_{R'}^+, \sO_{X_{R'}^+}(\pi'^*(K_{X_{R'}}+\Delta_{R'})))
        \end{array}
    \]
    where $\pi':X_{R'}^+\to X_{R'}$ and $d=\dim X$.  Notice that the base change of $\omega_X$ is $\omega_{X_{R'}}$.  The first map is injective by faithfully flat base change, and the second map is injective since $(X_{{R'}},\Delta_{R'})$ is globally $\bigplus$-regular by \autoref{lemma:BregularINpositiveCharacteristic} and using duality (see \autoref{prop.GlobalBRegularSplits} and \autoref{lem.B0AsInverseLimit}). Therefore the composition is injective. But as $X_{R'}^+\to X$ factors through $X^+$, we obtain that 
    $$\myH^d\myR\Gamma_\m\myR\Gamma(X, \sO_X(K_X))\to\myH^d\myR\Gamma_\m\myR\Gamma(X^+, \sO_{X^+}(\pi^*(K_X+\Delta)))$$
    is injective where $\pi: X^+\to X$. So using \autoref{prop.GlobalBRegularSplits} and \autoref{lem.B0AsInverseLimit} again we see that $(X, \Delta)$ is globally $\bigplus$-regular.
        \end{proof}

\begin{proposition} \label{proposition:pushforward-of-global-splinter}
If $f : X \to Y$ is a proper birational morphism between schemes satisfying the conditions at the start of this section, and $\Delta \geq 0$ is a $\bQ$-divisor on $X$ such that $(X, \Delta)$ is globally $\bigplus$-regular, then so is $(Y, f_* \Delta)$.  Hence $Y$ is also pseudo-rational (and so rational in the sense of \cite{KovacsRationalSingularities}), and if $K_Y + f_* \Delta$ is $\bQ$-Cartier, then $(Y, f_* \Delta)$ is klt.
\end{proposition}
\begin{proof}
Set $\Delta_Y = f_*\Delta$. Let $g : Z \to Y$ be a normal finite cover and let $W$ be the normalization of $X \times_Z Y$.  We have that following diagram:
\[
\xymatrix{
X \ar[d]_f & W \ar[d]^h \ar[l]_{\xi} \\
Y & Z \ar[l]^g.
}
\]
Since $X$ is globally $\bigplus$-regular, the map $\sO_X \to \xi_*\sO_W(\lfloor \xi^*\Delta \rfloor)$ splits. Let $U \subseteq Y$ be an open subset with complement of codimension at least two and such that $V := f^{-1}(U) \xrightarrow{f} U$ is an isomorphism. By restricting the above splitting to $V$, we get that
the map
\[
\sO_U \to g_*\sO_{g^{-1}(U)}(\lfloor g^*\Delta_Y|_U \rfloor)
\]
splits as well, and so does $\sO_Y \to g_*\sO_{Z}(\lfloor g^*\Delta_Y \rfloor)$, since $Y \setminus U$ is of codimension two and the sheaves are S2.  Hence, $(Y, \Delta_Y)$ is globally $\bigplus$-regular. The last assertion follows from \autoref{prop.GlobalBRegularRationalklt}.
\end{proof}

In the opposite direction, we have the following for \'etale covers.

\begin{proposition}
\label{Quasietale+reg}
    Suppose that $(X, \Delta)$ is globally $\bigplus$-regular, $X \to \Spec R$ is proper and $f : Y \to X$ is a finite quasi-\'etale\footnote{Meaning \'etale in codimension 1} cover from a normal integral scheme $Y$.  Then $(Y, f^* \Delta)$ is also globally $\bigplus$-regular.
\end{proposition}
\begin{proof}
    We may assume that $H^0(X, \sO_X) = R$.  Then we may assume that $R$ is complete and local by \autoref{cor.GlobalBRegularEqualsCompletely}, possibly working component by component on $Y$ after completion so that $H^0(Y, \sO_Y)$ is also local.
    Since $f$ is quasi-\'etale, we know that $f^* K_Y = K_X$.  Hence by \autoref{lem.B0AsInverseLimit} we see that 
    \[
        \myB^0(Y, f^* \Delta; \sO_Y) \twoheadrightarrow \myB^0(X, \Delta; \sO_X) = H^0(X, \sO_X)
    \]
    surjects.  But the induced trace map $H^0(Y, \sO_Y) \to H^0(X, \sO_X)$ sends the maximal ideal of $H^0(Y, \sO_Y)$ to the maximal ideal of $H^0(X, \sO_X)$, hence $\myB^0(Y, f^* \Delta; \sO_Y) = H^0(Y, \sO_Y)$.
\end{proof}

\subsubsection*{The local case}
We conclude the section by briefly describing some of the key features of \emph{local} $\bigplus$-regularity.

\begin{definition} \label{definition:plus-regular-on-local-ring}
    Suppose that $(A, \fram)$ is an excellent normal local ring whose residue field $A/\fram$ has positive characteristic and set $Y = \Spec A$.  Further suppose that $\Delta \geq 0$ is a $\bQ$-divisor on $Y$.  We say that $(Y, \Delta)$ is \emph{$\bigplus$-regular} if it is globally $\bigplus$-regular in the sense of \autoref{def:globally_B_regular}. 
    \end{definition}

    Notice that the $\fram$-adic completion $\widehat{A}$ of $A$ is $\bigplus$-regular if and only if so is $A$ (this is \autoref{cor.GlobalBRegularEqualsCompletely} applied to the identity map $X = \Spec A \to \Spec A$).
        Furthermore, if $\Delta=0$, then $A$ is $\bigplus$-regular if and only if it is a {\it splinter}. Lastly, if $(A,\Delta)$ is $\bigplus$-regular, then it is Cohen-Macaulay by \autoref{prop.GlobalBRegularRationalklt}.

\begin{lemma}
    \label{lem.StalkIsGlobally+RegularIfAmbient}
    If $(X, \Delta)$ as in the start of the section is globally $\bigplus$-regular, then for $x \in X$ whose stalk $A = \sO_{X,x}$ has positive characteristic residue field, $(\Spec A, \Delta|_{\Spec A})$ is globally $\bigplus$-regular.
\end{lemma}

\begin{remark}
    Suppose that $K_A + \Delta$ is $\bQ$-Cartier.  It follows from \autoref{prop.B0completion} by taking $X = \Spec A$ that $(A, \Delta)$ is globally $\bigplus$-regular if and only if $(\widehat{A}, \Delta_{\widehat A})$ is $\mathrm{BCM}_{\widehat{A^+}}$-regular.
\end{remark}

\subsection{Purely globally \texorpdfstring{$\bigplus$}{+}-regular schemes}

\begin{definition}
    \label{def.Purely+Regular}
    With notation as in the start of the section, suppose     that there exists a reduced divisor $S$ such that $\Delta = S + B$ for $B\geq 0$ with no common components with $S$.  Fix a reduced subscheme $S^+$ in $X^+$ as in the the second paragraph of \autoref{subsec.AdjointAnalogsOfS^0} with corresponding $\sum_{i = 1}^t S_{i,Y} = S_Y \to Y$ on each finite dominant map $f : Y \to X$ with $Y$ normal.   We say that $(X,S+B)$ is \emph{purely globally $\bigplus$-regular (along $S$)}, if for every finite dominant $f : Y \to X$ with $Y$ normal, the following map splits     \[
        \sO_X \to f_* \bigoplus_{i = 1}^t \sO_{Y}(-S_{i,Y} + \lfloor f^*(S + B) \rfloor).
    \]
    If we have $X \to \Spec R$ proper as in the start of the section, then we say that $(X, S+B)$ is \emph{completely purely globally $\bigplus$-regular over $R$ (along S)} if the base change to the completion $(X_{\widehat{R_z}}, \Delta_{\widehat{R_z}})$ along every closed point $z \in \mSpec R$ is purely globally $\bigplus$-regular.
\end{definition}

Note in the case that $S$ is integral and $f^*(S+B)$ has integer coefficients, this is simply asking that 
\[
    \sO_X \to f_* \sO_Y( f^*(S + B) - S_Y)
\]
splits.  

This definition is still meaningful even when $R$ is not complete although we will primarily work in the complete case, see the issues discussed in \autoref{rem.B0AdjointVersionUpToCompletion}.  In particular, we do \emph{not} have a full analog of \autoref{prop.GlobalBRegularSplits} or any analog of \autoref{cor.GlobalBRegularEqualsCompletely}.  However, see \autoref{cor.PurelyGlobally+RegularVsCompletelyAssumeBigAndSemiample} where we prove the equivalence of completely purely globally $\bigplus$-regular pairs with purely globally $\bigplus$-regular pairs when $-K_X - \Delta$ is big and semiample.

\begin{lemma}
    \label{lem.PurelyGloballyPlusRegularIsLocalOnTheBase}
   {Suppose we have $X\to\Spec(R)$}, the following are equivalent: 
    \begin{enumerate}
        \item $(X, S+B)$ is purely globally $\bigplus$-regular.
        \item for each closed point $z \in \mSpec R$ we have that the base change to the localization $(X_{R_z}, (S+B)_{R_z})$ is purely globally $\bigplus$-regular.        
    \end{enumerate}    
\end{lemma}
\begin{proof}
    We restrict to $Y$ large enough so that $f^* \Delta = f^* (S+B)$ has integer coefficients.
    The pair is globally $\bigplus$-regular if and only if the sum of the evaluation-at-1 maps
    \[
        \bigoplus_{i = 1}^t \Hom(f_* \sO_Y(-S_{i,Y} + f^*(S+B) ), \sO_X) \to H^0(X, \sO_X)
    \]
    surjects for each finite dominant $f : Y \to X$ with $Y$ normal.  Again, this surjectivity can be checked after localizing at closed points of $R$.  
\end{proof}

Similar to \autoref{prop.GlobalBRegularSplits}, we have the following alternate characterization of purely globally $\bigplus$-regular.  We recall the following notation from \autoref{rem.B0AdjointVersionUpToCompletion}.  If $H^0(X, \sO_X) = R$ is local and $X \to \Spec R$ is proper then   $\hat\myB^0_{S}(X, S+B; \sO_X) \subseteq H^0(X_{\widehat{R}}, \sO_{X_{\widehat{R}}})$ is the $R$-Matlis dual of 
    \[
        \mathrm{Im}\left( \myH^d \myR \Gamma_\m \myR \Gamma(X, \sO_X(K_X)) \to \myH^d \myR \Gamma_\m\myR\Gamma(X^+, \bigoplus_{i = 1}^t \sO_{X^+}(-S_i^+ + \pi^*(K_X + \Delta))) \right).
    \]

\begin{proposition}
    \label{prop.PurelyGlobalBRegularSplits}
    With notation as in \autoref{def.Purely+Regular}, suppose $R = H^0(X, \sO_X)$ is local and $X \to \Spec R$ is proper.  Then $(X, S+B)$ is purely globally $\bigplus$-regular if and only if $\hat\myB^0_{S}(X, S+B; \sO_X) = H^0(X_{\widehat{R}}, \sO_{X_{\widehat{R}}})$.  
    
    In particular, if $R$ is complete, then $(X, S+B)$ is purely globally $\bigplus$-regular if and only if $\myB^0_{S}(X, S+B; \sO_X) = H^0(X, \sO_X)$.  
\end{proposition}

\begin{proof}
    We work with covers large enough so that $f^*(K_X + S + B)$ has integer coefficients.
    The strategy is the same as in \autoref{prop.GlobalBRegularSplits}.  If 
    \[
        \sO_X \to f_* \bigoplus_{i = 1}^t \sO_{Y}(-S_{i,Y} + f^*(S + B))
    \] 
    splits for all $Y$, then twisting by $K_X$ and taking local cohomology, we see that each 
    \begin{equation}
        \label{eq.prop.PurelyGlobalBRegularSplits}
        \myH^d \myR \Gamma_\m \myR \Gamma(X, \sO_X(K_X)) \to \myH^d \myR \Gamma_\m\myR\Gamma(Y, \bigoplus_{i = 1}^t \sO_{Y}(-S_{i,Y} + f^*(K_X + S+B))) 
    \end{equation}
    is injective.  Hence $\hat\myB^0_{S}(X, S+B; \sO_X) = H^0(X_{\widehat{R}}, \sO_{X_{\widehat{R}}})$.
        
    Conversely, if each map of the form \autoref{eq.prop.PurelyGlobalBRegularSplits} injects, then 
    \[
        H^0(Y_{\widehat{R}},  \bigoplus_{i = 1}^t \sO_Y(K_Y + S_{i,Y} - f^*  (K_X + S + B)) _{\widehat{R}}) \to H^0(X_{\widehat{R}}, \sO_{X_{\widehat{R}}})
    \] 
    surjects.   Since $R \to \widehat{R}$ is faithfully flat, each 
    \[
        H^0(Y,  \bigoplus_{i = 1}^t \sO_Y(K_Y + S_{i,Y} - f^*  (K_X + S + B)) ) \to H^0(X, \sO_{X})
    \] 
    surjects.  
    Hence there exists 
    \[
        z \in H^0(Y,  \bigoplus_{i = 1}^t \sO_Y(K_Y + S_{i,Y} - f^*  (K_X + S + B) ) )
    \]
    mapped to $1 \in H^0(X, \sO_X)$.  Thus we have a map 
    \[
        \sO_X \to  f_* \bigoplus_{i = 1}^t \sO_Y(K_Y + S_{i,Y} - f^*  (K_X + S + B) )
    \]
    induced by sending $1 \mapsto z$ giving us a splitting.  Apply $\sHom(-, \sO_X)$ to obtain the desired result.
\end{proof}

\begin{lemma}
\label{lem.pureBregularBox}
    If $(X, S+B)$ is purely globally $\bigplus$-regular along a reduced divisor $S$ then $(X, aS+B)$ is globally $\bigplus$-regular for every $0 \leq a < 1$.
\end{lemma}
\begin{proof}
    This follows from \autoref{lem.comparison_betwen_B0_and_B0D} when $R$ is complete and $X \to \Spec R$ is proper.  Alternately, for the general case, note that for large enough covers $Y \to X$ we have a factorization:
    \[
        \sO_X \to \sO_Y(f^*(aS+B)) \to f_* \bigoplus_{i = 1}^t \sO_Y(-S_{i,Y} + f^*(S+B)).
    \]
    The splitting of the composition implies splitting of the left map.
\end{proof}

\begin{proposition} \label{proposition:pullback-of-global-splinter}
Suppose $X \to \Spec R$ is proper.  Additionally, let $f : Y \to X$ be a proper birational morphism between normal schemes.  Let $\Delta \geq 0$ be a $\bQ$-divisor on $X$ such that $(X, \Delta)$ is globally $\bigplus$-regular (completely purely globally $\bigplus$-regular over $R$, resp.). Suppose that $\Delta_Y \geq 0$, where $K_Y+\Delta_Y = f^*(K_X+\Delta)$. Then $(Y,\Delta_Y)$ is globally $\bigplus$-regular ( ly globally $\bigplus$-regular, resp.).
\end{proposition}
\begin{proof}
We can assume that $R$ is local and complete by \autoref{cor.GlobalBRegularEqualsCompletely}. Then this follows from \autoref{lemma-B0-under-pullbacks} and \autoref{lemma-B0S-under-pullbacks}.
\end{proof}

\begin{remark}
        If $(X, S+B)$ is purely globally $\bigplus$-regular, then we will see in \autoref{prop.GlobalPureBRegularPLT} that it is plt, and in \autoref{cor.NormalityOfS} that $S$ is normal.  
\end{remark}

\subsection{Summary of terminology}

We conclude this section by summarizing the terminology we have introduced.

Recall, saying that $(X, \Delta)$ is globally $\bigplus$-regular means that every finite surjective map $f : Y \to X$ between integral schemes, one has that $\sO_X \to \sO_Y(\lfloor f^* \Delta \rfloor )$ splits as a map of $\sO_X$-modules.  We then potentially add two different modifiers to this term.

\begin{enumerate}
    \item \emph{purely}, which should be thought of as a plt variant of $\bigplus$-regularity.  
    \item \emph{completely}, which makes (purely) globally $\bigplus$-regular a relative notion (over a base $\Spec R$), meaning that after completing at each closed point of the base, we have (pure) globaly $\bigplus$-regularity.    
\end{enumerate}

%% file: lifting.tex
\section{Lifting \texorpdfstring{$\bigplus$}{+}-stable sections from divisors}

In this section we aim to prove that we may lift global sections of $\myB^0$ from hypersurfaces in many cases.  In order to lift sections we need vanishing theorems, and the key vanishing theorem we use in this case is \autoref{cor.VanishingWithoutRestrictingToPFiber}.

\begin{setting} \label{setting:lifting-section}
In this section, $R$ is an excellent local domain with a dualizing complex and positive characteristic residue field.  
\end{setting}

Frequently, $R$ will even be complete.

\begin{theorem} \label{thm:main-lifting}
  Let $X$ be a normal integral scheme of dimension $d$ that is proper over a complete local Noetherian base $\Spec R$ with positive characteristic residue field.  Let $\Delta\geq0$ be a $\bQ$-divisor such that $K_X+\Delta$ is $\bQ$-Cartier.  Suppose that $\Delta = S+B$ where $S = \sum S_i$ is a sum of prime components of $\Delta$ of coefficient one with normalization $\nu : S^{\nm} \to S$, and $M$ is a Cartier divisor such that $M-K_X-\Delta$ is big and semiample.
  
  Set $\sM = \sO_X(M)$. Then the restriction map $H^0(X,\sM) \to H^0(S^{\nm}, \sM|_{S^{\nm}})$ induces a surjection
  \[
    \myB^0_{S}(X,\Delta;\sM)\twoheadrightarrow \myB^0(S^{\nm},\Delta_{S^{\nm}};\sM|_{S^{\nm}})
  \]
  where $\Delta_{S^{\nm}}$ is the different of $K_X + S+B$ along $S^{\nm}$ and the right side is defined as in \autoref{rem.NonintegralB^0} in the case where $S^{\nm}$ has multiple connected components (taking the direct sum).
  \end{theorem}

  For more information on the different (of $K_X + S+B$ along $S^{\nm}$), see for instance \cite[Section 4.1]{KollarKovacsSingularitiesBook}.

 \begin{proof}
  This argument is very closely related to, and inspired by, the proof of \cite[Theorem 3.1]{MaSchwedeTuckerWaldronWitaszekAdjoint}.  In the proof below, we frequently abuse of notation in the following way.  Let where $\pi : X^+ \to X$ be the natural map.  For a quasi-coherent sheaf $\sF$ on $X^+$ we will also write $\sF$ for $\pi_* \sF$.  This is essentially a harmless notational device as $\myR^i \pi_* \sF = 0$ for all $i > 0$ since $\pi$ is an affine morphism, \cite[\href{https://stacks.math.columbia.edu/tag/01XC}{Tag 01XC}]{stacks-project}, and in particular $\myR \Gamma_m \myR \Gamma(X^+, \sF) \cong \myR \Gamma_m \myR \Gamma(X, \myR \pi_* \sF) \cong \myR \Gamma_m \myR \Gamma(X, \pi_* \sF)$.  The same notational consideration applies to the affine morphisms $S \to X$, $S^{\nm} \to X$, $S^+ \to X$, etc.
  
  With this abuse of notation in mind, we have the following diagram
  \[
    \xymatrix{
      \sO_X(-S) \ar[r] \ar[d] & \sO_X \ar[d] \\
      \bigoplus_{i = 1}^t \sO_{X^+}(-S_{i,X^+}) \ar[r] & \bigoplus_{i = 1}^t \sO_{X^+}
                }  
  \]
  of quasi-coherent sheaves on $X$ as in \autoref{subsec.AdjointAnalogsOfS^0}.
  
  Set $$\sL := \sO_{X^+}(L) = \sO_{X^+}(\pi^*(K_X + S + B - M))$$ to be the line bundle on $X^+$ corresponding to $K_X + S + B - M$.  Twist the top row of the above diagram by $K_X + S - M$ and reflexify, then twist the bottom row by $\sL$.  Using the additional downward inclusions given that $B$ is effective, we obtain the leftmost square in the commutative diagram with exact rows:
  \[
    {\small
    \xymatrix@C=12pt{
      0 \ar[r] & \sO_X(K_X - M) \ar[r] \ar[d] & \sO_X(K_X + S - M) \ar[d] \ar[r] & {\sO_X(K_X + S - M) / \sO_X(K_X - M)} \ar[r] \ar@{.>}[d] & 0\\
       0 \ar[r] & \bigoplus_{i = 1}^t \sO_{X^+}(L -S_{i,X^+}) \ar[r] & \bigoplus_{i = 1}^t \sO_{X^+}(L)  \ar[r] & \sO_{S^+} \otimes \sL \ar[r] & 0
    }
    }
  \]
  Recall that ${S^+}$ is the disjoint union of the $S_i^+$ as in \autoref{subsec.AdjointAnalogsOfS^0}.
        Taking cohomology then gives the following commutative diagram, where the factorization of the left vertical arrow into surjective maps will be explained below.  
    \[
      {
      \small
      \xymatrix{
        \myH^{d-1} \myR \Gamma_{\fram}\myR\Gamma (S, \sO_X(K_X + S - M) / \sO_X(K_X - M)) \ar@{->>}[d]^{\rho} \ar[r] & \myH^d \myR \Gamma_{\fram}\myR\Gamma(X, \sO_X(K_X - M)) \ar@{->>}[dd] \\
        \myH^{d-1} \myR \Gamma_{\fram}\myR\Gamma(S, \omega_S \otimes (\sM^{-1}|_S))\ar@{->>}[d] \\
         \Image_S \ar@{^{(}->}[r]^{\kappa} \ar@{^{(}->}[d] & \Image_X \ar@{^{(}->}[d] \\
         \myH^{d-1} \myR \Gamma_{\fram}\myR\Gamma(S^+, \sL|_{S^+}) \ar@{^{(}->}[r] & \myH^d \myR \Gamma_{\fram}\myR\Gamma(X^+, \bigoplus_{i = 1}^t \sO_{X^+}(L -S_{i,X^+})) \\
      }
      }
    \]
    Here we define $\Image_S$ to be:
    \[ 
        \Image\Big( \myH^{d-1} \myR \Gamma_{\fram}\myR\Gamma (S, \sO_X(K_X + S - M) / \sO_X(K_X - M)) \to \myH^{d-1} \myR \Gamma_{\fram}\myR\Gamma(S^+, \sL|_{S^+}) \Big).
    \]
    Note that $\Image_X$ is the $R$-Matlis dual of $\myB^0_S(X,\Delta,\sM)$ by definition, see \autoref{setup_adjoint}.  A main goal of the rest of the proof is to show that $\Image_S$ is dual to $\myB^0(S^{\nm},\Delta_{S^{\nm}},\sM|_{S^{\nm}})$.  

    We first explain the injection of $\kappa$.  Observe that 
    \[
      \myH^{d-1} \myR \Gamma_\m\myR\Gamma(X^+, \bigoplus_{i = 1}^t \sL)
    \]
    surjects onto the kernel of the bottom row, and so, since $\sL^{-1}$ is big and semiample, we may apply \autoref{cor.VanishingWithoutRestrictingToPFiber} and see that the bottom row injects.  Thus $\kappa : \Image_S \hookrightarrow \Image_X$ also injects and hence its $R$-Matlis dual 
    \[
        \myB^0_S(X,\Delta,\sM) \twoheadrightarrow (\Image_S)^{\vee}
    \]
    surjects.

    We now explain origin and surjectivity of $\rho$.  Because $X$ is normal and so Cohen-Macaulay in codimension 2, the S2-ification on $S$ of $\sO_X(K_X + S - M) / \sO_X(K_X - M)$ is $\omega_S \otimes (\sM^{-1}|_S)$ (see  \cite[Subsection 2.1]{MaSchwedeTuckerWaldronWitaszekAdjoint}) and so we have a factorization of sheaves on $S$
    \[
        \sO_X(K_X + S - M) / \sO_X(K_X - M) \to \omega_S \otimes (\sM^{-1}|_S) \to \sL|_{S^+}
    \] 
    as well since $\sL|_{S^+} = \sO_{S^+}(L|_{S^+})$ is a colimit of S2 coherent sheaves. Applying local cohomology explains the origin of the map $\rho$.  We now explain the surjectivity of $\rho$ (in fact, we will show that $\rho$ is an isomorphism). Applying $\myR\Hom(-, \omega_X^\mydot)$ to $0\to \sO_X(-S)\to \sO_X \to \sO_S\to 0$ and taking cohomology, we obtain 
    \[ 0 \to \sO_X(K_X)\to \sO_X(K_X+S)\to \omega_S \to H^{-(d-1)}(\omega_X^\mydot) \to \cdots \]
    Since $X$ is normal and hence S2, we know that $\dim H^{-(d-1)}(\omega_X^\mydot) <d-2$. After twisting by $\sM^{-1}$, we observe that 
    the cokernel $C$ of 
    \begin{equation}
        \label{eq.SurjectionInCodim0InLiftingProof}
      \sO_X(K_X + S - M) / \sO_X(K_X - M) \to \omega_S \otimes (\sM^{-1}|_S)
    \end{equation}
    satisfies $\dim \Supp C < d -2$ (alternatively, one sees that \autoref{eq.SurjectionInCodim0InLiftingProof} is precisely the S2-ification map and thus an isomorphism in codimension one). It follows that $\myH^{d-2}\myR\Gamma_{\fram} \myR \Gamma(S, C)=\myH^{d-1}\myR\Gamma_{\fram} \myR \Gamma(S, C) = 0$ by \cite[\href{https://stacks.math.columbia.edu/tag/0A4R}{Tag 0A4R}]{stacks-project}. This implies that $\rho$ is an isomorphism. 
        Therefore $\Image_S$ is also the image of 
    \[
        \myH^{d-1} \myR \Gamma_{\fram}\myR\Gamma(S, \omega_S \otimes (\sM^{-1}|_S))  \to \myH^{d-1} \myR \Gamma_{\fram}\myR\Gamma(S^+, \sL|_{S^+}).
    \]

    Next notice that dual to $\sO_S \to \sO_{S^{\nm}}$ we obtain $\omega_{S^{\nm}} \to \omega_S$ which induces a map of sheaves on $S$
    \[
        \sO_{S^{\nm}}(K_{S^{\nm}} - M|_{S^{\nm}}) \to \omega_S \otimes (\sM^{-1}|_S).
    \]
    The cokernel of this map is supported in dimension $< d-1$, and so, arguing as above, we see that 
    \[
        \myH^{d-1} \myR \Gamma_{\fram}\myR\Gamma \big(S^{\nm}, \sO_{S^{\nm}}(K_{S^{\nm}} - M|_{S^{\nm}})\big) \twoheadrightarrow \myH^{d-1} \myR \Gamma_{\fram}\myR\Gamma (S, \omega_S \otimes (\sM^{-1}|_S))
    \]
    surjects.  

    In particular, we have the following composition:
    \[
        \begin{array}{rcl}
        \myH^{d-1} \myR \Gamma_{\fram}\myR\Gamma \big(S^{\nm}, \sO_{S^{\nm}}(K_{S^{\nm}} - M|_{S^{\nm}})\big)& \twoheadrightarrow & \myH^{d-1} \myR \Gamma_{\fram}\myR\Gamma (S, \omega_S \otimes (\sM^{-1}|_S))  \\
            & \twoheadrightarrow & \Image_S \\ &\hookrightarrow & \myH^{d-1} \myR \Gamma_{\fram}\myR\Gamma\big(S^+, \sL|_{S^+}\big)
    \end{array}
    \]
    Since $S^+ = (S^{\nm})^+$, it should be expected, using \autoref{lem.B0AsInverseLimit} \autoref{eq.lem.B0AsInverseLimit.DualImageForFinite}, that the $R$-Matlis dual of $\Image_S$ is $\myB^0(S^{\nm}, \Delta', \sM|_{S^{\nm}})$ for some $\bQ$-divisor $\Delta'$ on $S^{\nm}$.  We wish to show that this is true where $\Delta' = \Delta_{S^{\nm}}$ is the different of $K_X + S+B$ along $S^{\nm}$, see \cite[Section 4.1]{KollarKovacsSingularitiesBook} or \cite[Section 2.1]{MaSchwedeTuckerWaldronWitaszekAdjoint} for more discussion of the different.
    In other words, we will show that the composition $\sO_{S^{\nm}}(K_{S^{\nm}} - M|_{S^{\nm}}) \to \omega_S \otimes (\sM^{-1}|_S) \to \sL|_{S^+}$ may be identified with the canonical map (since the different $\Delta_{S^{N}}$ is effective) viewed as sheaves on either $S$ (or equivalently on $S^{\nm}$)
    \begin{equation}
        \label{eq.MapToIdentifyDifferent}
        \sO_{S^{\nm}}((K_{S^{\nm}} - M)|_{S^{\nm}}) \to \sO_{S^+}(\pi_{S^{\nm}}^* (K_S + \Delta_{S^{\nm}} - M)|_{S^+})
    \end{equation}
    where $\pi_{S^{\nm}} \colon S^+ = (S^{\nm})^+ \to S^{\nm}$ is the usual map.
        
    To conclude the proof, we must show that we have an isomorphism of $\sO_{S^+}$-modules
    \[
        \sL|_{S^+} \cong \sO_{S^+}(\pi_{S^{\nm}}^*(K_{S^{\nm}} + \Delta_{S^{\nm}}-M|_{S^{\nm}})),
    \]
    and that the map $\sO_{S^{\nm}}(K_{S^{\nm}} - M|_{S^{\nm}}) \to \sL_{S^+}$ we obtained by composition is the same as the map \autoref{eq.MapToIdentifyDifferent}.  The first isomorphism is an immediate consequence of the definition of the different which guarantees that 
    \[
        (K_X + S + B)|_{S^{\nm}} \sim_{\bQ} K_{S^{\nm}} + \Delta_{S^{\nm}}.
    \]
    The second assertion is local on $S$.      In particular, we may forget about $R$, set $M = 0$ and assume that $X = \Spec A$ is     normal and local, $S = \Spec (A/I)$ is     reduced and local so that $S^{\nm} = \Spec (A/I)^{\nm}$ is the spectrum of a semi-local normal ring.
          At this point, the argument is essentially identical to the argument of \cite[Theorem 3.1]{MaSchwedeTuckerWaldronWitaszekAdjoint} which we now explain in a slightly different way.  
    
    We first need a precise definition of the different.  We may write $K_X = -S + G$ where $G \geq 0$ does not contain any component of $S$ within its support.  This in fact determines a canonical divisor on $S^{\nm}$.  Consider a global section $y \in \sO_X(K_X + S) = \sO_X(G)$ determining $G$ (note we may take $y = 1 \in \sO_X(G)$).  The image of $y$, $\overline{y} \in \omega_S$ becomes a rational section of $\omega_{S^{\nm}}$ via the map $\omega_{S^{\nm}} \to \omega_S$, this rational section determines the divisor we call $K_{S^{\nm}}$.  Write $K_X + S + B = {1 \over m} \Div f$ for some $f \in A$.  Then, setting $\overline{f} \in (A/I)^{\nm}$ as the image of $f$, we define the different as 
    \[
        \Delta_{S^{\nm}} := {1 \over m} \Div_{S^{\nm}} \overline{f} - K_{S^{\nm}}.
    \]
    It is independent of our choices and always effective, see \cite[Section 4.1]{KollarKovacsSingularitiesBook}.
    With this in place, and the careful choice of $K_{S^{\nm}}$ described above, the map we constructed earlier in the proof
    \begin{equation}
        \label{eq.CompositionInducingStuff}
        \sO_{S^{\nm}}(K_{S^{\nm}}) \to \omega_S \to \sL|_{S^+}  = \sO_{S^+}( \pi^*(K_X + S + B)|_{S^+})
    \end{equation}
    sends the rational section $\overline{y}$ to an honest section of $\omega_S$ which came from the section $1 \in \sO_X(K_X + S) \subseteq \sO_{X^+}(\pi^*(K_X + S + B)) =  {1 \over {f}^{1/m}} \sO_{X^+}$.  In particular, in the composition \autoref{eq.CompositionInducingStuff}, $\overline{y}$ is sent to $1 \in {1 \over \overline{f}^{1/m}} \sO_{S^+} = \sL|_{S^+}$.
        On the other hand,     the map 
    \[
        \sO_{S^{\nm}}(K_{S^{\nm}}) \to \sO_{S^+}( \pi^*(K_{S^{\nm}} + \Delta_{S^{\nm}})) = {1 \over \overline{f}^{1/m}} \sO_{S^+}
    \]
    also sends $\overline{y}$ to $1$ by construction, and so the two maps agree since they are maps between rank-1 sheaves and so are determined by where they send any single nonzero (on any irreducible component) rational section.
       \end{proof}

  \begin{remark}
    One may also obtain an alternative proof in the case where $M-K_X-\Delta$ is ample, by passing to the affine cones, and using \autoref{prop.S+grIsBigCM} and \cite{MaSchwedeTuckerWaldronWitaszekAdjoint}. 
  \end{remark}

Recall from \autoref{rem.B0AdjointVersionUpToCompletion} that when $R$ is not necessarily complete, we define $\hat\myB^0_S(X, \Delta; \sM)$ to be the the Matlis dual of 
    \[
        \mathrm{Im}\bigg( \myH^d \myR \Gamma_\m \myR \Gamma(X, \sO_X(K_X - M)) \to \myH^d \myR \Gamma_\m\myR\Gamma(X^+, \bigoplus_{i = 1}^t \sO_{X^+}(-S_i^+ + \pi^*(K_X + \Delta - M))) \bigg).
    \]
\begin{corollary}
    \label{cor:main-lifting-non-complete version}    
    With the same assumptions as in \autoref{thm:main-lifting}, but with $H^0(X,\sO_X) = R$ and $R$ not necessarily complete but satisfying \autoref{setting:lifting-section}, we have that the restriction map 
    \[
        H^0(X, \sM) \otimes_R \widehat{R} = H^0(X_{\widehat{R}},\sM_{\widehat{R}}) \to H^0(S^{\nm}_{\widehat{R}}, \sM|_{S_{\widehat{R}}^{\nm}}) = H^0(S^{\nm}, \sM|_{S^{\nm}}) \otimes_R \widehat{R}
    \] 
    induces a surjection
    \[
        \hat\myB^0_{S}(X,\Delta;\sM)\twoheadrightarrow \myB^0(S^{\nm}_{\widehat{R}}; \Delta_{S^{\nm}_{\widehat{R}}}; \sM|_{S^{\nm}_{\widehat{R}}}).
    \]
\end{corollary}
\begin{proof}
    The proof is the same as that of \autoref{thm:main-lifting} in view of \autoref{prop.B0completion} since \autoref{cor.VanishingWithoutRestrictingToPFiber} does not require that $R$ is complete. 
\end{proof}

When $S$ is globally $\bigplus$-regular, we obtain the following important consequence.

\begin{corollary}[Adjunction and inversion of adjunction] \label{cor.inversion-of-B-adjunction} Let $(X,\Delta = S+B)$ be a pair proper over $R = H^0(X, \sO_X)$ satisfying \autoref{setting:lifting-section}. Assume additionally that $S$ is a reduced Weil divisor having no common components with $B$, and such that $-K_X - \Delta$ is {big and semiample.}   
  Let $\Delta_{S^{\nm}}$ denote the different of $K_X + S +B$ on $S^{\nm}$ with respect to $(X,S+B)$.  
  
  Then $(X,S+B)$ is purely globally $\bigplus$-regular (along $S$) if and only if $(S^{\nm}, \Delta_{S^{\nm}})$ is globally $\bigplus$-regular (in the sense of \autoref{rem.Globall+RegularForNonIntegralX} if $S$ is not irreducible).  
\end{corollary}
When $R$ is complete, notice that $R = \widehat{R}$ and $\myB^0_S = \hat\myB^0_S$.
\begin{proof}    
    By and using the notation of \autoref{cor:main-lifting-non-complete version}, we have a surjection 
  \[
    \hat\myB^0_S(X, \Delta, \sO_X) \twoheadrightarrow \myB^0(S^{\nm}_{\widehat{R}}, \Delta_{S^{\nm}_{\widehat{R}}}; \sO_{S^{\nm}_{\widehat{R}}}).
  \]
  Notice that $\hat\myB^0_S(X,S+B; \sO_X) \subseteq H^0(X_{\widehat{R}}, \sO_{X_{\widehat{R}}}) = \widehat{R}$.   
  
  First suppose that $(S^{\nm}, \Delta_{S^{\nm}})$ is globally $\bigplus$-regular.  Then so is the base change to the completion $(S^{\nm}_{\widehat{R}}, \Delta_{S^{\nm}_{\widehat{R}}})$ by \autoref{cor.GlobalBRegularEqualsCompletely}.  Notice that a priori,  $S^{\nm}_{\widehat{R}}$ may have even more components than $S^{\nm}$ since if $S_i^{\nm}$ is such a component of $S^{\nm}$, we may have that $H^0(S_i^{\nm}, \sO_{S_i^{\nm}})$ is only semilocal.  However, this will not cause a problem for us; $S^{\nm}$ already perhaps had multiple components.
  
  Regardless, $\myB^0(S^{\nm}_{\widehat{R}}, \Delta_{\widehat{R}}; \sO_{X_{\widehat{R}}}) = H^0(S^{\nm}_{\widehat{R}}, \sO_{S^{\nm}_{\widehat{R}}})$.  Our surjectivity then implies that $\hat\myB^0_S(X, \Delta, \sO_X)$ must contain an element $z$ of $H^0(X_{\widehat{R}}, \sO_{X_{\widehat{R}}}) = \widehat{R}$ mapping to $1 \in H^0(S^{\nm}_{\widehat{R}}, \sO_{S^{\nm}_{\widehat{R}}})$.  But such a section $z$ is necessarily a unit of $\widehat{R}$ and so $\hat\myB^0_S(X,S+B; \sO_X) = \widehat{R} = H^0(X_{\widehat{R}}, \sO_{X_{\widehat{R}}})$.  Hence $(X, \Delta)$ is purely globally $\bigplus$-regular along $S$.

  Conversely, if $(X, S+B)$ is purely globally $\bigplus$-regular then $\hat\myB^0_S(X,S+B; \sO_X)$ contains a unit, and hence so does its image $\myB^0(S^{\nm}_{\widehat{R}}, \Delta_{S^{\nm}_{\widehat{R}}}; \sO_{S^{\nm}_{\widehat{R}}}) \subseteq H^0(S^{\nm}_{\widehat{R}}, \sO_{S^{\nm}_{\widehat{R}}})$.  Thus $(S^{\nm}, \Delta_{S^{\nm}})$ is globally $\bigplus$-regular by \autoref{prop.GlobalBRegularSplits}.  
    \end{proof}

\begin{corollary}
    \label{cor.PurelyGlobally+RegularVsCompletelyAssumeBigAndSemiample}
    Let $X \to \Spec R$ be a proper morphism of schemes such that $H^0(X, \sO_X) = R$ satisfies \autoref{setting:lifting-section}.  Suppose that $(X, S+B)$ is a pair where $S$ and $B$ have no common components and $S$ is reduced.  Finally, assume that $-K_X - S - B$ is big and semiample over $\Spec R$.  Then $(X, S+B)$ is purely globally $\bigplus$-regular along $S$ if and only if it is completely purely globally $\bigplus$-regular  over $R$ along $S$.
\end{corollary}
Note that the assumptions of this corollary are satisfied when $X = \Spec R$.
\begin{proof}
    By \autoref{cor.inversion-of-B-adjunction}, we see that $(X, S+B)$ is purely globally $\bigplus$-regular if and only if $(S^{\nm}, \Delta_{S^{\nm}})$ is globally $\bigplus$-regular.  
    That is equivalent to $(S^{\nm}_{\widehat R}, \Delta_{S^{\nm}_{\widehat{R}}})$ being globally $\bigplus$-regular by \autoref{cor.GlobalBRegularEqualsCompletely}.  
    Hence applying \autoref{cor.inversion-of-B-adjunction} again, we see that this is equivalent to $(X_{\widehat R}, (S+B)_{\widehat{R}})$ being purely globally $\bigplus$-regular as desired.
\end{proof}

\begin{proposition}[Global to local]
\label{prop.GlobalPureBRegularPLT}
    Suppose $X$ is a normal Noetherian excellent integral scheme with a dualizing complex and where every closed point of $X$ has positive characteristic residue field.  Further suppose that $(X, S+B)$ is purely globally $\bigplus$-regular for a reduced divisor $S$ and that $K_X + S+B$ is $\bQ$-Cartier.  Then $(X, S+B)$ is plt.
\end{proposition}
\begin{proof}
    Choose $x$ a closed point and let $R = \widehat{\sO_{X,x}}$.  By \autoref{cor.PurelyGlobally+RegularVsCompletelyAssumeBigAndSemiample} we may assume $X = \Spec R$. 
    By \autoref{lem.B0AlongDAsInverseLimit}, the map induced by Grothendieck duality:
    \[
        H^0(Y, \sO_Y( K_Y + S_Y {- \lceil  f^* (K_X+ S+B)\rceil})) \to H^0(X, \sO_X)
    \]
    is surjective for every projective birational morphism $f \colon Y \to X$ from a normal integral scheme $Y$. This is the case exactly when $\lceil K_Y + S_Y - f^* (K_X+ S+B) \rceil$ is effective and exceptional over $X$ (cf.\ \autoref{lemma:pushforward}), which, in turn, is equivalent to the requirement that $\lfloor B \rfloor = 0$ and that all the exceptional divisors on $Y$ have discrepancy greater than $-1$.  As this is true for every projective birational morphism, $(X,\Delta)$ is plt.
\end{proof}

Our result also implies a surjectivity of $H^0$ under certain hypotheses, which implies that $S^{\nm}$ is connected.  Also compare with \cite[Theorem 5.48]{KollarMori} and \cite[5.7]{ShokurovThreeDimensionalLogFlips}.

\begin{corollary} \label{cor:lifting_from_BCM-regular} 
    Assume $X \to \Spec R$ is proper where $R$ satisfies \autoref{setting:lifting-section} and such that $H^0(X, \sO_X) = R$.
  Suppose that $(S^{\nm},\Delta_{S^{\nm}})$ is globally $\bigplus$-regular (in the sense of \autoref{rem.Globall+RegularForNonIntegralX} if $S$ is not integral) {and $M-K_X - \Delta$ is big and semiample.} Then
  \[
    H^0(X,\sM) \to H^0(S^{\nm}, \sM|_{S^{\nm}})
  \]
  is surjective.  As a consequence, if additionally $-K_X - \Delta$ is big and semiample, then $S^{\nm}_{\widehat{R}}$ is connected and thus integral and thus so is $S^{\nm}$.   \end{corollary}
\begin{proof}
    By \autoref{cor:main-lifting-non-complete version}, the map 
    \[
        \hat\myB^0_S(X,\Delta; \sM) \to \myB^0(S^{\nm}_{\widehat{R}}, \Delta_{S^{\nm}_{\widehat{R}}}; \sM|_{S^{\nm}_{\widehat{R}}})
    \]
    is surjective.  By \autoref{prop.GlobalBRegularSplits}, we know that $\myB^0(S^{\nm}_{\widehat{R}}, \Delta_{S^{\nm}_{\widehat{R}}}; \sM|_{S^{\nm}_{\widehat{R}}}) = H^0(S^{\nm}_{\widehat{R}}, \sM|_{S^{\nm}_{\widehat{R}}})$ and so since $\myB^0_S(X_{\widehat{R}},S^{\nm}_{\widehat{R}}+\Delta_{\widehat{R}};\sM_{\widehat{R}}) \subseteq H^0(X_{\widehat{R}},\sM_{\widehat{R}})$, we obtain that $H^0(X_{\widehat{R}},\sM_{\widehat{R}}) \to H^0(S^{\nm}_{\widehat{R}}, \sM|_{S^{\nm}_{\widehat{R}}})$ is surjective.  Thus $H^0(X,\sM) \to H^0(S^{\nm}, \sM|_{S^{\nm}})$ is surjective as well, since $\widehat{R}$ is faithfully flat over $R$.

    For the statement that $S^{\nm}_{\widehat{R}}$ is connected, notice that we have that $H^0(X_{\widehat{R}}, \sO_{X_{\widehat{R}}}) \to H^0(S^{\nm}_{\widehat{R}}, \sO_{S^{\nm}_{\widehat{R}}}) =: A$ surjects and that $H^0(X_{\widehat{R}}, \sO_{X_{\widehat{R}}}) = \widehat{R}$ is a local domain.  Thus $A$ is a normal local ring as well.  This implies that $A$ is an integral domain.  On the other hand $S^{\nm}_{\widehat{R}}$ is a disjoint union of normal integral schemes say $\coprod S_i$.  Thus 
    $A = \prod H^0(S_i, \sO_{S_i})$ is a product of domains, and so cannot be a domain itself unless there is only one $S_i$, meaning that $S^{\nm}_{\widehat{R}}$ is connected and integral as desired.
\end{proof}

In fact, we frequently also have that $S$ is normal.  

\begin{corollary}
  \label{cor.NormalityOfS}
  Suppose $(X, S+B)$ is a pair where $K_X + S + B$ is $\bQ$-Cartier, $S$ is reduced, and $B \geq 0$ is a $\bQ$-divisor with no common components with $S$. We assume that all closed points of $X$ are of positive characteristic. 
  Let $\Delta_{S^{\nm}}$ denote the different of $K_X + S +B$ on $S^{\nm}$ with respect to $(X,S+B)$.

  Suppose that  $(S^{\nm}, \Delta_{S^{\nm}})$ is globally $\bigplus$-regular or that $(X, S+B)$ is purely globally $\bigplus$-regular along $S$.   In either case, $S$ is normal.  
\end{corollary}
\begin{proof}    
    
    Fix a closed point $x \in S \in X$.  It suffices to show that $\sO_{S,x}$ is normal (note that such localizations still imply the various pairs are globally $\bigplus$-regular).  Thus we assume that $X = \Spec \sO_{X,x} = \Spec R$.
    
    In view of \autoref{cor.inversion-of-B-adjunction}, in either case we have that $(S^{\nm}, \Delta_{S^{\nm}})$ is globally $\bigplus$-regular.  Observe that $\Big(S^{\nm}_{\widehat{R}}, \Delta_{S^{\nm}}|_{S^{\nm}_{\widehat{R}}}\Big)$ is still globally $\bigplus$-regular by \autoref{cor.GlobalBRegularEqualsCompletely}.  Now, if $S_{\widehat{R}}$ is normal then so is $S$, so we may assume that $R = \widehat{R}$, $X = \Spec \widehat{R}$ and observe that $-(K_X + S + B)$ is big and semiample since we are working locally.
    But now by \autoref{cor:lifting_from_BCM-regular} we see that the composition $\sO_X \to \sO_{S} \to \sO_{S^{\nm}}$ is a surjection, implying that $S$ is normal.  
          \end{proof}

Note in the above we needed to assume every closed point has positive characteristic residue field since we do not believe that the globally $\bigplus$-regular hypothesis necessarily implies plt without it, see \autoref{rem.CharacteristicZeroSplinterNotAsGood} for some additional discussion.  

\begin{corollary}
  \label{cor.IntegralityOfS}
  Suppose that $X \to \Spec R$ is proper where $H^0(X, \sO_X) = R$ and $R$ satisfies \autoref{setting:lifting-section}.  
  Next assume that $(X, S+B)$ is a purely globally $\bigplus$-regular (or completely purely globally $\bigplus$-regular over $R$) pair along $S$ and $-K_X - S - B$ is big and semiample.  Then $S$ is normal and integral.
\end{corollary}
\begin{proof}
    Since $X$ is proper over $R$, every closed point of $X$ has positive characteristic residue field (since they must all map to the closed point of $\Spec R$).  We may now replace $R$ by its completion by \autoref{cor.PurelyGlobally+RegularVsCompletelyAssumeBigAndSemiample} since if $S_{\widehat{R}}$ is normal and integral so is $S$.  
  We see that $(S^{\nm}, \Delta_{S^{\nm}})$ is globally $\bigplus$-regular by \autoref{cor.inversion-of-B-adjunction} hence $S$ is normal by \autoref{cor.NormalityOfS}.  Furthermore, $S$ is connected by \autoref{cor:lifting_from_BCM-regular}.  This proves that $S$ is integral.
\end{proof}

As a corollary, we also recover the standard global generation result on Seshadri constants \cite{DemaillyANumericalCriterionForVeryAmple} (cf.\ \autoref{ss:seshadri-constants}, \cite[Chapter 5]{LazarsfeldPositivity1}). 
\begin{theorem} \label{theorem:seshadri}
	Let $(X,B)$ be a pair of dimension $n$ {proper} over $\Spec R$ where $R$ is Noetherian complete local and has positive characteristic residue field.  Let $x \in X$ be a closed point such that at the point $x$, $(X,B)$ is klt, $X$ is regular, and $\Supp B$ is simple normal crossing. Let $M$ be a Cartier divisor with $\sM = \sO_X(M)$ such that $A := M-(K_X+B)$ is big and semiample.  Further suppose that $\epsilon_{\mathrm{sa}}(A;x)>a(E,X,B)$ where $a(E,X,B)$ is the log discrepancy of (X,B) along the exceptional divisor $E$ of the blow-up $\pi \colon X' \to X$ at $x$.  
	
	Then $\myB^0(X,B;\sM)$ globally generates $\sM$ at $x$. In particular, $x$ is not a base point of $|M|$. 
\end{theorem}

If $X$ is nonsingular and $B = 0$, then $a(E, X,B) = \dim \sO_{X,x} = \dim X$ under our hypotheses.  Hence we recover the usual formulation of global generation via Seshadri constants: namely that $\epsilon_{\mathrm{sa}}(M-K_X; x) > \dim X$ implies that $M$ is globally generated at $x$.

\begin{proof}
Denote the log discrepancy $a(E,X,B)$ by $a$.
  By definition, $K_{X'}+\pi_*^{-1}B +(1-a)E = \pi^*(K_X+B)$. Notice that for each rational $0 \leq t < \epsilon_{\mathrm{sa}}(A;x)$ we have that $\pi^* A - tE$ is big and semiample.  Thus since $\epsilon_{\mathrm{sa}}(A) > a$, we have that 
  \[
    \pi^*M - (K_{X'}+E+\pi_*^{-1}B) = \pi^* A + K_{X'}+\pi_*^{-1}B +(1-a) E - (K_{X'}+E+\pi_*^{-1}B) = \pi^*A - aE
  \]
  is big and semiample, and so 
  \begin{equation}
    \label{theorem:seshadri.eq.B0Map}
    \myB^0_E(X', E + \pi_*^{-1}B; \pi^*\sM) \to \myB^0(E,B_E; \sO_E)
  \end{equation}
  is surjective by \autoref{thm:main-lifting}, where $K_E + B_E = (K_{X'}+E + \pi_*^{-1}B)|_E$. Notice that $E \simeq \mathbb{P}^{n-1}$ since $X$ is regular at $x$.  Furthermore, since $\Supp B$ is simple normal crossings, the components of $B$ are defined locally by part of a system of parameters of $\fram_x$.  Hence the support of $B_E$ is made up of coordinate hyperplanes, which thus remain coordinate hyperplanes even after base change of the residue field of $x$.  We would like to assert that $(E, B_E)$ is globally $F$-regular but the residue field $k(x)$ need not be $F$-finite and so neither is $E$.  However, because the coefficients of $B_E$ are $< 1$, we have that the base change of $(E, B_E)$ to the perfection of $k(x)$ is globally $F$-regular; indeed notice that the base change $B_E$ remains coordinate hyperplanes with coefficients $< 1$ and so $(E, B_E)$ is globally $F$-regular by \cite[Proposition 5.3]{SchwedeSmithLogFanoVsGloballyFRegular} (since the section ring pair is strongly $F$-regular).  Therefore $(E, B_E)$ is globally $\bigplus$-regular by \autoref{cor:BregularINpositiveCharacteristicWithoutF-finite}.  
  Thus, the right hand side of \autoref{theorem:seshadri.eq.B0Map} is equal to $H^0(E, \sO_E)$. Hence $\myB^0_E(X', E+ \pi_*^{-1}B; \pi^*\sM) \subseteq H^0(X', \pi^* \sM)$
  contains a section which does not vanish at $x$.  But now, for $1 \gg \varepsilon > 0$
  \[
    \myB^0_E(X', E+ \pi_*^{-1}B; \pi^*\sM) \subseteq \myB^0(X', (1-\varepsilon)E + \pi_*^{-1} B; \pi^* \sM ) \subseteq \myB^0(X, B; \sM).
  \]
  where first containment follows from \autoref{lem.comparison_betwen_B0_and_B0D} and the second follows from \autoref{lemma-B0-under-pullbacks}.  This completes the proof.
  \end{proof}
  
  \begin{remark}
    The condition that $B$ is simple normal crossing at $x$ was only used to guarantee that the exceptional divisor pair $(E, B_E)$ was globally $F$-regular (up to appropriate base change to make it $F$-finite).  One can weaken the simple normal crossing hypothesis if one instead assumes that $(E, B_E)$ is globally $F$-regular, the proof is unchanged.  
  \end{remark}

 \subsection{Globally \texorpdfstring{$\bigplus$}{+}-regular birational morphisms of surfaces}
{In this subsection, we give new proofs of \cite[Theorem 7.11 and Theorem G]{MaSchwedeTuckerWaldronWitaszekAdjoint}, in the case when the fixed big Cohen-Macaulay algebra is equal to $\widehat{R^+}$.  Explicitly, we show that, locally, two-dimensional klt and three-dimensional plt pairs $(X,\Delta)$ with standard coefficients and residue characteristics $p>5$ are globally $\bigplus$-regular and purely globally $\bigplus$-regular, respectively. In fact, we will show much stronger results, which we shall need in the proof of the existence of flips: that two-dimensional klt and three-dimensional plt Fano pairs are globally $\bigplus$-regular and purely globally $\bigplus$-regular relative to a birational map.}  {Our approach for proving these results is the same as in \cite{HaconWitaszekMMPp=5} which simplified the original strategy of \cite{HaconXuThreeDimensionalMinimalModel}.}

In what follows, we continue to assume \autoref{setting:lifting-section} that $(R,\fram)$ is an excellent local domain with residue characteristic $p>0$ and a dualizing complex. 

We start by stating the following lemma, which generalizes the existence of Koll\'ar's component for surfaces (c.f.\ \cite[Proof of Theorem 7.11]{MaSchwedeTuckerWaldronWitaszekAdjoint}).

\begin{lemma} \label{lemma:surface_mmp_lemma} Let $(X,B)$ be a two-dimensional klt pair admitting a projective birational map $f \colon X \to Z = \Spec R$ such that $-(K_X+B)$ is relatively nef, assuming that $R$ is as in \autoref{setting:lifting-section} and additionally has infinite residue field . Then there exist an $f$-exceptional irreducible curve $C$ on a blow-up of $X$ and projective birational maps $g \colon Y \to X$ and $h \colon Y \to W$ over $Z$ such that:
\begin{enumerate}
	\item $g$ extracts $C$ or is the identity if $C \subseteq X$,
	\item $(Y, C+ B_Y)$ is plt,
		\item $(W, C_W + B_W)$ is plt and $-(K_{W}+C_W+B_W)$ is ample over $Z$,
	\item $h^*(K_{W}+C_W+B_W) - (K_{Y}+C+B_Y) \geq 0$,
\end{enumerate} 
where $K_{Y}+bC + B_Y = g^*(K_X+B)$ for $C \not \subseteq \Supp B_Y$, $C_W := h_*C \neq 0$, and $B_W := h_*B_Y$.
\end{lemma}
\noindent {We warn the reader that it may happen that $g$ is the identity and $C$ lies on $X$.} Further, we added the assumption that $R/ \fram$ is infinite to avoid potential issues with tie-breaking (cf.\ \autoref{rem:general_element}).
\begin{proof}
This follows by exactly the same proof as \cite[Lemma 2.8]{HaconWitaszekMMPp=5}. This is a consequence of the two dimensional Minimal Model Program in mixed characteristic, see \cite{tanaka_mmp_excellent_surfaces}. Note that tie-breaking employs Bertini's theorem for regular schemes, which in our setting is known by \autoref{thm.FinalBertini}. 
\end{proof}

 In this and the next result, we add an additional $\varepsilon D$ to the boundary as it will be important in the proof of the existence of flips. 

\begin{theorem}\label{theorem:hacon_xu_surfaces} Let $(X,B)$ be a two-dimensional klt pair admitting a projective birational map $f \colon X \to Z = \Spec R$ such that $-(K_X+B)$ is relatively {nef}.  Here  the ring $R$ is as in \autoref{setting:lifting-section} and has residual characteristic $p>5$. Suppose further that $B$ has standard coefficients. Then, for every divisor $D \geq 0$ and $0 \leq  \varepsilon \ll 1$ depending on $D$, we have that $(X,B+\varepsilon D)$ is globally $\bigplus$-regular.\end{theorem}
\begin{proof}
We may assume that $R = H^0(X, \sO_X)$ is normal.  By \autoref{cor.GlobalBRegularEqualsCompletely} we may also assume that $R$ is complete. 
If the residue field of the complete ring $R$ is not infinite, we may further pass to the completion of the strict Henselization $R'$.  Indeed, checking that $\sO_X \to f_* \sO_Y(\lfloor f^* (B + \varepsilon D)\rfloor)$ splits for a finite dominant map $f : Y \to X$ can be checked after such a base change.  Hence we may assume that the residue field of $R$ is infinite.

We apply \autoref{lemma:surface_mmp_lemma} and use its notation. First, write $K_{C} + B_{C} = (K_W + C_W + B_W)|_{C}$ where $C$ is identified with $C_W$. Further, write $D_Y = g_*^{-1}D$ and pick a divisor $D_W$ on $W$ such that $C_W \not \subseteq \Supp D_W$ and $D_Y \leq h^*D_W + C$.  Since $-(K_{C} + B_{C})$ is ample and $B_{C}$ has standard coefficients, 
 we must have that $C_{\overline{k}}\cong\mathbb{P}^1$  since by \cite[\href{https://stacks.math.columbia.edu/tag/0C19}{Lemma 0C19}]{stacks-project}, $g(C)=0$, and if $C_{\overline{k}}$ was not normal, then $\deg    B_{C_{\overline{k}}} \geq 2$ for the anti-ample $\bQ$-Cartier divisor $K_{C_{\overline{k}}} + B_{C_{\overline{k}}} = (K_W + C_W + B_W)|_{C_{\overline{k}}}$ by \cite[Theorem 1.1]{patakfalvisingularities}.  Furthermore $B_{C_{\overline{k}}}$ also has standard coefficients, since  coefficient of $B_{C_{\overline{k}}}$ is either equal to a coefficient of $B_{C}$ or is at least $p$ times such a coefficient {(hence it is at least $\frac{p}{2})$}, and with $p>5$, the presence of such a coefficient would prevent ampleness of $-(K_{C_{\overline{k}}}+B_{C_{\overline{k}}})$.  Therefore, $(C_{\overline{k}},B_{C_{\overline{k}}})$ is globally $F$-regular, see \cite{WatanabeFregularFpureRings}, and so is  $(C_{\overline{k}},B_{C_{\overline{k}}} + \varepsilon D_W|_{C_{\overline{k}}})$ for $0 \leq \varepsilon \ll 1$. Hence by \autoref{cor:BregularINpositiveCharacteristicWithoutF-finite}, $(C,B_C+\varepsilon D_W|_{C})$ is globally $\bigplus$-regular. Thus, by inversion of adjunction in the form of \autoref{cor.inversion-of-B-adjunction}, $(W,C_W+B_W + \varepsilon D_W)$ is purely globally $\bigplus$-regular.

\autoref{proposition:pullback-of-global-splinter} and Condition (d) imply that $(Y,C+B_Y+\varepsilon h^*D_W)$ is purely globally $\bigplus$-regular. By \autoref{lem.GloballyBregularBox} and \autoref{lem.pureBregularBox}, $(Y,bC + B_Y + {{\varepsilon}}D_Y)$ is globally $\bigplus$-regular, and so is $(X,B+\varepsilon D)$ by \autoref{proposition:pushforward-of-global-splinter} where the notation is as in \autoref{lemma:surface_mmp_lemma}.
\end{proof}

\begin{corollary}
	\label{cor.3dim-plt-Fano-are-purely-plus-regular}
Let $(X,S+B)$ be a  three-dimensional plt pair and let $f \colon X \to Z = \Spec R$ be a projective birational map such that $-(K_X+S+B)$ is relatively {semiample}, where $R$ satisfies \autoref{setting:lifting-section} and is of residue characteristic $p > 5$.  Assume further that $B$ has standard coefficients,  $\lfloor B \rfloor = 0$, and  $S$ is reduced. 
Then $S$ is a normal prime divisor and setting $K_S+B_S = (K_X+S+B)|_S$ (with $B_S$ the different), we have that $(S,B_S+\varepsilon D)$ is globally $\bigplus$-regular for every Cartier divisor $D$ and $0 < \varepsilon \ll 1$.  Finally, $(X,S+B)$ is purely globally $\bigplus$-regular.
\end{corollary}
\begin{proof}
By \autoref{cor.GlobalBRegularEqualsCompletely} and \autoref{cor.PurelyGlobally+RegularVsCompletelyAssumeBigAndSemiample} we may assume that $R$ is complete. Let $S^{\nm} \to S$ be the normalization of $S$ and write $K_{S^{\nm}} + B_{S^{\nm}} = (K_X+S+B)|_{S^{\nm}}$. \autoref{theorem:hacon_xu_surfaces} implies that each component of $(S^{\nm},B_{S^{\nm}} + \varepsilon D|_{S^{\nm}})$ is globally $\bigplus$-regular (and hence $(S^{\nm},B_{S^{\nm}} + \varepsilon D|_{S^{\nm}})$ is globally $\bigplus$-regular in the sense of \autoref{rem.Globall+RegularForNonIntegralX}).
Hence $S$ is normal and integral and $(X,S+B)$ is purely globally $\bigplus$-regular by \autoref{cor.inversion-of-B-adjunction} and \autoref{cor.IntegralityOfS}.
\end{proof}

In fact, the proof even shows that $(X, S+B+\varepsilon H)$ is purely globally $\bigplus$-regular for every Cartier divisor $H$ on $X$ not containing any component of $S$, for $1 \gg \varepsilon > 0$.

We also observe that \autoref{theorem:hacon_xu_surfaces} gives a new proof of the following results of a subset of the authors, in the case when the fixed big Cohen-Macaulay algebra is equal to $\widehat{R^+}$.
\begin{corollary}[{\cite[Theorem 7.11]{MaSchwedeTuckerWaldronWitaszekAdjoint}}]
\label{thm.KLTSurfaceImpliesBCMReg}
	Let $(S,\Delta)$ be a klt pair with standard coefficients where $S = \Spec A$ for an excellent two-dimensional normal local domain  $(A,\fram)$ of mixed characteristic $(0,p)$ for $p>5$.
	{Then  $(S,\Delta)$ is globally $\bigplus$-regular.}
    \end{corollary}
\begin{proof}
  Apply \autoref{theorem:hacon_xu_surfaces} to $(X, B)=({S}, {\Delta})$ and $f$ the identity map.
\end{proof}

We also recall a special case of \cite[Theorem G]{MaSchwedeTuckerWaldronWitaszekAdjoint} in our setting. The following proof shows that the divisor $S$ of a three-dimensional plt pair $(X, S+B)$ is normal at every closed point where the residual characteristic is greater than $5$ as long as either $B$ has standard coefficients or $X$ is $\bQ$-factorial

\begin{corollary} 
  \label{cor.ThreefoldNormalityOfS}
  Let $(X, S+B)$ be a three-dimensional plt pair where $S$ is reduced, $B$ has standard coefficients and $\lfloor B \rfloor = 0$.  Then at every closed point $x \in X$ where $\mathrm{char}\, k(x) > 5$, we have that $S$ is normal at $x$ and if $S_x = \Spec \sO_{S,x}$,  then $(S_x, B_{S}|_{S_x})$ is globally $\bigplus$-regular at (here $B_S$ is the different of $K_X + S + B$ along $S$).
  
    Moreover, $S$ is normal at every closed point $x \in X$ of residue characteristic $p>5$ even when $B$ does not have standard coefficients, if we assume that $X$ is $\bQ$-factorial.
\end{corollary}
\begin{proof}
    Note first that since $(X, S+B)$ is plt, and since log resolutions exist for 3-dimensional excellent schemes see \autoref{subsec.ResolutionOfSings}, the completion of $(X, S+B)$ at any closed point is also plt. Hence replacing $X$ by its completion at a closed point $x \in X$, we may assume that $X = \Spec R$  for a three-dimensional complete local domain $(R, \fram)$ of residual characteristic $p > 5$.  Here we used that the completion of a ring is faithfully flat, and normality descends under faithfully flat extensions.  Notice that $X \to \Spec R$ is projective since it is the identity.
Let $S^{\nm} \to S$ be the normalization of $S$ and set $B_{S^{\nm}}$ to be the different of $K_X + S + B$ along $S^{\nm}$.  
By \cite[Lemma 4.8]{KollarKovacsSingularitiesBook} $(S^{\nm}, B_{S^{\nm}})$ is klt and so \autoref{thm.KLTSurfaceImpliesBCMReg} 
implies that $(S^{\nm},B_{S^{\nm}})$ is globally $\bigplus$-regular.  Hence $S$ is normal by \autoref{cor.IntegralityOfS}. 

The last part follows as $(X,S)$ is plt when $X$ is $\bQ$-factorial.
\end{proof}

\begin{remark}
Suppose $(X, S+B)$ is a $\bQ$-factorial three dimensional plt pair over {any excellent finite dimensional domain $R$ with a dualizing complex and whose residue characteristics at closed points} have characteristic zero or greater than $5$.  Then $S$ is normal.  Indeed, the above result implies that $S$ is normal at the closed points of positive residual characteristics.  At characteristic zero points this follows from the standard arguments \cite[Proposition 5.51]{KollarMori} in view of \cite{takumi}.
\end{remark}

%% file: flips.tex
\section{Existence of flips}
\label{Section:flips}

\begin{notation}
\label{notation:flips_base_ring}
All schemes in this section are defined over a complete normal Noetherian local domain $(R,\fram)$ with residue field $R/\fram$ of characteristic $p>0$. We set $Z=\Spec(R)$, which will serve as the base of our flipping contractions. 

{For pairs $(X,\Delta)$ in this section, we will always assume that $\Delta$ is a $\bQ$-divisor and $K_X + \Delta$ is $\bQ$-Cartier.}
\end{notation}

\begin{remark} In this remark, fix the fraction field $K$ of some excellent domain $A$. We say that $V = \bigoplus_i V_i$ is a \emph{function algebra} if  it is an $\bN$-graded $A$-algebra with $A \subseteq V_0$ being a finite extension, $V_i \subseteq K$ finitely generated over $A$, and the multiplication on $V$ induced from $K$ (that is, $V$ is a subalgebra of $K[T]$). We call  $V^{(j)} = \bigoplus_{j \mid i} V_i$ the \emph{$j$-Veronese subalgebra} of $V$. We say that two function algebras $V$ and $V'$ are \emph{Veronese equivalent}, if some Veronese subalgebra of $V$ is isomorphic to some Veronese subalgebra of $V'$. Finite generation of function algebras is stable under Veronese equivalence (cf.\ \cite[Lemma 2.3.3]{CortiFlipsFor3FoldsAnd4Folds}).
\end{remark}
We encourage the reader to recall \autoref{def:mobile_part} and \autoref{rem:mob}.
\begin{outline}
\label{outline:flips}
The goal of the present section is to prove the existence of flips for threefolds in the situation of \autoref{notation:flips_base_ring} when $p>5$. 

Let us start with presenting the general idea of our proof, which largely follows the argument of Hacon and Xu in positive characteristic \cite{HaconXuThreeDimensionalMinimalModel} (in turn motivated by the strategy of Shokurov in characteristic zero, see \cite{CortiFlipsFor3FoldsAnd4Folds}). As explained in \cite[Lem 6.2]{KollarMori}, the main goal is to show that for a pair $(X,\Delta)$ with mild singularities (such as klt) and with a flipping contraction $f: X \to Z$ of relative Picard rank one\footnote{in our actual proof this assumption will be weakened}, the sheaf of $\sO_Z$-algebras 
\begin{equation*}
\bigoplus_{m \in \bN} f_* \sO_X(\lfloor m(K_X + \Delta) \rfloor )\end{equation*}is finitely generated; the flip is then given as the  $\Proj$ of this algebra. For this purpose we may assume that $Z$ is affine, which reduces the problem to showing that the section ring 
\begin{equation}
\label{eq:flips_outline}
    \sectionRingR(X,K_X+\Delta)=\bigoplus_{m \in \bN} H^0(X, \sO_X(\lfloor m(K_X + \Delta)\rfloor))
\end{equation} 
is finitely generated over $R=H^0(Z, \sO_Z)$.
An obvious way to approach this is to prove that $K_X+\Delta$ is semiample over $R$. Unfortunately, this will never happen as, by the definition of a flipping contraction, $K_X + \Delta$ is anti-ample. This suggests that we should find a semiample $\bQ$-divisor to which $\sectionRingR(X,K_X+\Delta)$ can be related. More precisely, we want to find a projective birational morphism $\pi \colon Y \to X$ and $i>0$ such that
\begin{equation} \label{equation:intro-flips}
M_i := \Mob\big(i\pi^*(K_X+\Delta)\big) \text{ is base point free}, \quad \text{ and } \quad kM_i = M_{ik} \text{ for all } k>0.
\end{equation}
Then $\sectionRingR(X, K_X+B)$ and $\sectionRingR(Y, M_i)$ are Veronese equivalent by the definition of the mobile part and by the stabilization. In particular, since the latter algebra is finitely generated by base-point-freeness, so is the former. 

It turns out that it is very hard to prove such a statement. For every $i$ we can find a resolution for which $M_i$ is base point free, but the resolution will \emph{a priori} depend on $i$, and there are not enough tools to prove that $kM_i = M_{ik}$ directly on $Y$. As usual in birational geometry we address this problem by restricting to a divisor.

Suppose that there exists an irreducible relatively anti-ample divisor $S$ with singularities so mild {(and being \emph{sufficiently transversal} to $\Delta$)} that we can increase its coefficient at $\Delta$ so that $\lfloor \Delta \rfloor = S$ and $(X,\Delta)$ is plt. Since the relative Picard rank is one and $S$ is anti-ample, $-(K_X+\Delta)$ is still ample and the new canonical ring is Veronese  equivalent to the old one. Hence, it is enough to show that our new $\sectionRingR(X,K_X+\Delta)$ is finitely generated. Flipping contractions for which such $S$ exists are called pl-flipping; note that this is quite a restrictive condition: in the spirit of Bertini, we should be able to find a very ample divisor with mild singularities, but not necessary an anti-ample one. Nevertheless, it is a standard argument that if you can prove the existence of flips for pl-flipping contractions (called \emph{pl-flips}), then you can use them to construct all flips (cf.\ \autoref{proposition:reduction-to-pl-flips}); briefly speaking, you pick an arbitrary anti-ample $S$ and then improve its singularities by running a special MMP on a log resolution and this only needs pl-flips. The huge advantage of pl-flipping contractions is that $(K_X+\Delta)|_S = K_S + \Delta_S$ for a klt pair $(S,\Delta_S)$  which suggests a possibility for applying induction.

Alas, the finite generation of $\sectionRingR(S,K_S+\Delta_S)$ is not enough to deduce the finite generation of $\sectionRingR(X,K_X+\Delta)$ as the restriction map $\sectionRingR(X,K_X+\Delta) \to \sectionRingR(S,K_S+\Delta_S)$ need not be surjective. It is easy to see, however, that if 
\[
    \sectionRingR_S = \mathrm{image}\,\big(\sectionRingR(X,K_X+\Delta) \to \sectionRingR(S,K_S+\Delta_S)\big)
\]
is finitely generated, then so is $\sectionRingR(X,K_X+\Delta)$ (see the proof of \autoref{theorem:flips-exist}). To this end, we apply the idea mentioned earlier: we find a projective birational morphism $\pi \colon Y \to X$ such that $M_i|_{S'}$ satisfies the conditions  of \autoref{equation:intro-flips}, where $S'$ is the strict transform of $S$, and so $\sectionRingR(S', M_i|_{S'})$ is finitely generated. It turns out, after some work, that $\sectionRingR_S$ is Veronese equivalent to $\sectionRingR(S', M_i|_{S'})$, and so is finitely generated as well, concluding the proof.

The finite generation of \autoref{eq:flips_outline}  for pl-flips  is shown in the present section (\autoref{cor:flips-exist2}), and the next section contains the reduction to pl-flips. Below, we introduce the notation needed to make the present outline more precise. Then, we give a more detailed explanation in \autoref{outline:flips2}. The assumption that $p>5$ {and $\Delta$ has standard coefficients will be needed so that $(S,B_S)$ is globally $+$-regular.} \end{outline}

\begin{definition}
\label{def:pl_flipping_contraction}
In the situation of \autoref{notation:flips_base_ring}, a \emph{pl-flipping} contraction  over $Z=\Spec R$ is a projective birational morphism $f \colon X \to Z$ of a plt pair $(X, S+B)$ with $\lfloor B\rfloor=0$ and $S$ irreducible   such that $f$ is small (that is, $\mathrm{Exc}(f)$ is of codimension at least two), and $-S$ and $-(K_X+S+B)$ are $f$-ample. 
\end{definition}

\noindent {Note that we do not assume in \autoref{def:pl_flipping_contraction}, as is usually the case, that $\rho(X/Z)=1$.}

\subsection{Finite generation of the restriction algebra} \label{subsection:flips-key-of-the-argument}
In the entire present subsection we work in the framework of the following notation. We do not state separately that this setting is assumed.

\begin{setting}
\label{notation:flips_general}
In the situation of \autoref{notation:flips_base_ring}, we assume additionally that {$R/\fram$} is infinite. Let $f \colon X \to Z$ be a three-dimensional pl-flipping contraction of a plt pair $(X,S+B)$, where $\dim R =3$ and $Z=\Spec R$. Since $X$ admits a small birational morphism to an affine scheme, we can replace $K_X$ by a linearly equivalent divisor so that $K_X+S+B$ is an effective $\Q$-divisor {and does not} contain $S$ in its support. This choice of $K_X$ is fixed for the whole section.

We also assume that $S$ is normal and $(S,B_S+\varepsilon D)$ is \emph{globally $\bigplus$-regular}  for every effective divisor $D$ on $S$ and $0 \leq  \varepsilon \ll 1$ (depending on $D$), where $K_S + B_S = (K_X+S+B)|_S$. This is the case, for example, if $B$ has standard coefficients and $p>5$ (\autoref{cor.3dim-plt-Fano-are-purely-plus-regular}).
\end{setting}
{Under the above hypothesis, $K_X$ is not effective and $K_X+S+B$ may contain some components of $B$ in its support. We choose $K_X$ in the way as above so that $B_S$ is the different of $(X,S+B)$ along $S$, where $(K_X+S+B)|_S = K_S + B_S$.

\begin{remark}
\label{rem:general_element}
The  residue field  $R/\fram$ in \autoref{notation:flips_general} is assumed to be infinite for the sole purpose so that if we have
\begin{itemize}
    \item a normal Noetherian excellent separated scheme $X$ over $R$,
    \item a base-point-free line bundle $L$ on $X$, and
    \item finitely many points $x_1,\dots, x_n \in X_{\fram}$,
\end{itemize} 
then there is  an element in the linear system $|L|$ which does not vanish at any of the points $x_i$.
\end{remark}

\begin{notation}
\label{notation:log-resolutions}
For a log resolution $\pi \colon Y \to X$ of $(X,S+B)$ as in \autoref{notation:flips_general}, we introduce the following notation:
\begin{itemize}
\item  $S'$ is the strict transform of $S$, 
    \item $K_Y + S' + B' = \pi^*(K_X+S+B)$,
    \item $K_{S'} + B_{S'} = (K_Y+S'+B')|_{S'}$, where we choose for $K_{S'}$ the representative $(K_Y + S')|_{S'}$,
    \item  $A' = -B'$, and 
    \item $A_{S'} = -B_{S'}$, so that  $A_{S'} = A'|_{S'}$.
\end{itemize}
Furthermore, for every integer $i>0$ such that $i(K_X+S+B)$ is Cartier, we set
\begin{align}
\label{eq:M_i_M_i_S_prime}
	M_i &:= \Mob(i(K_Y+S'+B')), \text{ and }\\ 
	M_{i,S'} &:= M_i|_{S'}, \nonumber
\end{align}
which makes sense as $M_i$ does not have $S'$ within its support. We note that it is vital to remember that $M_{i, S'}$ is the restriction of the mobile part, as opposed to the mobile part of the restriction.  
Additionally write
\begin{equation*}
    D_i := {\frac{1}{i}}M_i, \qquad 
    D_{i,S'} := { \frac{1}{i}}M_{i,S'}.
\end{equation*} 
{By definition, $D_j \leq D_i$ whenever $j(K_X+S+B)$ is Cartier and $j \mid i$.}
\end{notation}
\begin{remark}
As $\pi$ is a log resolution of $(X, S + B)$, the induced morphism $\pi|_{S'} : S' \to S$ is a log-resolution as well. Additionally, $B_{S}$ is defined in a way such that $K_{S'} + B_{S'} = \pi|_{S'}^* (K_S + B_S)$ holds. This implies that   $A_{S'}= - B_{S'}$ is the discrepancy divisor of the pair $(S, B_S)$ on the log resolution $S' \to S$. 

Since $(X,S+B)$ is plt, $(S, B_S)$ is also klt. Therefore, by the definition of $A'$ and by the previous paragraph, we have that $\lceil A' \rceil$ and $\lceil A_{S'} \rceil$ are effective and exceptional over $X$ and $S$, respectively.  We will also repeatedly use that every line bundle on $Y$ or $S'$ is automatically big (as $Y \to \Spec R$ and $S' \to f(S)$ are birational).
\end{remark}

\begin{definition} \label{def:compatible-resolution} 
In the situation of \autoref{notation:log-resolutions}, let $\pi \colon Y \to X$ be a log resolution of $(X,S+B)$. We say that it is \emph{good} if 
\begin{itemize}
    \item it is also a log resolution of $(X, S+B+ (K_X+S+B))$ for $K_X+S+B$ being the explicit effective $\bQ$-divisor fixed in \autoref{notation:flips_general}, and
    \item  $S' \to S$ factors through the terminalization $\overline{S}$ of $(S,B_S)$ (which is unique as $S$ is two-dimensional)
\end{itemize}
Let $i>0$ be an integer such that $i(K_X+S+B)$ is Cartier. Then we say that $\pi \colon Y \to X$ is \emph{compatible with $i$}, if it is good and it is a resolution of the linear system $|i(K_X+S+B)|$. The latter condition is equivalent to $|M_i|$ being base point free.
\end{definition}
Observe that $\overline{S}$ is a terminal surface hence it is regular. }

\begin{remark} \label{remark:compatible-resolutions-make-everything-log-smooth}
 {If $\pi \colon Y \to X$ is a good log resolution of $(X,S+B)$, then $\Supp(S'+B'+M_i+\mathrm{Ex}(\pi))$ is simple normal crossing for every integer $i>0$ such that $i(K_X +S +B) $ is Cartier.}
\end{remark}

When $Y$ is compatible with $i$, which essentially will always need to be the case for us, then the choice of $\pi$, $Y$, and $S'$ \emph{depends on $i$}. Note that given $i,j \in \mathbb{N}$, we can always construct $Y$ which is compatible with both $i$ and $j$ (cf.\  \cite[Lemma 4.5]{CP19}). However, \emph{a priori}, it might not be possible to construct $Y$ which is compatible with all $i \in \mathbb{N}$ simultaneously; \emph{a posteriori} such $Y$ exists as a consequence of the existence of flips.

\begin{remark} \label{remark:pullback-of-b-divisor}
Given a sequence of maps $\tilde Y \xrightarrow{h} Y \to X$ such that $\tilde Y \to X$ and $Y \to X$ are resolutions compatible with $i$, we have that $h^*D_i = \tilde D_i$, where $\tilde D_i$ is calculated for $\tilde Y$ exactly as $D_i$ is calculated for $Y$. The same property holds for $D_{i,S'}$.

We emphasize that if $Y \to X$ is not compatible with $i$, then although $D_i = h_*\tilde D_i$, it need not even be true that $D_{i,S'} = (h|_{\tilde{S}'})_*{\tilde D}_{i,\tilde{S}'}$, where $\tilde{S}'$ is the strict transform of $S'$ and ${\tilde D}_{i,\tilde{S}'} = \tilde{D}_i|_{\tilde{S}'}$ (pushing forward for divisors does not commute with restrictions).
\end{remark}

\begin{outline}
\label{outline:flips2}
Having introduced the above notation, we are able to provide a more detailed version of \autoref{outline:flips}. As explained therein, our goal is to show that
\[
\sectionRingR_S = \mathrm{image}\big( \sectionRingR(X,K_X+S+B) \to \sectionRingR(S,K_S+B_S)\big)
\]
is finitely generated.  We will prove that, up to taking a Veronese subalgebra, 
\begin{equation} \label{eq:flips:candidate-for-restriction-algebra}
    \sectionRingR_S = \bigoplus_i H^0\big(\overline{S}, \sO_{\overline{S}}(iD_{\overline{S}})\big)
\end{equation}
for a semiample $\bQ$-divisor $D_{\overline{S}}$ on $\overline{S}$, where $(\overline S, B_{\overline S})$ is the terminalization of $(S,B_S)$. In particular, this implies that $\sectionRingR_S$ is finitely generated. 
 
The $\bQ$-divisor $D_{\overline{S}}$ is constructed as follows. {First, for a log resolution $\pi \colon Y \to X$ compatible with $i \in \bN$} we show that $D_{i,S'}$ is a pullback of a  $\Q$-divisor $D_{i,\overline{S}}$ on $\overline{S}$ (\autoref{proposition:flips-descend-of-b-divisors}). Then, we define an $\bR$-divisor  $D_{\overline{S}}$ as the limit of $D_{i,\overline{S}}$ for $i\to \infty$ and show that, in fact, it is a $\bQ$-divisor (\autoref{proposition:flips-b-divisor-is-rational}). Next, we show that $D_{i,\overline{S}}$ coincides with $D_{\overline{S}}$ for divisible enough $i>0$ (\autoref{proposition:flips-restriction-of-key-identity}). Last, we prove the validity of  \autoref{eq:flips:candidate-for-restriction-algebra} (\autoref{claim:proof-of-pl-flips}), and conclude that $\sectionRingR_S$ is finitely generated (\autoref{proposition:restricted-algebra-fin-gen}). 

Let us emphasize that we use in an essential way that $S$ is a surface, and so we cannot run the above limiting process directly on a birational  model of $X$.

\end{outline}

The key to our strategy is to understand the divisors $M_i|_{S'}$ which are restrictions of mobile parts of $i\pi^*(K_X+S+B)$ to $S'$. Since $\Mob$ does not commute with restrictions in general, we want to find a property of the divisors $M_i$ that could also be shared by $M_i|_{S'}$. The following technical lemma identifies such a property. 
\begin{lemma} \label{lemma:flips-key-identity}
For every log resolution $\pi \colon Y  \to X$ of $(X,S+B)$, if
$i,j >0$ are integers such that $i(K_X + S + B)$ and $j(K_X+S+B)$ are Cartier, then  $\Mob \lceil jD_{i} + A' \rceil \leq jD_{j}$.
\end{lemma}
\noindent Let us point out that from the viewpoint of Kawamata-Viehweg or $\myB^0$-lifting, the divisors of the form $\lceil jD_i + A' \rceil$ work well, see \autoref{eq:technical-lifting-sections-in-flips:adjunction1}.
\begin{proof}
Since $\lceil A' \rceil \geq 0$ is exceptional over $X$, we have that $\pi_*\sO_Y(j(K_Y+S'+B') + \lceil A' \rceil) = \sO_X(j(K_X+S+B)) = \pi_*\sO_Y(j(K_Y+S'+B'))$ (cf.\ \autoref{lemma:pushforward}). This implies that $D \mapsto D + \lceil A' \rceil$ yields a bijection between $|j(K_Y+S'+B')|$ and $|j(K_Y+S'+B') + \lceil A' \rceil|$, which is what we use in the first equality of the following computation:
\[
\Mob(\lceil jD_{i} + A' \rceil) \leq \Mob(j(K_Y+S'+B') +  \lceil A' \rceil) = \Mob(j(K_Y+S'+B')) = jD_{j}. \qedhere
\]
\end{proof}
In fact, to deduce the properties of $D_{\overline{S}}$ mentioned in \autoref{outline:flips2}, it is enough to show that the identity of \autoref{lemma:flips-key-identity} holds also after restricting to $S'$, and the rest of the argument for the existence of flips would be mostly characteristic free except for some technical issues with Bertini. In characteristic zero, this can be achieved by the Kawamata-Viehweg vanishing. {More precisely, the surjectivity of $H^0(Y, \sO_Y(\lceil jD_i + A'\rceil)) \to H^0(S', \sO_{S'}(\lceil jD_{i,S'} + A_{S'}\rceil))$ in characteristic zero (cf.\ \autoref{eq:technical-lifting-sections-in-flips:adjunction1})
immediately implies that 
$\Mob \lceil jD_{i,S'} + A_{S'} \rceil \leq jD_{j,S'}$  Alas, it seems impossible to show this surjectivity directly in positive and mixed characteristic, so we obtain the above surjectivity only towards the end of this subsection in \autoref{proposition:flips-restriction-of-key-identity}. }

\begin{remark}
\label{rem:f_S_Q_factorial}
In the following proof we will use that the normalization $f(S)^{\mathrm{N}}$ of $f(S)$ is $\bQ$-factorial. Indeed, by \autoref{lemma:add_ample} 
in dimension two,
we can pick an effective ample $\bQ$-divisor  $H_S \sim_{\bQ} -(K_S + B_S)$ such that $(S,B_S + H_S)$ is klt.  Hence $(f(S)^{\mathrm{N}}, (f|_S)_*B_S + (f|_S)_*H_S)$ is klt as well.
\end{remark}

\begin{proposition}
\label{proposition:flips-descend-of-b-divisors}
{Let $i>0$ be an integer such that $i(K_X+S+B)$ is Cartier and  let $\pi \colon Y \to X$ be a log resolution of $(X,S+B)$ which is compatible with $i$. Then the divisor $M_{i,S'}$ descends to $\overline{S}$: it is a pullback of some divisor $M_{i,\overline{S}}$ on $\overline{S}$.}
\end{proposition}
\noindent {We emphasize here that $M_{i,\overline{S}}$ is \emph{not} defined as $M_{i,S'}$ was defined in \autoref{eq:M_i_M_i_S_prime}, i.e.\ by restricting a divisor from an ambient space; it is the pushforward of $M_{i,S'}$ to $\overline{S}$.}

\begin{proof}
As $M_i$ and $M_{i,S'}$ are integral, we have that
\begin{align}
	\label{eq:FirstEqualityIn-flips-descend-of-b-divisors}
\lceil M_{i} + A' \rceil &=  K_{Y} + S' + \{B'\} + M_i - \pi^*(K_X+S+B), \text{ and }\\ 
\label{eq:SecondEqualityIn-flips-descend-of-b-divisors}
\lceil M_{i,S'} + A_{S'} \rceil &= K_{S'} + \{B_{S'}\} + M_{i,S'} - (\pi|_{S'})^*(K_S+B_S).
\end{align}
Since $M_i - \pi^*(K_X+S+B)$ is big and semiample (as both $M_i$ and $-\pi^*(K_X+S+B)$ are big and semiample), \autoref{thm:main-lifting} yields a surjection
\begin{equation}
\label{eq:flips-descend-of-b-divisors:surjection}
\myB^0_{S'}(Y, S'+\{B'\}; \sO_Y(\lceil M_i + A' \rceil)) \twoheadrightarrow \myB^0(S', \{B_{S'}\}; \sO_{S'}(\lceil M_{i,S'}+A_{S'}\rceil)).
\end{equation}
Applying \autoref{lemma:flips-key-identity} with $j = i$
yields $\Mob(M_i +  \lceil A' \rceil) = M_i$. Combining this with \autoref{eq:flips-descend-of-b-divisors:surjection} we obtain that every section in the vector space $\myB^0(S', \{B_{S'}\}; \sO_{S'}(\lceil M_{i,S'}+A_{S'}\rceil))$ vanishes along $\lceil A_{S'} \rceil$.

As the support of $\lceil A_{S'} \rceil$ is equal to the exceptional locus of $g \colon S' \to \overline{S}$ (by definition of terminalization), to prove the proposition it is enough to show that there exists an element of $|M_{i,S'}|$ which does not intersect $\lceil A_{S'} \rceil$. Assume by contradiction the opposite. Then, as  $|M_{i,S'}|$ is free, there exists an element $M \in |M_{i,S'}|$ which  does not contain any component of $\lceil A_{S'}\rceil$ in its support, see \autoref{rem:general_element}. By our contradiction assumption we may then choose  $x \in \Supp M \cap \Supp \lceil A_{S'} \rceil$.

 Note that the exceptional locus of {$S' \to f(S)$} is simple normal crossing and the normalization of $f(S)$ is   $\bQ$-factorial by \autoref{rem:f_S_Q_factorial}.  Therefore, we can pick an effective $\Q$-divisor $F$ on $S'$ which is {anti-ample and} exceptional over $f(S)$ and  such that $(S',\{B_{S'}\}+F)$ is both klt at $x$ and simple normal crossing at $x$.  Furthermore, by taking a suitable positive multiple of $F$, for any $0<\delta\ll1$ (to be determined later) we may find $F_{\delta}\geq F$ such that at least one of the exceptional divisors passing through $x$ has coefficient $1-\delta$ in $\{B_{S'}\}+F_{\delta}$ (and $(S',\{B_{S'}\}+F_{\delta})$ is still klt at $x$). In particular, the log discrepancy of the exceptional divisor of the blow-up at $x$ with respect to this pair is at most $1+\delta$. Then,
\[
\myB^0(S', \{B_{S'}\}; \sO_{S'}(\lceil M_{i,S'}+A_{S'}\rceil)) \supseteq \myB^0(S', \{B_{S'}\}+F_{\delta}; \sO_{S'}(\lceil M_{i,S'}+A_{S'}\rceil)),
\]
and the latter space, hence also the former, is free at $x$ for sufficiently small $\delta$ by \autoref{theorem:seshadri}. Indeed,
\begin{multline}
\label{eq:flips-descend-of-b-divisors:seshadri}
\epsilon_{\mathrm{sa}}(\lceil M_{i,S'}+A_{S'}\rceil - (K_{S'}+\{B_{S'}\}+F_{\delta});x) \geq \epsilon_{\mathrm{sa}}(M-F_\delta;x)
\\ \geq \epsilon_{\mathrm{sa}}(M-F;x) > { \epsilon_{\mathrm{sa}}(M;x)\geq 1},
\end{multline}
where $\lceil M_{i,S'}+A_{S'}\rceil - (K_{S'}+\{B_{S'}\}+F_{\delta}) \sim_{\bQ} M - F_{\delta}- (\pi|_{S'})^*(K_S+B_S)$ is semiample \autoref{eq:SecondEqualityIn-flips-descend-of-b-divisors}, and
\begin{itemize}
    \item in the first inequality we used that {$-(\pi|_{S'})^*(K_S+B_S)$} is big and semiample,
    \item in the second and third inequality we used that $-F$ is ample, and that $F_\delta$ is a positive multiple of $F$, and
    \item the last inequality is a direct consequence of \autoref{lem:flips_seshadri_lower_bound}. 
    
    \end{itemize}
    Using \autoref{eq:flips-descend-of-b-divisors:seshadri}, we may now choose $0<\delta\ll1$ to be such that  
         $\epsilon_{\mathrm{sa}}(M-F,x)>1+\delta$, 
  resulting in $F_{\delta}$ satisfying 
  $\epsilon_{\mathrm{sa}}(M-F_{\delta},x)>1+\delta$, 
 which allows us to apply \autoref{theorem:seshadri}.
 
The freeness of $\myB^0(S', \{B_{S'}\}; \sO_{S'}(\lceil M_{i,S'}+A_{S'}\rceil))$ at $x$
 is a contradiction to the fact that every section of this linear system vanishes along $\lceil A_{S'}\rceil$. \qedhere
 \end{proof}

Note that $M_{i,\overline{S}}$ is independent of the choice, in the above proposition, of a log resolution $\pi \colon Y \to X$ compatible with $i$ by \autoref{remark:pullback-of-b-divisor}, and so it exists and is unique for every integer $i>0$ such that $i(K_X+S+B)$ is Cartier.
Therefore, we may introduce the following notation, which is assumed until the end of this subsection.

\begin{notation}
\label{notation:D_overline_S}
We set
\begin{equation*}
    D_{i,\overline{S}} := \frac{1}{i}M_{i,\overline{S}}, 
    \qquad \qquad
    D_{\overline{S}} := \lim_{i \to \infty} D_{i,\overline{S}},
    \end{equation*}
where the limit is taken over all integers $i>0$ such that $i(K_X+S+B)$ is Cartier; it exists by \autoref{lem:limit-of-b-divisors}. Note that $D_{i,\overline{S}}$ is an $\bR$-divisor.

Further, for any good log resolution $\pi \colon Y \to X$ of $(X,S+B)$ (in the situation of \autoref{notation:log-resolutions}), write $D_{S'} := g^*D_{\overline{S}}$, where $S' \xrightarrow{g} \overline{S} \xrightarrow{h} S$ is the given factorization. Last, set $ K_{\overline{S}} + B_{\overline{S}} = h^*(K_{S} + B_{S})$.  \end{notation}
\noindent We emphasize that $D_{S'}$ cannot be defined as a limit of $D_{i,S'}$ directly on $S'$ as $D_{i,S'}$ does not have good properties unless the log resolution $\pi \colon Y \to X$ is compatible with $i$ (cf.\ \autoref{remark:pullback-of-b-divisor}), in which case $S'$ depends on $i$.

We need the following lemmas. 
\begin{lemma} \label{lem:limit-of-b-divisors}
    The limit $D_{\overline{S}}$, as defined in \autoref{notation:D_overline_S}, exists.
           It is a nef $\mathbb{R}$-divisor, and moreover, $D_{j,\overline{S}} \leq D_{i,\overline{S}}$ when $j(K_X+S+B)$ is Cartier and $j \mid i$. In particular, $D_{j, \overline{S}} \leq D_{\overline{S}}$.
\end{lemma}
\begin{proof}
    Let $i,j > 0$ be integers such that $i(K_X+S+B)$ and $j(K_X+S+B)$ are Cartier. Pick a log resolution $\pi \colon Y \to X$ which is compatible with $i$, $j$, and $i+j$. By definition, $M_i + M_j \leq M_{i+j}$. Hence, $M_{i,S'} + M_{j,S'} \leq M_{i+j, S'}$, and so $M_{i,\overline S} + M_{j,\overline S} \leq M_{i+j, \overline S}$. In turn, this gives
    \[
    \frac{i}{i+j}D_{i,\overline S} + \frac{j}{i+j}D_{j,\overline S} \leq D_{i+j,\overline S}.
    \]
    Further, note that $D_{i,\overline{S}} \leq K_{\overline{S}}+B_{\overline{S}}$ (recall that $K_X+S+B$ is an explicit effective $\bQ$-Cartier $\bQ$-divisor without $S$ in its support; by restricting to $S$ and pulling back to $\overline{S}$ this determines the right hand side as an effective $\Q$-divisor).  In particular this ensures that there is a fixed finite set of irreducible divisors {which contain the support of} every $D_{i,\overline S}$.
    
    The existence of the limit now follows from the fact that any sequence of real numbers $\{a_i\}_{i \in \bZ_{>0}}$ which is bounded from above and satisfies $\frac{i}{i+j}a_i + \frac{j}{i+j}a_j \leq a_{i+j}$ is convergent. Moreover, this condition implies that $a_j \leq a_i$ when $j \mid i$, and so $a_j \leq \lim_{i \to \infty} a_i$.
\end{proof}

\begin{lemma} \label{lemma:technical-lifting-sections-in-flips}
{Let $i>0$ be an integer such that $i(K_X + S + B)$ is Cartier and let $\pi \colon Y \to X$ be a log resolution of $(X,S+B)$ which is compatible with $i$.} Then the following map is surjective {for every $j>0$}:
\[
\myB^0_{S'}(Y, S'+\{B'-jD_i\}; \sO_{Y}(\lceil jD_i + A' \rceil)) \to \myB^0(S', \{B_{S'} - jD_{i,S'}\}; \sO_{S'}(\lceil jD_{i,S'}+A_{S'}\rceil)).
\]
\end{lemma}
\begin{proof}
Using the identity $\lceil L \rceil = L + \{-L \}$ for any $\bQ$-divisor $L$, we have that
\begin{align}
\label{eq:technical-lifting-sections-in-flips:adjunction1}
\lceil jD_{i} + A' \rceil &= K_{Y} + S' + \{B'-jD_i\} + jD_i - \pi^*(K_X+S+B), \text{ and }\\ 
\label{eq:technical-lifting-sections-in-flips:adjunction2}
\lceil jD_{i,S'} + A_{S'} \rceil &= K_{S'} + \{B_{S'}-jD_{i,S'}\} + jD_{i,S'} - (\pi|_{S'})^*(K_S+B_S).
\end{align}
Since $jD_i - \pi^*(K_X+S+B)$ is big and semiample, we obtain the sought-after surjection by \autoref{thm:main-lifting}. {Here, we used that $(K_{Y} + S' + \{B'-jD_i\})|_{S'} = K_{S'} + \{B_{S'}-jD_{i,S'}\}$ which follows from \autoref{remark:compatible-resolutions-make-everything-log-smooth}.}
\end{proof}

\begin{lemma}\label{lemma:diophantine-approximation}
Fix $a_1, \ldots, a_k \in \bR$, and let $G$
be the image of the additive semigroup 
\[
\big\{(ja_1,\ldots,ja_k) \ \big| \ j \in \bZ_{\geq 0}\big\}
\]
under the natural projection $\lambda \colon \mathbb{R}^k \to \mathbb{R}^k/\mathbb{Z}^k$ of topological groups. Let $\overline{G}$ be the closure of  $G$ and set $T:= \mathbb{R}^k/\mathbb{Z}^k$. Then:
\begin{enumerate}
\item 
\label{itm:diophantine-approximation:closed_subgroup}
$\overline{G}$ is a closed topological subgroup of $T$, and hence it is a disconnected union of finitely many translates of the connected component $\overline{G}^0$  of the identity, and 
    \item $\lambda^{-1}(\overline{G}^0)= \bZ^k + L$ for an $\bR$-linear subspace $L$ of $\bR^k$, 
\end{enumerate}
 In particular,  for every $\varepsilon > 0$ we can find a natural number $j>0$ and integers $m_1, \ldots, m_k$ such that $|m_i - ja_i| < \varepsilon$ for every $i$. 
 
  \end{lemma}
 \begin{proof}
By  the main theorem of \cite{Wright56}, $\overline{G}$ is a closed subgroup of $T$. In particular it is compact, which implies  \autoref{itm:diophantine-approximation:closed_subgroup}. The rest  follows from \cite[Ch.\ 7.2, Thm.\ 2]{BourbakiTopology5-10}.
 \end{proof}
 
\begin{theorem} \label{proposition:flips-b-divisor-is-rational}
The $\mathbb{R}$-divisor $D_{\overline{S}}$ is in fact a {semiample} $\Q$-divisor.
\end{theorem}
\begin{proof}

First,  as $(\overline{S},B_{\overline{S}})$ is klt, it is $\bQ$-factorial. Second, by the base point free theorem for Noetherian excellent surfaces \cite[Theorem 4.2 and Remark 4.3]{tanaka_mmp_excellent_surfaces} {and since $-(K_{\overline{S}}+B_{\overline{S}})$ is nef and big}, we know that every nef $\Q$-divisor {on $\overline{S}$} is not only $\bQ$-Cartier, but also semiample. This we will use multiple times during the proof. Additionally,  it also reduces our goal to showing that $D_{\overline{S}}$ is a $\bQ$-divisor.

As $\overline{S} \to f(S)$ is a projective birational  morphism of Noetherian excellent surfaces, there are finitely many irreducible curves $E_1,\dots, E_s$ on $\overline{S}$ that are exceptional over $f(S)$. Additionally, we can reorder them so that
 $E_1, \ldots, E_r$ for some integer $r>0$ are exactly the curves for which $E_i \cdot D_{\overline{S}} = 0$. As $D_{\overline S}$ is nef, we have that $D_{\overline S} \cdot E_i > 0$ for $r < i \leq s$. Set 
\begin{equation*}
    V = \big\{\, D \ \big| \
    D \cdot E_i =0 \text{ for } 1 \leq i \leq r \, \big\} \subseteq \mathrm{Div}(\overline S) \otimes_{\bZ} \mathbb{R},
\end{equation*}
{where $\mathrm{Div}(\overline S)$ is a free abelian group on all (not necessarily exceptional) prime divisors on $\overline S$.}
{We endow $\mathrm{Div}(\overline S) \otimes \mathbb{R}$ with the standard Euclidean metric by setting $\Vert G \Vert=1$ for every irreducible divisor $G$.}

Since $V$ is defined over $\Q$, we can pick $\mathbb{Z}$-divisors $N_1, \ldots, N_k \in V$  such that  $D_{\overline{S}} = \sum a_iN_i$ for some \emph{positive real} numbers $a_1, \ldots, a_k$. Moreover, we can re-choose $N_i$ to be in a  $\delta$-neighborhood of $qD_{\overline{S}}$ for some $q\gg 0$, where $\delta$ is  the diameter of the fundamental parallelepiped in the lattice spanned by the originally chosen $N_i$, so that we obtain: 
\begin{equation}
\label{eq:flips-b-divisor-is-rational:close}
    \Big\Vert D_{\overline{S}} - \frac{N_i}{q} \Big\Vert \ll 1,
\end{equation} 
Explicitly, $q$ is chosen big enough so that
$(D_{\overline{S}} - \frac{N_i}{q}) \cdot E_j < D_{\overline{S}} \cdot E_j$ for all $r < j \leq s$ (this is possible as the right hand side is positive). In particular, $N_i \cdot E_j > 0$ for all $r < j \leq s$. Moreover, since $N_i \in V$, we have that $N_i \cdot E_j = 0$ for all $1 \leq j \leq r$. This implies:
\begin{itemize}
    \item for every $1 \leq j \leq s$ we have  $N_i \cdot E_j = 0$ if and only if $D_{\overline{S}} \cdot E_j = 0$, and
    \item  $N_i$ are nef (and hence semiample) over $f(S)$.
\end{itemize}

Therefore, by replacing $N_i$ by their multiples {(this might render $\big\Vert D_{\overline{S}} - \frac{N_i}{q} \big\Vert \ll 1$ invalid but we will not need this going forward)}, we may assume that:
\begin{itemize}
    \item the linear systems $|N_i|$ define the same birational morphism $a \colon \overline{S} \to S^+$,
    \item $a$ contracts exactly the curves $E_1,\ldots, E_r$, and 
    \item  $N_i = 3a^*N^+_i$, where $N^+_i$ is a very ample divisor on $S^+$.
\end{itemize}
Recall that $D_{\overline{S}} = \sum a_iN_i$. Thus $D_{\overline S} = a^*D_{S^+}$
for the $\mathbb{R}$-divisor $D_{S^+}= \sum 3a_i  N_i^+$ on $S^+$. We also set $B_{S^+} = a_* B_{\overline S}$ and $A_{S^+} = a_* A_{\overline S}$.

Assume by contradiction that $D_{\overline{S}}$ is \emph{not} a $\Q$-divisor. Under this assumption, we claim that we can find an integer $j>0$ and a base point free Weil divisor $N$ on $\overline{S}$ such that 
\begin{enumerate}
    \item 
    \label{itm:flips-b-divisor-is-rational:close}
    $\left\Vert jD_{\overline{S}}-N \right\Vert \ll 1$, and
    \item
        \label{itm:flips-b-divisor-is-rational:not_effective}
    $jD_{\overline{S}} - N$ is not {effective}. 
\end{enumerate}
For  condition \autoref{itm:flips-b-divisor-is-rational:close}, we can just set $N=m_1N_1 + \ldots m_kN_k$ for positive integers $m_1,\ldots,m_k$ and $j>0$ as in \autoref{lemma:diophantine-approximation}. However to guarantee also  condition \autoref{itm:flips-b-divisor-is-rational:not_effective} we have to do a more involved argument. We consider the image $W \subseteq V$ of the vector space $L$ from \autoref{lemma:diophantine-approximation} under the linear map $\phi \colon \mathbb{R}^k \to \mathrm{Div}(\overline S) \otimes \mathbb{R}$ given by $\phi \colon (x_1,\ldots,x_k) \mapsto x_1N_1 + \ldots + x_kN_k$. Note that $W$ is a non-trivial vector space; indeed, otherwise the classes of $jD_{\overline{S}} = \sum ja_iN_i$ in $\mathrm{Div}(\overline S) \otimes \mathbb{R}\big/\mathrm{Div}(\overline S)$ for integers $j>0$ would belong to a finite subset. Hence, $D_{\overline{S}}$ would be a $\bQ$-divisor, contradicting our assumption.

The effective cone in $W$ (that is, the subset of all $\bR$-divisors in $W$ with coefficients at prime divisors being at least $0$) is a \emph{closed} cone which does not contain a line. Hence, we can pick $\Gamma \in W$ in a small neighborhood of $0$ which is not effective. Additionally, by the definition of $L$  in \autoref{lemma:diophantine-approximation} and by the closedness of the effective cone, we can find $j>0$ and positive integers $m_1, \ldots, m_k$ such that $jD_{\overline{S}} - N$ is sufficiently close to $\Gamma$, where $N = m_1N_1 + \ldots + m_kN_k$. Hence, $jD_{\overline{S}} - N$ is not effective  and additionally $\Vert jD_{\overline{S}}-N \Vert \ll 1$. This concludes the above claim and the proof of conditions \autoref{itm:flips-b-divisor-is-rational:close} and \autoref{itm:flips-b-divisor-is-rational:not_effective}. Since both $jD_{\overline{S}}$ and  $N$ are pullbacks from $S^+$, we obtain that in fact $jD_{S^+}-a_*N$ is not effective. 
Further we remark that in the above construction we may assume that $j$ is divisible enough so that $j(K_X+S+B)$ is Cartier. Indeed, for this we just have to replace $(a_1,\dots,a_k)$ with an adequate multiple at the beginning of the argument.   

Since $N_i = 3a^*N^+_i$ for every $i$, we see that $\frac{1}{3}N$ is a pullback of a very ample divisor from $S^+$, and  we can  pick a curve $C \sim \frac{1}{3}N$ on $\overline{S}$ which does not contain any of the exceptional divisors of $\overline{S} \to f(S)$ in its support (\autoref{rem:general_element}). 
\begin{claim} \label{claim:rationality-of-b-divisor}
$\lceil jD_{i,\overline{S}} + A_{\overline S} \rceil - jD_{j,\overline S}$ is $a$-exceptional {for all $i \in \bN$ such that $i\gg j$ and $j \mid i$.}
\end{claim} 
\noindent Explicitly, we pick $i \gg j$ so that $||jD_{i,\overline S} - N|| \ll 1$. Note that since the images of $D_{i,\overline{S}}$ and $D_{j,\overline{S}}$ agree on $f(S)$, we have that $\lceil jD_{i,\overline{S}} + A_{\overline S} \rceil - jD_{j,\overline S}$ is automatically exceptional over $f(S)$.
\begin{proof}[Proof of claim]
{Let $\pi \colon Y \to X$ be a log resolution compatible with $i$ and $j$ and let $S'$ be the strict transform of $S$ on $Y$ as before. By re-choosing $C$ we can assume that it does not contain the image of $\mathrm{Exc}(S' \to \overline {S})$ under the map $S' \to \overline{S}$ in its support. Let $C'$ be the strict transform of $C$ on $S'$. By the above, we can assume that $C'$ is a pullback of $C$, contains no curves exceptional over $f(S)$ in its support, and is disjoint from the exceptional locus of $S' \to S^+$. Note that $C'$ intersects every curve which is exceptional over $f(S)$ but horizontal over $S^+$.}

Since $\Mob\lceil jD_i + A' \rceil \leq jD_j$ (see \autoref{lemma:flips-key-identity}), every section of $H^0(Y, \sO_Y(\lceil jD_i + A' \rceil))$ vanishes along $\lceil jD_i + A' \rceil - jD_j \geq 0$ (here, $jD_j=M_j$ is integral, $D_i \geq D_j$, and $\lceil A' \rceil \geq 0$). Thus, by \autoref{lemma:technical-lifting-sections-in-flips}, all the sections of 
\[
\myB^0(S', \{B_{S'} - jD_{i,S'}\}; \sO_{S'}(\lceil jD_{i,S'}+A_{S'}\rceil))
\]
vanish along $E := \lceil jD_{i,S'} + A_{S'} \rceil - jD_{j,S'} \geq 0$. But the above space is base point free at every point $x \in C' \cap E$ by \autoref{theorem:seshadri}, and so there is no such point, concluding the proof. We can invoke \autoref{theorem:seshadri}, because 
\begin{align*}
\epsilon_{\mathrm{sa}}(\lceil jD_{i,S'}+A_{S'}\rceil - (K_{S'}+\{B_{S'}- jD_{i,S'}\});x) &\geq \epsilon_{\mathrm{sa}}(jD_{i,S'};x)
\\ &= \epsilon_{\mathrm{sa}}(jD_{i,\overline S};y),
\\ &= \epsilon_{\mathrm{sa}}((jD_{i,\overline S}-N) + 3C;y) >2,
\end{align*}
where $y$ is the image of $x$ on $\overline S$. Here:
\begin{itemize}
    \item the first inequality holds, because $\lceil jD_{i,S'}+A_{S'}\rceil - (K_{S'}+\{B_{S'} - jD_{i,S'}\}) - jD_{i,S'}$ is big and semiample (see \autoref{eq:technical-lifting-sections-in-flips:adjunction2} in the proof of \autoref{lemma:technical-lifting-sections-in-flips})
    \item the first equality holds, because $D_{i,S'}$ is a pullback of $D_{i,\overline S}$ and $x$ is not contained in the exceptional locus of $S' \to \overline{S}$,
    \item the second equality holds, because $3C \sim N$, and
    \item the second inequality holds by \autoref{lem:flips_seshadri_lower_bound} for $k=3$ as $y \in C$ and  $\Vert jD_{i,\overline S}-N \Vert \ll 1$.    \end{itemize}
\end{proof}

\autoref{claim:rationality-of-b-divisor} implies that
$
a_*N \leq \lceil jD_{S^+} + A_{S^+} \rceil = {jD_{j, S^+} } \leq jD_{S^+},
$
where the first inequality follows from $0 < \Vert jD_{S^+} - a_*N \Vert \ll 1$ and the fact that $A_{S^+}$ has coefficients in $(-1,0)$. Here we put the norm on $\mathrm{Div}(S^+)$ the same way as on $\mathrm{Div}(\overline{S})$, and hence $0 < \Vert jD_{S^+} - a_*N \Vert \ll 1$ follows from $0 < \Vert jD_{\overline{S}} - N \Vert \ll 1$, as the former contains a subset of the non-zero coefficients of the latter. 

Therefore, we obtained $a_*N \leq  jD_{S^+}$, which
 contradicts the fact that $jD_{S^+}-a_*N$ is not an effective $\bR$-divisor. 
\qedhere
\end{proof}

\begin{proposition} \label{proposition:flips-restriction-of-key-identity}
{Let $i, j >0$ be integers such that $i$ is divisible by $j$ and $i \gg j$. Further assume that $j(K_X+S+B)$ and $jD_{\overline S}$ are Cartier. Let $\pi \colon Y \to X$ be a log resolution compatible with $i$ and $j$.} Then the following identity holds:
\[
\Mob \lceil jD_{i,S'} + A_{S'} \rceil \leq jD_{j,S'}.
\]
In particular, if $j$ is chosen so that $jD_{\overline{S}}$ is base point free, then $D_{j,\overline{S}} = D_{\overline{S}}$.
\end{proposition}
{ \noindent Explicitly, we pick $i \gg j$ so that $\lceil jD_{i,\overline{S}}\rceil = jD_{\overline S}$ and the coefficients of $\{-jD_{i,\overline S}\}$ are $ \ll 1$.}
\begin{proof}
Since $\Mob \lceil jD_{i} + A' \rceil \leq jD_{j}$ by \autoref{lemma:flips-key-identity}, it suffices to show that
\[
H^0(Y,\sO_Y(\lceil jD_{i} + A' \rceil)) \to H^0(S',\sO_{S'}(\lceil jD_{i,S'} + A_{S'} \rceil))
\]
is surjective. By \autoref{lemma:technical-lifting-sections-in-flips}, we have a surjection
\[
\myB^0_{S'}(Y, S'+\{B'-jD_i\}; \sO_{S'}(\lceil jD_i + A' \rceil)) \to \myB^0(S', \{B_{S'} - jD_{i,S'}\}; \sO_{S'}(\lceil jD_{i,S'}+A_{S'}\rceil)),
\]
and so we will be done if we show that the right hand side equals $H^0(S', \sO_{S'}(\lceil jD_{i,S'}+A_{S'}\rceil))$.\\

Since $jD_{\overline{S}}$ is integral, $jD_{i,\overline{S}} \leq jD_{\overline{S}}$ (\autoref{lem:limit-of-b-divisors}), and for $j \ll i$, we obtain that $\lceil jD_{i,\overline{S}} \rceil = jD_{\overline{S}}$ and the coefficients of $\{-jD_{i,\overline S}\}$ are $ \ll 1$. In particular, $(\overline{S}, B_{\overline S} + \{-jD_{i,\overline S}\})$ is klt and globally $\bigplus$-regular for $i \gg 0$. Indeed, this follows from \autoref{proposition:pullback-of-global-splinter} as $(S,B_S + \varepsilon G)$ is globally $\bigplus$-regular for every effective divisor $G$ and $0 < \varepsilon \ll 1$ by assumptions.
Therefore,
\begin{align*}
\myB^0(S', \{B_{S'} - jD_{i,S'}\}; \sO_{S'}(\lceil jD_{i,S'}+A_{S'}\rceil)) &=  \myB^0(\overline{S}, B_{\overline{S}} + \{- jD_{i,\overline{S}}\}; \sO_{\overline{S}}(\lceil jD_{i,\overline{S}}\rceil)) \\
&=  \myB^0(\overline{S}, B_{\overline{S}} + \{- jD_{i,\overline{S}}\};  \sO_{\overline{S}}(jD_{\overline{S}})) \\
&=H^0(\overline{S}, \sO_{\overline{S}}(jD_{\overline{S}})) \\
&=H^0(S', \sO_{S'}(\lceil jD_{i,S'}+A_{S'}\rceil)),
\end{align*}
where the first equality follows by \autoref{lemma-B0-under-pullbacks-fancy} and \autoref{proposition:flips-descend-of-b-divisors},
 the third one by the global $\bigplus$-regularity of $(\overline{S},B_{\overline{S}} + \{- jD_{i,\overline{S}}\})$ (see \autoref{lem.B0EqualsH0ForGlobally+Regular}), {and the fourth one by \autoref{lemma-B0-under-pullbacks-fancy} again}. This concludes the proof of the first part of the proposition.

Since $H^0(\overline{S}, \sO_{\overline{S}}(jD_{\overline{S}})) = H^0(S', \sO_{S'}(\lceil jD_{i,S'}+A_{S'}\rceil))$ by \autoref{lemma-B0-under-pullbacks-fancy}, we get that 
\[
\Mob(jD_{\overline{S}}) = g_*\Mob \lceil jD_{i,S'} + A_{S'} \rceil \leq jD_{j,\overline S}.
\]
As $jD_{\overline{S}}$ is base point free, we thus get $jD_{\overline S} \leq  jD_{j,\overline S}$. By \autoref{lem:limit-of-b-divisors}, the other inequality holds true, too, hence $jD_{j,\overline{S}} = jD_{\overline{S}}$. 
\end{proof} 
An important difficulty in the proof of the above result is that \emph{a priori} $\{-jD_{i,S'}\} \ll 1$ and $\lceil jD_{i,S'}\rceil = jD_{S'}$ need not hold for $i \gg 0$ (because $S'$ depends on $i$). 

{
\begin{proposition}\label{proposition:restricted-algebra-fin-gen} With notation as above,  the restricted algebra
\[
    \sectionRingR_S = \bigoplus_{i \in \bN} \im \big(H^0(X, \sO_X(\lfloor i(K_X+S+B)\rfloor)) \to H^0(S, \sO_S(\lfloor i(K_S+B_S)\rfloor ))\big) 
\]
is finitely generated. \end{proposition}}
\begin{proof}
 The proof proceeds as in characteristic zero and is based purely on \autoref{proposition:flips-restriction-of-key-identity} (see \cite[Chapter 2]{CortiFlipsFor3FoldsAnd4Folds}). For the convenience of the reader, we provide a slightly different argument that avoids a direct use of b-divisors.

First, it is enough to show that any Veronese subalgebra of $\sectionRingR_S$ is finitely generated (cf.\ \cite[Lemma 2.3.3]{CortiFlipsFor3FoldsAnd4Folds}). We will show that $\sectionRingR_S^{(j)}$ is finitely generated for $j>0$ as in \autoref{proposition:flips-restriction-of-key-identity}, that is, satisfying that $j(K_X+S+B)$ is a Cartier divisor and $jD_{\overline S}$ is a Cartier base point free divisor. 

\begin{claim} \label{claim:proof-of-pl-flips} For every $i>0$ divisible by $j$ and resolution $\pi \colon Y \to X$ compatible with $i$, the following map is surjective
\[
	\begin{array}{rl}
	& H^0(X,\sO_X(i(K_X+S+B))) = H^0(Y, \sO_Y(\lceil iD_i + A' \rceil)) \\
	\to & H^0(S',\sO_{S'}(\lceil iD_{i,S'} + A_{S'} \rceil)) = H^0(\overline{S}, \sO_{\overline{S}}(iD_{i,\overline S})).\\
	\end{array}
\]
\end{claim}
Assuming the claim, we finish the proof. Using \autoref{proposition:flips-restriction-of-key-identity} we have that $iD_{i,\overline S}=iD_{\overline S}$ and so $\sectionRingR^{(j)}_S$ is equal to
\[
\bigoplus_{j \mid i} H^0(\overline{S}, \sO_{\overline{S}}(iD_{\overline{S}})) \subseteq \bigoplus_{j \mid i} H^0(\overline{S}, \sO_{\overline{S}}(i(K_{\overline{S}}+B_{\overline{S}}))) = \bigoplus_{j \mid i} H^0(S, \sO_S(i(K_S+B_S))).
\]
Since $D_{\overline{S}}$ is semiample (\autoref{proposition:flips-b-divisor-is-rational}), $\sectionRingR^{(j)}_S$ is finitely generated. 
{
\begin{proof}[Proof of \autoref{claim:proof-of-pl-flips}]
The proof is completely  analogous to that of \autoref{proposition:flips-restriction-of-key-identity}. Note that $iD_i$ and $iD_{i,S'}=iD_{S'}$ are integral, and so it is not necessary to assume that $i\gg 0$. Moreover, the first equality in the statement of the claim holds by \autoref{lemma:flips-key-identity} for $i=j$, while the second identity is a consequence of \autoref{lemma:pushforward} as $\lceil A_{S'} \rceil \geq 0$.

Recall that $(S,B_S)$ is globally $\bigplus$-regular, and so is $(\overline S, B_{\overline S})$ by \autoref{proposition:pullback-of-global-splinter}. Therefore,
\begin{align*}
\myB^0(S', \{B_{S'}\}; \sO_{S'}(\lceil iD_{i,S'}+A_{S'}\rceil)) &= \myB^0(\overline{S}, B_{\overline{S}};  \sO_{\overline{S}}(iD_{i,\overline{S}})) \\
&=H^0(\overline{S}, \sO_{\overline{S}}(iD_{i,\overline{S}})) \\
&=H^0(S', \sO_{S'}(\lceil iD_{i,S'}+A_{S'}\rceil)),
\end{align*}
where the first equality is a very special case of \autoref{lemma-B0-under-pullbacks-fancy},
 the second one follows by the global $\bigplus$-regularity of $(\overline{S},B_{\overline{S}})$ (see \autoref{lem.B0EqualsH0ForGlobally+Regular}), and the third one by \autoref{lemma:pushforward}. 

By \autoref{lemma:technical-lifting-sections-in-flips}, we have a surjection
\[
\myB^0_{S'}(Y, S'+\{B'\}; \sO_{S'}(\lceil iD_i + A' \rceil)) \to \myB^0(S', \{B_{S'}\}; \sO_{S'}(\lceil iD_{i,S'}+A_{S'}\rceil))=H^0(S', \lceil iD_{i,S'}+A_{S'}\rceil)
\]
which  concludes the proof of the claim.
\end{proof}}
The claim completes the proof.
\end{proof}

\subsection{Conclusion} \label{subsection:flips-conclusion}
In this subsection we conclude the proof of the existence of flips. For the sake of precision, we abandon the notions introduced in Subsection~\ref{subsection:flips-key-of-the-argument}, but we keep \autoref{notation:flips_base_ring} introduced at the beginning of \autoref{Section:flips}.
That is our base is a complete Noetherian local domain $(R,\fram)$ with residue field $R/\fram$ of characteristic $p>0$, and  $Z=\Spec(R)$.

\begin{theorem}\label{theorem:flips-exist} Let $f \colon X \to Z$ be a three-dimensional pl-flipping contraction of a plt pair $(X,S+B)$ with $\bQ$-boundary over the affine scheme $Z=\Spec R$ with $S=\lfloor S + B \rfloor$  an irreducible, normal, $\bQ$-Cartier divisor. Suppose that $R/\fram$ is infinite, $K_X+S+B \sim_{Z,\bQ} bS$ for some $b \in \Q$, and that $(S,B_S +\varepsilon D)$ is globally $\bigplus$-regular for every effective divisor $D$ and $0< \varepsilon \ll 1$, where $K_S+B_S = (K_X+S+B)|_S$.
Then the canonical ring

\[
    \sectionRingR(X,K_X+ {S+B})=\bigoplus_{m \in \bN} H^0(X, \sO_X(\lfloor m(K_X + {S+B}) \rfloor))
\]
is finitely generated. In particular, the pl-flip of $(X,S+B)$ over $Z$ exists. \end{theorem} 
\begin{proof}
This is a consequence of $\sectionRingR_S$ being finitely generated by \autoref{proposition:restricted-algebra-fin-gen} (note that the assumptions of \autoref{notation:flips_general} are satisfied). Explicitly, we follow the explanation from \cite[Lemma  2.3.6]{CortiFlipsFor3FoldsAnd4Folds}. Consider a divisor $G \sim S$ which does not contain $S$ in its support. Let $k,l \in \bN$ be such that $k(K_X+S+B) \sim lS$. It is enough to show that the Veronese subalgebra $\sectionRingR^{(k)}(X,K_X+{S+B})$ is finitely generated, and so that $\sectionRingR^{(l)}(X,S)$ is finitely generated. Finally, this reduces to showing that  $\sectionRingR := \sectionRingR(X,G)$ is finitely generated. From \autoref{proposition:restricted-algebra-fin-gen} we can deduce, following a similar argument to that above, that
\[
\sectionRingR^0 = \mathrm{image}(\sectionRingR(X,G) \to \bigoplus_{i \in \bN} K(S)) 
\]
is finitely generated, where $K(S)$ is the fraction field of $S$. Here, the map is induced by the restriction $\sO_X(iG) \to K(S)$.

Let $K(X)$ be the fraction field of $X$ and choose $t \in K(X)$ such that $\mathrm{div}(t) + G = S$. By definition $t\in \sectionRingR_1$. We claim that the kernel of the above map $\sectionRingR(X,G) \to \sectionRingR(S,G|_S)$ is the principal ideal generated by $t$ which concludes the proof. Indeed, then $\sectionRingR(X,G)$ is generated by $t$ and any homogeneous lifts of the homogeneous generators of $\sectionRingR^0$.
To show the claim suppose that the image of $\phi \in \sectionRingR_n$ is equal  to $0 \in \sectionRingR^0$. Then $\mathrm{div}(\phi) + nG - S \geq 0$. Hence, we can write $\phi = t\phi'$, where $\mathrm{div}(\phi') + (n-1)G  \geq 0$. In particular, $\phi' \in \sectionRingR_{n-1}$, and $\phi \in (t)\sectionRingR$.
\end{proof}

\begin{corollary} \label{cor:flips-exist2}
Let $f \colon X \to Z$ be a pl-flipping contraction of a three-dimensional plt pair $(X,S+B)$ over the affine scheme $Z=\Spec R$ where $S=\lfloor S + B \rfloor$ is a $\bQ$-Cartier irreducible divisor, $B$ has standard coefficients and $p>5$. Suppose that $K_X+S+B \sim_{Z,\bQ} bS$ for some $b \in \Q$. Then the pl-flip of $(X,S+B)$ over $Z$ exists.
\end{corollary}
\begin{proof}
This follows from \autoref{theorem:flips-exist} and \autoref{cor.3dim-plt-Fano-are-purely-plus-regular}, except that the former result assumes that $R/\fram$ is infinite. However, one can reduce the statement to the case when  $R/\fram$ is infinite by  applying the base change to the completion of the strict henselization of $R$, see \autoref{lem:plt_dlt_base_change}. This works first because of the statement of \autoref{lem:plt_dlt_base_change}, and second  because this is a faithfully-flat base-change, so $\sectionRingR(X, K_X +S +B)$ is finitely generated if and only if $\sectionRingR(X', K_{X'} + S' + B')$ is finitely generated, where $X'$, $S'$ and $B'$ are the base-changes of $X$, $S$ and $B$, respectively. Moreover, $S'$ is irreducible; indeed, 
since it is anti-ample over the base change $Z'$ of $Z$, it must contain the exceptional locus of $X' \to Z'$ which is the fiber over $\fram'$ and is necessarily connected. As $S'$ is a disjoint union of its irreducible components by \autoref{lem:properties-of-plt}, each of which must intersect the fiber over $\fram'$, $S'$ can only have a single irreducible component.
\end{proof}

{As it will be needed for running a non-$\bQ$-factorial MMP, we also prove the following proposition inspired by \cite{HW19a}. It shows the existence of ``one-complemented'' flips for arbitrary residual characteristics even when the coefficients are not standard.  It is called one-complemented because the divisor $A$ in the boundary has coefficient $1$.
\begin{proposition}\label{proposition:one-complemented-pl-flips-exist}
Let $f \colon X \to Z$ be a small projective birational contraction of a three-dimensional dlt pair $(X,S+A+B)$ over the affine scheme $Z=\Spec R$ such that $S$ and $A$ are {effective} $\bQ$-Cartier Weil divisors, $S$ is irreducible, and $B$ is an effective $\bQ$-divisor satisfying $\lfloor B \rfloor =0$. Assume that $-(K_X+S+A+B)$, $-S$, and $A$ are $f$-ample. Further, suppose that $K_X+S+A+B \sim_{Z,\bQ} bS \sim_{Z,\bQ} cA$ for some $b,c \in \Q$. Then the canonical ring 
\[
    \sectionRingR(X,K_X+S+A+B)=\bigoplus_{m \in \bN} H^0(X, \sO_X(\lfloor m(K_X + S+A+B) \rfloor))
\]
is finitely generated.
\end{proposition}
\begin{proof}
By the same argument as in \autoref{cor:flips-exist2} (applying  \autoref{lem:plt_dlt_base_change} and \autoref{lem:properties-of-plt} to $(X,S+B)$) we can assume that $R/\fram$ is infinite. Further, since $K_X+S+B$ is $f$-anti-ample and $\bQ$-linearly equivalent to a multiple of $K_X+S+A+B$, it is enough to show that $\sectionRingR(X, K_X +S +B)$ is finitely generated.

Write $K_{\tilde S} + A_{\tilde S} + B_{\tilde S} = (K_X+S+A+B)|_{\tilde S}$, where $A_{\tilde S} = A|_{\tilde S}$ and $\tilde S$ is the normalization of $S$. By adjunction, $(\tilde S, A_{\tilde S} + B_{\tilde S})$ is dlt. Since $\tilde S$ is $\bQ$-factorial, we may perturb $A_{\tilde S}$ a bit, to a $\bQ$-divisor $A'_{\tilde S}$ such that $C = \lfloor A'_{\tilde S} \rfloor$ is a prime divisor which is not contracted, $(\tilde S, A'_{\tilde S} + B_{\tilde S})$ is plt, and $-(K_{\tilde S} + A'_{\tilde S} + B_{\tilde S})$ is ample.  By \autoref{lem:one-complemented-is-plus-regular}, $(\tilde S, A'_{\tilde S} + B_{\tilde S}+\varepsilon D)$ is purely globally $\bigplus$-regular for every effective Cartier divisor $D$ with no common component with $C$ and $0 < \varepsilon \ll 1$. Hence the log Fano pair $(\tilde S, B_{\tilde S} + \varepsilon D)$ is globally $\bigplus$-regular for \emph{every} Cartier divisor $D$ and $0 < \varepsilon \ll 1$; in particular, $S$ is normal by \autoref{cor.NormalityOfS}. 
Therefore, $\sectionRingR(X, K_X +S +B)$ is finitely generated by \autoref{theorem:flips-exist}. \qedhere

\end{proof}
\begin{lemma}[{cf.\ \cite[Lemma 4.1]{HaconWitaszekMMP4fold}}]  \label{lem:one-complemented-is-plus-regular} Let $(S,C+B)$ be a two-dimensional plt pair admitting a projective birational {(onto its image)} morphism $f \colon S \to \Spec R$ such that $C$ is not contracted and $-(K_S+C+B)$ is $f$-ample. Then $(S,C+B)$ is purely globally $\bigplus$-regular.
\end{lemma}
\begin{proof}
    {We replace $\Spec R$ by the normalization of the image of $S$.} To show that $(S, C + B)$ is purely globally $\bigplus$-regular, it suffices to apply \autoref{cor.inversion-of-B-adjunction} and the following claim (here $K_C+B_C = (K_S+C+B)|_S$).
\begin{claim}
	The pair $(C,B_C)$ is globally $\bigplus$-regular.
\end{claim}
\begin{proof}
    If there was no pair, this would just be the direct summand theorem for 1-dimensional rings\footnote{This just uses that if $C \subseteq D$ is a finite extension, then $D$ is finite flat and hence $C \subseteq D$ splits.}.  {In general}, we pass to a finite cover to remove the boundary $B_C$.  
	Note $C$ is affine, one-dimensional, normal and hence regular, and $\lfloor B_C \rfloor = 0$.  It suffices to show that for any finite cover $\kappa: C' \to C$ (with $C'$ integral and $\kappa^* B_C$ integral), $\sO_C \to \kappa_*\sO_{C'}(\kappa^* B_C)$ splits.  Since $C$ is affine, this may be checked at the stalk of a closed point $Q$ of $C$.  Thus consider a DVR $V = \sO_{C, Q}$ with uniformizer $v$ and $B_C|_{\Spec V} = {a \over b} \Div(v)$ with $a < b$ {coprime} integers.  Form the extension $V' = V[v^{1/b}]$.  The map $V \to V'$ sending $1 \mapsto v^{a/b}$ splits by construction.  Since $V'$ is also regular, any further finite extension $V' \subseteq W$ is split.
		Hence the map $V \to W$ sending $1 \mapsto v^{a/b}$ splits.  This shows that $\sO_{C,Q} \to (\kappa_*\sO_{C'}(\kappa^* B_C))_Q$ splits and proves the claim.
\end{proof}
The claim completes the proof.\end{proof}}

%% file: mmp.tex
\section{Minimal Model Program}
\label{section:MinimalModelProgram}

We develop the Minimal Model Program for arithmetic threefolds.

\begin{setting}\label{MMP_setting}

{In this section we work over a base scheme $T$ which (for us) is always quasi-projective over a finite dimensional excellent ring $R$ admitting a dualizing complex.}
Note that this includes the cases where $T$ is purely of zero or positive characteristic. 

{Throughout this section, the dualizing complex on $R$ is fixed. This in turn defines a unique dualizing complex, and so a canonical sheaf, on all schemes which are quasi-projective (or constructed therefrom by ways of localisation or completion) over $R$.}

Whenever we use the word \emph{curve}, it will implicitly mean \emph{curve over $T$}, that is a one dimensional scheme which is proper over a closed point of $T$.  Recall that {curves} can be of codimension one even when $X$ is of dimension three (cf.\ \autoref{remark:divisors-of-unexpected-dimension}).

Unless otherwise stated, a field $k$ will refer to the residue field of $T$ at a suitable closed point.
Furthermore, in this section, all boundary divisors $\Delta$ will be $\mathbb{R}$-divisors, unless otherwise stated. Notions such as semiampleness or nefness are assumed to be relative, typically over the base $T$.
\end{setting}
Recall from \autoref{sec:surface_mmp}, that the key examples of $T$ include quasi-projective schemes over Dedekind domains or spectra of complete Noetherian local domains.\\

{The argument has the following steps:}

\begin{itemize}[leftmargin=1.5cm]
    \item[\textbf{Step 1}] We prove the cone theorem and the existence of pl-contractions in the pseudo-effective case.
    \item[\textbf{Step 2}] We construct flips with arbitrary coefficients in the $\bQ$-factorial setting using the existence of pl-flips with standard coefficients proven in the previous section.
    \item[\textbf{Step 3}] We prove the base point free theorem for nef and big line bundles using the existence of ``one-complemented'' pl-flips (\autoref{proposition:one-complemented-pl-flips-exist}).
    \item[\textbf{Step 4}] We show the termination of any sequence of flips when $K_X+\Delta$ is pseudo-effective using \cite{AHK07}, and conclude the proof of the MMP in this case. 
    \item[\textbf{Step 5}] {We prove the base point free theorem in its most general form, for non-big line bundles.}
    \item[\textbf{Step 6}] {We show the full cone theorem, and deduce termination with scaling and the existence of Mori fiber spaces when $K_X+\Delta$ is not pseudo-effective.}
\end{itemize}
Steps 2 and Step 3 are independent: Step 2 is based on \autoref{cor:flips-exist2}.  {It requires the assumptions of $\bQ$-factoriality and characteristics different than 2,3 or 5, but otherwise has no special requirements.}

On the other hand, in Step 3 we need to run a non-$\bQ$-factorial MMP in the case relative to a birational morphism \cite{Kollar2020RelativeMMPWithoutQfactoriality}. This means that  we cannot apply \autoref{cor:flips-exist2} directly, because it assumes that the coefficients are standard, and we cannot apply the existence of flips with arbitrary coefficients obtained in Step 2 either due to the $\bQ$-factoriality restrictions. Since we work with a special  MMP relative to a birational morphism, the flipping contractions occurring in this MMP are ``one-complemented'', and so we can apply \autoref{proposition:one-complemented-pl-flips-exist}.  

Also, observe that the argument of \cite{AHK07} used in Step 4 works when $K_X+\Delta\sim_{\mathbb{R}} M$ for some effective $\mathbb{R}$-divisor $M$
and terminalizations exist. The former condition holds automatically when $K_X+\Delta$ is pseudo-effective, $(X,\Delta)$ is klt, 
and $X$ is not defined over a closed point of $T$ (for example, when $X$ is of mixed characteristic) by applying the non-vanishing theorem for varieties of dimension at most $2$ over the generic point of the image of $X$ in $T$ \cite[Theorem 7.2]{fujino_minimal_2012}.  To construct a terminalization we run an MMP which terminates for terminal pairs by Shokurov's argument.

\begin{remark} \label{remark:history} The cone theorem in the pseudo-effective case holds by the same arguments as in \cite{KeelBasepointFreenessForNefAndBig,DW19} while the existence of pl-contractions follows from \cite{Witaszek2020KeelsTheorem}. The argument behind Step 2 is due to \cite{Birkar16}. The base point free theorem for nef and big line bundles was proven in characteristic $p>0$ in \cite{Birkar16} and \cite{Xu15Bpf} based on Keel's theorem and the generalized MMP (\cite{HaconXuThreeDimensionalMinimalModel,Birkar16}). {The existence of log minimal models in the pseudo-effective case in positive characteristic} was proven in \cite{Birkar16}. The general version of the cone theorem, the termination with scaling, the existence of Mori fiber spaces, and the base point free theorem for nef line bundles in positive characteristic is due to \cite{BW17} (see \cite{CTX15} for partial results). The generalization of some of the above results from algebraically closed fields to arbitrary $F$-finite fields is due to \cite{DW19} {(cf.\ \cite{GNT06} for the case of perfect fields)}. {We give different proofs for most of these results in the relative situation.}
\end{remark}

\begin{remark}
In \cite{HW19a} it is proven that the Minimal Model Program is valid over three-dimensional singularities and in semi-stable families in all
characteristics $p>0$.  In the process of showing the base point free theorem, we generalize the former result to mixed characteristic (\autoref{lem:non-q-factorial-mmp}), and the latter should go through with almost no modifications. Similarly, most of our results  can be extended to include $p=5$ in the general case using the arguments of \cite{HaconWitaszekMMPp=5} as has been verified in \cite{xie_xue}. \end{remark}

{
\begin{remark}
The only place in this section where the theory of $\bR$-divisors is used in an essential way is the proof of the non-$\bQ$-factorial MMP (\autoref{lem:non-q-factorial-mmp}) which in turn is employed to show the base point free theorem in the big case (\autoref{corollary:bpf-theorem-no-standard-coefficients}). In particular, readers interested in the case of $\bQ$-boundaries only, may assume in the remaining steps that all the boundaries are $\bQ$-divisors (in \autoref{thm:finiteness_of_models} which is used to prove termination with scaling, \autoref{thm:termination_scaling}, one should only consider points of the polytope which are rational). Note that in \cite{BW17} it was essential to consider the full power of the MMP for $\bR$-divisors as they come up as limits of $\bQ$-boundaries in an essential way. This is not the case in our arguments in Steps 5--6, as we employ a different strategy of proof. 
\end{remark}
}

Before proceeding, we recommend the reader to review \autoref{remark:divisors-of-unexpected-dimension}, \autoref{remark:divisors-of-unexpected-dimension2}, and \autoref{remark:divisors-of-unexpected-dimension3}, which discuss the unexpected behaviour of the dimension of Cartier divisors and localisation at $\mathbb{Q}$.

\subsection{Existence of flips and background on termination}

We start by stating the existence of  pl-flips in our setting, and recalling the statement of special termination.  First we tackle the case in which $X$ is a scheme of pure characteristic zero -- we must deal with the generalization from varieties to Noetherian excellent schemes.  Our argument above can be adapted to this situation, where we would use the fact that $\myB^0_{\alt}=H^0$ for a klt scheme of characteristic zero and deduce the relevant liftings from \cite{takumi}.  However, we believe it is more straightforward for the reader to follow the original argument as explained in \cite{CortiFlipsFor3FoldsAnd4Folds} which goes through verbatim, given the appropriate vanishing theorems:

\begin{proposition}\label{prop:char_zero_pl}
Suppose in addition to \autoref{MMP_setting} that $R$ is a domain with all residue characteristics being zero. Let $f\colon X\to Z$ be a three-dimensional pl-flipping contraction where {$Z$ is quasi-projective over $R$,}  $(X,S+B)$ is plt,  $S = \lfloor S + B \rfloor$ is a $\bQ$-Cartier prime divisor, $B$ is an effective $\bQ$-divisor, and $K_X+S+B$ is $\bQ$-linearly equivalent to a multiple of $S$.  Then the pl-flip of $(X,S+B)$ over $Z$ exists.
\end{proposition}
\begin{proof}
As mentioned above, this follows from the proof of \cite[Theorem 2.2.25]{CortiFlipsFor3FoldsAnd4Folds} {(it is assumed therein that $X$ is $\bQ$-factorial and $\rho(X/Z)=1$, but our weaker assumptions are sufficient)}.  There are various ingredients, which mirror the steps used in \autoref{Section:flips}, and all of which go through using the existing proofs as application of \cite[Theorem A]{takumi} in characteristic zero.  We have normality of plt centers \cite[Corollary 17.5]{KollarFlipsAndAbundance}, plt inversion of adjunction \cite[Theorem 5.50]{KollarMori} and existence of projective resolutions of singularities with ample exceptional divisors  (\autoref{proj-resolutions}).  
\end{proof}

\begin{proposition}\label{proposition:flips-exist} Let $f \colon X \to Z$ be a three-dimensional pl-flipping contraction of a plt pair $(X,S+B)$ where $S = \lfloor S + B \rfloor$ is a $\bQ$-Cartier prime divisor, $B$ is an effective $\bQ$-divisor with standard coefficients, $K_X+S+B$ is $\bQ$-linearly equivalent to a multiple of $S$, and $Z$ is a quasi-projective scheme over $R$. {Suppose that none of the residue fields of $R$ have characteristic $2$, $3$ or $5$.} Then the pl-flip of $(X,S+B)$ over $Z$ exists. \end{proposition}

\begin{proof}
By \autoref{prop:char_zero_pl} and localisation, we may assume that $Z$ is the spectrum of a local ring with positive residue characteristic 
(note that it will not be quasi-projective over $R$ any more).
We need to show that some Veronese subalgebra of the canonical ring $\sectionRingR(X,K_X+S+B) = \bigoplus_{i \in \bN} H^0(X,\sO_X(\lfloor i(K_X+S+B)\rfloor))$ is finitely generated. This is equivalent to verifying that there exists a divisible enough $j>0$ such that the multiplication map
\begin{equation} \label{eq:surjectivity-for-the-canonical-ring}
f_*\sO_X(j(K_X+S+B))^{\oplus i/j} \to f_*\sO_X(i(K_X+S+B))
\end{equation}
is surjective for every $i>0$ divisible by $j$.  Let $\big(\widehat{X}, \widehat{S}+\widehat{B} \big)$ be the completion of $(X,S+B)$ at $z = f(\mathrm{Exc}(f)) \in Z$. By \autoref{lem:plt_dlt_base_change}, $\big(\widehat{X}, \widehat{S}+\widehat{B}\big)$ is plt. Moreover, by \autoref{lem:properties-of-plt}, $\widehat{S}$ is a disjoint union of its irreducible components. Since $\widehat{S}$ is anti-ample over the completion $\widehat{Z}$ of $Z$ at $z$, it must contain the exceptional locus of $\widehat{X} \to \widehat{Z}$ which is the fiber over the closed point of $\widehat{Z}$ and is connected.  As every component of $\widehat{S}$ must also intersect this exceptional locus, this is only possible when $\widehat{S}$ is irreducible.

The condition that $K_X+S+B \sim_{Z,\bQ} -bS$, for some $b \in \bQ_{>0}$, is preserved under completion. 
Hence, \autoref{eq:surjectivity-for-the-canonical-ring} is surjective after completion by \autoref{cor:flips-exist2}, and since surjectivity of finitely generated modules can be verified after completion, the proposition follows.
\end{proof}

\begin{theorem} \label{thm:special-termination}
Let $(X,\Delta)$ be a three-dimensional $\bQ$-factorial dlt pair with $\bR$-boundary which is projective over $T$,  and let 
\[
(X,\Delta) \dashrightarrow (X_1,\Delta_1) \dashrightarrow (X_2,\Delta_2) \dashrightarrow \cdots
\]
be a sequence of $(K_X+\Delta)$-flips and divisorial contractions over $T$. Then after finitely many steps all the maps are flips and the flipped and flipping loci are disjoint from $\lfloor \Delta_i \rfloor$. 
\end{theorem}
\begin{proof}
Since divisorial contractions decrease the Picard rank (cf.\ \autoref{remark:relative-Picard-rank}), we can assume that the above sequence consists only of flips.  

The result then follows by the same argument as in \cite[Theorem 4.2.1]{fujino05}. The proof employs the two dimensional MMP (\autoref{thm:surface-excellent-mmp}). Implicitly, this reference assumes the normality of the irreducible components of $\lfloor \Delta \rfloor$, but what is only needed is normality up to a universal homeomorphism (see \cite{HW19a}) which follows from \autoref{lem:properties-of-plt}.  We point out that the irreducible components $Y \subseteq X_i$ of the  flipping and flipped loci cannot be contained in the prime divisors $D \subseteq  \Supp \Delta_i$ satisfying $\dim D = 1$ (otherwise, $D=Y$, and so $Y$ would be a divisor). Similarly, the flipped contraction is small (\cite[Lemma 6.2]{KollarMori}), thus the flipped locus must also have codimension at least $2$ and therefore it cannot contain a divisor of dimension $1$.  Thus, no new phenomena show up and the proof is really exactly as in \cite[Theorem 4.2.1]{fujino05}.
\end{proof}

\begin{theorem} \label{thm:special-termination-2}
Let $(X,\Delta)$ be a three-dimensional $\bQ$-factorial dlt pair with $\bR$-boundary which is projective over $T$, and suppose that all three-dimensional $\bQ$-factorial klt pairs projective over $T$ and with  underlying scheme birational to $X$ admit terminalizations.  Let
\[
(X,\Delta) \dashrightarrow (X_1,\Delta_1) \dashrightarrow (X_2,\Delta_2) \dashrightarrow \cdots
\]
be a sequence of $(K_X+\Delta)$-flips and divisorial contractions over $T$. Then after finitely many steps all the flipped and flipping loci in the above sequence are disjoint from $\Supp \Delta_i$.
\end{theorem}

\begin{proof}
Suppose by contradiction that there exists an infinite sequence of flips $(X,\Delta) \dashrightarrow (X_1,\Delta_1) \dashrightarrow (X_2,\Delta_2) \dashrightarrow \cdots$ for which the statement fails. By \autoref{thm:special-termination}, we can assume that the flipping loci are disjoint from $\lfloor \Delta_i \rfloor$. Hence, by decreasing the coefficients of $\Delta$, we can assume that $(X,\Delta)$ is klt; the sequence $X \dashrightarrow X_1 \dashrightarrow \cdots$ is still a $(K_X+\Delta)$-MMP as all the flipping loci are disjoint from the divisors whose coefficients were decreased.   

Now the proof follows from the argument of Alexeev-Hacon-Kawamata (\cite[Proposition 2.10]{HaconWitaszekMMP4fold},  \cite{AHK07}); although the statement assumes that the schemes are defined over a field and the boundaries are $\bQ$-divisors, it is valid in our setting as well (in particular, the proof of \cite[Lemma 2.11]{HaconWitaszekMMP4fold} goes through for arbitrary Noetherian excellent surfaces). Note that \cite[Proposition 2.10]{HaconWitaszekMMP4fold} requires the existence of terminalizations (which is assumed in \autoref{thm:special-termination-2}), and the existence of proper resolutions of singularities (\autoref{thm:proper-resolutions}). Finally, we point out that, as explained in the proof of \autoref{thm:special-termination}, the divisors $D \subseteq \Supp \Delta_i$ satisfying $\dim D=1$ do not cause any problems.
\end{proof}

\subsection{Step 1: Partial cone and contraction theorems}

In what follows, given a ray $\Sigma$ and a $\bQ$-Cartier divisor $D$, we shall write, by abuse of notation, that $\Sigma \cdot D >0$ when $\Sigma$ is $D$-positive (and analogously for $\Sigma \cdot D = 0$ and $\Sigma \cdot D <0$), although the number $D \cdot \Sigma$ is not well defined.

\begin{theorem} \label{thm:keel_cone}
Let $(X,\Delta)$ be a normal $\mathbb{Q}$-factorial three-dimensional pair with $\bR$-boundary and coefficients in $[0,1]$, which is projective over $T$. If $K_X+\Delta\equiv_T M$ for some effective $\mathbb{R}$-Cartier divisor $M$,
then there exists a countable set of curves over $T$, denoted $\{C_i\}$, such that
\begin{enumerate}
\item
\label{itm:keel_cone:sum}
\[
\overline{\mathrm{NE}}(X/T) = \overline{\mathrm{NE}}(X/T)_{K_X+\Delta\geq 0} + \sum_i \bR_{\geq 0}[C_i].
\]
\item 
\label{itm:keel_cone:accumulation}
The rays $[C_i]$ do not accumulate in the half space $(K_X+\Delta)_{<0}$, and

    \item 
    \label{itm:keel_cone:bound}
    For all but finitely many $i$, $$0<-(K_X+\Delta)\cdot_k {C_i}\leq 4 d_{C_i}$$ where $k$ is the residue field of the closed point on $T$ which is the image of $C_i$, $d_{C_i}$ is the constant from \autoref{lem:d_C} such that if $L$ is any Cartier divisor on $X$, then $L\cdot_k C_i$ is divisible by $d_{C_i}$.
\end{enumerate}

\end{theorem}

\begin{remark} \label{remark:non_vanishing_in_the_relative_setting}
Note that the condition $K_X+\Delta\equiv_TM\geq 0$ is automatic whenever $K_X+\Delta$ is pseudo-effective and the image of $X$ in $T$ is at least one-dimensional, by non-vanishing applied to the generic fiber of $X\to T$ (see  \cite{fujino_minimal_2012} and \cite{TanakaAbundanceImperfectFields}).
In particular, {the latter condition holds when}  $X$ has mixed characteristic.

\end{remark}

\begin{proof}

Let $\Sigma$ be a $(K_X+\Delta)$-negative extremal ray.  
Choose an irreducible component $E$ of ${\Supp}\, M$ which is negative on $\Sigma$. If $\dim E=1$, then $\Sigma = \mathbb{R}E$. Since there are only finitely many irreducible components of $M$, we may assume that $E$ is among the set of curves $\{C_i\}$. Thus, we are henceforth free to assume that $\dim E = 2$.

We first aim to show that $\Sigma$ contains a curve satisfying the required bound. 

\begin{claim}\label{claim_NE}
$\Sigma$ is in the image of $\overline{NE}(\widetilde{E})\to\overline{NE}(X)$ where $\widetilde{E}$ is the normalization of $E$.
\end{claim}
\begin{proof}[Proof of claim]
Fix an ample $\mathbb{Q}$-divisor $H$ sufficiently small that $\Sigma$ is also $(K_X+\Delta+H)$-negative. 
Fix a non-zero cycle $\Gamma$ in $\Sigma$, and write $\Gamma$ as a limit of effective cycles: $\Gamma=\lim_{j}\Gamma_{j}$.  {Further,} write $\Gamma_j=\sum_i a_{i,j} C_{i}+\sum_i b_{i,j} D_{i}$ where {$C_{i}\cdot E<0$ and $D_i\cdot E\geq 0$} for each $i$.  Letting $A$ by an ample Cartier divisor, and after replacing by a subsequence,
we may assume that \[\sum_ia_{i,j}+\sum_ib_{i,j}\leq \sum_i a_{i,j} C_{i}\cdot A+\sum_i b_{i,j} D_{i}\cdot A=\Gamma_j\cdot A<\Gamma\cdot A+1\]
{for some fixed ample Cartier divisor $A$}.  This shows that the $a_{i,j}$ and $b_{i,j}$ are all bounded independently of $i$ and $j$.
 Let $a_E$ be such that $\Delta+a_EM$ has coefficient $1$ in $E$.  Then by \autoref{thm:surface_cone}\autoref{itm:surface_cone_ample} and adjunction   of $K_X+\Delta+a_EM+H$ to the normalization $\widetilde{E}$ of $E$, $C_i$ may be taken to be from finitely many extremal rays on $E$.  It follows that we may take all the $C_i$ to come from a fixed finite set, and so after replacing by a subsequence, $a_i=\lim_j a_{i,j}$ is a well defined non-negative number.

{
It follows that  $\lim_j(\sum_i a_{i,j} C_{i})$ is a well defined pseudo-effective $1$-cycle, {and it is non-zero} since it intersects negatively with $E$.  As a result, $\lim_j(\sum_i b_{i,j}D_i)$ exists as a class in $N_1(X)$ as it is the difference of $\Gamma$ and a converging sequence.  Then as $\Gamma = \lim_j(\sum_i a_{i,j} C_{i})+\lim_j(\sum_ib_{i,j}D_i)$
is a decomposition into a sum of pseudo-effective cycles, we must have that $\lim_j(\sum_i a_{i,j} C_{i})$ is in {$\Sigma$} by extremality.  Then the fact that $C_i\cdot E<0$ for each $i$ means that each $C_i$ is contained in $\Supp(E)$ 
and so $\Sigma$ is contained in the image of $\overline{NE}(\widetilde{E})\to\overline{NE}(X)$, and the claim is proved.}
\end{proof}

Returning to the proof of the Cone Theorem, by adjunction 
there is an effective divisor $\Delta_{\widetilde{E}}$ on $\widetilde{E}$ satisfying $(K_X+\Delta+a_EM)|_{\widetilde{E}}=K_{\widetilde{E}}+\Delta_{\widetilde{E}}$, where $a_E$ is such that $\Delta + a_EM$ has coefficient $1$ in $E$.
  Thus $\Sigma$ is in the image of some $(K_{\widetilde{E}}+\Delta_{\widetilde{E}})$-negative extremal ray $\Sigma_{\widetilde{E}}$ via the map $\overline{NE}\big(\widetilde{E}\big)\to\overline{NE}(X)$.
By \autoref{thm:surface_cone}, any $(K_{\widetilde{E}}+\Delta_{\widetilde{E}})$-negative extremal ray either contains a curve satisfying the required bound or a curve in $\Supp(\Delta_{\widetilde{E}})$.  Note that there are only finitely many possibilities for the latter curves independently of the choice of $E$ as they lie in $\Sing(\Supp(\Delta+M))\cup\Sing(X)$.  We have proved that every $(K_X+\Delta)$-negative extremal ray $\Gamma$ contains a curve $C$ such that $C$ either satisfies the bound in \autoref{itm:keel_cone:bound} or is an element of a fixed finite set of curves.

Next we show that the extremal rays do not accumulate in $\overline{NE}(X/T)_{K_X+\Delta<0}$.  Suppose otherwise, so we have a sequence of distinct $(K_X+\Delta)$-negative extremal rays $\Sigma_i$ which converge to a $(K_X+\Delta)$-negative ray $\Sigma$.  Fix a component $E$ of $M$ which is negative on $\Sigma$.  By passing to a subsequence we may assume that $E$ is also negative on $\Sigma_i$ for all $i$, and so by \autoref{claim_NE}, $\Sigma$ and $\Sigma_i$ are all in the image of $\iota_*:\overline{NE}\big(\widetilde{E}\big)\to\overline{NE}(X)$ where $\tilde{E}$ is the normalization of $E$.
For each $i$, choose a $(K_{\widetilde{E}}+\Delta_{\widetilde{E}})$-negative extremal ray $\Sigma_i^E$ such that $\iota_*\Sigma_i^E=\Sigma_i$ where $\iota_*:\overline{NE}\big(\widetilde{E}\big)\to\overline{NE}(X)$.  By \autoref{thm:surface_cone}, the rays $\Sigma_i^E$ do not accumulate to a  $(K_{\widetilde{E}}+\Delta_{\widetilde{E}})$-negative ray.  But by compactness of $\overline{NE}(\widetilde{E})$ intersected with the unit ball,
by again taking a subsequence we may assume that $\Sigma_i^E$ do converge in $\overline{NE}\big(\widetilde{E}\big)$, and so converge to a ray $\Sigma^E$ satisfying 
\[
    0\leq (K_{\widetilde E}+\Delta_{\widetilde E})\cdot\Sigma^E= (K_X+\Delta+a_EM)\cdot \iota_*\Sigma^E \leq (K_X+\Delta)\cdot \iota_*\Sigma^E.
\]
This shows that the rays $\Sigma_i$ could not converge to a $(K_X+\Delta)$-negative ray.  This concludes the proof of \autoref{itm:keel_cone:accumulation}.

It remains to prove the countability of the set of curves in  \autoref{itm:keel_cone:sum}. Fix an ample divisor $H$.  For each $n\in \mathbb{N}$, the previous paragraph implies that there are only finitely many $(K_X+\Delta+\frac{1}{n}H)$-negative extremal rays.  Then there can be only countably many $(K_X+\Delta)$-negative rays because each is $(K_X+\Delta+\frac{1}{n}H)$-negative for some $n$. 
\end{proof}

\begin{proposition}[{cf.\ \cite[Proposition 4.4]{HaconWitaszekMMP4fold}}] \label{prop:partial-contraction-theorem} Let $(X,S+B)$ be a $\mathbb{Q}$-factorial three-dimensional projective dlt pair over $T$, where $S$ is a prime divisor and $B$ is an effective $\bR$-divisor. Suppose that $K_X+S+B$ is pseudo-effective over $T$.
Let $\Sigma$ be a $(K_X+S+B)$-negative extremal ray over $T$ such that $\Sigma$ is $S$-negative.  Then the contraction $f\colon X\to Z$ of $\Sigma$ exists so that $f$ is a projective morphism with $\rho(X/Z)=1$. 
\end{proposition}

\begin{proof}
First, we reduce to the plt case with $\bQ$-boundary. By graded prime avoidance (see \cite[Tag 00JS]{stacks-project}), we may pick an ample Cartier divisor $A$ which does not contain any log canonical center of $(X,S+B)$, so that $(X,S+B+\varepsilon A)$ is dlt for $\varepsilon\ll1$ and $K_X+S+B+\varepsilon A$ is big and negative on $\Sigma$.  Now replacing $(X,S+B)$ with $(X,S+B'+\varepsilon A)$ where $B'$ is a $\mathbb{Q}$-divisor which is a small perturbation of $B$ such that $K_X+S+B'+\varepsilon A$ is still big and negative on $\Sigma$, we may assume that $B$ is a  $\mathbb{Q}$-divisor. Furthermore, by decreasing all the coeficients of $\lfloor B\rfloor$, we may assume that $(X,S+B)$ is plt.

By  \autoref{thm:keel_cone}, we may pick an ample (over $T$) $\Q$-divisor $H$ such that $L=K_X+S+B+H$ is nef and $L^\perp \subseteq \overline{\mathrm{NE}}(X/T)$ is spanned by $\Sigma$. Let $A$ be another ample $\Q$-divisor such that $(S+A)\cdot \Sigma =0$. Again, by \autoref{thm:keel_cone}, we have that $L_{\varepsilon}=K_X+S+B+H_\varepsilon$ is nef over $T$ and $(L_\varepsilon)^\perp$ is spanned by $\Sigma$ for any $0<\varepsilon\ll 1$, where $H_\varepsilon = H+ \varepsilon (S+A)$ is an ample $\bQ$-divisor. Explicitly, by \autoref{thm:keel_cone}(b), there are finitely many $(K_X+S+B+\frac{1}{2}H)$-negative extremal rays: $\Sigma, \Sigma_1, \ldots, \Sigma_l$. For every $\varepsilon$ such that $H_{\varepsilon}-\frac{1}{2}H$ is ample, $L_{\varepsilon}$ is positive on all extremal rays except possibly these $\Sigma, \Sigma_1, \ldots, \Sigma_l$. By decreasing $\varepsilon$ further we can assume that $L_{\varepsilon}$ is also positive for on $\Sigma_j$ for all $1 \leq j \leq l$. Last, $L_{\varepsilon} \cdot \Sigma =0$ holds for all $\varepsilon$.

Moreover, we have that $\mathbb E(L_\varepsilon)\subset S$. Indeed, if  $V\subset X$ is a an integral subscheme not contained in $S$, then $L_\varepsilon |_V=(L+\varepsilon (S+A))|_V$ is nef and big over $T$. Replacing $L$ by $L_\varepsilon$, we may assume that  $\mathbb E(L)\subset S$.

Now, over closed points of residual characteristic $p>0$ the proposition follows from \autoref{prop:bpf_plt}, and over closed points of residual characteristic zero from \autoref{prop:char_zero_bpf} applied to a klt perturbation of $(X,S+B)$ .
\end{proof}

\subsection{Step 2: Construction of flips with arbitrary coefficients}

We recall the {standard} argument for reducing the existence of flips to pl-flips.

\begin{proposition} \label{proposition:reduction-to-pl-flips} Let $(X,B)$ be a $\bQ$-factorial klt pair of dimension three, where $B$ is a $\bQ$-divisor with standard coefficients. Let $f \colon X \to Z$ be a flipping contraction over an affine scheme $Z = \Spec R$ such that $\rho(X/Z)=1$. Suppose that none of the residue fields of $R$ have characteristic $2$, $3$ or $5$. Then, the flip $X^+ \to Z$ of $f$ exists.
\end{proposition}
\begin{proof}
We closely follow the presentation from \cite[Theorem 4.1]{HW19a}. Fix $B_Z = f_*B$ and let $H_Z$ be a reduced Cartier divisor on $Z$ where the following hold:
\begin{enumerate}
	\item $f^*H_Z$ contains the exceptional set of $f$, 
	\item $H_Z$ and $B_Z$ have no irreducible components in common,
	\item\label{item:red_pl_condition_support} for any projective birational morphism $h \colon Y \to Z$ where $Y$ is $\bQ$-factorial, $N^1(Y/Z)$ is generated by the $h$-exceptional divisors and the irreducible components of the strict transform of $H_Z$.
\end{enumerate}
For the (non-trivial) condition \autoref{item:red_pl_condition_support}, we use that the relative group of divisors up to numerical equivalence of a birational morphism of $\bQ$-factorial varieties is generated by the exceptional divisors.
In view of this, we may pick $H_Z$ so that it satisfies \autoref{item:red_pl_condition_support} for a single resolution of singularities, and condition \autoref{item:red_pl_condition_support} is then satisfied for every larger resolution of singularities as well. The statement follows since $h$ as above is a factor of some resolution and the group of divisors on $Y$ over $Z$ is the image of the group of divisors of any projective birational cover.

Fix a log resolution $h \colon Y \xrightarrow{p} X \xrightarrow{f} Z$ of $(Z,B_Z + H_Z)$ which factors through $X$. We may assume that $H_Z$ contains the image of each $h$-exceptional divisor, and we claim that we can run a $(K_Y+B_Y+H_Y)$-MMP over $Z$ where $H_Y$ is the strict transform of $H_Z$ and $B_Y := h_*^{-1}B_Z + \mathrm{Exc}(h)$. The cone theorem is valid by \autoref{thm:keel_cone}, and note that every extremal ray $\Sigma$ over $Z$ is contained in the support of $h^*H_Z$. 
By condition \autoref{item:red_pl_condition_support}, there is a component of the support of $h^*H_Z$ having non-zero intersection with $\Sigma$. Since $\Sigma\cdot h^*H_Z=0$, there is a component  $E$ of the support of $h^*H_Z$ with $\Sigma \cdot E <0$. In particular, we have $E \subseteq \lfloor B_Y + H_Y \rfloor$. Hence, contractions exist by \autoref{prop:partial-contraction-theorem}, the necessary flips exist by \autoref{proposition:flips-exist} applied to a plt perturbation 
of $(Y,B_Y+H_Y)$, and special termination follows by \autoref{thm:special-termination}.

Now replace $(Y,B_Y+H_Y)$ by its minimal model over $Z$, and $H_Y$ by its pushforward under the map to the minimal model. While $Y$ need no longer admit a map to $X$, it still necessarily maps  to $Z$, which we denote by $h \colon Y \to Z$. 

Write $B_Y = D + B^{<1}_Y$, where $D = \sum_{i=1}^m D_i$ is the sum of exceptional divisors and $\lfloor B^{<1}_Y \rfloor = 0$.
As $H_Y$ is contained in the pullback of $H_Z$ from $Z$, we have 
\[
H_Y \equiv_h -\sum_j b_j D_j, 
\]
where $b_j \in \bQ_{\geq 0}$. Run a $(K_Y+B_Y)$-MMP over $Z$ with scaling of $H_Y$, noting that an MMP with scaling is well-defined by the existence of bounds on extremal rays from \autoref{thm:keel_cone}. Arguing as above, to show that such an MMP can be run, it suffices to show that flips and contractions exist. Let $0 < \lambda \leq 1$ be such that $K_Y+B_Y + \lambda H_Y$ is $h$-nef and there exists a $(K_Y+B_Y)$-negative extremal ray $\Sigma$ satisfying $(K_Y+B_Y+\lambda H_Y) \cdot \Sigma = 0$. Since $(K_Y+B_Y) \cdot \Sigma < 0$, we have that $H_Y \cdot \Sigma > 0$, and the equivalence above implies that $D_j \cdot \Sigma < 0$ for some $j$. It follows that  the contraction of $\Sigma$ exists by \autoref{prop:partial-contraction-theorem}, and in the case that the contraction is small the flip exists by \autoref{proposition:flips-exist}. Once again, the MMP terminates by special termination as above.

Denote by $(X^+,B^+)$ an output of this MMP, so that $K_{X^+}+B^+$ is nef over $Z$, and notice that the projection $f^+ \colon X^+ \to Z$ is small. Indeed, the negativity lemma applied to a resolution of indeterminacies $\pi_1 \colon W' \to X$ and $\pi_2 \colon W' \to X^+$ shows that
\[
    G:= \pi_1^*(K_X+B) - \pi_2^*(K_{X^+} + B^+),
\]       
{is effective (and non-zero)}
which is only possible when $\lfloor B^+ \rfloor = 0$ since $(X,B)$ is klt. As all of the exceptional divisors on $X^+$ over $Z$ are contained in $\lfloor B^+ \rfloor$, this shows that $f^+$ is small. Moreover, { $K_{X^+}+B^+$} is ample over $Z$; 
otherwise, as $\rho(X^+/Z)=1$ ({here $X$ and $X^+$ are $\bQ$-factorial and $f, f^+$ are small over $Z$, so $\rho(W/X)=\rho(W/X^+)$ is equal to the number of exceptional divisors, thus\footnote{The additivity of the Picard rank here follows again from the $\bQ$-factoriality of $X$ and $X^+$} $\rho(X^+/Z) = \rho(W/Z)- \rho(W/X^+) = \rho(W/Z)-\rho(W/X)= \rho(X/Z)=1$}), { $K_{X^+}+B^+$}  would be numerically trivial over $Z$, and so $G$ would be numerically trivial over $X$. As $G$ is exceptional and non-zero,  this contradicts the negativity lemma (over $X$). Hence, $f^+$ is the flip of $f$. 
\end{proof}

The following technique was discovered in \cite{Birkar16}; we closely follow the presentation from \cite[Proof of Theorem 1.1]{HaconWitaszekMMPp=5}. We emphasize that $\Delta$ is allowed to have arbitary $\mathbb{R}$-coefficients.

\begin{theorem} \label{thm:full_flips} If $(X,\Delta)$ is a dlt pair with $\bR$-boundary and $f \colon X \to Z$ is a three-dimensional $\bQ$-factorial flipping contraction to a quasi-projective scheme $Z$ over $R$, whose residue fields do not have characteristic $2$, $3$ or $5$, with $\rho(X/Z)=1$, then the flip of $(X,\Delta)$ exists.
\end{theorem}
\begin{proof}
We begin with a number of reductions. By perturbing $B$ and using that $X$ is $\bQ$-factorial, we may assume that {$\Delta$} is a $\bQ$-divisor. After replacing  $\Delta$ with $\Delta - \frac{1}{l}\lfloor \Delta \rfloor$ for $l\gg 0$, we can further assume that $(X,\Delta)$ is klt. Finally, we may also assume that every component of $\Supp \Delta$ is relatively antiample, as removing the ample {and numerically-trivial} components will not affect the anti-ampleness of $K_X+\Delta$.

In case $\Delta$ has standard coefficients, the theorem follows from \autoref{proposition:reduction-to-pl-flips}. In the remainder, we proceed with a proof by induction on the number $\zeta (\Delta)$   of components of $\Delta$ with coefficients outside of the standard set $\{ 1-\frac 1 m \mid m\in \mathbb N \}\cup \{ 1\}$. Assuming $\zeta (\Delta )>0$, write $\Delta =aS+B$ where $a\not \in \{ 1-\frac 1 m \mid m\in \mathbb N \}\cup \{ 1\}$.
 
 Consider a log resolution $\pi \colon W \to X$ of $(X,S+B)$ with reduced exceptional divisor $E$. Setting $B_W := \pi^{-1}_*B + E$ and $S_W := \pi^{-1}_*S$, since $K_X + \Delta \equiv_Z \mu S$ for some $\mu > 0$ and $S$ is relatively anti-ample as it is a component of $\Supp \Delta$, we have that
\[
K_W+S_W+B_W = \pi^*(K_X+\Delta) + (1-a)S_W + F
 			\equiv_Z (1-a+\mu)S_W + F',
\]
where $F$, $F'$ are effective exceptional $\Q$-divisors over $X$.

We now run a $(K_W+S_W+B_W)$-MMP over $Z$.  As $\zeta(S_W+B_W)<\zeta(\Delta)$ and by decreasing the coefficients by $\frac{1}{l}\lfloor S_W + B_W \rfloor$ for $l\gg 0$ so as to make the pair klt without affecting $\zeta$, all flips exist in this MMP. Additionally, as every extremal ray is negative on $(1-a+\mu)S_W + F'$ ({and so on an} irreducible component of $\lfloor S_W + B_W \rfloor$), all contractions in this MMP exist by \autoref{prop:partial-contraction-theorem}. The cone theorem is valid by \autoref{thm:keel_cone}, and the MMP terminates by the special termination in \autoref{thm:special-termination}. Let $h \colon W \dasharrow Y$ be an output of this MMP where $S_Y$, $B_Y$, and $F'_Y$ are the strict transforms of $S_W$, $B_W$, and $F'$ on $Y$, respectively.

Next, we run a $(K_Y + aS_Y + B_Y)$-MMP over $Z$ with scaling of $(1-a)S_Y$. 
If $\Sigma$ is an extremal ray, then $\Sigma \cdot S_Y > 0$ and 
$(K_Y+B_Y)\cdot \Sigma < 0$.
As $\zeta(B_Y) < \zeta(\Delta)$, again decrease the coefficients by $\frac{1}{l}\lfloor B_Y \rfloor$ for $l\gg 0$ to make the pair klt without affecting $\zeta$, all the flips in this MMP exist by induction. Noting  
\[
K_Y+aS_Y+B_Y \equiv_Z \mu S_Y + F'_Y,
\]
every extremal ray is negative on $\mu S_Y+F'_Y$, hence on $F'_Y$ {(as $\Sigma \cdot S_Y >0$)} and so on an irreducible component of $\lfloor B_Y \rfloor$. It follows from \autoref{prop:partial-contraction-theorem} that all contractions in this MMP exist. As in the paragraph above, the cone theorem and termination are both valid in this setting. Set $(X^+, aS^+ + B^+)$ to be an output of this MMP.

To conclude the proof, we show that $(X^+, aS^+ + B^+)$ is the flip of $(X,aS+B)$. Notice that the negativity lemma applied to a common resolution of indeterminacies $\pi_1 \colon W \to X$ and $\pi_2 \colon {W} \to X^+$ implies that
\[
G:= \pi_1^*(K_X+aS+B) - \pi_2^*(K_{X^+} + aS^+ + B^+)
\]       
is effective and non-zero. Since $(X,aS+B)$ is klt{, we get that} $\lfloor B^+ \rfloor = 0$, and so all the exceptional divisors were contracted and $X \dashrightarrow X^+$ is an isomorphism in codimension one. Moreover, since $X$ and $X^+$ are $\bQ$-factorial we have that $\rho(W/X)=\rho(W/X^+)$ is equal to the number of exceptional divisors, and it follows that  $\rho(X^+/Z)=1$ using that $\rho(W/X) + \rho(X/Z) = \rho(W/Z) =  \rho(W/X^+) + \rho(X^+/Z)$ and $\rho(X/Z)=1$. Again we must now have that $K_{X^+}+aS^++B^+$ is relatively ample over $Z$, else $K_{X^+}+aS^++B^+$ is relatively numerically trivial over $Z$ and then $G$ is exceptional and numerically trivial over $X$, contradicting the negativity lemma once more. It follows that $(X^+, aS^+ + B^+)$ is the flip of $(X, \Delta)$ as desired.
\end{proof}

\subsection{Step 3: Base point free theorem for nef and big line bundles}

The following theorem is key in our proof of the base point free theorem. Here, condition (e) may be thought of as numerical klt-ness of $(X,\pi_*\Delta)$. When $X$ is $\bQ$-factorial, then this is a mixed characteristic variant of \cite[Theorem 1.1]{HW19a} (cf.\ \cite[Theorem  4.6]{TakamatsuYoshikawaMMP}).

\begin{theorem}[{\cite[Theorem 1]{Kollar2020RelativeMMPWithoutQfactoriality}}] \label{lem:non-q-factorial-mmp}
Let $(Y,\Delta)$ be a three-dimensional dlt pair with $\bQ$-boundary and let $\pi \colon Y \to X$ be a projective birational map of quasi-projective schemes over $R$ with irreducible exceptional divisors $E_1, \ldots, E_r$. Suppose that
\begin{enumerate}[series=nonqfactorialmmp]
    \item \label{itm:non-q-factorial-mmp:exists_ample} there exists an ample exceptional $\bQ$-divisor $\Lambda$ on $Y$, 
    \item \label{itm:non-q-factorial-mmp:E_i_Q_Cartier} all $E_i$ are $\bQ$-Cartier,
    \item \label{itm:non-q-factorial-mmp:boundary_round_down} $\lfloor \Delta \rfloor = E_1 + \ldots + E_r$,
    \item \label{itm:non-q-factorial-mmp:log_canonical} $K_Y+\Delta \equiv_X \sum e_iE_i$ for $e_i \in \bQ$, and
    \item \label{itm:non-q-factorial-mmp:sub_klt} $(X, \pi_*\Delta)$ is klt, or more generally that there exists a sub-klt\footnote{satisfying the same conditions as klt but not requiring that the boundary divisor is effective.} pair $(Y,\Delta')$ such that $K_Y+\Delta' \equiv_X 0$ and $\Delta-\Delta'$ is exceptional.
\end{enumerate}
We can run a $(K_Y+\Delta)$-MMP over $X$ in the sense of \cite{Kollar2020RelativeMMPWithoutQfactoriality} and it terminates with $X$. In particular, every $\bQ$-Cartier $\bQ$-divisor $D$ on $Y$ such that $D \equiv_X 0$ satisfies $D \sim_{\bQ,X} 0$ (in other words, some multiple of $D$ descends to $X$).
\end{theorem}
\noindent 
We decompose $Y\to X$ into pl-contractions and pl-flips, such that $D$ descends under each operation.  This uses 
Koll\'ar's non-$\bQ$-factorial MMP \cite[Theorem 1]{Kollar2020RelativeMMPWithoutQfactoriality}, in which all contractions behave as if they were of Picard rank one with respect to exceptional divisors, and so their $\bQ$-Cartierness is in fact preserved. For the convenience of the reader we write down a detailed explanation below. Unless otherwise stated, the exceptionality and ampleness below is always relative to $X$.

\begin{proof}
It is enough for the last sentence of the statement to show that $\pi_*D$ is $\bQ$-Cartier, as then $D = \pi^*\pi_*D \sim_{\bQ, X} 0$ by the negativity lemma. 
By Conditions \autoref{itm:non-q-factorial-mmp:exists_ample} and \autoref{itm:non-q-factorial-mmp:E_i_Q_Cartier}, we can pick an $\bR$-divisor $H$ and $h'\in \bR_{>0}$ such that
\begin{enumerate}[resume*=nonqfactorialmmp]
    \item \label{itm:non-q-factorial-mmp:H} $H= \sum \gamma_i E_i$ where $\{\gamma_1, \ldots, \gamma_r\}$ are linearly independent over $\bQ$, and
    \item \label{itm:non-q-factorial-mmp:ample} $K_Y+\Delta+h'H$ is $\bR$-ample. 
\end{enumerate}
To this end, we may initially take $H$ to be $\bR$-ample, but we will not use that in proofs since it will not be stable under the procedure described below. 
Further, note that Condition \autoref{itm:non-q-factorial-mmp:exists_ample} implies that $\mathrm{Exc}(Y/X)$ is a divisor (and hence equal to $\lfloor \Delta \rfloor$ {by Condition  \autoref{itm:non-q-factorial-mmp:boundary_round_down}}); indeed, $-\Lambda$ is effective by the negativity lemma (\autoref{lem:negativity}), and so if there exists an irreducible component $C \subseteq \mathrm{Exc}(Y/X)$ such that $C \not \subseteq  \Supp E_1\cup \ldots \cup \Supp E_r$ (in particular, $C$ is a component of codimension at least two, and since it cannot be a point it must be a curve), then $C \cdot \Lambda \leq 0$, contradicting the ampleness of $\Lambda$.

We start by establishing the \emph{cone theorem}.
Set $(K_{Y} + \Delta)|_{\tilde E_i} = K_{\tilde E_i}+\Delta_{\tilde E_i}$ for the normalization $\tilde E_i$ of $E_i$, which makes sense as $E_i \subseteq \lfloor \Delta \rfloor$. Since $\Exc(Y/X)$ is a divisor, the map $\sum_{i=1}^r \overline{\NE}(\tilde E_i/X)\to \overline{\NE}(Y/X)$ is surjective and by applying  \autoref{thm:surface_cone} we get
\[
\overline{\NE}(Y/X) = \overline{\NE}(Y/X)_{K_{Y}+\Delta \geq 0} + \sum_{t \geq 0} \bR_{\geq 0} [C_{t}],
\]
for a countable set of curves $\{C_t\}$ and positive integers $d_{C_t}$  such that $0 < -(K_{Y} + \Delta_{Y}) \cdot C_{t} \leq 4d_{C_t}$, and $L \cdot_k C_t$ is divisible by $d_{C_t}$ for every Cartier divisor $L$.  We also obtain that the rays $\bR_{\geq 0}[C_t]$ do not accumulate in the half space $\overline{\NE}(Y/X)_{K_{Y}+\Delta < 0}$. This concludes the proof of the cone theorem.

We run a $(K_Y + \Delta)$-MMP with scaling of $H$ as in \cite{Kollar2020RelativeMMPWithoutQfactoriality} and explain \emph{that {this determines the choice of extremal faces so that they behave as if they were} one-dimensional with respect to exceptional divisors}. We construct it explicitly. Let  $h \in \bR_{\geq 0}$ be the smallest number such that $K_{Y}+\Delta + tH$ is nef for all $h \leq t < h'$. If $h=0$, then move to the last paragraph of the proof.

Since 
\begin{equation}\label{eq:KYDeltahH-ample}
K_Y+\Delta + hH = (1-h/h')(K_Y+\Delta) + (h/h')(K_Y+\Delta + h'H)
\end{equation}
and $K_Y+\Delta+h'H$ is $\bR$-ample, we see that $K_Y+\Delta + hH$ is positive on $\mathrm{NE}(Y/X)_{K_Y+\Delta \geq 0}$ and on all but a finite number of extremal rays by the non-accumulating property of extremal rays. In particular, $(K_{Y}+\Delta + hH)\cdot \Sigma_j = 0$ for all such extremal rays $\Sigma_1, \ldots, \Sigma_c$. Set $V = \mathrm{span}(\Sigma_1, \ldots, \Sigma_c) = (K_X+\Delta+hH)^{\perp}$. We have  $H \cdot \Sigma_j = -\frac{1}{h}(K_Y+\Delta) \cdot \Sigma_j$, and so
\[
\textstyle
\sum_i \gamma_iE_i \cdot \Sigma_j = H \cdot \Sigma_j \in \frac{1}{h}\Q.
\]
Since $\gamma_i$ are linearly independent over $\bQ$, the number $\frac{1}{h}$ has a unique presentation as a linear combination of $\gamma_i$, and so we get that the vectors $(E_1\cdot \Sigma_j, \ldots, E_r \cdot \Sigma_j) \in \bQ^r$ are colinear (that is, $\bQ$-multiples of one another) for different $j$. It follows then that $E_i \in W$ for every $i$, where $W \subseteq \mathrm{Div}_{\bQ}(X)$ is  the subspace of $\bQ$-Cartier $\bQ$-divisors which are colinear with $E_1$ as functionals on $V$. Since $K_Y+\Delta \equiv \sum e_iE_i$, we also have that $K_Y+\Delta \in W$.

As the ample $\bQ$-divisor $\Lambda$ is exceptional over $X$, we get that $\Lambda \in W$, too. In particular, every exceptional divisor $E_i$ is colinear with a multiple of $\Lambda$ as a functional on $V$, and so is either entirely positive, trivial, or negative on $V$. Since $\Lambda$ is ample, and so anti-effective by the negativity lemma, it cannot happen that every exceptional divisor is positive or trivial on $V$. Hence there must exist an exceptional irreducible divisor $S \subseteq \Supp \lfloor \Delta \rfloor$ such that $S$ is negative on $V$.

\emph{We construct the contraction of $V$.} First, we claim that there exists an ample $\bQ$-divisor $G$ such that $V = {(K_Y}+\Delta+G)^{\perp}$. To this end, {define the $\bR$-divisor}
\begin{equation*}
G' = \frac{h/h'}{1-h/h'}\sum (e_i+h'\gamma_i)E_i \equiv_X \frac{h/h'}{1-h/h'}(K_Y+\Delta+h'H).
\end{equation*}
Note that 
${K_Y}+\Delta+G' \equiv_X \frac{1}{1-h/h'}(K_Y+\Delta+hH)$
 by \autoref{eq:KYDeltahH-ample}. {So, $G'$ satisfies all the requirements for $G$, except that it is not a $\bQ$-divisor. However,}  since the  irreducible components  of $G'$ are contained in $W$ we can perturb it in $W$ { to obtain the claimed $\bQ$-divisor $G$. Note that as the perturbation happens in $W$, {we can make}  $K_X+\Delta+G$ trivial on $V$. Additionally, by the non-accumulating property of extremal rays for small enough perturbation, $K_X + \Delta+G$ is still positive on all extremal rays not in  $(K_X+\Delta+G')^{\perp}=V$, and hence $(K_X+\Delta+G)^{\perp}=V$ holds.}
 
 Having shown the claim, we can invoke \autoref{prop:nonqfactorial-contraction-theorem} to construct a contraction $f \colon Y \to Z$ of $V$. Moreover, $K_{Y} + \Delta + hH \in W \otimes_{\bQ} \bR$ descends to an $\bR$-Cartier $\bR$-divisor on $Z$, which must be $\bR$-ample by the Nakai-Moishezon criterion \autoref{lem:Nakai-Moishezon}. Indeed, to apply the Nakai-Moishezon criterion we need to check that a closed integral subscheme $Q \subseteq Y$ over a field is contracted by $f$ if and only if $(K_{Y} + \Delta + hH)|_{Q}$ is not big. This is automatic when $\dim Q = 1$ as $(K_{Y} + \Delta + hH)^{\perp}=V$, and so we can assume that $Q = E_i$ for some $i$ and $E_i$ is defined over a field. 

But $(K_{Y} + \Delta + hH)|_{E_i}$ is semiample (by \autoref{eq:KYDeltahH-ample} and \cite[Theorem 1.1]{TanakaAbundanceImperfectFields}), and hence it is big if and only if $E_i$ is not contracted.

\emph{We construct the flip}. Suppose that $f$ is small. We have that $K_Y+\Delta+hH \equiv_{X,\bR} \sum (e_i+h\gamma_i)E_i$ and the $\bR$-divisor on the right descends to an $\bR$-Cartier $\bR$-divisor $\sum (e_i+h\gamma_i)f_*E_i$ on $Z$ by \autoref{prop:nonqfactorial-contraction-theorem}.\autoref{itm:nonqfactorial-contraction-theorem:num_triv} (with $f_*E_i\neq 0$, as $f$ is small), which is $\bR$-ample. Hence, by the negativity lemma,  $e_i+h\gamma_i<0$ for all $i$. If $C$ is a curve contracted by $f$, then $C \cdot \sum (e_i+h\gamma_i)E_i = 0$. As $C \cdot S < 0$, there must exist another irreducible exceptional divisor $A$ such that $C \cdot A > 0$. Since $A \in W$, we get that $A$ is $f$-ample. {We use the divisors $A$ and $S$ to construct the flip $f^+ \colon Y^+ \to Z$ of $f$.

By \autoref{prop:nonqfactorial-contraction-theorem}.\autoref{itm:nonqfactorial-contraction-theorem:num_triv} (applied to $T=Z$), 
all $\bR$-divisors in $W \otimes_{\bQ} \bR$ are in fact $\bR$-linearly equivalent over $Z$ to multiples of each other; similarly for $\bQ$-divisors. Hence, the existence of $f^+$ follows from  \autoref{proposition:one-complemented-pl-flips-exist} as in the proof of 
\autoref{proposition:flips-exist} (we can peform a necessary perturbation so that $\lfloor \Delta \rfloor = S + A$ as the irreducible components of $\lfloor \Delta \rfloor$ are $\bQ$-Cartier and contained in $W$). 

Additionally, we note that 
\begin{equation} \label{eq:non-q-factorial-flips-preserve-q-factoriality}
\textrm{the strict transform } D^+ \textrm{ of }  D \in W \otimes_{\bQ} \bR \textrm{ is } \bR\textrm{-Cartier.}
\end{equation}
In fact,  in this situation  $D + a(K_X+\Delta) \equiv_Z 0$ for some $a \in \bR$. Hence, $D + a(K_X+\Delta)$ descends to $Z$ by \autoref{prop:nonqfactorial-contraction-theorem}.\autoref{itm:nonqfactorial-contraction-theorem:num_triv}, and so $D^+ + a(K_{X^+}+\Delta^+)$ is $\bR$-Cartier. Since $K_{X^+}+\Delta^+$ is $\bQ$-Cartier, so is  $D^+$.
 }
Let $\phi \colon Y \dashrightarrow Y^+$ be the induced rational map.

\emph{We show that the above procedure can be repeated.} Pick $g \colon Y \dashrightarrow \overline{Y}$ as follows: $\overline{Y}=Z$ (with $g=f$) when $f$ is divisorial and $\overline{Y}=Y^+$ (with $g = \phi$) when $f$ is small. We claim that we can replace $Y$, $\Delta$, $\Delta'$, $D$, $H$, $h'$ by $\overline Y$, $\Delta_{\overline Y}$, $\Delta'_{\overline Y}$, $D_{\overline Y}$, $H_{\overline Y}$, $h-\varepsilon$, respectively (with $0 < \varepsilon \ll 1$ and the corresponding divisors being their strict transforms on $\overline{Y}$), so that \autoref{itm:non-q-factorial-mmp:exists_ample}-\autoref{itm:non-q-factorial-mmp:ample} hold and the algorithm can be run again:

\begin{itemize}
\itemsep0.2em 
    \item note that $g_*M$ is $\bQ$-Cartier for every $\bQ$-divisor $M \in W$ by \autoref{prop:nonqfactorial-contraction-theorem}.\autoref{itm:nonqfactorial-contraction-theorem:CArtier} and \autoref{eq:non-q-factorial-flips-preserve-q-factoriality}. In particular, if $M \equiv_X 0$, then $M \in W$, and so $g_*M \equiv_X 0$ is $\bQ$-Cartier (cf.\ \autoref{prop:nonqfactorial-contraction-theorem}.\autoref{itm:nonqfactorial-contraction-theorem:num_triv}). This immediately yields \autoref{itm:non-q-factorial-mmp:E_i_Q_Cartier},\autoref{itm:non-q-factorial-mmp:log_canonical},\autoref{itm:non-q-factorial-mmp:sub_klt} shows that $K_{\overline Y}+\Delta_{\overline Y}$ is $\bQ$-Cartier and that $D_{\overline Y}\equiv_X 0$ and is $\bQ$-Cartier (by setting $M=E_i,K_Y+\Delta-\sum e_iE_i, K_Y+\Delta', K_Y+\Delta, D$, respectively). 
    \item we have $H_{\overline Y} = \sum \gamma_i g_*E_i$, and those $\gamma_i$ for which $g_*E_i \neq 0$ comprise a subset of $\{\gamma_1, \ldots, \gamma_r\}$, and so are linearly independent over $\bQ$; hence \autoref{itm:non-q-factorial-mmp:H} holds,
        \item the $\bR$-divisor $K_{\overline Y}+\Delta_{\overline Y} + (h-\varepsilon)H_{\overline Y}$ is $\bR$-ample, and so \autoref{itm:non-q-factorial-mmp:ample} holds. This is automatic in the divisorial case as $f_*(K_{ Y}+\Delta + hH)$ is ample, and in the flipping case it follows from $
    K_{Y^+} + \Delta_{Y^+} + (h-\varepsilon)H_{Y^+} = (f^+)^*(K_Z + \Delta_Z + hH_{Z}) - \varepsilon H_{Y^+}$,
    where $H_{Y^+}$ is anti-ample over $Z$ as $H \in W$ was ample over $Z$. Here, divisors with subscripts denote appropriate strict transforms. 
    \item \autoref{itm:non-q-factorial-mmp:exists_ample} is satisfied for $\Lambda = \sum (e_i+(h-\varepsilon)\gamma_i) g_*E_i \equiv_X g_*(K_Y+\Delta+(h-\varepsilon) H)$ by the above paragraph.
        Note that such a chosen $\Lambda$ is only an $\bR$-divisor, but since each irreducible component of $\Supp \Lambda$ is $\bQ$-Cartier, we can perturb it so that it is an ample $\bQ$-divisor. 
    \end{itemize}  
\noindent However, note that $H_{\overline Y}$ is not necessarily $\bR$-ample any more.

Now, repeat the above procedure. It eventually stops by the same argument as in special termination (\cite[Theorem 4.2.1]{fujino05}, cf.\ \autoref{thm:special-termination}). 

Thus, we can assume that $K_Y+\Delta$ is nef. Then $K_Y + \Delta - (K_{Y}+\Delta')$ is nef, effective, and exceptional, hence zero by the negativity lemma. 
Since $\mathrm{Exc}(Y/X) = \lfloor \Delta \rfloor$, this implies that $Y=X$ and conclude the proof as $D$ is $\bQ$-Cartier.\qedhere

\end{proof}

In the above proof, we used the following results.
First, we state a variant of the contraction theorem. Here we say that \emph{the cone theorem is valid for a pair $(X,\Delta)$ over $T$} if there exists a countable set of curve $\{C_i\}$ such that conditions (a), (b), and (c) of \autoref{thm:keel_cone} are satisfied.

\begin{proposition} \label{prop:nonqfactorial-contraction-theorem} Let $(X,\Delta)$ be a three-dimensional dlt pair which is projective over $T$, and let $G$ be an ample $\bQ$-divisor {on $X$} such that :
\begin{itemize}
    \item $K_X+\Delta$ is pseudo-effective over $T$.
    \item $\Delta$ is a $\bQ$-divisor such that all irreducible components of $\lfloor \Delta \rfloor$ are $\bQ$-Cartier.
\item The cone theorem holds for $(X,\Delta)$ over $T$.     \item $L = K_X+\Delta+G$ is nef.
    \item  $V = L^{\perp}  \subseteq \overline{\mathrm{NE}}(X/T)$ is an extremal face.
    \item  There exists a prime divisor $S \subseteq \lfloor \Delta \rfloor$, which is negative on $V$ and is contained in $W$, where  $W \subseteq \mathrm{Div}_{\bQ}(X)$ is the subspace of $\bQ$-Cartier $\bQ$-divisors which are colinear with $K_X+\Delta$ as functionals on $V$.
\end{itemize} 
Then the contraction $f\colon X\to Z$ of $V$ exists.
Moreover :
\begin{enumerate}
    \item \label{itm:nonqfactorial-contraction-theorem:num_triv}
    If $D\equiv_Z 0$ is a $\bQ$-Cartier $\bQ$-divisor, then $D$ descends to $Z$; the same holds for $D \in W \otimes_{\bQ} \bR$ satisfying $D \equiv_Z 0$,
    \item \label{itm:nonqfactorial-contraction-theorem:CArtier} If $f$ contracts an irreducible divisor $E \in W$, then $f_*D$ is $\bR$-Cartier for every $D \in W \otimes_{\bQ} \bR$ (in particular, if $D \in W$, then $f_*D$ is $\bQ$-Cartier).
\end{enumerate} 
\end{proposition}
Note that $W \otimes \bR$ agrees with the subspace of $\bR$-Cartier $\bR$-divisors which are colinear with $K_X+\Delta$ as functionals on $V$.
\begin{proof} 
 Since $L^\perp = V$, we have that $L$ is trivial on $V$, and so $G \in W$. Further, $K_X+\Delta$ is negative on $V$. 
 Pick an ample $\bQ$-divisor $A \in W$ (by rescaling $G$) such that $S+A$ acts trivially on $V$ (this is possible as $S,G\in W$).
 
 We claim that, $L_{\varepsilon}=K_X+\Delta+G_\varepsilon$ is nef over $T$ and $(L_\varepsilon)^\perp=V$ for any $0<\varepsilon\ll 1$, where $G_\varepsilon = G+ \varepsilon (S+A)$ is an ample $\bQ$-divisor. Indeed, by {non-accumulating property of the cone theorem}, there are finitely many $(K_X+\Delta+\frac{1}{2}G)$-negative extremal rays: $\Sigma_1, \ldots, \Sigma_l$. We may assume that $V = \mathrm{span}(\Sigma_1, \ldots, \Sigma_k)$ for some $k\leq l$. For every $\varepsilon$ such that $G_{\varepsilon}-\frac{1}{2}G$ is ample, $L_{\varepsilon}$ is positive on all extremal rays except possibly these $\Sigma_1, \ldots, \Sigma_l$. By decreasing $\varepsilon$ further we can assume that $L_{\varepsilon} \cdot \Sigma_j$ is close enough to $L \cdot \Sigma_j$, and so it is also positive for $k < j \leq l$. Last, $L_{\varepsilon} \cdot \Sigma_j =0$ for $1 \leq j \leq k$ holds for all $\varepsilon$ as $L \cdot \Sigma_j = (S+A) \cdot \Sigma_j = 0$. 
 
 Moreover, we have that $\mathbb E(L_\varepsilon)\subset S$. Indeed, if  $V\subset X$ is an integral subscheme not contained in $S$, then $L_\varepsilon |_V=(L+\varepsilon (S+A))|_V$ is nef and big over $T$. Replacing $L$ by $L_\varepsilon$, we may assume that $\mathbb E(L)\subset S$. Now, the contraction exists by \autoref{prop:bpf_plt} and \autoref{prop:char_zero_bpf}.

As for condition \autoref{itm:nonqfactorial-contraction-theorem:num_triv}, {the case of $\bQ$-Cartier $\bQ$-divisor $D$ } follows directly from \autoref{prop:bpf_plt} and \autoref{prop:char_zero_bpf}. {So, we only have to prove the case of  } ${0 \equiv_Z }D \in W \otimes_{\bQ} \bR$. {In this case,} $D = \sum a_i D_i$ for $D_i \in W$ and $a_i \in \bR$. Pick $b_i \in \bQ$ such that $D_i \equiv_Z b_i S$. Then $D = \sum a_i(D_i - b_iS) + (\sum a_ib_i)S$ and $D_i - b_iS$ descend to $Z$ 
by \autoref{prop:bpf_plt}. {As} $D \equiv_Z 0$, {we have} $\sum a_ib_i = 0$, and {hence} $D$ descends to $Z$.

{It remains to show point \autoref{itm:nonqfactorial-contraction-theorem:CArtier}.}
If $f$ contracts $E \in W$, then $E=S$ and  $f_*D = f_*(D - cS)$ is $\bR$-Cartier for every $D \in W \otimes \bR$ by condition \autoref{itm:nonqfactorial-contraction-theorem:num_triv} proved in the above paragraph, where $c \in \bR$ is chosen so that $D-cS \equiv_Z 0$. 
\end{proof}

\begin{theorem} \label{thm:bpf} \label{corollary:bpf-theorem-no-standard-coefficients} Let $(X,B)$ be a {$\bQ$-factorial} three-dimensional klt pair, with $\bR$-boundary,
which is projective over $T$. Let $L$ be a nef and big $\mathbb{Q}$-Cartier divisor on $X$ such that $L-(K_X+B)$ is {nef and big}. Then $L$ is semiample.
\end{theorem}

\begin{proof}

By a small perturbation, since $X$ is $\bQ$-factorial we may assume that $B$ is a $\bQ$-divisor, and $L-(K_X+B)$ is ample.
By \autoref{thm:keel_EWM_imperfect}, there exists a proper birational $T$-morphism $f \colon X \to Z$ to a proper algebraic space $Z$ over $T$ such that a proper integral subscheme $V \subseteq X$ is contracted if and only if $L|_V$ is not relatively big. In particular, $L\equiv_Z 0$.  

We claim that $\sO_X(mL)=f^*\mathcal{M}$ for some $m>0$ and a line bundle $\mathcal{M}$ on $Z$. This will conclude the proof of the theorem as $\mathcal{M}$ must then be ample by the Nakai-Moishezon criterion (\autoref{lem:Nakai-Moishezon}).
The  assumptions of Nakai-Moishezon are satisfied as $L^{\dim V} \cdot V = 0$ for a proper integral subscheme $V \subseteq X$ (over a field) if and only if $V \subseteq \mathrm{Exc}(f)$.

To show the claim, it is enough to prove that $f_*\sO_X(mL)$ is a line bundle for some $m>0$ which can be verified \'etale locally ($\bQ$-factoriality may be lost, but it will not be needed again). Thus, we can assume that $Z$ is {the spectrum of a Noetherian local ring.}
The assumptions of \autoref{lemma:add_ample} are satisfied, and as $A = L-(K_X+B)$ is ample, we can assume that $(X,B')$ is klt for $B' = B+A$. Set $B'_Z = f_* B'$. Note that $-(K_X+B)$ is relatively ample over $Z$.

{Let $h \colon Y \to Z$ be a log resolution of $(Z,B'_Z)$ which admits a factorization $\pi \colon Y \to X$ and such that there exists an ample exceptional divisor (see \autoref{proj-resolutions})}. Set $\Delta_Y = h^{-1}_*B'_Z + \mathrm{Exc}(h)$. Note that $K_Y+B'_Y \equiv_Z 0$ for $K_Y+B'_Y=\pi^*(K_X+B') = \pi^*L$. Further $K_Y+\Delta_Y \equiv_Z K_Y+\Delta_Y - (K_Y+B'_Y)$ and the latter $\Q$-divisor is exceptional over $Z$. Thus the assumptions of \autoref{lem:non-q-factorial-mmp} for $(Y,\Delta_Y)$ over $Z$ are satisfied, and so ${ \pi^*L}$ descends to $Z$. Hence, $f_*\sO_X(mL) = h_*\sO_Y(m{ \pi^*L})$ is a line bundle for $m$ divisible enough by the projection formula. 
\end{proof}

\begin{corollary}[Contraction theorem for birational extremal rays]\label{cor:bir-contraction}
Let $(X,B)$ be a $\mathbb{Q}$-factorial dlt pair with $B$ an $\mathbb{R}$-divisor.  Suppose that $\Sigma$ is a $(K_X+B)$-negative extremal ray such that there is some nef and big divisor {$D$} with $\Sigma={D}^\perp$.  Then there is a projective contraction $f:X\to Z$ of $\Sigma$. 
\end{corollary}
\begin{proof}
As $X$ is $\mathbb{Q}$-factorial,  we may decrease the coefficients of $B$ to assume it is a $\mathbb{Q}$-divisor and $(X,B)$ is klt,  while maintaining that $\Sigma$ is $(K_X+B)$-negative.  By a standard argument using \autoref{thm:keel_cone}, we may find an ample $\mathbb{Q}$-divisor $A$ such that $\Sigma=(K_X+B+A)^\perp$. 
Now we may apply \autoref{corollary:bpf-theorem-no-standard-coefficients} to $L=K_X+B+A$.
\end{proof}

\subsection{Step 4: MMP in the pseudo-effective case}

Next, we note that  projective terminalizations  of klt pairs can be constructed. This is used in the proof of termination below.

\begin{proposition} \label{prop:existence-of-terminalizations} Let $(X,B)$ be a three-dimensional quasi-projective klt pair with $\bR$-boundary over $R$ as in \autoref{MMP_setting} where additionally the residue fields of $R$ do not have characteristic $2$, $3$ or $5$. Then there exists a projective birational morphism $g \colon Y \to X$ and a terminal pair $(Y,B_Y)$ such that  $K_Y+B_Y = g^*(K_X+B)$.
\end{proposition}

\begin{proof}
By \cite[Proposition 2.36]{KollarMori} there are only finitely many divisors over $X$ with log discrepancy at most 1.  Therefore by \cite[Lemma 2.45]{KollarMori} and  \autoref{proj-resolutions} we may find a projective log resolution $g\colon Y\to X$ of $(X,B)$ which extracts all divisors of log discrepancy at most $1$ with respect to $(X,B)$.  
Define 
\[K_Y+B_Y\sim g^*(K_X+B)+F-E\]
where $E$ and $F$ {are effective $\bR$-divisors with no common prime divisors in their support,} and $B_Y$ is the strict transform of $B$.
By repeatedly blowing up strata of $(Y,B_Y+E)$ we may assume that the irreducible components of $\Supp (B_Y + E)$  do not meet.  If we replace $Y, F$ and $E$ by the result of this process, all new exceptional components will be added to $F$.  As a result we may assume that the irreducible components of {$\Supp(B_Y + E)$} are disjoint and hence that $(Y, B_Y+E)$ is terminal.

Run a $(K_Y+B_Y+E)$-MMP over $X$, which uses the cone theorem  \autoref{thm:keel_cone}, contractions theorem \autoref{cor:bir-contraction}, and existence of flips \autoref{thm:full_flips}.  
This LMMP terminates by a standard argument involving Shokurov's difficulty \cite[Theorem 6.17]{KollarMori}.  Note that \cite{KollarMori} deals with only $\mathbb{Q}$-boundaries, however the same argument works in the $\mathbb{R}$-boundary case.  It uses the fact that the variety underlying a terminal surface pair is regular in codimension $2$ which holds in our case by \cite[Theorem 2.29]{KollarKovacsSingularitiesBook}, and also uses the fact that there are only finitely many components of log discrepancy at most one \cite[Proposition 2.36]{KollarMori}.
Let $\phi\colon W\to X$ be the outcome of this MMP and let $E_W$ and $F_W$ be the images of $E$ and $F$, respectively.  We know that $\phi$ contracts every component of $F$ since by construction $F_W-E_W$ is nef and $\phi_*(E_W-F_W)=0$, so $F_W=0$ by the negativity lemma.   Since  $K_Y+B_Y+E\sim_{\mathbb{Q},g} F$, this means that every divisorial contraction which occurs is negative for $F$, and hence the contracted divisor is a component of $F$.  As a result we see that this MMP contracts exactly the components of $F$ and so produces the required terminalization.
\end{proof}

\begin{proposition}\label{prop:psef_termination}
Let $(X,B)$ be a {$\bQ$-factorial} three-dimensional {dlt} pair with {$\bR$}-boundary which is projective and surjective over $T$ with $\dim(T)>0$ and such that none of the residue fields have characteristic $2$, $3$ or $5$. Suppose that $K_X+B$ is pseudo-effective. Then {we can run a $(K_X+B)$-MMP and} {any sequence of the steps of the MMP terminates}.  {As a result, $(X,B)$ has a log minimal model.}
\end{proposition}
\begin{proof}

First, note that $K_X+B\sim_T M\geq 0$. Indeed, it is enough to show that $\kappa(K_{X_{\eta}}+B|_{X_{\eta}}) \geq 0$, where $X_{\eta}$ is the fiber over a generic point $\eta \in T$, and this follows by the two-dimensional non-vanishing theorem in equicharacteristic \cite{fujino_minimal_2012, TanakaAbundanceImperfectFields}.  

{We can apply {\autoref{thm:keel_cone}, \autoref{thm:full_flips}} and \autoref{cor:bir-contraction} to run a $(K_X+B)$-MMP, and it remains to show that it terminates.}

Suppose we have an infinite sequence of $(K_X+B)$-flips $X_i\dashrightarrow X_{i+1}$. By the first assertion in \autoref{thm:special-termination}, eventually the flipping loci are disjoint from $\rddown{B}$. Thus, by replacing $X$ by $X_i$ for $i \gg 0$, we can assume that all these flips are $(K_X+\Delta)$-flips for $\Delta = B-\lfloor B \rfloor +\varepsilon M$ and $0 < \varepsilon \ll 1$. Explicitly, we pick $\varepsilon$ so that $(X,\Delta)$ is klt. Then the statement follows from  \autoref{thm:special-termination-2}, where the existence of terminalizations is a consequence of \autoref{prop:existence-of-terminalizations}. Note that $K_X+\Delta$ is not necessary pseudo-effective any more, but  here we only use that the extremal rays of $X_i \dashrightarrow X_{i+1}$ are negative on $M$, and so also on an irreducible component of $\Supp \Delta$.  
\end{proof}

\begin{corollary}\label{cor:dlt_modification}
\label{cor:dlt_models}
Let $X$ be a variety which is quasi-projective over $\Spec R$, such that $X$ has no residue fields of characteristic $2$, $3$ or $5$. 

Let $\Delta=\sum a_i\Delta_i$ be an $\mathbb{R}$-divisor such that $\Delta_i$ are prime divisors and such that $K_X+\Delta$ is $\mathbb{R}$-Cartier.
Let 
\[\Gamma=\sum_{i: a_i>1} \Delta_i+\sum_{i: a_i\leq1} a_i\Delta_i.\]

Then there exists a dlt modification  of $(X,\Delta)$, which is a projective birational morphism $\pi\colon Y\to X$ with the properties listed below.  First define $\Delta_Y$ by $K_Y+\Delta_Y:=\pi^*(K_X+\Delta)$ and $\Gamma_Y$ by $\Gamma_Y=\pi_*^{-1}\Gamma+\mathrm{Ex}(\pi)$.  Then $\pi$ satisfies:
\begin{enumerate}
    \item $Y$ is $\mathbb{Q}$-factorial.
       \item $(Y, \Gamma_Y)$ is dlt, 
    \item  $K_Y + \Gamma_Y$ is nef over $X$, 
     \item $\Delta_Y-\Gamma_Y\geq 0$\label{itm:cor:dlt_modification:effective}, and 
    \item\label{rem:dlt_model_connected} for every $x\in X$, either $\pi^{-1}(x)$ is contained in $\Supp(\Delta_Y-\Gamma_Y)$ or is disjoint from it. \label{itm:cor:dlt_modification:(b)}
\end{enumerate}

\end{corollary}

\begin{proof}

Take $\pi \colon Y \to X$ to be a log resolution of $(X,\Delta)$. Then a minimal model of  $(Y, \pi^{-1}_*\Gamma + \mathrm{Ex}(\pi))$ over $X$, which exists by \autoref{prop:psef_termination}, is a dlt modification of $(X,\Delta)$.  The first three properties may be verified by the same argument as in \cite[Theorem 10.4]{FujinoFundamental}.
For \autoref{itm:cor:dlt_modification:effective} note that
\[
 \Delta_Y - \Gamma_Y = K_Y + \Delta_Y -  (K_Y + \Gamma_Y)
\]
is  anti-nef over $X$ and its pushforward via $\pi$ is effective. Thus, it is effective by the negativity lemma (\autoref{lem:negativity})  concluding \autoref{itm:cor:dlt_modification:effective}. 
For \autoref{itm:cor:dlt_modification:(b)}, note that $\pi^{-1}(x)$ is connected for every $x\in X$, and if $C$ is a curve in $\pi^{-1}(x)$ which intersects $\Supp(\Delta_Y-\Gamma_Y)$ but is not contained in it, then $C\cdot (\Delta_Y-\Gamma_Y)>0$, contradicting the fact that $\Delta_Y-\Gamma_Y$ is anti-nef over $X$. 
\end{proof}

\begin{remark} \label{rem:non-q-factorial-dlt-modification}
Even when $X$ does admit residue characteristics $2$, $3$, or $5$, one can still construct a dlt modification of $X$ by \cite{Kollar2020RelativeMMPWithoutQfactoriality} (cf.\  \autoref{lem:non-q-factorial-mmp}). However, it will not necessarily be $\bQ$-factorial unless $X$ is $\bQ$-factorial as well.
\end{remark}

\subsection{Step 5: Base point freeness}

In this subsection, we prove the full basepoint freeness theorem. 
We do this only in the case of $\dim(T)>0$, an assumption that automatically holds in the  arithmetic situation which is the main motivation of our article. 
The case of a projective variety over a field appears in  \cite[Theorem 3.3]{KollarMori} when the field has characteristic zero (see \cite{kawamata_pluricanonical_1985} for the original proof stated less generally), \cite{BW17} when it is algebraically closed of characteristic $p>5$ and  \cite{gongyo_rational_2015} when it is perfect of characteristic $p>5$.  We leave open the case of a variety projective over an imperfect field.

While many of the arguments of \cite{BW17} go through in our situation of a positive dimensional base, there are several things which do not work, such as Tsen's theorem.  However, the relative situation provides advantages which enable us to avoid those problems.  In the first version of this article we directly referred to the arguments of \cite{BW17} wherever possible, while below we provide simpler proofs which make full use of the advantages offered by the relative situation.

{First, we} prove the abundance theorem for {semi-log canonical curves and log canonical} surfaces, for which we were unable to find a reference in sufficient generality.

\begin{lemma}\label{slc-curve}
Let $(X,\Delta)$ be a semi-log canonical curve pair with {$\mathbb{Q}$-boundary}, such that $K_X+\Delta$ is nef.  Then $K_X+\Delta$ is semiample.
\end{lemma}
\begin{proof}
By Keel's theorem {(\cite{KeelBasepointFreenessForNefAndBig})} we can reduce to the case of $K_X+\Delta\equiv 0$, and furthermore assume that $X$ is connected. 
If $\Delta=0$, we need only show that $h^0(X, \omega_X) \neq 0$ which follows from the general Riemann-Roch theorem for reduced curves \cite[Thm VII.3.26]{liu_algebraic_2002}  (note that $X$ is Cohen-Macaulay):
\begin{equation*}
   \dim_k H^0(X, \omega_X) = \dim_k H^1(X, \omega_X) + 0 + \chi(X, \sO_X)  = 2- \dim_k H^1(X,\sO_X) = 2- \dim_k H^0(X, \omega_X)
\end{equation*}
When $\Delta \neq 0$ on the other hand, we claim that $X$ is a chain of curves $C$ such that $\left( C_{\overline{k}}\right)_{\red} \cong \bP^1$. Let $C_1$ be  an irreducible component which supports a component of the boundary.  Then it can meet at most one other irreducible component $C_2$, at a single point.  Since $C_2$ gains a non-zero conductor component in the normalization, it can meet at most one other component $C_3$ at a single point, and $C_3$ is disjoint from $C_1$.  The argument continues to produce the required chain.  Since normalization produces a non-zero conductor (or boundary) on each component, we must have a chain of curves $\left( C_{\overline{k}}\right)_{\red} \cong \bP^1$ as claimed.
Hence, $K_X+\Delta$ is semiample on all irreducible components.
Thus we may conclude by \cite[Cor 2.9]{KeelBasepointFreenessForNefAndBig} and induction on the number of components.
\end{proof}

\begin{theorem}\label{abundance}
Let $(X,\Delta)$ be a log canonical pair of dimension $2$, projective and surjective over $T$ with $\mathbb{Q}$-boundary, and assume  in addition that that $T$ is the spectrum of a local ring with positive residue characteristic.  If $K_X+\Delta$ is nef over $T$, then it is semiample over $T$.
\end{theorem}
\begin{proof}
By \cite{TanakaAbundanceImperfectFields} we may assume that $X$ is surjective over $T$ with $\dim(T)>0$, and by \autoref{thm:surface-bpf-theorem} we may assume that $\dim(T)=1$. We may replace $T$ by its normalization to assume that it is a {spectrum of a} DVR of positive residue characteristic.  By taking a dlt modification, we may assume that $(X,\Delta)$ is $\mathbb{Q}$-factorial and dlt.  

We first deal with the case where $K_X+\Delta$ is big by adapting the argument of \cite[Theorem 1.1]{waldron2017lmmp} to the two dimensional case.
Firstly, $(K_X+\Delta)|_{X_{\mathbb{Q}}}$ is semiample since $\dim X_{\mathbb{Q}}=1$ {(here it is crucial that $T$ is a spectrum of a DVR, cf.\ Remark \ref{remark:divisors-of-unexpected-dimension3})}, in which case abundance is straightforward. 
So by \autoref{thm:mixed-characteristic-Keel} it is enough to show that $(K_X+\Delta)|_{\mathbb{E}(K_X+\Delta)}$ is semiample.  
Run a $(K_X+\Delta-\varepsilon\rddown{\Delta})$-MMP, with scaling of $\rddown{\Delta}$.  By taking $\varepsilon$ sufficiently small, we may assume that this only contracts $(K_X+\Delta)$-trivial curves, and also that $K_X+\Delta-\varepsilon\rddown{\Delta}$ is big.  Once the MMP terminates, we obtain $\psi:X\to Y$, such that $\psi^*(K_Y+\Delta_Y)=K_X+\Delta$, and  $(Y, \Delta_Y-\varepsilon \rddown{\Delta_Y})$ is klt.  It follows from the base-point free theorem {\autoref{thm:surface-bpf-theorem})} that $K_Y+\Delta_Y-\varepsilon\rddown{\Delta_Y}$ is semiample.   
As $Y$ is a surface, every irreducible component of $\mathbb{E}(K_Y+\Delta_Y)$ is one dimensional, and a curve {$C$} is in $\mathbb{E}(K_Y+\Delta_Y)$ if and only if $(K_Y+\Delta_Y)|_C\equiv 0$. 
By construction, $Y$ contains no $(K_Y+\Delta_Y)$-trivial curves which intersect $\rddown{\Delta_Y}$ positively, and so every connected component of $\mathbb{E}(K_Y+\Delta_Y)$ is either contained in $\rddown{\Delta_Y}$ or disjoint from it.

Suppose first that $E$ is a connected component of $\mathbb{E}(K_{Y}+\Delta_{Y})$ which is completely disjoint from $\rddown{\Delta_{Y}}$.  Then $(K_{Y}+\Delta_{Y})|_E\sim_{\mathbb{Q}}(K_{Y}+\Delta_{Y}-\varepsilon\rddown{\Delta_{Y}})|_E$, which as noted earlier is semiample by the base-point free theorem.  This implies that $(K_{Y}+\Delta_{Y})|_E$ is semiample.  
On the other hand, if $E$ is a connected component of $\mathbb{E}(K_{Y}+\Delta_{Y})$ which is contained entirely in $\rddown{\Delta_{Y}}$,  then if $K_E+\Delta_E=(K_{Y}+\Delta_{Y})|_E$, we have that  $(E, \Delta_{E})$ is a semi-log canonical pair by \cite[Corollary 3.35]{KollarKovacsSingularitiesBook} and so $(K_{Y}+\Delta_{Y})|_{E}$ is semiample by adjunction and \autoref{slc-curve}.

{Assume now that $K_X+\Delta$ is not big.  Since we may assume as above that $T$ is the spectrum of a DVR, the semiampleness now
follows from \cite[Lemma 2.17]{CasciniTanaka2020}.}
\end{proof}

The following result on descending nef divisors is an adaptation of \cite[Lemma 5.6]{BW17} and \cite[Proposition 2.1]{kawamata_pluricanonical_1985}.

\begin{lemma}\label{lem:BW_5.6}
	Let $f:X\to T$ be a projective and surjective contraction between normal  integral schemes over $R$. 
		Let $L$ be {a $\bQ$-Cartier  $\bQ$-divisor} on $X$, nef over $T$, such that $L|_F$ is semiample, for the generic fiber $F$ of $f$. 
		Assume $\dim X\leq 3$.  Then there exists a commutative diagram 
	\[
        \xymatrix{ X'\ar[r]^\phi\ar[d]_{f'}& X\ar[d]^{f}\\
	    Z\ar[r]^{\psi} & T}
	\]
	with $\phi $ and $\psi$ projective and $\phi$ birational, where $f'$ agrees with the map induced by $\phi^*L$
	over the generic point of $T$, and with  {$\bQ$-Cartier $\bQ$-}divisor $D$ on $Z$ {satisfying} $\phi^*L\sim_{\mathbb{Q}}f'^*D$.
	\end{lemma}

\begin{proof}
Up to replacing $X$ by a projective birational cover, we may pick a projective surjective morphism $X \to Z'$ to a normal projective scheme $Z'$ over $T$ such that its restriction to the generic fibre is the fibration defined by $L|_F$.

    Now take  a {flattening} (see \cite[Theorem 5.2.2]{rg71}):
   
    \[
        \xymatrix{ X''\ar[r]^{\phi''}\ar[d]_{f''}& X\ar[d]^{f}\\
        Z''\ar[r]^\pi & Z'.}
    \]
	Here $f''$ is flat {(hence equidimensional, see \cite[Tag 0D4J]{stacks-project})}, and $\phi''$ and $\pi$ are birational.  
    We can then replace $Z''$ with a resolution $Z$ and $X''$ with the normalization $X'$ {of the irreducible component  of} $X''\times_{Z''}Z$
    {which is dominant over $Z$} to assume that $Z$ is regular and $X'$ is normal. {Note  that $f'$ may not be flat, but it stays equidimensional.}  Denote $\phi:X'\to X$ and $f':X'\to Z$. {By \cite[Lemma 2.17]{CasciniTanaka2020},
we get that  $\phi^* L\sim_{\mathbb{Q}} f'^*D$ for some $\bQ$-divisor $D$ on $Z$.}
\qedhere
      
	\end{proof}

\begin{lemma}\label{lem:fiberwise-semiampleness}
Let $X$ be a three-dimensional normal integral scheme, projective over $T$.  Suppose $L$ is a nef $\bQ$-Cartier $\mathbb{Q}$-divisor which is not big over $T$ and such that $L|_{X_\mathbb{Q}}$ is semiample (if $X_{\bQ}$ is not empty), and $L|_G$ is semiample where $G$ is {the} fiber over the generic point of the image of $X$ in $T$.  
 Then $L$ is endowed with a map $f:X\to V$ over $T$ to an algebraic space $V$ proper over $T$.
 
 Moreover if $L|_F\sim_{\mathbb{Q}}0$ for every fiber $F$ of $f$,  then $L$ is semiample over $T$.
\end{lemma}
\begin{proof}

{
First we may replace $T$ {by the image of $X$ in $T$} to assume that $X\to T$ is {a surjective contraction}.  Let $g:X'\to Z$ be the morphism given by \autoref{lem:BW_5.6}.  Replacing  $X$ birationally (which we can do as $X$ is normal so $X' \to X$ is a contraction)
we may assume that $X=X'$, so that there is a {$\bQ$-}divisor $D$ on $Z$ such that $L\sim_{\mathbb{Q}}g^*D$.}

 To show that $L$ is EWM, it suffices to show that $D$ is EWM.   If $\dim(Z)=2$, then $D$ is EWM by \autoref{lem:auxiliary-for-EWM-bpf-theorem}.  If $\dim(Z)=1$
 then we may assume that $\dim T = 0$ or  $Z=T=\Spec(R)$ for $R$ a Dedekind domain, and
 then there is nothing to prove.

For {the second part of the lemma}, first localize $T$ at a closed point of positive characteristic, which we may do by semiampleness of $L|_{X_{\mathbb{Q}}}$.  Let $f:X\to V$ be the map associated to $L$, and assume that $L$ is semiample on every fiber of $f$.    It is enough to show that the divisor $D$ on $Z$ is semiample.  Furthermore, we may assume $\dim(Z)=2$ otherwise we are done.  As $D$ is big and EWM on $Z$, $\mathbb{E}(D)$ is a finite set of curves contracted to points on $T$,
 whose pre-images on $X$ are therefore contained in fibers of $f$. {Hence $f^*D|_{f^{-1}(\mathbb{E}(D))}$ is semiample, and so is $D|_{\mathbb{E}(D)}$ by \cite[Lemma 2.11(3)]{CasciniTanaka2020} as $f^{-1}(\mathbb{E}(D)) \to \mathbb{E}(D)$ has geometrically connected fibers.} We are done by 
 \autoref{thm:mixed-characteristic-Keel}. 
\end{proof}

We now prove the base point free theorem.

\begin{theorem}\label{thm:MMP_bpf}
Let $(X,B)$ be a three-dimensional $\mathbb{Q}$-factorial klt pair with $\mathbb{R}$-boundary admitting a projective morphism $f:X\to T$, such that the image of $f$ has positive dimension, and none of the residue characteristics of $T$ are $2$, $3$ or $5$.

Suppose that $L$ is an $f$-nef $\mathbb{Q}$-divisor such that $L-(K_X+B)$ is $f$-big and $f$-nef.  Then $L$ is $f$-semiample.
\end{theorem}

\begin{proof}
{By \autoref{thm:bpf}}, it remains to prove the case where $L$ is not big.  By a small perturbation we may assume that $B$ is a $\mathbb{Q}$-boundary, and that $L-(K_X+B)$ is $f$-ample.

By the base point free theorem in dimension $1$ and $2$, $L|_G$ is semiample, where $G$ is the fiber over the generic point of $\mathrm{Im}(f)$.
By  \autoref{lemma:add_ample} we may choose $0\leq A\sim_{\mathbb{Q}}L-(K_X+B)$ such that $(X,B+A)$ is klt, and it suffices to show that $K_X+B+A$ is semiample.  By \autoref{lem:fiberwise-semiampleness} {and \autoref{prop:char_zero_bpf}}, $K_X+B+A$ is EWM over $T$, and let $g:X\to V$ be the associated map.  Note that in particular, $K_X+B+A\equiv_V 0$.  By {the second part of} \autoref{lem:fiberwise-semiampleness}, it is enough to show that $L|_F\sim_{\mathbb{Q}}0$ for every fiber $F$ of $g$.   This is satisfied over the generic point of $V$ by the base point free theorem in lower dimensions and furthermore holds over the points of characteristic zero by \autoref{prop:char_zero_bpf}.  So we may fix a point $v\in V$ of positive residue characteristic, not equal to the generic point, for which we must test semiampleness on the fiber $F$ over $v$. 

Let $h:V'\to V$ be an \'etale cover of a neighbourhood of $v\in V$ by an affine scheme, and fix $v'\in h^{-1}(v)$.  Since $F_{v'}$ is only a base change of $F$ by an extension of the ground field, it is enough to check semi-ampleness of $L|_{F_{v'}}$.  Hence after performing a small $\mathbb{Q}$-factorialization of the base change $X\times_V V'$, we may assume that $V$ is an affine scheme, and furthermore by passing to the localization at $v$ we may assume it is the spectrum of a local ring with positive residue characteristic.

Fix a Cartier divisor $D$ on $V$ which contains the point $v$ (which we may do because {$v$} is not the generic point), and furthermore that $D$ is of pure characteristic $p$.  It follows that $\Supp(f^*D)$ contains the fiber $F$. Note that if $X$ is not purely of characteristic $p$, we can just take $D=(p)$.

Let $k=\mathrm{lct}(X,B+A,g^*D)\in\mathbb{Q}$.  After shrinking $V$ and replacing $k$  we may assume that all log canonical centers of $(X,B+A+kg^*D)$ intersect $F$.
 After tie breaking by changing $A$ up to linear equivalence, we may assume that $(X,B+A+kg^*D)$ has only one log canonical place. Note that to perform the tie breaking argument of \cite[Section 8.7]{CortiFlipsFor3FoldsAnd4Folds}, it is enough to have a log resolution with an ample exceptional divisor \autoref{proj-resolutions} and log Bertini for a sufficiently ample divisor \autoref{log_bertini}, which holds in complete generality in our setting.  Let $\pi \colon Y\to X$ be a $\mathbb{Q}$-factorial dlt modification of $(X,B+A+kg^*D)$, see \autoref{cor:dlt_models}, and let $K_Y+\Delta_Y=\pi^*(K_X+B+A+kg^*D)$, where we have $\rddown{\Delta_Y}:=S$ irreducible and therefore $K_Y+\Delta_Y$ is plt. The divisor $S$ is not disjoint from $F_Y$, the fiber of $Y\to V$ over $v$.    
Since $\pi$ has connected fibres, so does $\pi|_{F_Y}:F_Y\to F$ since this is set theoretically a union of fibres of $\pi$.  Hence by \cite[2.11(3)]{CasciniTanaka2020}, it is enough to show that $(K_Y+\Delta_Y)|_{F_Y}$ is semiample. Furthermore, the converse is also true since semi-ampleness is preserved under pullback.  We will use this trick repeatedly in what follows: if we take a morphism with connected fibres for which $K_Y+\Delta_Y$ descends or pulls back, it is enough to show semi-ampleness of $L$ restricted to the new fiber.

Run a $(K_Y+\Delta_Y- S)$-MMP over $V$ with scaling of $S$ (which terminates by \autoref{prop:psef_termination} { as $K_Y+\Delta_Y - S \equiv_V -S$ is pseudo-effective over $V$ being equivalent to an effective $\bQ$-divisor $af^*D - S$ for $a \gg 0$})
to reach $Y'$ on which $-S$ is nef over $V$.  By construction this cannot have contracted $S$, as each step intersects it positively.
Again, the fiber $F_{Y'}$ over $v\in V$ is not disjoint from $S$.  But any curve $\Gamma$ in $F_{Y'}$ satisfies $S\cdot\Gamma\leq 0$ and so  $F_{Y'}$ is either contained in $S$ or disjoint from it.  However we know that it cannot be disjoint, and so $F_{Y'}\subset S$.  The divisor $K_Y+\Delta_Y$ is trivial for every step in the prior MMP since $K_Y+\Delta_Y\equiv_V 0$, and so it descends to every step.  
As a result, by repeatedly applying \cite[2.11(3)]{CasciniTanaka2020} at every step of the MMP as explained above, it is enough to show that  $(K_{Y'}+\Delta_{Y'})|_{F_{Y'}}\sim_{\mathbb{Q}}0$, and for this it is enough to see that $(K_{Y'}+\Delta_{Y'})|_{S'}$ is semiample, but this follows from  \autoref{abundance} and \autoref{cor.ThreefoldNormalityOfS}, since $(Y', \Delta_{Y'})$ is plt as it has the same non-klt places as $(Y,\Delta)$: which are $S'$ and $S$ respectively.
\end{proof}

{The proof of the base point free theorem for $\bR$-line bundles will be given in the next section (\autoref{thm:bpf_for_R_boundary}) as it requires the cone theorem.}

\subsection{Step 6: Cone theorem and Mori fiber spaces}

The first goal of this section is to prove the full cone theorem:

\begin{theorem} \label{thm:full-cone-theorem}
Let $(X,\Delta)$ be a  three-dimensional $\mathbb{Q}$-factorial dlt pair with $\bR$-boundary projective and surjective over $T$, which has positive dimension and no residue fields of characteristic $2$, $3$ or $5$.  Then there exists a countable collection {of} curves  $\{\Gamma_i\}$ such that 
\begin{enumerate}
    \item\label{itm:new_cone_rays} 
\[
\overline{\mathrm{NE}}(X/T) = \overline{\mathrm{NE}}(X/T)_{K_X+\Delta\geq 0} + \sum_i\mathbb{R} [\Gamma_i],
\]
\item\label{itm:new_cone_accumulate} The rays $\mathbb{R}[\Gamma_i]$ do not accumulate in the half space $(K_X+{\Delta})_{<0}$,
\item\label{itm:new_cone_bound}
For each $\Gamma_i$, \[-4d_{\Gamma_i}<(K_X+\Delta)\cdot \Gamma_i<0 \] where $d_{\Gamma_i}$ is such that for any Cartier divisor $L$ on $X$, we have $L\cdot \Gamma_i$  divisible by $d_{\Gamma_i}$.
\end{enumerate}
\end{theorem}

The cone theorem is a formal consequence of \autoref{lem:nef_threshold}, our proof of which is inspired by the flip case of \cite[Lemma 3.2]{BW17}.  {We are unable to apply the other cone theorem arguments of \cite[Section 3]{BW17} directly due to the possibility that we work over a local base where general closed fibers need not exist.}

\begin{lemma}\label{lem:nef_threshold}
Let $(X,B)$ be a $\mathbb{Q}$-factorial {klt} threefold with $\mathbb{Q}$-boundary, projective and surjective over $T$ with $\dim(T)>0$ and having no residue fields of characteristic $2$, $3$ or $5$. Suppose that $K_X+B$ is not nef.  Then there exists an integer $n$ depending only on $(X,B)$ such that if $H$ is an ample Cartier divisor, and  \[\lambda=\min\{t\geq 0\mid K_X+B+tH\mathrm{\ is\ nef}\}\]
then $\lambda=\frac{n}{m}$ for some natural number $m$.  

Furthermore, there is a $(K_X+B+\lambda H)$-trivial curve  $\Gamma$ satisfying 
\[-4d_\Gamma\leq (K_X+B)\cdot \Gamma<0\]
{where $d_{\Gamma}$ is such that for any Cartier divisor $L$ on $X$, $L\cdot \Gamma$ is divisible by $d_{\Gamma}$.}
\end{lemma}

\begin{proof}

First suppose that $K_X+B+\lambda H$ is big. 
Then $K_X+B+(\lambda-\varepsilon)H$ is also big for sufficiently small $\varepsilon$, and by definition of $\lambda$, it fails to be nef.  By \autoref{thm:keel_cone}
there are only finitely many $(K_X+B+(\lambda-\varepsilon)H)$-negative extremal rays for $\varepsilon$ sufficiently small, and these rays are isolated.  Therefore at least one of these rays $R$ must satisfy $R\cdot L=0$, and $R$ has a projective contraction $f:X\to Z$ by \autoref{cor:bir-contraction} which contracts a curve $C$.  This satisfies $(K_X+B)\cdot C=-\lambda H\cdot C$ and therefore $\lambda$ is rational as $K_X+B$ and $H$ are both $\mathbb{Q}$-Cartier.  We now show that $f$ contracts a curve satisfying the required bound.

Suppose that $f$ is a divisorial contraction, contracting a divisor $S$. Let $A=\lambda H$, which we have just seen is $\mathbb{Q}$-Cartier, so that by \autoref{thm:bpf} $L=K_X+B+A\sim_{\mathbb{Q},Z} 0$.  Note that it is sufficient to find a curve $\Gamma $ such that 
\[0< A\cdot\Gamma\leq 4d_\Gamma.\]
Let $\phi:W\to X$ be a log resolution of $(X,B+S)$, let $B_W$ be the sum of the birational transform of $B$ and the reduced exceptional divisor of $\phi$, $S_W$ be the birational transform of $S$, and let $A_W=\phi^*A$.  By the projection formula, it is enough to find a curve $\Gamma_W$ on $W$ which satisfies 
\[0< A_W\cdot\Gamma_W\leq 4d_{\Gamma_W}.\]

Let $a$ be such that {$S$ has coefficient $1$ in $B+aS$}.  We have $K_W+B_W+A_W+aS_W\sim_{\mathbb{R},Z} E+aS_W$ for some exceptional$/X$ effective {$\bQ$-}divisor $E$.  This means that $E+S_W$ is in fact effective and exceptional over $Z$, and $\mathrm{Ex}(f\circ\phi)=\rddown{B_W+A_W+aS_W}$.   Run a $(K_W+B_W+A_W+aS_W)$-MMP over $Z$, which must terminate on $Z$ by the negativity lemma and the fact that $Z$ is $\mathbb{Q}$-factorial.  Suppose that a step $W\dashrightarrow W'$ of this MMP contracts a curve over $X$.  Then $A_W$ descends to $A_{W'}$ for it is a pullback from $X$, and again it is enough to find a curve {$\Gamma_{W'}$} satisfying \[0< A_{W'}\cdot\Gamma_{W'}\leq 4d_{\Gamma_{W'}}.\]  We are reduced to the same problem for the next step of the MMP.  As the MMP terminates on $Z$, we must eventually reach a step contracting a ray $R$ which is not over $X$.  Then as $A_W$ is ample over $X$, we have $A_W\cdot\Gamma>0$ and so the step is also negative for $K_W+B_W+aS_W$.  But since this MMP is negative for $E+aS_W$, whose support is equal to the reduced boundary, we can choose a component $F$ of $E+S_W$ on which $R$ is negative.  By restricting to $F$ and applying adjunction \cite[Section 4.1]{KollarKovacsSingularitiesBook}, we find that $(K_W+B_W+aS_W)|_F=K_F+B_F$ for some dlt pair $(F,B_F)$.  If $F$ has dimension $1$,  then it follows that $F=\Gamma$ satisfies \[-2d_{\Gamma_W}{\leq (K_W+\Gamma_W)\cdot\Gamma_W}\leq (K_W+B_W+aS_W)\cdot\Gamma_W<0\]
e.g. by \cite[Lemma 4.4]{DW19}.
Meanwhile if $F$ is two dimensional  we
see by \autoref{thm:surface_cone} that there is a curve $\Gamma_W\subset F$ in $R$ satisfying 
\[-4d_{\Gamma_W}\leq (K_W+B_W+aS_W)\cdot\Gamma_W<0.\]
{In either case, } since the ray is also negative for $K_W+B_W+A_W+aS_W$, we  find that \[0< A_W\cdot\Gamma_W\leq 4d_{\Gamma_W}\] as required.

Now suppose that $f:X\to Z$ is a flipping contraction, and $z\in Z$ is the image of the flipping locus.  In this case, the argument for the flipping case in  \cite[Lemma 3.2]{BW17} applies directly, with the only change being to insert ${d_\Gamma}$ in appropriate places. {The reference to \cite[3.4]{Birkar16} in \cite[Lemma 3.2]{BW17} can be replaced by {the argument in the first paragraph of this proof using  \autoref{thm:keel_cone}.}}

Next suppose that {the $\bR$-divisor} $L=K_X+B+\lambda H$ is not big.  
Let $\xi$ be the generic point of $T$.  By \autoref{thm:surface-bpf-theorem}, $L|_{X_{\xi}}$  is semiample, and by \autoref{prop:char_zero_bpf}  we may assume that  $L|_{X_{\mathbb{Q}}}$ is semiample, if this fiber is non-empty. As $L$ is not big, there is a curve $C$ in $X_{\xi}$ (over the residue field of $\xi$) which is contracted by the induced map. This satisfies 
\[(K_X+B)|_{X_{\xi}}\cdot C=-\lambda H|_{X_{\xi}}\cdot C\] and therefore because $K_X+B$ is a $\mathbb{Q}$-Cartier $\bQ$-divisor and $H$ is an ample Cartier divisor, $\lambda\in\mathbb{Q}$ and $L$ is a $\mathbb{Q}$-divisor.
{Let $A=\lambda H$, where after changing $A$ up to $\mathbb{Q}$-linear} equivalence we may assume that $(X,B+A)$ is klt {(see \autoref{lemma:add_ample})}.  $L$ is semiample by \autoref{thm:MMP_bpf},
and so let $f:X\to Z$ be the induced contraction.  We may assume that $Z$ is normal and projective over $T$. 

Choose a Cartier divisor $D_Z\subset Z$. Let $\pi:W\to X$ be a dlt {modification} of $(X,(B+A+f^*D_Z)^{\leq 1})$ (see \autoref{cor:dlt_modification}), where $D^{\leq 1}$ denotes the divisor obtained by truncating the coefficients of $D$ at $1$.  Then let $A_W=\pi^*A$ and $B_W$ be the sum of the strict transform of $B$ and the unique effective $\mathbb{Q}$-divisor necessary to ensure that $B_W+A_W$ has coefficient one at every component of {$\Supp(\pi^*f^*D_Z)$}.
As in the divisorial case, it suffices to find a curve $\Gamma$ on $W$ which satisfies
$0<A_W\cdot\Gamma\leq 4d_{\Gamma}$.

We have \begin{equation}\label{eqn:contained_in_fiber}K_W+B_W+A_W\sim_{\mathbb{Q},Z} E\end{equation} where $E$ is effective and each component of $E$ is supported over $D_Z$. {In particular this implies that $K_W+B_W+A_W$ is not big over $Z$.}
Furthermore, $\rddown{B_W+A_W}$ and $E$ both contain every component of $\Supp(\pi^*f^*D_Z)$.

Run a $(K_W+B_W+A_W)$-MMP over $Z$, which exists and terminates by \autoref{prop:psef_termination}.  If the first step of the MMP, 
$W\dashrightarrow W'$, is over $X$  then exactly as before, $A_W$ descends to $A_{W'}$, and so we may replace $W$ by $W'$ and continue.
On the other hand, suppose that a step of the MMP contracting a ray $R$ is not over $X$.  As before, since $A$ is ample on $X$ we see that $A_W\cdot R>0$, and as a result $(K_W+B_W)\cdot R<0$
But as \autoref{eqn:contained_in_fiber} implies that the curves contracted are contained in the reduced boundary, we find a  curve $\Gamma_W$ which satisfies
\[-4d_{\Gamma_W}\leq (K_W+B_W)\cdot\Gamma_W<0\] by the log canonical case of \autoref{thm:surface_cone}.  But as this ray was chosen to be negative for $K_W+B_W+A_W$, it follows that we must also have $0<A_W\cdot \Gamma_W\leq 4d_{\Gamma_W}$ as required.

Hence we may assume that the entire MMP is over $X$ and terminates with a model $Y$ with maps $\phi:Y\to X$ and $\psi:Y\to Z$,
and such that $K_Y+B_Y+A_Y$ is nef over $Z$, where $A_Y=\phi^*A$.  
Now $K_Y+B_Y+A_Y\sim_{\mathbb{Q},\psi}K_Y+B_Y+A_Y-\varepsilon \psi^*D_Z$, and the pair $(Y,B_Y+A_Y-\varepsilon \psi^*D_Z)$ is klt for any sufficiently small $\varepsilon$.  Hence by \autoref{thm:MMP_bpf}, using the fact that $A_Y$ is big, we see that $K_Y+B_Y+A_Y$ is $\psi$-semi-ample.  
{Let $\sigma:Y\to V$ be the morphism induced by $K_Y+B_Y+A_Y$, so that $K_Y+B_Y+A_Y\sim_{\phi}0$.  Since $K_Y+B_Y+A_Y$ is not big over $Z$, $\dim(V)<\dim(Y)$.  We have varieties and morphisms:}
{
$$\xymatrix{Y\ar^\phi[r]\ar^\sigma[d]\ar^\psi[dr] & X\ar[d]\\
V\ar[r] & Z.}$$
}
Choose a component 
$S$ of $\psi^*D_Z$ which is not contracted over $X$, or equivalently which is the strict transform on $Y$ of a component of $f^*D_Z$.
As $S$ is not contracted over $X$, $A_Y|_S$ is big. {However, since $\dim(S)\geq\dim(V)$, and $S$ is not horizontal over $Z$ and hence not horizontal over $V$, we see that $S$ is contracted over $V$.}

{Hence $S$ contains a curve $C$
which is vertical over $V$ and which satisfies $A_Y|_S\cdot C>0$, since $A_Y|_S$ is big and $S$ is contracted over $V$. 
{The divisor} $S$ is contained in $\rddown{B_Y+A_Y}$, so by adjunction let}
$K_S+B_S=(K_Y+B_Y)|_S$, and if $S$ is one dimensional {(as in Remark \ref{remark:divisors-of-unexpected-dimension})}, set $\Gamma=S$.  {The latter satisfies the required bounds exactly as in the birational case above.}  Otherwise {apply the cone theorem} \autoref{thm:surface_cone} over $Z$ to $K_S+B_S$. This finds an extremal ray which is $(K_S+B_S+A|_S)$-trivial {(here we use that $K_Y+B_Y+A_Y\sim_{\sigma}0$)}  and so $A|_S$-positive  which contains a curve $\Gamma$ such that
\[-4d_\Gamma\leq (K_S+B_S)\cdot\Gamma= (K_Y+B_Y)\cdot\Gamma=-A_Y\cdot\Gamma<0.\] 
{Taking the pushforward of $\Gamma$ to $X$ gives the required curve as in the birational case, since $A_Y$ is the pullback of $A$ from $X$ and the curve $\Gamma$ is contracted over $Z$. }

Now to prove the statement about $\lambda$,  let $I$ be the Cartier index of $K_X+B$.  Then we have that $I(K_X+B)\cdot \Gamma$ is an integer divisible by $d_\Gamma$, and so is an integer between $-4I$ and $-1$. 

{Since} 
\[
\lambda=\frac{-I(K_X+B)\cdot\frac{\Gamma}{d_{\Gamma}}}{IH\cdot\frac{\Gamma}{d_{\Gamma}}},
\]
 we can take $n=(4I)!$.
\end{proof}

\begin{definition}\label{def:extremal_curve}
Let $X$ be a scheme with a projective morphism $f:X\to T$, and $R$ an extremal ray over $T$.  Let $H$ be an $f$-ample Cartier divisor on $X$. We say that a curve $\Gamma\in R$ is \emph{extremal} if \[H\cdot\frac{\Gamma}{d_{\Gamma}}=\min\{H\cdot\frac{C}{d_C}\mid C\in R\}.\] 
\end{definition}

The extremality of a curve does not depend on the ample divisor $H$, since if $H'$ is a different ample divisor, there is $\lambda>0$ such that  $H\cdot C= \lambda H'\cdot C$ for any $C$ in $R$, and so 
\[\frac{H\cdot\frac{\Gamma}{d_\Gamma}}{H'\cdot\frac{\Gamma}{d_{\Gamma}}}=\frac{H\cdot\frac{C}{d_C}}{H'\cdot\frac{C}{d_{C}}}\] for any other curve $C\in R$.  Similarly, if $D$ is a $\mathbb{Q}$-divisor such that $D\cdot R<0$, we have 
\[D\cdot\frac{\Gamma}{d_{\Gamma}}=\max\{D\cdot\frac{C}{d_C}\mid C\in R\}.\]

Finally we are ready to prove the cone theorem. 

\begin{proof}[Proof of \autoref{thm:full-cone-theorem}]
If we assume that $\Delta$ is a $\mathbb{Q}$-divisor, and $(X,\Delta)$ is klt, $(a)$ and $(b)$ follow formally from \autoref{lem:nef_threshold} using \cite[Theorem 3.15]{KollarMori} {(one can also use the standard proof of the cone theorem in the smooth case \cite[Theorem 1.24]{KollarMori} as \autoref{lem:nef_threshold} recovers a singular variant of Mori's bend-and-break  \cite[Theorem 1.13]{KollarMori})}.

 Now suppose that $\Delta$ is an $\mathbb{R}$-divisor or $(X,\Delta)$ is not klt. We first prove  that there are only countably many $(K_X+\Delta)$-negative extremal rays and that they do not accumulate in $(K_X+\Delta)_{<0}$.   For each integer $n$, choose a klt $\mathbb{Q}$-boundary $\Delta_n$ such that $\Supp(\Delta_n)=\Supp(\Delta)$ and $|\Delta-\Delta_n|<\frac{1}{n}$.  Each $(K_X+\Delta)$-negative extremal ray is $(K_X+\Delta_n)$-negative for some $n$, and so the collection of $(K_X+\Delta)$-negative extremal rays is a subset of a countable union of countable sets, hence countable.  Furthermore, if there is a sequence of $(K_X+\Delta)$-negative extremal rays which accumulate in $(K_X+\Delta)_{<0}$,  then they accumulate to a ray in $(K_X+\Delta_n)_{<0}$ for some $n\gg0$.  Therefore by truncating the sequence of extremal rays we obtain a  sequence of $(K_X+\Delta_n)$-negative rays which accumulate in $(K_X+\Delta_n)_{<0}$, contradicting $(b)$ in the klt $\mathbb{Q}$-divisor case.

{Now we move to (c). Let $R$ be a $(K_X+\Delta)$-negative extremal ray.  Then let $\Delta_n$ be a sequence of klt $\mathbb{Q}$-boundaries with $\lim_n\Delta_n=\Delta$, and such that $R$ is $(K_X+\Delta_n)$-negative for every $n$.  For each $n$, we can find an ample divisor $A_n$ such that $R=(K_X+\Delta_n+A_n)^\perp$, and then \autoref{lem:nef_threshold} shows that there is a curve $C_n$ in $R$, which satisfies \[-4d_{C_n}\leq (K_X+\Delta_n)\cdot C_n<0\]
for every $n$.   Then as $R$ contains a curve, it contains an extremal curve $C$, which satisfies 
\[-4\leq(K_X+\Delta_n)\cdot \frac{C_n}{d_{C_n}} \leq (K_X+\Delta_n)\cdot \frac{C}{d_C}<0.\]
It then follows that
\[-4d_C\leq (K_X+\Delta)\cdot C<0\] as required.} 
\end{proof}

Our next result is finiteness of log minimal models, for which we first recall the setup. 

\begin{setup}\label{setup:shokurov_polytope}
Let $X$ be a three dimensional, klt $\mathbb{Q}$-factorial integral scheme, projective over $T$, such that the image of $X$ in $T$ is positive dimensional.  Let $A\geq 0$ be a $\mathbb{Q}$-divisor and $V$ a finite dimensional rational affine subspace of the vector space of $\bR$-Weil divisors on $X$.  Then we define the Shokurov polytope 
$$\mathcal{L}_A(V)=\{\Delta \mid 0\leq \Delta-A\in V \textrm{\ and \ } (X,\Delta) \textrm{\ log canonical}\}.$$
As we know that projective log resolutions exist in our situation, this is a rational polytope by \cite[1.3.2]{ShokurovThreeDimensionalLogFlips}.
\end{setup}

The proof of the next proposition closely follows that given in \cite[Proposition 3.8]{BW17}.   Note that the proofs of parts (4) and (5) of \cite[Proposition 3.8]{BW17} do not work in our situation, but we do not need them.

\begin{proposition}\label{prop:polytope}
    Let $X$, $T$, $V$, and $\mathcal{L}$ be as above, and fix $B\in \mathcal{L}$.  Then there are real numbers $\alpha, \delta>0$ depending only on $(X,B)$ and $V$, such that 
    \begin{enumerate}
        \item\label{itm:bnd_1} If $\Gamma$ is an extremal curve on $X$ and if $(K_X+B)\cdot\Gamma>0$, then $(K_X+B)\cdot\frac{\Gamma}{d_\Gamma}>\alpha$,
        \item\label{itm:bnd_2} if $\Delta\in\mathcal{L}$ and $||\Delta-B||<\delta$ and $(K_X+\Delta)\cdot R\leq 0$ for an extremal ray $R$ then $(K_X+B)\cdot R\leq 0$.
        \item\label{itm:nef_polytope} {Let $\{R_t\}_{t\in S}$ be a family of extremal rays of $\overline{NE}(X/T)$.  Then the set 
        \[
        \mathcal{N}_S=\{\Delta\in\mathcal{L}\mid (K_X+\Delta)\cdot R_t\geq 0\mathrm{\ for\ any\ }t\in S\}
        \]
        is a rational polytope.}
    \end{enumerate}
\end{proposition}
\begin{proof}
{The proofs of the corresponding statements in \cite[Proposition 3.8]{BW17} work here, by replacing every appearance of a curve $\Gamma$ with $\frac{\Gamma}{d_{\Gamma}}$.}
\end{proof}

The following base point free theorem for $\mathbb{R}$-divisors is used in the upcoming proof of finiteness of log minimal models.

\begin{theorem} \label{thm:bpf_for_R_boundary}
{Let $(X,B)$ be a $\mathbb{Q}$-factorial three dimenaional klt pair with $\mathbb{R}$-boundary, projective over $T$ and such that the image of $X$ in $T$ has positive dimension and that none of the residue fields of $T$ have characteristic $2$, $3$ or $5$.  Suppose that $D$ is a nef $\mathbb{R}$-divisor such that $D-(K_X+B)$ is nef and big.  Then $D$ is semiample.}
\end{theorem}

\begin{proof} 
{Let $A=D-(K_X+B)$. It is sufficent to prove the statement after localizing at a point $t\in T$.  Thus we may change $A$ and $B$ using \autoref{lemma:add_ample} to assume that $(X,\Delta:=B+A)$ is klt and $A$ is an ample $\mathbb{Q}$-divisor.  By \autoref{prop:polytope}\autoref{itm:nef_polytope} there are $\mathbb{Q}$-boundaries $\Delta_j$ such that $\Delta=\sum_ja_j\Delta_j$ for $a_j>0$, $||\Delta-\Delta_j||$ are sufficiently small, $\Delta_j \geq A$, $(X,\Delta_j)$ are klt and $K_X+\Delta_j$ are all nef.  By \autoref{thm:bpf} $K_X+\Delta_j$ are all semiample, so $K_X+\Delta$ is also semiample.}
\end{proof}

\begin{theorem}\label{thm:finiteness_of_models}
In the situation of \autoref{setup:shokurov_polytope}, assume that $A$ is also big over $T$, and the image of $X$ in $T$ is positive dimensional.  Let $\mathcal{C}\subset \mathcal{L}_A(V)$ be a rational polytope such that $(X,B)$ is klt for every $B\in\mathcal{C}$.  Then there exist finitely many birational maps $\phi_i:X\dashrightarrow Y_i$ over $T$ such that for each $B\in\mathcal{C}$ for which $K_X+B$ is pseudo-effective over $T$, there is some $i$ such that $(Y_i,B_{Y_i})$ is a log minimal model of $(X,B)$ over $T$. 
\end{theorem}
\begin{proof}
The proof is identical to that of \cite[Proposition 4.2]{BW17}, with the inputs being \autoref{prop:polytope}, the base point free theorem \textcolor{ForestGreen}{\autoref{thm:bpf_for_R_boundary}} and the existence of log minimal models \autoref{prop:psef_termination}.
\end{proof}

\begin{theorem}\label{thm:termination_scaling}
 Let $(X,B)$ be a $\mathbb{Q}$-factorial three-dimensional klt pair with $\bR$-boundary, projective over $T$, such that the image of $X$ in $T$ has positive dimension and that none of the residue fields of $T$ have characteristic $2$, $3$ or $5$.   Suppose $A$ is an ample $\mathbb{R}$-divisor such that $K_X+B+A$ is nef over $T$.  Then we can run the $(K_X+B)$-MMP over $T$ with scaling of $A$ and it terminates.
\end{theorem}
\begin{proof}
This follows by the arguments of \cite[Proposition 4.3]{BW17} using \autoref{thm:finiteness_of_models}, and \cite[Proof of Theorem 1.6]{BW17} except that we replace the reference to \cite[Proposition 4.5]{BW17} with \autoref{prop:psef_termination}.
\end{proof}

Note that if we assume that $T$ is a curve with finitely many closed points, for instance if $T=\Spec(\mathbb{Z}_p)$, we get a stronger termination result:

\begin{proposition}\label{prop:termiation_finitely_many_points}
    Let $(X,B)$ be a $\mathbb{Q}$-factorial three dimensional klt pair with $\bR$-boundary projective and surjective over $T$ of positive dimension, and $T$ has only finitely many closed points, none of which have residue fields of characteristic $2$, $3$ or $5$.  Then any sequence of $(K_X+B)$-flips terminates.
\end{proposition}
\begin{proof}
{By  \autoref{thm:special-termination}, after finitely many flips both the flipping and flipped loci are disjoint from the birational transform of the boundary.  Given this, note that any $(K_X+B)$-MMP is also a $(K_X+B+\varepsilon \sum_i F_i)$-MMP {for $0 < \varepsilon \ll 1$} where $F_i$ are the pullbacks of Cartier divisors on $T$ which contain the finitely many closed points of $T$, and so the flips eventually terminate.}
\end{proof}

\begin{theorem}\label{thm:MFS}
{Let $(X,B)$ be a three-dimensional $\mathbb{Q}$-factorial dlt pair, with $\bR$-boundary, projective over $T$ such that the image of $X$ in $T$ has positive dimension and none of the residue fields of $T$ have characteristic $2$, $3$ or $5$.   Suppose that $K_X+B$ is not pseudo-effective over $T$.  Then we can run a $(K_X+B)$-MMP with scaling of an ample divisor which terminates with a Mori fiber space.}
\end{theorem}
\begin{proof}
If $(X, B)$ is klt, this follows by combining the  \autoref{thm:full-cone-theorem}, \autoref{thm:bpf_for_R_boundary}, \autoref{thm:full_flips}, and \autoref{thm:termination_scaling}.  

If it is not klt, fix an ample divisor $A$ and run a $(K_X+B)$-MMP with scaling of $A$.  The cone theorem holds by \autoref{thm:full-cone-theorem}, contractions and flips exist by perturbing the boundary to a klt boundary and then applying \autoref{thm:bpf} and \autoref{theorem:flips-exist}.  It remains to show termination.

Fix $\delta$ sufficiently small that $K_X+B+\delta A$ is not pseudo-effective over $T$.  Now choose $\varepsilon\ll \delta$ sufficiently small that  $\varepsilon B+\delta A$ is ample over $T$.
Note that since $K_X+B+\delta A$ is not pseudo-effective, a $(K_X+B)$-MMP with scaling of $A$ is also a $(K_X+B+\delta A)$-MMP with scaling of $(1-\delta)A$.  

For any point $t\in T$, we may localize over $t$, apply \autoref{lemma:add_ample} and then spread out over some open subset $t\in U\subset T$ and its preimage $X_U$ in $X$, to find a divisor $H\sim_{\mathbb{R}}\varepsilon B_{X_U}+\delta A_{X_U}$ such that $(X_U,(1-\varepsilon)B_{X_U}+H)$ is klt.
Therefore by \autoref{thm:termination_scaling} our MMP terminates over $U$ since it is also an MMP for $K_{X_U}+(1-\varepsilon)B_{X_U}+H$.  Since we can cover $T$ with finitely many such open sets, we see that the $(K_X+B)$-MMP with scaling of $A$ terminates everywhere.
\end{proof}

%% file: applications.tex
\section{Applications to moduli of stable surfaces}
\label{sec:applications}

The goal of this section is to show the existence of the moduli stack $\overline{\sM}_{2,v}$ of stable surfaces of volume $v$ over $\bZ[1/30]$ as an Artin stack with finite stabilizers and of finite type over $\bZ[1/30]$. By the Keel-Mori theorem \cite{Keel_Mori_Quotients_by_groupoids,Conrad_The_Keel_Mori_theorem_via_stacks} this then  also implies the existence of the coarse moduli space $\overline{\mathrm{M}}_{2,v}$ of stable surfaces of volume $v$  as an algebraic space over $\bZ[1/30]$.  We refer to \cite[Sec 1.3]{PatakfalviProjectivityStableSrufaces} for the precise definitions of the moduli functor of $\overline{\sM}_{2,v}$. 

The starting point is that in \cite[Thm 9.7]{PatakfalviProjectivityStableSrufaces} it was proven that one has to only show a special case of inversion of adjunction: if $f : X \to T $ is a $1$-parameter flat projective family of geometrically demi-normal varieties with semi-log central fiber, then $X$ is semi-log canonical. By passing to the normalization of $X$ this follows from the log canonical inversion of adjunction. So, this version of inversion of adjunction is  our first goal, which is a consequence of the following existence statement for dlt-models.

\begin{corollary}[Log canonical inversion of adjunction]
\label{cor:lc_inversion_of_adjunction}
In the situation of \autoref{MMP_setting} suppose that none of the residue fields of $R$ have characteristic $2$, $3$ or $5$.  Let $(X,D)$ be a normal pair of dimension $3$ such that $K_X + D$ is $\bQ$-Cartier, and with a prime divisor $S$  that has coefficient $1$ in $D$.  Let $S^N$ be the normalization of $S$.  If $(S^{\mathrm{N}}, D_{S^{\mathrm{N}}})$ is log canonical, where  $D_{S^{\mathrm{N}}}$ is the different of $D$ along $S$, then so is $(X,D)$ in a neighborhood of $S$.
\end{corollary}

\begin{proof}

Consider a $\bQ$-factorial dlt-model $g\colon (Z, \Gamma) \to (X,D)$ constructed in \autoref{cor:dlt_models}. Here $\Gamma$ is the boundary used in \autoref{cor:dlt_models}, that is, it can be obtained by  lowering to $1$ all the greater than $1$ coefficients of $g^{-1}_* D$  and additionally adding in all the $g$-exceptional divisors  with coefficient $1$. 
Let $T$ be the component of $\Gamma$ dominating $S$.
Since $Z$ is $\bQ$-factorial we use \autoref{cor.ThreefoldNormalityOfS} and a pertubation argument to see that  $T$ is normal.  We fix the following notation for the induced morphisms:
\begin{equation*}
\xymatrix{
T \ar@/^1pc/[rr]^{\gamma} \ar[r]_{\alpha} & S^{\mathrm{N}} \ar[r]_{\beta} & S.
}
\end{equation*}
Let $\Delta$ be the crepant boundary on $Z$, that is for which $K_Z + \Delta = {g}^* (K_X + D)$.  Note that by point \autoref{itm:cor:dlt_modification:effective} of \autoref{cor:dlt_modification},  $\Delta - \Gamma$ is effective and it is non-zero exactly at {each} prime divisor $E$ of $Z$ for which $\coeff_E \Delta >1$. 
Note also that $(T, \Delta_T)$ is a crepant dlt-model for $\left(S^{\mathrm{N}}, D_{S^{\mathrm{N}}} \right)$,  where $\Delta_T$ is the different of $\Delta$ along $T$.  In other words, we have that 
\begin{equation}
\label{eq:lc_inversion_of_adjunction:crepant}
K_T + \Delta_T=\alpha^* \left( K_{S^{\mathrm{N}}} + D_{S^{\mathrm{N}}} \right).
\end{equation}
 Additionally, $(T, \Gamma_T)$ is dlt, where $\Gamma_T$ is the different of $\Gamma$ along $T$.
As $\left(S^{\mathrm{N}}, D_{S^{\mathrm{N}}}\right)$ is log canonical,  by  \autoref{eq:lc_inversion_of_adjunction:crepant}, we see that   
the coefficients of $\Delta_T$ are at most $1$.    By the surface inversion of adjunction applied at the codimension $1$ points of $T$, this means that the  coefficients of $\Gamma$  are at most $1$ in a neighborhood of $T$.  We  note that here we crucially use the $\bQ$-factoriality of $Z$, which implies that divisors on $Z$ can only meet $T$ in codimension $1$ points of $T$. Since at all divisors in  $ \Supp (\Delta - \Gamma)$, the coefficient of $\Delta$ is $1$,  we obtain that  the divisors $\Delta$ and $\Gamma$ agree in a neighborhood of $T$.
However, \autoref{cor:dlt_modification}\autoref{rem:dlt_model_connected} tells us that for each fiber, $\Supp (\Delta - \Gamma)$ either contains it or is disjoint from it. So, we obtain that $g\big(\Supp (\Delta - \Gamma) \big)$ is a closed set that is disjoint from $S$. This concludes our proof as $(X, D)$ is log canonical over $X \setminus g \big(\Supp (\Delta - \Gamma) \big)$.
\end{proof}
In fact, we believe that the above result works even when $R$ has arbitrary residue characteristics, by using the non-$\bQ$-factorial dlt modification as in \autoref{rem:non-q-factorial-dlt-modification} and replacing $T$ in the proof by its normalisation.

In the proofs of the following statements we use the language of almost Cartier divisors on $S_2$ and $G_1$, Noetherian  schemes, as introduced in \cite{HartshorneGeneralizedDivisorsOnGorensteinSchemes}, for the canonical divisor of demi-normal schemes and their one-parameter families. Furthermore, for such families the canonical divisor is compatible with base-change, as they contain a relatively Gorenstein  open set, the complement of which has codimension two in every fiber (for the arbitrary Gorenstein base-change see \cite[Sec 3.6]{ConradGDualityAndBaseChange}).

\begin{corollary}[{Existence of $\overline{\sM}_{2,v}$ over $\bZ[1/30]$}]
\label{cor:moduli_exists}
With notation as above:
\begin{enumerate}
    \item 
    \label{itm:M_2_v_exists:fine}
The moduli stack $\overline{\sM}_{2,v}$ of stable surfaces of volume $v$ over $\bZ[1/30]$ exists as a separated Artin stack of finite type over $\bZ[1/30]$ with finite diagonal.
\item 
    \label{itm:M_2_v_exists:coarse}
The  coarse moduli space $\overline{\mathrm{M}}_{2,v}$ of stable surfaces of volume $v$  over $\bZ[1/30]$  exists as a separated algebraic space of finite type over $\bZ[1/30]$.
\end{enumerate}
\end{corollary}

\begin{proof}
Point \autoref{itm:M_2_v_exists:coarse} follows from point \autoref{itm:M_2_v_exists:fine} using 
the  Keel-Mori theorem \cite{Keel_Mori_Quotients_by_groupoids,Conrad_The_Keel_Mori_theorem_via_stacks}. So, we only have to show \autoref{itm:M_2_v_exists:fine}.
By
\cite[Thm 9.7]{PatakfalviProjectivityStableSrufaces}
we have to show that if $f: X\to T$ is a flat family of geometrically demi-normal projective schemes over the spectrum of a DVR with $t$ being the closed point and $X_{\overline{t}}$ being a stable {surface}, then $X$ has slc singularities. (We note that \cite[Thm 9.7]{PatakfalviProjectivityStableSrufaces} is based on \cite{Hacon_Kovacs_On_the_boundedness_of_SLC_surfaces_of_general_type}, \cite{Alexeev_Boundedness_and_K_2_for_log_surfaces} and \cite{KollarHullsAndHusks}.)

{First, we show the corollary under an assumption that $X_t$ is slc.} Let $g : (Y,D) \to X$ be the normalization, where $D$ is the conductor. As $X$ is demi-normal, $D$ has only coefficients $1$. We have to show that $(Y,D)$ is log canonical.  Note that as $X$ is regular at every generic point of every fiber of $f$, $Y  \to X$ is an isomorphism at these points. In particular $Y_t \to X_t$ is an isomorphism around the generic points of $X_t$. As $Y$ is $S_2$, $Y_t$ is $S_1$. So, all embedded points of $Y_t$ are at generic points which implies that $Y_t$ is reduced. Hence, the normalization of $Y_t$ and  of $X_t $ agree. Let us write $\delta : Z := X_t^{\mathrm{N}} \to {X_t}$ for this normalization. {Take the boundary $D_Z$ on $Z$ which is crepant to $(Y,D)$, that is, $K_Z + D_Z =  \alpha^* (K_Y + D)$, where $\alpha : Z  \to Y_t \to Y$ is the induced composition morphism. In fact, this boundary is also crepant to $X_t$, that is $K_Z + D_Z = \delta^* K_{X_t}$. This follows from the fact that { both $K_{X_t}$ and $K_Y+D$ are pullbacks of $K_X$. To sum up, we have the following commutative diagram, where every arrow connects crepant equivalent pairs (i.e., the log-canonical divisors  are compatible via pull-backs by any of the arrows):} 
{\begin{equation*}
\begin{tikzcd}[column sep=70pt,row sep=20pt]
  (Z,D_Z) \arrow{d}[swap]{\alpha} \arrow{r}{\delta}[swap]{\textrm{normalization}} & X_t \arrow{d}{\textrm{central fiber}} \\
  (Y, D) \arrow{r}{g}[swap]{\textrm{normalization}} & X.
\end{tikzcd}
\end{equation*}
}
} 
By the definition of $X_t$ being slc, $(Z,D_Z) $ is lc, hence by \autoref{cor:lc_inversion_of_adjunction} $(Y,D)$ is also lc, and hence $X$ is slc. 

{Second, we show} that $X_{\overline{t}}$ being slc implies that $X_t$ is slc (note: we know that $X_t$ is geometrically demi-normal and hence geometrically reduced). This is a standard argument: we need to show that $(Z, D_Z)$ is log canonical. Let $\rho \colon V \to Z$ be a log resolution of singularities with $D_V$ {{so that $K_V + D_V = \rho^*(K_Z + D_Z)$. In other words $D_V$ is a crepant sub-boundary}}. We need to show that $D_V$ has coefficients at most $1$.

Let $\xi : W \to V_{\overline{k}}$ be the normalization  of $V_{\overline{k}}$, where $k=k(t)$. {Let $D_W$ be a $\bQ$-divisor on $W$ such that $K_W+ D_W = \xi^* (K_V + D_V)_{\overline{k}}$. It is crepant to both $(V,D_V)$ and to $X_{\overline{t}}$; in the latter case, we use that $\nu^*K_{X_t} = K_{X_{\overline t}}$ as relative canonical divisors are stable under base change.}
Let $\phi: W \to V$ be the induced morphism, and let $E$ be the boundary on $W$ that makes $(W, E) \to V$ crepant.  In other words, as $V$ is geometrically reduced by being generically isomorphic to $X_t$, $E$ is the conductor of $W \to V_{\overline{k}}$. In particular $E \geq 0$. By the definition of the respective divisors we see that $D_W = E + \phi^* D_V $. {To sum up, we have the following commutative diagram, where every arrow connects crepant equivalent pairs:}
{\begin{equation*}
\begin{tikzcd}[column sep=70pt,row sep=20pt]
(W,D_W) \arrow{rd}[swap]{\phi} \arrow{r}{\xi}[swap]{\textrm{normalization}} &    (V_{\overline k},D_{\overline k}) \arrow{d} \arrow{rr} &  &   X_{\overline t} \arrow{d}{\textrm{base-extension to } \overline{k(t)}}[swap]{\nu} \\
&  (V,D_V) \arrow{r}{\rho}[swap]{\textrm{log-resolution}} & (Z,D_Z) \arrow{r}{\delta}[swap]{\textrm{normalization}} & X_t
\end{tikzcd}
\end{equation*}}
As $X_{\overline{t}}$ is slc, and $(W, D_W)$ is crepant to $X_{\overline{t}}$, we obtain that the coefficients of $D_W$ are at most $1$. Coupling this up with the equation $D_W = E + \phi^* D_V $ and with the effectivity of $E$, we obtain that the coefficients of $D_V$ are in fact at most $1$ too.
\end{proof}

 \autoref{cor:moduli_exists} implies different modular lifting statements on stable varieties. A sample one is the following which gives the lifting to be over a localisation of a finite extension of $\bZ$ (alas, we need to assume that the base field is finite). One can also show that if the surface is defined over a perfect field $k$, then there exists a lifting over $W(k)$.

\begin{corollary}
\label{cor:lifting_stable_surfaces}
For every rational number $v >0$ there exists a prime $p(v)$ with the following property: for all stable surfaces $X$ of volume $v$ over a finite field of characteristic $p \geq p(v)$, there is a  family of stable surfaces $\sX$ over an open set  of the spectrum of the ring of integers of a number field such that $X$ is a fiber of $\sX$.
\end{corollary}

\begin{remark}
The  point of \autoref{cor:lifting_stable_surfaces}, where we think that \autoref{cor:moduli_exists} is essentially used, is that it states a lifting to a stable family, not only to an arbitrary flat family. We think that for this type of application one essentially needs the openness of the stable locus in adequate flat families, which was our main contribution to the proof of \autoref{cor:moduli_exists}. 

\end{remark}

The following theorem uses  the notion of a Lefschetz pencil of a smooth projective variety $X$ over an field $k$. By definition \cite[Sec XVII.2.2 on page 215]{SGA7_2}, this is a pencil $\phi: X' \to \bP^1_k$ of hyperplane sections of $X$ such that general fibers of $\phi$ are regular and every singular point of every fiber has quadratic singularity. The latter in dimension $1$ means nodal singularity. Note that by the virtue of being a pencil, $\phi$ fits into a commutative diagram as follows:
\begin{equation*}
\begin{tikzcd}[column sep=80pt,row sep=15pt]
X & \arrow{l}[swap]{\textrm{birational}} X' \arrow[hook]{r}{\textrm{closed embedding}} \arrow{dr}{\phi}  & X \times \bP^1_k   \arrow{d} \\
& & \bP^1_k
\end{tikzcd}
\end{equation*}

\begin{remark}
\label{rem:existence_Lefschetz}
Let $X$ be a smooth projective variety over $k$. It is shown in \cite[Sec XVII, Thm 2.5]{SGA7_2} that for any projective embedding  of $X$ given by a very ample line bundle $L$, Lefschetz pencils exist for the projective embedding  given by $L^{\otimes 2}$. Additionally, over algebraically closed fields Lefschetz pencils can be obtained as general pencils of hyperplane sections \cite[Sec XVII, Cor 3.2.1]{SGA7_2}. 
\end{remark}

\begin{theorem}
\label{thm:closure_moduli_proper}
Fix an integer $v>0$ and let 
\begin{equation*}
    d=\prod_{p \textrm{ prime, } p \leq \beta(v)} p,
        \qquad \textrm{ where } \qquad
        \beta(v)= \left\{
    \begin{array}{ll}
    393 & \textrm{if $v=1$} \\[10pt]
    213 v + 48 \qquad & \textrm{if $v \geq 2$.}
    \end{array}
    \right.
\end{equation*}
Then, the closure $\overline{\sM}_{2,v}^{\sm}$ of the locus of smooth surfaces  in $\overline{\sM}_{2,v}$ is proper over $\bZ[ 1 / d ]$. Additionally, it admits a projective coarse moduli space $\overline{\mathrm{M}}_{2,v}^{\sm}$ over $\bZ[ 1 / d ]$.
\end{theorem}

\begin{proof}
{\scshape Reduction to the existence of limits:} First, let us note that \cite[Thm 1.2]{PatakfalviProjectivityStableSrufaces} shows the projectivity of $\overline{\mathrm{M}}_{2,v}^{\sm}$  contingent upon the properness of $\overline{\sM}_{2,v}^{\sm}$. We note here that  \cite[Thm 1.2]{PatakfalviProjectivityStableSrufaces} is unfortunately not stated as precisely as needed here, but {its} (few paragraph long) proof exactly shows this, using \cite[Thm 1.1]{PatakfalviProjectivityStableSrufaces}. So, we are left to show the properness of $\overline{\sM}_{2,v}^{\sm}$.

By \autoref{cor:moduli_exists}, we know that $\sM_{2,v}$ is an Artin stack of finite type over $\bZ[1/30]$ with finite diagonal. So, we only have to show that $\sM$  is closed under limits. As the properness of $\overline{\sM}_{2,v}$ is known in characteristic zero, it is enough to show the $\sM$ is closed under limits of characteristic $p>0$. That is, we have to show that if $f^0 : X^0 \to T^0 $ is a  smooth canonically polarized surface over the spectrum {of a}  field $K$, and $R$ is a DVR of $K$ with residue field characteristic $p$ greater than $\beta(v)$, then $f^0$ extends to a family of stable surfaces $f : X \to T = \Spec R$, after possibly replacing $K$ and $R$ with finite extensions and $f^0$ with the corresponding base-change.  We may even assume that the residue field of $R$ is algebraically closed.

{\scshape The plan of showing the existence of limits:}
The construction of $f$ happens in the following steps:
\begin{itemize}
    \item We construct a birational model  $Y^0 \to X^0$ admitting a fibration $Y^0 \to \bP^1_K$ with certain singularity and boundedness properties. 
    \item The above singularity and boundedness properties are tailored exactly, so that \cite[Corollary 2]{Saito_Log_smooth_extension_of_a_family_of_curves_and_semi_stable_reduction} provides a semi-stable extension $f_Y : Y \to T$, after possibly applying a finite base-change.
    \item We run an MMP to turn the semi-stable extension into a stable family. 
\end{itemize}
{\scshape Existence of semi-stable limits:}
To state the above mentioned singularity and boundedness properties, let
 $Y^0 \to X^0$  be a projective birational morphism from another smooth surface over $T^0$, and let $\overline{f}^0 : Y^0 \to T^0$ be the composition. 
\cite[Corollary 2]{Saito_Log_smooth_extension_of_a_family_of_curves_and_semi_stable_reduction} tells us that in this situation we can find at least a semi-stable extension $f_Y : Y \to T$ of $\overline{f}^0$ if we can produce a diagram as follows
\begin{equation}
\label{eq:compactifying_smooth_surfaces:fibration_by_Lefschetz_pencil}
\xymatrix@C=50pt{
Y^0 \ar[r]_{g} \ar@/^1.5pc/[rr]^{\overline{f}^0} & \bP^1_K \ar[r]_{h} & T^0
}
\end{equation}
 such that:
 
 \begin{enumerate}
 \item $g$ is projective and surjective,
     \item \label{itm:compactifying_smooth_surfaces:genus_fiber} for the degree $d$ of the canonical sheaves of the fibers of $g$ and of $h$ we have $p \geq d+4$, which is guaranteed if  $\beta(v) \geq d+4$,
     \item \label{itm:compactifying_smooth_surfaces:number_singular_fibers} for
     \begin{equation*}
        \qquad \qquad D=\left\{
        \begin{array}{ll}
        \parbox{300pt}{the reduced discrimant divisor $D_g$} & \textrm{if $\deg D_g \geq 3$} \\[3pt] 
        \textrm{a reduced divisor of degree $3$ containing $D_g$ in its support,}  & \textrm{otherwise}
        \end{array}
        \right.
     \end{equation*}
     we have $p > \deg D$, which again is guaranteed if $\beta(v)\geq \deg D$, and
     \item \label{itm:compactifying_smooth_surfaces:etale}  
     $D$ is \'etale over $T^0 = \Spec K$, that is, all the residue fields of $D$ are separable extensions of $K$
     \item \label{itm:compactifying_smooth_surfaces:at_least_zero}   the degree  of the canonical sheaf of the fibers of $g$ is greater than 0.
 \end{enumerate}
{In the above, the discriminant divisor is the divisor over which the non-smooth points of $g$ lie.}
 We note that for this application of \cite[Corollary 2]{Saito_Log_smooth_extension_of_a_family_of_curves_and_semi_stable_reduction}, one needs to set  $X_1=\bP^1$, $D_1=D$, $U_1=X_1 \setminus D$, $X_2=Y^0$, $D_2=0$. We also note that using the notation of \cite[Corollary 2]{Saito_Log_smooth_extension_of_a_family_of_curves_and_semi_stable_reduction}, \begin{itemize}
 \item condition \autoref{itm:compactifying_smooth_surfaces:genus_fiber} guarantees that $p \geq 2g_1 +2$ and $p \geq 2 g_2 +2$, \item condition  \autoref{itm:compactifying_smooth_surfaces:number_singular_fibers} guarantees that $p > r_1$ and that $2g_1-2+r_1>0$,  
 \item condition \autoref{itm:compactifying_smooth_surfaces:etale} guarantees that $D_1$ is \'etale over $U_0$, and 
 \item condition \autoref{itm:compactifying_smooth_surfaces:at_least_zero} guaratees that $2g_2-2 \geq 0$, where we took into account that $r_2=0$.
 \end{itemize}

 We construct the $Y^0$ above and  the fibration  \autoref{eq:compactifying_smooth_surfaces:fibration_by_Lefschetz_pencil} by taking a Lefschetz pencil of $X^0_{\overline{K}}$, and descending it to a finite extension $K'$ of $K$. Note, this descent can be done, as the Lefschetz pencil  is defined by finitely many equations. We may even assume that over $K'$ the discriminant divisor $D_g$ consists of only $K'$-rational points. As throughout our process we can freely replace $K$ be a finite extension, we may assume that in fact $K=K'$. Additionally, for a Lefschetz pencil one always needs to  fix a projective embedding, and as we explained \autoref{rem:existence_Lefschetz}, then one has to post-compose this projective embedding by the second Veronese embedding. As the linear systems  $\left|4K_{X_{\overline{K}}}\right|$ if $v>1$ and $\left|5 K_{X_{\overline{K}}}\right|$ if $v=1$ are very ample by \cite[p 13, Main Thm]{Ekedahl_Canonical_models_of_surfaces_of_general_type_in_positive_characteristic},  we obtain a Lefschetz pencil for the embedding $\left|8K_{X_{\overline{K}}}\right|$ if $v>1$, and for $\left|10 K_{X_{\overline{K}}}\right|$ if $v=1$.

So, we are able to construct \autoref{eq:compactifying_smooth_surfaces:fibration_by_Lefschetz_pencil} itself, but we also need to verify conditions \autoref{itm:compactifying_smooth_surfaces:genus_fiber}, \autoref{itm:compactifying_smooth_surfaces:number_singular_fibers}, \autoref{itm:compactifying_smooth_surfaces:etale} and \autoref{itm:compactifying_smooth_surfaces:at_least_zero}.
Condition \autoref{itm:compactifying_smooth_surfaces:etale} is automatic as we choose our Lefschetz pencil so that the discriminant consists only of $K$-rational points.  Conditions \autoref{itm:compactifying_smooth_surfaces:genus_fiber} and \autoref{itm:compactifying_smooth_surfaces:at_least_zero} have to be verified only for the fibers of $g$, as the only fiber of $h$ is isomorphic to $\bP_K^1$. Additionally, when $v >1$, then the degree of the canonical sheaf of the fibers of $g$  by adjunction is:
\begin{equation*}
  0 <  (K_X + 8 K_X) \cdot 8 K_X = 72 K_X^2 = 72 v < 213 v + 44 = \beta(v)-4.
\end{equation*}
If $v=1$, then by the corresponding computation we obtain that the degree is $110 \leq 373-4=369$.
So, this concludes the verification of conditions \autoref{itm:compactifying_smooth_surfaces:genus_fiber}, \autoref{itm:compactifying_smooth_surfaces:etale} and \autoref{itm:compactifying_smooth_surfaces:at_least_zero}.

Hence, we are left to verify condition \autoref{itm:compactifying_smooth_surfaces:number_singular_fibers}. For this we use the formula that the degree of the discriminant locus of a Lefschetz pencil associated to a very ample line bundle $L$ on $X$ is:
\begin{equation}
\label{eq:closure_moduli_proper:degeneracy}
    3 L^2 + 2 L \cdot K_X + c_2(\Omega_X).
\end{equation}
We learned this formula from \cite{JasonStarrMathOverflowBoundingCritical}. As we did not find a proof in the literature, we briefly indicate the argument using the notation of \autoref{eq:compactifying_smooth_surfaces:fibration_by_Lefschetz_pencil}: one can base-change to the algebraic closure of $K$, then one uses the Giambelli-Thom-Porteous formula that the cycle given by the degeneracy locus of $\sT_Y \to g^* \sT_U$ is given by plugging into the Chern number $c_1^2 -c_2$ the virtual bundle\footnote{that is an element of the Grothendieck group $K^0$} $ \sO_U(2G) - \sT_Y$, where $G$ is a fiber of $g$. The total chern class of this virtual bundle is $ 1 + (2G+K_Y) + \big( (2G + K_Y) \cdot K_Y - c_2( \Omega_Y)\big)$. Hence, the degree of the degeneracy locus in terms of the invariants of $Y$ is $c_2(\Omega_Y) + 2 G \cdot K_Y$, from which it is straight-forward to deduce \autoref{eq:closure_moduli_proper:degeneracy}.

Plugging $8K_X$ into the $L$ of \autoref{eq:closure_moduli_proper:degeneracy} yields that the degree of the degeneracy locus is at most
 \begin{multline*}
     (3 \cdot 8^2 + 2 \cdot 8)  K_X^2 +c_2(\Omega_X)
          \explshift{50pt}{=}{$c_2(\Omega_X)= 12 \chi(\sO_X) - K_X^2$ by Grothendieck-Riemann-Roch applied to $\sO_X$} 
          208 K_X^2 + 12 \chi(\sO_X) - K_X^2 = 207 K_X^2 + 12 \chi(\sO_X) 
          \\ 
          \expl{\leq}{Noether's inequality \cite[Thm 2.1]{Liedtke_Algebraic_surfaces_of_general_type_with_small_c_1_2_in_positive_characteristic}}
          207 K_X^2 + 12 \left( \frac{1}{2} (K_X^2 + 6)+1\right) = 213 K_X^2 + 48 = 213 v + 48 = \beta(v).
 \end{multline*}
When $v=1$, we have $L=10K_X$, for which  the same computation gives $325 v +48=373$. This concludes then the verification of \autoref{itm:compactifying_smooth_surfaces:number_singular_fibers} too. 

 {\scshape Existence of stable limits:}
Therefore,  we arrive at a semi-stable extension $\overline{f} : Y \to T$ of $\overline{f}^0$. Then we run {a $K_Y$-}MMP on $Y$ over $T$ or equivalently {a $(K_Y + Y_0)$-MMP} over $T$. Note that $X^0$ is the canonical model of $Y^0$ over $T^0$. Hence, we obtain the extension $f : X \to T$ of $f^0$, where $X$ is a canonical model over $T$. At the same time $(X, X_0=f^{-1}(0))$ is also a log canonical model over $T$, where $0 \in T$ is the closed point. By adjunction we obtain that $(X_0^{\mathrm{N}},\Diff_{X_0^{\mathrm{N}}})$ is log canonical, {where $X_0^{\mathrm{N}}$ is the normalisation of $X_0$ and $K_{X_0^{\mathrm{N}}} + \Diff_{X_0^{\mathrm{N}}} = (K_X + X_0)|_{X_0^{\mathrm{N}}}$}.  This implies that $X_0$ is slc, as soon as 
we can show that $X_0$ is $S_2$. By looking at the exact sequence
\begin{equation*}
    \xymatrix{
    0 \ar[r] & \sO_X(-X_0) \ar[r] & \sO_X \ar[r] & \sO_{X_0} \ar[r] & 0
    }
\end{equation*}
we see that it is enough to show that
 $X$ is Cohen-Macaulay. This was shown in \cite[Thm 2 \& Thm 17]{Bernasconi_Kollar_Vanishing_theorems_for_threefolds_in_characteristic_p_greater_than_5} (this article depends on \cite{Kollar2020RelativeMMPWithoutQfactoriality}, which in turn uses the earlier sections of the present article, however it does not use the present section).
\end{proof}